\documentclass[12pt]{amsart}
\usepackage{fullpage,tikz,scalerel}
\usepackage[utf8]{inputenc}
\usepackage{color,amsfonts,amsmath,amssymb,amsthm,cases,wasysym}
\usepackage[normalem]{ulem}
\usepackage{svg}
\usepackage{enumitem}
\usepackage{booktabs}
\usepackage{mathtools}
\usepackage{hyperref} 

\title{Flows on Gentle Algebras}
\author[JB]{Jonah Berggren}
\address{Department of Mathematics, University of Kentucky, Lexington, KY, United States}
\email{jrberggren@uky.edu}

\date{}

\newtheorem{thm}{Theorem}[section]
\newtheorem{prop}[thm]{Proposition}
\newtheorem{lemma}[thm]{Lemma}
\newtheorem{cor}[thm]{Corollary}

\newtheorem{thmIntro}{Theorem}    
    
\theoremstyle{definition}
    
\newtheorem{defn}[thm]{Definition}
 
\newtheorem{remk}[thm]{Remark}
\newtheorem{example}[thm]{Example}

\usetikzlibrary {arrows.meta, decorations.markings}
\newcommand{\spiral}{\scalerel*{\tikz[xscale=-1]{\useasboundingbox(-.2,-.4)rectangle(.2,.3);
    \draw [line width=1.5, -{[length=2]}, line cap=round, domain=10:685 ,variable=\t, smooth, samples=30] plot (\t: 0.00047*\t);}}{(}
}

\newcommand{\spircle}{\scalerel*{\tikz[xscale=-1]{\useasboundingbox(-.2,-.4)rectangle(.2,.3);
    \draw [line width=1.5, -{[length=2]}, line cap=round, domain=10:685 ,variable=\t, smooth, samples=30] plot (\t: 415*0.00047);}}{(}
}

\newcommand{\C}{\mathcal C}

\newcommand{\fac}{\textup{Fac}\ }

\renewcommand{\P}{\mathcal P}

\renewcommand{\k}{k}
\newcommand{\e}{\varepsilon}
\newcommand{\f}{\zeta}
\newcommand{\inn}{\textup{in}}
\newcommand{\out}{\textup{out}}

\renewcommand{\hom}{\textup{Hom}}
\renewcommand{\epsilon}{\varepsilon}
\newcommand{\sttilt}{\textup{s}\tau\textup{-tilt}}

\newcommand{\F}{\mathcal F}

\newcommand{\tL}{{\tilde\Lambda}}
\renewcommand{\L}{{\Lambda}}
\newcommand{\I}{\mathcal{I}}
\newcommand{\g}{\mathfrak{g}}
\renewcommand{\for}{\textup{For}}
\newcommand{\bac}{\textup{Back}}
\newcommand{\xx}{\textbf{x}}
\newcommand{\yy}{\textbf{y}}

\newcommand{\B}{\mathcal{B}}
\newcommand{\R}{\mathcal{R}}
\newcommand{\GCD}{\textup{GCD}}

\renewcommand{\g}{\textbf{g}}
\renewcommand{\epsilon}{\varepsilon}
\newcommand{\bK}{\bar{K}}
\newcommand{\K}{K}
\renewcommand{\B}{\mathcal{B}}

\renewcommand{\int}{\circ}
\newcommand{\D}{\Delta}
\newcommand{\Sbs}{\mathcal S}
\newcommand{\Tt}{\mathcal T}

\renewcommand{\R}{R}
\renewcommand{\emph}{\textbf}

\begin{document}

\maketitle

\begin{abstract}
	The space of unit flows on a directed acyclic graph (DAG) is known to admit regular unimodular triangulations induced by framings of the DAG.
	Amply framed DAGs and their triangulated flow polytopes have recently been connected with the representation theory of certain gentle algebras. We expand on this connection by defining a flow on the fringed quiver of an arbitrary gentle algebra. We call the space of unit flows its turbulence polyhedron. We show that support $\tau$-tilting modules of a gentle algebra index a unimodular triangulation of its turbulence polyhedron. In the representation-infinite case, this triangulation is not complete and we give two different larger polyhedral dissections given by adding lower-dimensional walls to the picture.
	The turbulence polyhedron has a quotient map to what we define as the g-polyhedron lying in the ambient space of the g-vector fan, proving that gentle algebras are g-convex. Moreover, the images of our two types of walls under this quotient map provide two different interpretations for the complement of the g-vector fan of a gentle algebra.
\end{abstract}

\setcounter{tocdepth}{1}
\tableofcontents

\section{Introduction}

We split the introduction into three parts. First, we give a brief overview of the main results of the paper. We then go into more detail about our results on turbulence polyhedra and g-polyhedra.

\subsection{Overview of the main ideas}

Flow polytopes, which model the space of unit \emph{flows} on a directed acyclic graph (DAG), are a fundamental object of combinatorial optimization and have relations to many fields such as representation theory~\cite{WV}, Grothendieck polynomials~\cite{LMD}, and algebraic geometry~\cite{EM}.
Danilov, Karzanov, and Koshevoy~\cite{DKK} introduced \emph{framings} on DAGs and defined a notion of pairwise compatibility on routes (paths from source to sink). Using the correspondence between routes on a DAG and vertices of the flow polytope, the complex of \emph{cliques}, or sets of pairwise compatible routes, of a framed DAG $(G,\mathfrak{F})$ serves as a combinatorial model for a regular unimodular \emph{DKK triangulation} of the associated flow polytope. Many important classes of polytopes and their canonical triangulations appear in this way, such as associahedra, generalized permutahedra~\cite{MD}, $s$-permutahedra~\cite{Y2}, and many order polytopes~\cite{LMD}.

Introduced in~\cite{AIR}, $\tau$-tilting theory has proven to be a successful generalization of tilting and cluster-tilting theory~\cite{BMRRT,BB,HR,BMR}.
Gentle algebras are a rich source of examples in the representation theory of algebras which have received much study since their introduction in the 1980's~\cite{AH81,AS87,BR87}.
It was shown in~\cite{PPP,BDMTY} that given a gentle algebra $\L$, one may construct a \emph{fringed quiver} $\tL$ of $\L$ by adding some extra ``fringe'' vertices and arrows; then the \emph{clique complex} (or non-kissing complex) of routes (certain walks between fringe vertices) on $\tL$ describes the $\tau$-tilting theory of $\L$.
Recently, a connection between flow polytopes and the $\tau$-tilting theory of gentle algebras was established in showing that certain amply framed DAGs $(G,\mathfrak{F})$ give rise to a fringed quiver $\tL(G,\mathfrak{F})$ whose clique complex agrees with the clique complex of $(G,\mathfrak{F})$~\cite{WIWT}.

The aim of the current article is to broaden this connection between flow polytopes and gentle algebras.
To this end, we define a (nonnegative) \emph{unit flow} on an arbitrary fringed quiver $\tL$. The \emph{turbulence polyhedron} $\F_1(\tL)$ of $\tL$ is the space of unit flows on $\tL$. 

When the gentle algebra $\tL$ is \emph{paired} and \emph{representation-finite}, it may be associated to an amply framed DAG $(G,\mathfrak{F})$ as in~\cite{WIWT} and the turbulence polyhedron $\F_1(\tL)$ is equal to the flow polytope $\F_1(G)$.
Dropping the assumption of representation-finiteness corresponds on the level of framed DAGs to losing the assumption of acyclicity, and results in unbounded turbulence polyhedra. Dropping the assumption of pairedness of the gentle algebra preserves the boundedness of the turbulence polyhedron, and corresponds in some sense to losing the assumption of ``directedness'' of the framed directed acyclic graph. Hence, gentle algebras generalize amply framed DAGs by relaxing the conditions of acyclicity and directedness. The fringed quiver of Figure~\ref{KRONINTRO} is paired but not representation-finite, while the fringed quiver of Figure~\ref{GEMINTRO} is representation-finite but not paired.
{We remark that the concurrent work~\cite{UQAM} studies a different generalization of amply framed DAGs to allow cycles, obtaining triangulation results analogous to our own (see Remark~\ref{remk:UQAM}).}

A \emph{string} is a walk across the fringe quiver $\tL$ which does not cross any relations of $\tL$. We say that a \emph{route} of $\tL$ is a maximal string starting and ending at fringe vertices, and a \emph{band} is a string appearing as a repeatable cycle. We refer to routes and bands collectively as \emph{trails}.
The indicator vector $\I(p)$ of a trail $p$ is always a flow, and it is unit if $p$ is a route.
We classify certain routes and bands as \emph{elementary} in order to show that the vertices of $\F_1(\tL)$ are precisely the indicator vectors of {elementary} routes, and that the indicator vectors of {elementary} bands form a minimal generating set for the recession cone of $\F_1(\tL)$.

Next, we obtain triangulation and dissection results for the turbulence polyhedron $\F_1(\tL)$ of a fringed quiver $\tL$.
We show that the clique complex of~\cite{PPP,BDMTY} of routes on the fringed quiver $\tL$ models through indicator vectors a unimodular \emph{clique triangulation} of $\F_1(\tL)$. 
Note that we require our triangulations to cover only a dense subset of $\F_1(\tL)$, and indeed when $\tL$ is representation-infinite there will be a nonempty positive-codimension space of flows missed by the clique triangulation of $\F_1(\tL)$.
In response to this, we give two different larger dissections by adding some positive-codimension cells to the clique triangulation of $\F_1(\tL)$.

First, we extend the compatibility condition on routes defining the clique complex to a compatibility condition on all trails, including both routes and bands. A \emph{bundle} is a maximal collection of pairwise compatible trails. We show that the complex of bundles on $\tL$ induces a \emph{bundle subdivision} of $\F_1(\tL)$. All maximal cliques are also maximal bundles, meaning that the clique triangulation is contained in the bundle subdivision. Maximal bundles containing bands give additional lower-dimensional walls which provide information about the complement of the clique triangulation. While the bundle subdivision may still miss some irrational flows of $\F_1(\tL)$, we show that all rational flows of $\F_1(\tL)$ are contained in a cell of the bundle subdivision.

Finally, we provide a \emph{vortex dissection} of $\F_1(\tL)$ indexed by what we call \emph{band-stable cliques}. Like the bundle subdivision, the vortex dissection contains the clique triangulation as well as some extra lower-dimensional walls. We use the word dissection rather than subdivision because the intersection of two cells may fail to be a common face. The vortex dissection has the benefit of being complete, meaning that it covers all points of $\F_1(\tL)$.

After proving our results on turbulence polyhedra, we will treat the g-vector fan $\g_{\geq0}(\L)$ of the underlying gentle algebra $\L$ of $\tL$.
The g-vector fan of a finite-dimensional algebra $\L$ was introduced in~\cite{AIR} and is modelled by the support $\tau$-tilting complex of certain extended modules over the gentle algebra $\L$.
Hence, when $\L$ is a gentle algebra, its g-vector fan $\g_{\geq0}(\L)$ is modelled by the (reduced) clique complex of routes on $\tL$~\cite{PPP,BDMTY}. We provide a normalized-volume-preserving quotient map $\phi$ from $\F_{1}(\tL)$ to the \emph{g-polyhedron} $\g_1(\L)$ of $\L$, which lies within the ambient space $\mathbb R^E$ of the g-vector fan $\g_{\geq0}(\L)$. In particular, this shows a version of g-convexity for gentle algebras in the sense of~\cite{AHIKM}.

We then transport our triangulation and dissection results for $\F_1(\tL)$ through the map $\phi$. The image of the clique triangulation of $\F_1(\tL)$ gives a \emph{clique triangulation} for $\g_1(\L)$, which is merely the restriction of the usual representation-theoretic g-vector fan to the g-polyhedron $\g_1(\L)$ (in other words, the cone over the clique triangulation of $\g_1(\L)$ recovers the g-vector fan).
Similarly, we obtain the \emph{bundle subdivision} and \emph{vortex dissection} of $\g_1(\L)$ by adding to the clique triangulation some lower-dimensional walls indexed by maximal bundles and band-stable cliques, respectively.
These walls provide two different interpretations for the complement of the g-vector fan and it would be interesting to connect them with other structures in the representation theory of algebras.

\subsection{Results on turbulence polyhedra}

We now go into more detail on our results about turbulence polyhedra of fringed quivers.

Recall that the vertices of a unit flow polytope are precisely the indicator vectors of its routes. We obtain a similar presentation for turbulence polyhedra.
Generalizing routes of a framed DAG are \emph{routes} of a gentle algebra (maximal strings starting and ending at fringe vertices) as well as \emph{bands} of a gentle algebra (strings appearing as repeatable cycles). We refer to routes and bands collectively as \emph{trails}.
The indicator vector $\I(p)$ of a trail $p$ is always a flow, and it is unit if $p$ is a route.
We define some routes and bands to be \emph{elementary} (see Definitions~\ref{defn:elband} and~\ref{defn:elroute} for the technical definition or Figure~\ref{ELEMENTARYINTRO} for the pictorial summary) to obtain the following presentation of $\F_1(\tL)$:
\begin{thmIntro}[{Theorem~\ref{thm:turb-vert-dir}}]\label{thmZ}
	The map $p\mapsto\I(p)$ bijects elementary routes to vertices of $\F_1(\tL)$. The map $B\mapsto\I(B)$ bijects elementary bands to a minimal generating set for the recession cone of $\F_1(\tL)$.
\end{thmIntro}
\begin{figure}
\centering
	\begin{minipage}[b]{.145\textwidth}
  \centering
	\def\svgscale{.21}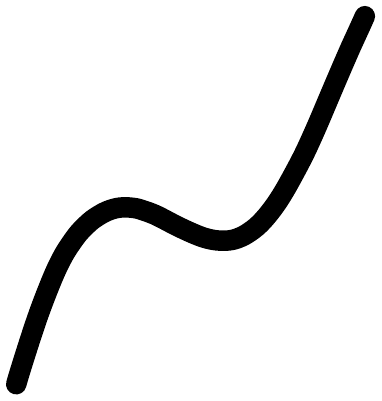

	simple
\end{minipage}
	\begin{minipage}[b]{.145\textwidth}
  \centering
	\def\svgscale{.21}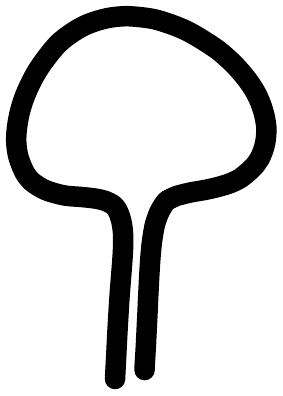

	lollipop
\end{minipage}
	\begin{minipage}[b]{.145\textwidth}
  \centering
	\def\svgscale{.21}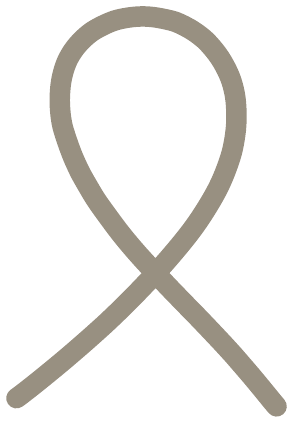

	nonelementary
\end{minipage}
	\begin{minipage}[b]{.145\textwidth}
  \centering
	\def\svgscale{.21}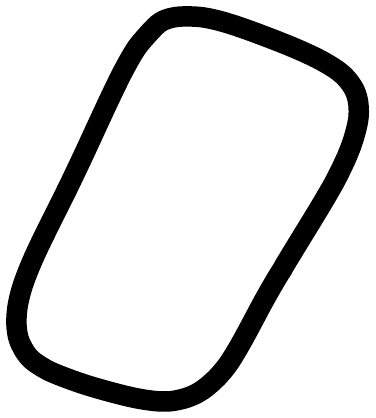

	simple
\end{minipage}
	\begin{minipage}[b]{.145\textwidth}
  \centering
	\def\svgscale{.21}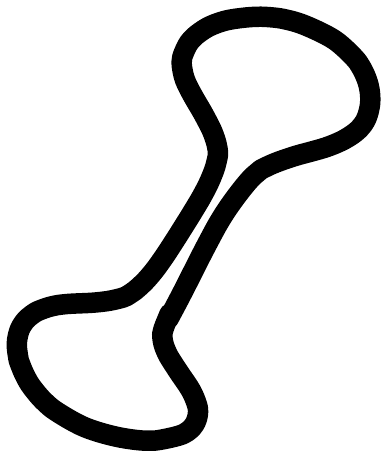

	barbell
\end{minipage}
	\begin{minipage}[b]{.145\textwidth}
  \centering
	\def\svgscale{.21}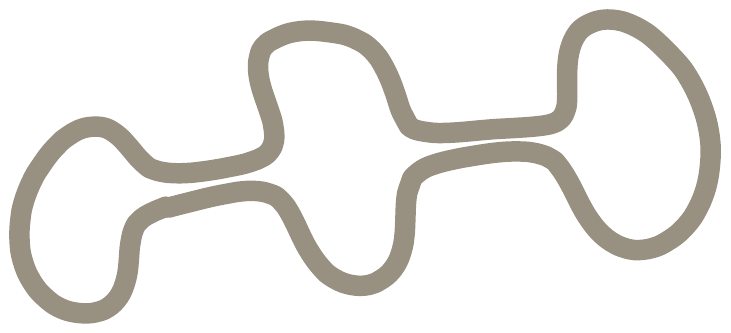

	nonelementary
\end{minipage}
	\caption{Examples of elementary (black) and nonelementary (gray) routes and bands.}
	\label{ELEMENTARYINTRO}
\end{figure}
Note, then, that a turbulence polyhedron is unbounded if and only if its fringed quiver $\tL$ has no bands. This is equivalent to representation-finiteness and hence $\tau$-tilting finiteness of $\L$ by~\cite{Plamondon}.

\begin{example}\label{ex:KRONINTRO0}
	See the Kronecker fringed quiver of Figure~\ref{KRONINTRO}. The three-dimensional turbulence polyhedron is drawn on the top-right; closed points denote vertices and open points denote lattice points which are not vertices.
	There are four elementary routes: the straight routes $\{e_1e_2e_3,f_1f_2f_3\}$ (whose indicator vectors are the top and bottom vertices of the turbulence polyhedron), the orange route $e_1f_1^{-1}$ (giving the orange vertex), and the blue route $e_3^{-1}f_3$ (giving the blue vertex). All other routes are nonelementary. There is only one elementary band $e_2f_2^{-1}$, whose indicator vector generates the ray moving to the right of the page.
\end{example}

\begin{figure}
	\centering
	\def\svgscale{.21}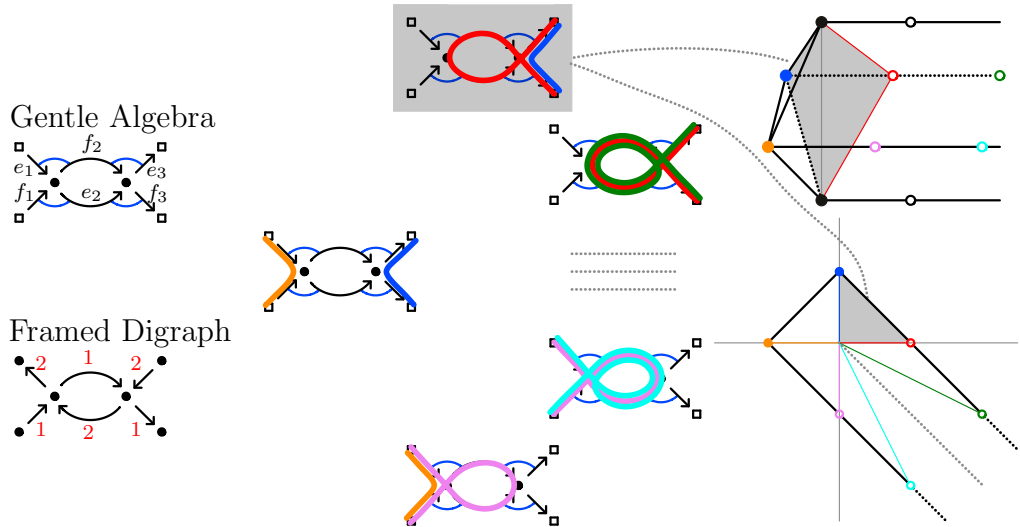
	\caption{The Kronecker fringed quiver with five maximal cliques (drawn without straight routes). On the right is its three-dimensional turbulence polyhedron and two-dimensional g-polyhedron.}
	\label{KRONINTRO}
\end{figure}

We now phrase three generalizations of the theory of DKK triangulations to the setting of gentle algebras and their turbulence polyhedra.
Routes come equipped with a notion of pairwise compatibility corresponding to $\tau$-rigidity of associated string modules over $\L$; a \emph{clique} is a collection of pairwise compatible routes.
If $\K$ is a clique of $\tL$, then a unit \emph{($\K$-)clique combination} is a convex combination of indicator vectors of routes of $\K$. 
It is \emph{positive} if routes of $\K$ appear in the combination with positive coefficient.
If $\K$ is a clique, then the \emph{clique simplex $\D_1(\K)$} is the space of unit $\K$-clique combinations.
\begin{thmIntro}[{Theorem~\ref{thm:comb}, Lemma~\ref{prop:b}, Corollary~\ref{cor:clique-subd}}]\label{ITHM:CLIQUE}
	Any flow has at most one representation as a clique combination.
	A dense subset of flows $F\in\F_{\geq0}(\tL)$ may be obtained as a clique combination. If $F$ is an integer flow, then its clique combination has integral coefficients.
	Consequently, the \emph{clique triangulation}
	\[\Tt(\F_1(\tL)):=\{\D_1(\K)\ :\ \K\text{ is a maximal clique of }\tL\}\]
	is a unimodular triangulation of $\F_1(\tL)$. 
\end{thmIntro}

We note that we require our triangulations and subdivisions to cover only a dense subset of a polyhedron. When $\tL$ is representation-finite, the turbulence polyhedron $\F_1(\tL)$ is bounded and the clique triangulation is a unimodular triangulation covering all points of $\F_1(\tL)$ (see Figure~\ref{GEMINTRO}). When $\tL$ is paired and representation-finite, the clique triangulation of $\F_1(\tL)$ is the same as the DKK triangulation of the flow polytope of an associated amply framed DAG as in~\cite{WIWT}. Note that DKK triangulations are known to be regular, though we do not show regularity in this paper.

The results of this introduction have analogs for the cone of nonnegative flows $\F_{\geq0}(\tL)$. For example, we obtain a clique triangulation of $\F_{\geq0}(\tL)$ into clique cones, which are the cones generated by the clique simplices. For brevity, we phrase results only for $\F_1(\tL)$.

Outside of clique simplices, we will give two larger polyhedral dissection of $\F_1(\tL)$ by adding extra lower-dimensional cells to the maximal clique simplices.
The first comes from extending the notion of compatibility of routes to a notion of compatibility on all trails, including bands.
A \emph{bundle} is a (necessarily finite) collection of pairwise compatible trails. All cliques are then bundles; in fact, all maximal cliques are maximal bundles.
If $\bK$ is a bundle of $\tL$ and $\K$ is the clique given by routes of $\bK$, then we say that a \emph{nonnegative ($\bK$-)bundle combination} is a $\K$-clique combination plus a nonnegative sum of indicator vectors of bands of $\bK$.
The unit \emph{bundle space $\D_1(\K)$} is the space of unit $\bK$-bundle combinations.
If $\bK$ is a clique, then the bundle space of $\bK$ is equal to the clique simplex of $\K$ and is full-dimensional in $\F_1(\tL)$. If $\bK$ contains at least one band, $\D_1(\K)$ is not full-dimensional in $\F_1(\tL)$ and we call it a \emph{bundle wall}.

\begin{thmIntro}[{Theorem~\ref{thm:comb}, Lemma~\ref{prop:b}, Corollary~\ref{cor:bundle-subd}}]\label{ITHM:BRIQUE}
	Any flow has at most one representation as a bundle combination.
	A dense subset of flows $F\in\F_{\geq0}(\tL)$, including all rational flows, may be obtained as a bundle combination.
	Consequently, the \emph{bundle subdivision}
	\[\Sbs^{\spircle}(\F_1(\tL)):=\{\D_1(\bK)\ :\ \bK\text{ is a maximal bundle of }\tL\}\]
	formed by adding the bundle walls to the clique triangulation is a polyhedral subdivision of $\F_1(\tL)$ which covers every rational point of $\tL$.
\end{thmIntro}

{The technical core of this paper is a constructive \emph{flow algorithm} to take an arbitrary flow $F\in\F_{\geq0}(\tL)$ and express it as a bundle combination if such a representation exists. We prove Theorem~\ref{ITHM:BRIQUE} by studying this algorithm.}

In some cases, there may be (irrational) points of the turbulence polyhedron which cannot be realized as a bundle combination and hence the bundle subdivision is not complete. We now provide an alternative dissection called the \emph{vortex dissection} (a dissection is slightly weaker than a subdivision in that the intersection of two cells may fail to be be a common face). Like the bundle subdivision, this will amount to adding extra lower-dimensional walls to the maximal clique simplices, but the vortex dissection will always be complete.

A \emph{vortex} is a flow which is zero on all fringe arrows. If $\K$ is a clique of $\tL$, then a vortex is \emph{$\K$-compatible} if it is a positive linear combination of indicator vectors of $\K$-compatible bands. A \emph{($\K$-)vortex combination} is a sum of the form $W+\sum_{p\in\K}a_p\I(p)$, where $W$ is a $\K$-compatible vortex and $a_p\geq0$ for all $p\in\K$. The vortex combination is \emph{unit} if $\sum_{p\in\K}a_p=1$, and it is \emph{positive} if $a_p>0$ for all $p\in\K$. 
Note that any clique combination is also a vortex combination.

If $\K$ is a clique, then the \emph{(unit) vortex space} $\D^{\spiral}_1(\K)$ is the space of $\K$-vortex combinations. This is a polyhedron whose vertices are the indicator vectors of routes of $\K$ and whose recession cone is generated by indicator vectors of $\K$-compatible bands.
The maximal vortex spaces are indexed by \emph{band-stable cliques} (Definition~\ref{defn:bandstable}).
If $\K$ is a maximal clique, then $\K$ is band-stable and the vortex space $\D^{\spiral}_1(\K)$ is equal to the clique simplex $\D_1(\K)$. If $\K$ is band-stable but not maximal, then $\D^{\spiral}_1(\K)$ is unbounded and lower-dimensional and we call it the \emph{vortex wall} of $\K$.
Modeling Theorem~\ref{ITHM:BRIQUE}, we show that adding the vortex walls to the clique triangulation gives a complete polyhedral dissection of $\F_1(\tL)$:
\begin{thmIntro}[{Theorem~\ref{thm:vort-decomp}, Theorem~\ref{thmCbody}}]\label{ITHM:VORTEX}
	Every flow $F\in\F_{\geq0}(\tL)$ has a unique description as a positive vortex combination.
	Consequently, the \emph{vortex dissection} \[\Sbs^{\spiral}(\F_1(\tL)):=\{\D_1^{\spiral}(\K)\ :\ \K\text{ is a band-stable clique of }\tL\}\] is a complete polyhedral dissection of the turbulence polyhedron $\F_1(\tL)$.
\end{thmIntro}
Since every maximal clique is band-stable, the vortex dissection contains every clique simplex and hence contains the clique triangulation. If the bundle subdivision is complete (as in Figure~\ref{KRONINTRO}), then the bundle subdivision and vortex dissection agree.

\begin{example}\label{ex:KRONINTRO}
	Consider the Kronecker fringed quiver of Figure~\ref{KRONINTRO}.
	The full list of maximal bundles is
	\begin{align*}
		&\{e_1e_2e_3,\ f_1f_2f_3,\ e_2f_2^{-1}\}, \ \ \
		\{e_1e_2e_3,\ f_1f_2f_3,\ e_1f_1^{-1},\ e_3^{-1}f_3\}, \\
		&\{e_1e_2e_3,\ f_1f_2f_3,\ e_1(e_2f_2^{-1})^mf_1^{-1},\ e_1(e_2f_2^{-1})^{m+1}f_1^{-1}\}\text{ (for any $m\geq0$)},\text{ and }\\ 
		&\{e_1e_2e_3,\ f_1f_2f_3,\ e_3^{-1}(e_2^{-1}f_2)^mf_3,\ e_3^{-1}(e_2^{-1}f_2)^{m+1}f_3\}\text{ (for any $m\geq0$)}.
	\end{align*}
	Note that all maximal bundles contain the straight routes $\{e_1e_2e_3,f_1f_2f_3\}$ and that there is only one maximal bundle containing a band (namely, $e_2f_2^{-1}$).
	In the middle of the figure are depicted five maximal cliques, with the caveat that the straight routes $\{e_1e_2e_3,f_1f_2f_3\}$ are not drawn for readability.
	Five maximal cliques are drawn in the figure.
	Each maximal clique gives a full-dimensional clique simplex consisting of the two black vertices (which correspond to the straight routes) and two of the colored points, and these clique simplices cover a dense subset of the turbulence polyhedron. The maximal bundle $\{e_1e_2e_3,f_1f_2f_3,e_2f_2^{-1}\}$ gives a bundle space whose vertices are the black vertices and which continues infinitely to the right to make a two-dimensional strip. Note that every point of the turbulence polyhedron is contained in a maximal bundle space.
	The maximal clique $\{e_1e_2e_3,f_2f_2f_3,e_3^{-1}f_3,e_3^{-1}e_2^{-1}f_2f_3\}$ is highlighted and connected to the corresponding clique simplex of $\F_1(\tL)$.

	The only band-stable clique of $\tL$ which is not maximal is the clique $\{e_1e_2e_3,f_1f_2f_3\}$.
	In fact, the vortex wall $\D_1^{\spiral}(\{e_1e_2e_3,f_1f_2f_3\})$ of this clique is equal to the bundle wall $\D_1(\{e_1e_2e_3,f_1f_2f_3,e_2f_2^{-1}\})$ of the unique maximal bundle with a band. This means that the bundle subdivision and vortex dissection of $\tL$ are the same for this fringed quiver, though this is not the case in general.
\end{example}

\subsection{Results on g-polyhedra}

We connect our results on turbulence polyhedra to the g-vector fan $\g_{\geq0}(\L)$ by providing an explicit normalized-volume-preserving quotient map $\phi$ from $\F_{1}(\tL)$ to the ambient space $\mathbb R^E$ of the g-vector fan $\g_{\geq0}(\L)$. The map $\phi$ sends the indicator vector $\I(p)$ of any trail to its \emph{g-vector $\g(p)$}. Up to unimodular equivalence, this map on $\F_1(\tL)$ is a quotient by the affine span of the indicator vectors of straight routes of $\tL$, which are the only routes not corresponding to extended modules over $\L$. We call its image the \emph{g-polyhedron} $\g_1(\L)$. When $\L$ is $\tau$-tilting finite, this coincides with the g-polytope of $\L$ as defined in~\cite{AHIKM}. In particular, this shows that g-polytopes of $\tau$-tilting finite gentle algebras are convex polytopes, hence that $\tau$-tilting finite gentle algebras are g-convex in the sense of~\cite{AHIKM}. It is noted in~\cite[Theorem 5.10]{AHIKM} that if $\L$ is $\tau$-tilting infinite, then its g-polytope is never convex; in this sense, g-polytopes are studied only for $\tau$-tilting finite algebras. On the other hand, if a gentle algebra $\L$ is $\tau$-tilting infinite then our g-polyhedron is an unbounded analog of the g-polytope which coincides with the completion of the g-polytope. Hence, we argue that even $\tau$-tilting infinite gentle algebras ought to be considered as g-convex.

Furthermore, we may translate our results about presentations and dissections of turbulence polyhedra through the map $\phi$ to obtain results about g-polyhedra and even g-vector fans of gentle algebras.
Since $\phi$ precisely collapses the affine span of the indicator vectors of straight routes, we show that the map $p\mapsto\g(p)$ is a bijection from elementary non-straight routes to vertices of $\g_1(\L)$, and that the map $B\mapsto\g(B)$ on elementary bands is a bijection onto a minimal generating set for the recession cone of $\g_1(\L)$.
More generally, we obtain descriptions of all faces of $\g_1(\L)$ and their defining inequalities (Theorem~\ref{thm:faces-crooked} and Theorem~\ref{thm:facets-crooked}).

We may also take our triangulation and dissection results through the map $\phi$.
By slightly modifying the relevant definitions from the flow setting, we may define (unit or nonnegative) g-clique combinations, g-clique spaces, g-bundle combinations, g-bundle spaces, g-vortex combinations, and g-vortex spaces.
Mirroring our triangulatation and dissection results on turbulence polyhedra, we show that 
		the g-clique triangulation $\Tt(\g_1(\L))$ is a unimodular triangulation of $\g_1(\L)$,
the g-bundle subdivision $\Sbs^{\spircle}(\g_1(\L))$ is a subdivision of $\g_1(\L)$ which covers every rational point, and that
the g-vortex dissection $\Sbs^{\spiral}(\g_1(\L))$ is a complete dissection of $\g_1(\L)$.
Figure~\ref{KRONINTRO} also includes the g-polyhedron with a g-simplex highlighted in the bottom right.

This completes our dissection and presentation results about the g-polyhedron $\g_1(\L)$.
The cone over the g-clique triangulation recovers the usual g-vector fan, and the cones over the g-bundle and g-vortex dissections add some lower-dimensional walls to this fan to obtain larger dissections.
Note, then, that the g-bundle walls and g-vortex walls provide two different interpretations for the complement of the g-vector fan.
It would be interesting to study these additions to the g-vector fan in connection with other similar structures in the representation theory of algebras: for example, torsion classes~\cite{DIRRT} (in particular, the semistable and morphism torsion classes of~\cite{AI24}), the wall-and-chamber picture of $\L$~\cite{ITW,BSH,BHIT,BRIDGE,BST}, the purely non-rigid region of $\L$~\cite{Asa22}, or scattering diagrams~\cite{GHKK,READING}.

In Section~\ref{sec:example} we study the fringed quiver of Figure~\ref{fig:doubkron-g} as an example whose bundle and vortex dissections differ.
We are able to use new techniques derived from the flow algorithm to understand in detail its maximal bundles containing bands and the complement of its g-vector fan.

{We finally remark that a broad strategy of the paper is to prove results about flows on fringed quivers before using the quotient map $\phi$ to turn them into results about g-vectors. This reflects the idea that working with flows and indicator vectors is in many ways more intuitive than working with g-vectors. For example, the facet-defining hyperplanes of $\F_1(\tL)$ always amount to zeroing out flow through some arrow, while facet-defining hyperplanes of $\g_1(\L)$ are more complicated (Theorem~\ref{thm:facets-crooked}). Moreover, it is not even clear what the analog of the flow algorithm on the g-vector picture ought to be. We hope that the strategy of proving results on flows and transporting them through the map $\phi$ will continue to be a useful strategy in the study of g-vector fans of gentle algebras.}

\subsection*{Structure of the Paper and Acknowledgments}

The structure of this paper is as follows. In Section~\ref{sec:bac}, we give background on gentle algebras and $\tau$-tilting theory. In Section~\ref{sec:DAG}, we give  background on framed DAGs. In Section~\ref{sec:tp}, we define flows and turbulence polyhedra on gentle algebras. We give some basic results about dimension, define bundles, and connect flows on gentle algebras to flows on amply framed directed graphs.
In Section~\ref{sec:triang}, we define the flow algorithm to obtain bundle subdivisions, proving Theorem~\ref{ITHM:BRIQUE}. In Section~\ref{sec:vert-unb}, we use this theory to prove our presentation result Theorem~\ref{thmZ} of turbulence polyhedra.
In Section~\ref{sec:g}, we begin our work on g-vector fans and g-polyhedra. We define the transformation $\phi$ and define its image to be the g-polyhedron $\g_1(\L)$. By pulling back density of the g-vector fan, we are able to prove that the clique simplices are dense in $\F_1(\tL)$ and hence complete the proof of Theorem~\ref{ITHM:CLIQUE}. In Section~\ref{sec:CBS}, we transport our presentation and bundle subdivision theory from turbulence polyhedra to g-polyhedra. In this way, we prove our presentation
and g-bundle subdivision
results about g-polyhedra and g-vector fans. We then describe the faces, vertices, and unbounded directions of g-polyhedra.
In Section~\ref{sec:vortex}, we obtain vortex dissections of turbulence polyhedra and then use the map $\phi$ to get vortex dissections of g-polyhedra and the g-vector fan.
In Section~\ref{sec:example} we will use the flow algorithm to study a fringed quiver whose bundle and vortex dissections differ.
Finally, Section~\ref{sec:appendix} is an appendix where we give the (new) specialization of the flow algorithm to amply framed DAGs and a skeleton for some of our methods in the language of framed directed graphs.

The author thanks Khrystyna Serhiyenko, Benjamin Braun, Hugh Thomas, and Alejandro H. Morales for helpful discussions and comments on a preliminary version.
The author was supported by the NSF grant DMS-2054255.

\section{Background on Gentle Algebras and $\tau$-Tilting Theory}
\label{sec:bac}

We give background on $\tau$-tilting theory, particularly as it applies to gentle algebras.
Throughout, let $\k$ be an algebraically closed field. 
A {quiver} is a directed graph.

\subsection{$\tau$-tilting theory}\label{sec:ttt}

We collect some definitions and results on $\tau$-tilting theory, introduced in~\cite{AIR}. We first treat the $\tau$-tilting theory of an arbitrary finite-dimensional algebra $\L$; in the following subsections we will specialize to the particularly combinatorial setting when $\Lambda$ is gentle and its $\tau$-tilting theory is described by a simplicial complex of routes on its fringed quiver.
Let $\Lambda$ be a finite dimensional algebra.
If $P$ is a projective module over $\Lambda$, the symbol $P[1]$ denotes the associated formal \emph{shifted projective}. We say that the set of \emph{extended modules} is the set of modules and shifted projectives. We allow ourselves to form direct sums of shifted projectives with other shifted projectives and with ordinary modules.

A module $M$ over $\Lambda$ is \emph{$\tau$-rigid} if $\hom(M,\tau M)=0$, where $\tau$ is the Auslander-Rieten translate. By convention, the shifted projectives are $\tau$-rigid.

\begin{defn}
	Following~\cite{AIR}, we say that $M\oplus P[1]$ is
		\emph{$\tau$-rigid} (over $\Lambda$) if
		\begin{enumerate}
			\item $M$ is a module over $\Lambda$ and $P$ is a projective module over ${\Lambda}$,
			\item $M$ is $\tau$-rigid, and
			\item $\hom_\Lambda(P,M)=0$.
		\end{enumerate}
		If, in addition, the number of pairwise non-isomorphic indecomposable summands of $M\oplus P$ is the number of vertices of $Q$, then we say that $M\oplus P[1]$ is a \emph{support $\tau$-tilting module}.
		We write $\sttilt\Lambda$ for the set of support $\tau$-tilting modules of $\Lambda$.
	We say that two extended modules $M$ and $N$ are \emph{$\tau$-rigid} if $M\oplus N$ is $\tau$-rigid.
	A module is \textit{basic} if its direct summands are pairwise nonisomorphic.
\end{defn}
\begin{thm}[\cite{AIR}]
	Let $M=\oplus_{i\in[n]}M_i$ be a basic support $\tau$-tilting module. Then for any $i\in[n]$, there exists a unique indecomposable extended module $M'_i$ not isomorphic to $M_i$ such that
	\[\mu_i(M):=M_i'\oplus\bigoplus_{j\neq i}M_j\]
	is a support $\tau$-tilting module, called the \emph{mutation of $M$ at $M_i$}.
	Moreover, either $\fac M\subsetneq\fac\mu_i(M)$ or $\fac M\supsetneq\fac\mu_i(M)$.
\end{thm}

\begin{defn}
	Let $\L$ be a finite-dimensional algebra.
	\begin{enumerate}
		\item The \emph{support $\tau$-tilting complex} $\C(\Lambda)$ is the simplicial complex whose vertices are the isomorphism classes of indecomposable extended modules of $\Lambda$ and whose faces are the collections of extended modules $\{M_1,\dots,M_k\}$ such that $M_1\oplus\dots\oplus M_k$ is $\tau$-rigid.
		\item The \emph{support $\tau$-tilting poset} $\P_\Lambda$ is the poset whose ground set is the set of support $\tau$-tilting objects of $\Lambda$, and $M\geq N$ if $\fac M\supsetneq\fac N$. The top element of $\P_\Lambda$ is $\Lambda$ (considered as a module over itself) and the bottom element is $\Lambda[1]$.
	\end{enumerate}
\end{defn}

We say that $\Lambda$ is \emph{$\tau$-tilting finite} if the set of isomorphism classes of support $\tau$-tilting modules is finite. In this case, $\P_\Lambda$ is a lattice.
We now define the g-vector of a module over $\L$.

\begin{defn}\label{defn:g-vector}
	Let $M$ be a module over $\L$ with minimal projective presentation $P_1^M\to P_0^M\to M\to0$.
	The \emph{g-vector} of $M$ is the element $\g(M)=[P_0^M]-[P_1^M]$ of the Grothendieck group $K_0(\text{proj}\ \L)$ of the additive category $\text{proj}\ \L$ of projective modules of $\L$. If $P$ is a projective module, then the g-vector of the shifted projective $P[1]$ is $\g(P[1]):=-[P]$.
\end{defn}

Note that if $\L=\k Q/I$ is an admissible path algebra with relations then $\text{proj}\ \L$ is isomorphic to $\mathbb Z^{Q_0}$. We may then think of the g-vector $\g(M)$ as a tuple in $\mathbb R^{Q_0}$.

\begin{defn}\label{defn:g-stuff}
If $M=\oplus_{i=1}^nM_i$ is a support $\tau$-tilting $\Lambda$-module, define the \emph{g-simplex} 
	\[\g_{1}(M):=\textup{conv}(\hat0,\g(M_1),\dots,\g(M_n))\] as the convex hull of the g-vectors of indecomposable summands of $M$. 
	The \emph{g-cone} $\g_{\geq0}(M)$ of $M$ is the cone of nonnegative combinations of the g-vectors $\{\g(M_1),\dots,\g(M_n)\}$.  The \emph{g-vector fan} $\g_{\geq0}(\L)$ of $\L$ is the fan whose cells are the cones $\g_{\geq0}(M)$, as $M$ ranges over all support $\tau$-tilting $\Lambda$-modules.
\end{defn}

The g-vectors $\g(M_1),\dots,\g(M_n)$ form a basis for the Grothendieck group $K_0(\text{proj}\L)$~\cite[Theorem 5.1]{AIR}.
This gives the following result:
\begin{thm}[{\cite[Theorem 5.1]{AIR}}]\label{thm:uni-cite}
	If $M=\oplus_{i=1}^nM_i$ is a support $\tau$-tilting $\Lambda$-module, then 
	the g-simplex $\g_{1}(M)$
	is an $n$-dimensional unimodular simplex in $\mathbb R^n$.
\end{thm}

\subsection{Gentle algebras}

We recall the definition of a gentle algebra from~\cite{AS87}.
For an arrow $\alpha$ of a quiver $Q$, we write $h(\alpha)$ for the head (or endpoint) of $\alpha$ and $t(\alpha)$ for the tail (or start point) of $\alpha$. We consider the formal inverse $\alpha^{-1}$ to have head $t(\alpha)$ and tail $h(\alpha)$. A \emph{signed arrow} is an expression of the form $\alpha^{\e}$, where $\alpha$ is an arrow and $\e\in\{1,-1\}$.
We consider the head of $\alpha^{-1}$ to be $h(\alpha^{-1}):=t(\alpha)$ and the tail of $\alpha^{-1}$ to be $t(\alpha^{-1}):=h(\alpha)$.

\begin{defn}
	A finite-dimensional $\k$-algebra $\Lambda$ is \emph{gentle} if it is (Morita equivalent to) a bound path algebra $\k Q/I$, where $Q$ is a quiver and $I$ is an admissible ideal in $\k Q$ such that
	\begin{enumerate}
		\item for each vertex $v\in Q_0$, there are at most two arrows starting at $v$ and at most two arrows ending at $v$,
		\item for each arrow $\alpha\in Q_1$, there is at most one arrow $\beta$ with $h(\alpha)=t(\beta)$ such that $\alpha\beta\not\in I$, and there is at most one arrow $\gamma$ with $t(\alpha)=h(\gamma)$ such that $\gamma\alpha\not\in I$,
		\item for each arrow $\alpha\in Q_1$, there is at most one arrow $\beta$ with $h(\alpha)=t(\beta)$ such that $\alpha\beta\in I$, and there is at most one arrow $\gamma$ with $t(\alpha)=h(\gamma)$ such that $\gamma\alpha\in I$, and
		\item the ideal $I$ is generated by length-two monomial relations.
	\end{enumerate}
	We say that $(Q,I)$ is a \textit{gentle bound quiver}.
\end{defn}

	See the left of Figure~\ref{fig:bloss} for an example of a gentle algebra $\L$. The ideal $I$ is generated by the relations $\beta\gamma$ and $\gamma\beta$ of $\L$. The relation $\beta\gamma$ (resp. $\gamma\beta$) is marked visually by a blue arc connecting $\beta$ to $\gamma$ (resp. $\gamma$ to $\beta$) around the vertex $h(\beta)=t(\gamma)$ (resp. $h(\gamma)=t(\beta)$). In the following, we will express the relations of a gentle algebra $\L=\k Q/I$ using this visual convention without writing them in text.

Henceforth, let $\Lambda=\k Q/I$ be a gentle algebra.
We collect some facts about gentle algebras, following~\cite{PPP}.

\begin{defn}
		A \emph{string} of $\Lambda$ is a word of the form $p=\alpha_1^{\e_1}\alpha_2^{\e_2}\dots\alpha_m^{\e_m}$, where
			\begin{enumerate}
				\item $\alpha_i\in Q_1$ and $\e_i\in\{-1,1\}$ for all $i\in[m]$,
				\item $h(\alpha_i^{\e_i})=t(\alpha_{i+1}^{\e_{i+1}})$ for all $i\in[m-1]$,
				\item for any $\beta\gamma\in I$, neither $\beta\gamma$ nor $\gamma^{-1}\beta^{-1}$ appears in $p$, and
				\item no factor $\alpha\alpha^{-1}$ or $\alpha^{-1}\alpha$ appears in $p$, for $\alpha\in Q_1$.
			\end{enumerate}
			The integer $m$ is called the \emph{length} of $p$. We write $t(p):=t(\alpha_1^{\e_1})$ and $h(p):=h(\alpha_m^{\e_m})$. For each vertex $v\in Q_0$, there is a \emph{lazy string} of length zero, denoted by $e_v$, starting and ending at $v$. A \emph{substring} of $p$ is a word of the form $\alpha_i^{\e_i}\dots\alpha_j^{\e_j}$, for some $1\leq i\leq j\leq m$. We also consider a lazy string $e_v$, where $v$ is some vertex through which $p$ passes, to be a substring. We may sometimes write lazy strings into our strings for emphasis: for example, if $\alpha^\e\beta^\f$ is a string, then $\alpha^\e e_{h(\alpha^\e)}\beta^\f$ refers to the same string while emphasizing the lazy substring $e_{h(\alpha^\e)}$ in the middle.
\end{defn}

We consider the two inverse strings $p$ and $p^{-1}$ to be \emph{equivalent}, as they will give rise to the same string module over $\Lambda$.
We use the symbol $p^{\pm1}$ to refer to the (equivalent) strings $p$ and $p^{-1}$ simultaneously as an abuse of notation. For example, we may write ``$q$ is a substring of $p^{\pm1}$'' to mean, ``$q$ is a substring of $p$ or $p^{-1}$.''

	The strings of the gentle algebra on the left of Figure~\ref{fig:bloss}, up to equivalence, are 
	\[e_1\ ,e_2\ ,e_3,\ \alpha,\ \beta,\ \gamma,\ \alpha\beta,\text{ and }\alpha\gamma^{-1}\] (labelling the vertices from 1 to 3, left to right).
	Note, for example, that the string $\alpha$ is equivalent to $\alpha^{-1}$ and $\gamma\alpha^{-1}$ is equivalent to $\alpha\gamma^{-1}$.

\begin{defn}
	A substring $p=\alpha_i^{\e_i}\dots\alpha_j^{\e_j}$ of a string $q=\alpha_1^{\e_1}\dots\alpha_m^{\e_m}$ is said to be:
	\begin{enumerate}
		\item a \emph{top substring of $q$} if $q$ either ends or has an outgoing arrow at each endpoint of $p$. In other words, if $i=1$ or $\e_{i-1}=-1$, and $j=m$ or $\e_{j+1}=1$.
		\item a \emph{bottom substring of $q$} if $q$ either ends or has an incoming arrow at each endpoint of $p$. In other words, if $i=1$ or $\e_{i-1}=1$, and $j=m$ or $\e_{j+1}=-1$.
	\end{enumerate}
\end{defn}

In the gentle algebra on the left of Figure~\ref{fig:bloss}, the lazy string $e_2$ at the middle vertex is a bottom substring of $\alpha\gamma^{-1}$ and a top substring of $\beta$.

A \emph{band} of $\Lambda$ is a string $B=\alpha_1^{\e_1}\dots\alpha_m^{\e_m}$ of length at least one such that $h(B)=t(B)$, the walk $\alpha_m^{\e_m}\alpha_1^{\e_1}$ is a string, and $B$ is not itself a power of a strictly smaller string.
The gentle algebra of Figure~\ref{fig:bloss} has no bands.
We consider a \emph{substring} of $B$ to be a substring of any power of $\alpha_1^{\e_1}\dots\alpha_m^{\e_m}$.
Two bands $B$ and $B'$ are \emph{equivalent} if the underlying string of $B'$ is a substring of $B^2$ or $(B^{-1})^2$ (equivalently, if $B^{\pm1}$ and $(B')^{\pm1}$ are cyclically equivalent).

Strings and bands of a gentle algebra $\L$ are important because they describe the indecomposable modules over $\L$.
Given a string $p$, one may define the \emph{string module} $M(p)$ over $\Lambda$. We note that $M(p)$ is isomorphic to $M(p^{-1})$, which is why we consider strings up to equivalence.
Similarly, given a band $B$ of a gentle algebra $\Lambda$, a power $n\in\mathbb Z_{>0}$, and $\lambda\in\k^\times$, we may define a \textit{band module} $M(B,n,\lambda)$ which depends on $B$ only up to equivalence. The specific constructions of band and string modules are not important in this paper, and hence we refer to~\cite{BR87},~\cite[\S2]{BDMTY},~\cite[\S1]{PPP} for more information on band modules and string modules. For context, we cite the following theorem.

\begin{thm}[{\cite[p. 161]{BR87}}]
	The string and band modules are a complete list of the indecomposable $\Lambda$-modules up to isomorphism. Moreover,
	\begin{enumerate}
		\item a string module is never isomorphic to a band module,
		\item two string modules $M(p)$ and $M(p')$ are isomorphic if and only if $p'=p^{\pm1}$, and
		\item two band modules $M(B,n,\lambda)$ and $M(B',n',\lambda')$ are isomorphic if and only if $n=n'$, $\lambda=\lambda'$, and the bands $B$ and $B'$ are equivalent.
	\end{enumerate}
\end{thm}

The above shows that the combinatorics of strings carries representation-theoretic information about the corresponding string modules. In the next section, we will get at the $\tau$-tilting theory of $\Lambda$ by using certain string modules over a larger gentle algebra.

\subsection{$\tau$-tilting theory of gentle algebras and fringed quivers}\label{ssec:bloss}

The $\tau$-tilting theory of gentle algebras has a particularly nice combinatorial interpretation, developed independently in~\cite{BDMTY} and~\cite{PPP}. We refer to either of these papers for more information on this topic.

If $v\in Q$, we say that $\textup{indeg}(v)$ is the number of arrows with head $v$ and we say that $\textup{outdeg}(v)$ is the number of arrows with tail $v$. If $\Lambda$ is gentle, then every vertex of $\Lambda$ has $\textup{indeg}(v)\leq2$ and $\textup{outdeg}(v)\leq2$.

\begin{defn}[{\cite[Definition 2.1]{PPP},\cite[Definition 3.2]{BDMTY}}]
	The \emph{fringed (bound) quiver} of a gentle algebra $\Lambda=\k Q/I$ is the gentle bound quiver $(\tilde Q,\tilde I)$ obtained by adding at each vertex $v\in Q_0$:
	\begin{itemize}
		\item\label{g1} arrows from new vertices in and out of $v$ such that its in-degree and out-degree both become $2$, and
		\item relations such that vertex $v$ fulfills the gentle bound quiver conditions.
	\end{itemize}
	Any arrow added through this construction is called a \emph{fringe arrow}, and is considered to be incident to $v$ as well as a new \emph{fringe vertex}. Vertices of $\tL$ coming from $\L$ are called \emph{internal vertices}. By construction, all {internal vertices} of $\tL$ have in-degree and out-degree 2, and all fringe vertices are incident to only one (fringe) arrow.
	We use the symbols $V_{\text{int}}$, $V_{\text{fringe}}$, $E_{\text{int}}$, and $E_{\text{fringe}}$ to refer to the internal and fringe vertices and edges of $\tL$.
	The \textit{fringed algebra} (also called the \emph{blossoming algebra}) of $\Lambda$ is $\tilde\Lambda=\k\tilde Q/\tilde I$.
\end{defn}

See Figure~\ref{fig:bloss} for an example of a gentle algebra and its fringed quiver.
When we draw fringed quivers, we draw fringe vertices as squares.

Abusing notation, we will use the term \emph{fringed quiver} to refer to the fringed algebra $\tL$ as well as its bound quiver $(\tilde Q,\tilde I)$.
We may refer to vertices or arrows of $\tilde Q$ as vertices or arrows of the fringed quiver $\tilde\Lambda$, respectively.
We say that the \emph{relations} of $\tL$ at an internal vertex $v$ are those compositions $\alpha_1\alpha_2$ and $\beta_1\beta_2$ such that $h(\alpha_1)=h(\beta_1)=t(\alpha_2)=t(\beta_2)=v$ and neither $\alpha_1\alpha_2$ nor $\beta_1\beta_2$ is a string of $\tL$. 

\begin{figure}
	\centering
	\def\svgscale{.21}
\begingroup%
  \makeatletter%
  \providecommand\color[2][]{%
    \errmessage{(Inkscape) Color is used for the text in Inkscape, but the package 'color.sty' is not loaded}%
    \renewcommand\color[2][]{}%
  }%
  \providecommand\transparent[1]{%
    \errmessage{(Inkscape) Transparency is used (non-zero) for the text in Inkscape, but the package 'transparent.sty' is not loaded}%
    \renewcommand\transparent[1]{}%
  }%
  \providecommand\rotatebox[2]{#2}%
  \newcommand*\fsize{\dimexpr\f@size pt\relax}%
  \newcommand*\lineheight[1]{\fontsize{\fsize}{#1\fsize}\selectfont}%
  \ifx\svgwidth\undefined%
    \setlength{\unitlength}{1141.06018066bp}%
    \ifx\svgscale\undefined%
      \relax%
    \else%
      \setlength{\unitlength}{\unitlength * \real{\svgscale}}%
    \fi%
  \else%
    \setlength{\unitlength}{\svgwidth}%
  \fi%
  \global\let\svgwidth\undefined%
  \global\let\svgscale\undefined%
  \makeatother%
  \begin{picture}(1,0.16803498)%
    \lineheight{1}%
    \setlength\tabcolsep{0pt}%
    \put(0,0){\includegraphics[width=\unitlength,page=1]{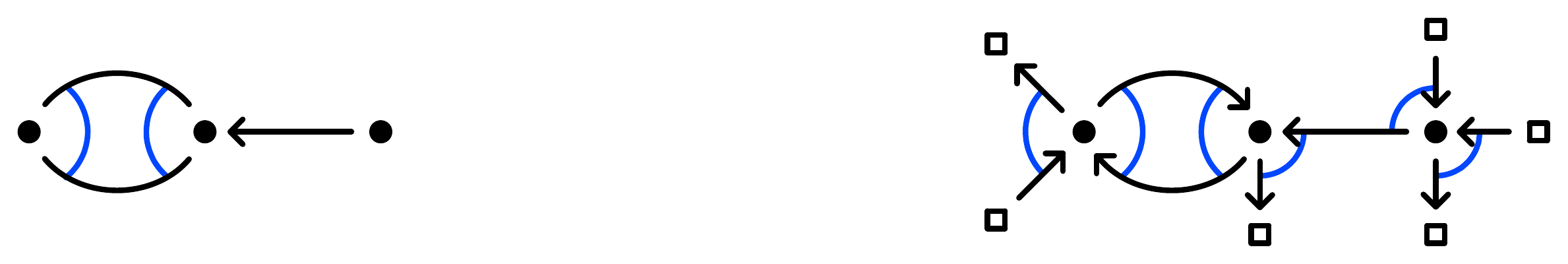}}%
    \put(0.73080615,0.05943422){\color[rgb]{0,0,0}\makebox(0,0)[lt]{\lineheight{1.25}\smash{\begin{tabular}[t]{l}$\scriptstyle{\beta}$\end{tabular}}}}%
    \put(0.72712019,0.13727586){\color[rgb]{0,0,0}\makebox(0,0)[lt]{\lineheight{1.25}\smash{\begin{tabular}[t]{l}$\scriptstyle{\gamma}$\end{tabular}}}}%
    \put(0.84160075,0.09682401){\color[rgb]{0,0,0}\makebox(0,0)[lt]{\lineheight{1.25}\smash{\begin{tabular}[t]{l}$\scriptstyle{\alpha}$\end{tabular}}}}%
    \put(0,0){\includegraphics[width=\unitlength,page=2]{genttoblos.pdf}}%
    \put(0.05774884,0.05944474){\color[rgb]{0,0,0}\makebox(0,0)[lt]{\lineheight{1.25}\smash{\begin{tabular}[t]{l}$\scriptstyle{\beta}$\end{tabular}}}}%
    \put(0.05406279,0.1372855){\color[rgb]{0,0,0}\makebox(0,0)[lt]{\lineheight{1.25}\smash{\begin{tabular}[t]{l}$\scriptstyle{\gamma}$\end{tabular}}}}%
    \put(0.16854335,0.09683453){\color[rgb]{0,0,0}\makebox(0,0)[lt]{\lineheight{1.25}\smash{\begin{tabular}[t]{l}$\scriptstyle{\alpha}$\end{tabular}}}}%
  \end{picture}%
\endgroup%

	\caption{{A gentle algebra (left) and its fringed quiver (right).}}
	\label{fig:bloss}
\end{figure}

\begin{defn}
	A \emph{route} of $\tL$ is a maximal string in $\tilde \Lambda$ (thus beginning and ending at fringe vertices of $\tilde \Lambda$).
	A route is \emph{straight} if it is an oriented walk in $\tilde \Lambda$, and \emph{bending} otherwise.
\end{defn}

\begin{defn}\label{defn:tilt-complex}
	Let $p$ and $q$ be two strings of $\tilde\Lambda$. We say that $p$ and $q$ \emph{kiss} if without loss of generality
	there exists a (possibly lazy) string $\sigma$ which is a top substring of $p^{\pm1}$ and a bottom substring of $q^{\pm1}$.
	We call $\sigma$ an \emph{incompatibility} between $p$ and $q$.
\end{defn}

\begin{defn}
	We say that two routes $p$ and $q$ are \emph{compatible} if $p$ does not kiss $q$.
	A \emph{clique} is a set of pairwise compatible (self-compatible) routes.
	The \emph{clique complex} $\R(\tL)$ of $\tL$ is the simplicial complex of cliques.
\end{defn}

The clique complex is also called the \emph{non-kissing complex}; we use the term clique complex to differentiate it from the bundle complex, to be defined later.
We particularly care about the \emph{maximal cliques}. 
A clique is \emph{reduced} if it contains no straight route. 
Note that no string may be incompatible with a straight route, hence every maximal clique contains every straight route. This means that the \emph{reduced clique complex} $\R_{\text{red}}(\tL)$, defined as the simplicial complex of reduced cliques, has the same dual graph as the clique complex. More specifically, given a clique $\K$, we say that the $\K_{\text{red}}$ is the clique given by the bending routes of $\K$ (without any straight routes). Then maximal cliques are in bijection with maximal reduced cliques through the map $\K\mapsto\K_{\text{red}}$.
See Figure~\ref{GEMINTRO} for an example of a fringed quiver and its maximal reduced cliques.

\begin{example}\label{ex:kiss}
	Shown in Figure~\ref{fig:kiss} are two routes which kiss -- the string $\alpha$ is at the top of the red route and at the bottom of the blue route.
	\begin{figure}
		\centering
		\def\svgscale{.21}
\begingroup%
  \makeatletter%
  \providecommand\color[2][]{%
    \errmessage{(Inkscape) Color is used for the text in Inkscape, but the package 'color.sty' is not loaded}%
    \renewcommand\color[2][]{}%
  }%
  \providecommand\transparent[1]{%
    \errmessage{(Inkscape) Transparency is used (non-zero) for the text in Inkscape, but the package 'transparent.sty' is not loaded}%
    \renewcommand\transparent[1]{}%
  }%
  \providecommand\rotatebox[2]{#2}%
  \newcommand*\fsize{\dimexpr\f@size pt\relax}%
  \newcommand*\lineheight[1]{\fontsize{\fsize}{#1\fsize}\selectfont}%
  \ifx\svgwidth\undefined%
    \setlength{\unitlength}{426.14199829bp}%
    \ifx\svgscale\undefined%
      \relax%
    \else%
      \setlength{\unitlength}{\unitlength * \real{\svgscale}}%
    \fi%
  \else%
    \setlength{\unitlength}{\svgwidth}%
  \fi%
  \global\let\svgwidth\undefined%
  \global\let\svgscale\undefined%
  \makeatother%
  \begin{picture}(1,0.41174772)%
    \lineheight{1}%
    \setlength\tabcolsep{0pt}%
    \put(0,0){\includegraphics[width=\unitlength,page=1]{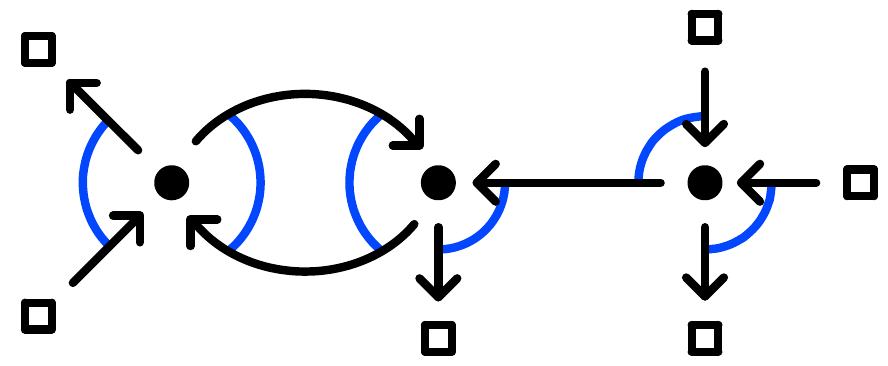}}%
    \put(0.29828508,0.14004721){\color[rgb]{0,0,0}\makebox(0,0)[lt]{\lineheight{1.25}\smash{\begin{tabular}[t]{l}$\scriptstyle{\beta}$\end{tabular}}}}%
    \put(0.28841537,0.34848008){\color[rgb]{0,0,0}\makebox(0,0)[lt]{\lineheight{1.25}\smash{\begin{tabular}[t]{l}$\scriptstyle{\gamma}$\end{tabular}}}}%
    \put(0.5949545,0.24016407){\color[rgb]{0,0,0}\makebox(0,0)[lt]{\lineheight{1.25}\smash{\begin{tabular}[t]{l}$\scriptstyle{\alpha}$\end{tabular}}}}%
    \put(0,0){\includegraphics[width=\unitlength,page=2]{kiss.pdf}}%
  \end{picture}%
\endgroup%

		\caption{Two routes which kiss.}
		\label{fig:kiss}
	\end{figure}
\end{example}

\begin{defn}
	Let $p=\alpha_1^{\e_1}\dots\alpha_m^{\e_m}$ be a bending route of $\tL$. We define a module $\tilde M(p)$ of $\Lambda$ as follows:

	Let $i$ be minimal such that $\alpha_{i-1}=-1$ and let $j$ be maximal such that $\alpha_{j+1}=1$. If $i\leq j$, then define the string $S(p):=\alpha_i^{\e_i}\dots\alpha_j^{\e_j}$. If $i=j+1$, then 
\[\e_k=\begin{cases}
	1&k<j\textup{ or }k=j+1=i\\
	-1&k>j+1\textup{ or }k=j=i-1.
\end{cases}\]
	In this case, let the string $S(p)$ be the lazy string $S(p):=e_{t(\alpha_i)}=e_{h(\alpha_j^{-1})}$. In either of these cases, let $\tilde M(p)$ be the string module $\tilde M(S(p))$. Otherwise, we must have $i=j+2$, and 
\[\e_k=\begin{cases}
	1&k\leq j+1\\
	-1&k\geq j+1.
\end{cases}
	\]
	In this case, let $\tilde M(p)$ be the shifted projective indecomposable module at the vertex $t(\alpha_{i-1}^{-1})=h(\alpha_{j+1})$.
\end{defn}

\begin{thm}[{\cite[Theorem 2.46]{PPP},\cite[Theorem 5.1]{BDMTY}}]
	\label{thm:kiss-compat}
	The above map $\tilde M$ gives a bijection between bending routes of $\tL$ and string modules of $\Lambda$.
	Moreover,
	\begin{enumerate}
		\item a route $p$ is self-compatible if and only if $\tilde M(p)$ is $\tau$-rigid, and
		\item a set of routes $\K$ is a clique if and only if $\oplus_{p\in\K}\tilde M(p)$ is $\tau$-rigid.
	\end{enumerate}
	Hence, the bijection $\tilde M$ induces an isomorphism between the reduced clique complex of bending routes of the fringed quiver $\tL$ and the support $\tau$-tilting complex of the gentle algebra $\L$.
\end{thm}

Band modules are not associated to routes and are never $\tau$-rigid. On the other hand, a band of $\tL$ may naturally be considered as a band of $\L$.

Moreover, g-vectors of string modules can be calculated conveniently using their routes over the fringed quiver.

\begin{prop}[{\cite[\S1.2.1]{PPP}}]\label{prop:g-vect}
	Let $p$ be a route of $\tL$.
	Then $\g(\tilde M(p))=\textbf{a}-\textbf{b}$, where $\textbf{a}=(a_v)_{v\in V_{\text{int}}}$ such that $a_v$ is the number of times the lazy string $e_v$ is a top substring of $p$, and $\textbf{b}=(b_v)_{v\in V_{\text{int}}}$ such that $b_v$ is the number of times the lazy string $e_v$ is a bottom substring of $p$.
\end{prop}

For example, in the Kronecker fringed quiver of Figure~\ref{fig:kron1}, the route $p_1:=e_1e_2f_2^{-1}e_2f_2^{-1}f_1^{-1}$ has g-vector $\g(p_1)=(1,-2)$.
Since all $\tau$-rigid indecomposable modules have an associated route of $\tL$, and support $\tau$-tilting modules over $\L$ correspond to maximal reduced cliques of $\tL$, Proposition~\ref{prop:g-vect} allows us to understand the g-vector fan of $\L$ using the clique complex of $\tL$.

We cite one more fact about the g-vector fan of a gentle algebra.
\begin{thm}[{\cite{AY,KPY}}]\label{thm:g-vector-dense}
	If $\L$ is a gentle algebra, then the g-vector fan of $\L$ is dense in $\mathbb R^{V_{\text{int}}}$.
\end{thm}

\section{Background on Polyhedra, DAGs, and Ample Framings}\label{sec:DAG}

Unit flow polytopes are a fundamental and widely studied class of polytopes which may be defined from a directed acyclic graph (DAG). \emph{Framings} of a DAG were introduced by Danilov, Karzanov, and Koshevoy~\cite{DKK} and shown to induce regular unimodular triangulations on the flow polytope.
In~\cite{WIWT}, a connection was developed between certain gentle algebras and \emph{amply framed} directed acyclic graphs (DAGs), which we now include.
The results of this section are not directly used in the rest of the paper, but the intuition is fundamental to our definitions and our results will pertain to this theory.

\subsection{Triangulations, subdivisions, and dissections}
\label{ssec:subd}

Before talking about flow polytopes, we define what we ask of a subdivision of a polyhedron.

\begin{defn}\label{defn:subd}
	Let $P$ be a lattice polyhedron. A \emph{(polyhedral) dissection} of $P$ is a set $\Sbs$ of lattice polyhedra satisfying:
\begin{enumerate}
	\item \textbf{Density:} $\cup_{Q\in\Sbs}Q$ is a dense subset of $P$, and
    \item \textbf{Weak Intersection Property:} for distinct $Q_1,Q_2\in\Sbs$, the polyhedron $Q_1 \cap Q_2$ has strictly lower dimension than $Q_1$ and $Q_2$.
\end{enumerate}
	The dissection is \emph{complete} if $P=\cup_{Q\in\Sbs}Q$.
	We say that a dissection is a \emph{subdivision} if it satisfies the \emph{strong intersection property}: for distinct $Q_1,Q_2\in\Sbs$, the intersection $Q_1\cap Q_2$ is a proper face of both $Q_1$ and $Q_2$.
	A \emph{triangulation} is a strong subdivision in which each cell is a simplex of dimension $\text{dim}(P)$.
	A triangulation is \emph{unimodular} if each cell has normalized volume 1 (equivalently, if the volume of each cell is the factorial of $\dim P$).
\end{defn}

Other sources may require all dissections to be complete, or require each cell of a dissection to be full-dimensional.
Note that if $P$ is a lattice polytope, then every dissection of $P$ is complete.
We may also use the term \emph{(unimodular) triangulation} to refer to the cone over a unimodular triangulation as in Definition~\ref{defn:subd}.

\subsection{Flow polytopes}

Let $G=(V,E)$ be a finite directed acyclic graph with vertex set $V$ and edge set $E$. For each $v\in V$, let $\inn(v)$ and $\out(v)$ denote the incoming and outgoing edges of $v$, respectively. A vertex $v$ is called a \emph{source} if $\inn(v)=\emptyset$ and it is called a \emph{sink} if $\out(v)=\emptyset$.
All other vertices are called \emph{internal vertices}. 
An edge $\alpha\in E$ is \emph{internal} if it is between two internal vertices, and otherwise is a \emph{source/sink edge}. More specifically, it is either a \emph{source edge} (if its tail is a source) and/or a \emph{sink edge} (if its head is a sink). 
A \emph{route} of $G$ is a maximal (directed) path in $G$.

\begin{defn}\label{defn:flow-polytope}
	A \emph{flow} $F$ on a directed graph $G$ is a function $F:E\to\mathbb R$ which preserves flow at each internal vertex, i.e., for every internal vertex $v$ we have
	\[\sum_{e\in\inn(v)}F(e)=\sum_{e\in\out(v)}F(e).\]
	A flow $F$ is \emph{nonnegative} if $F(e)\geq0$ for all edges $e\in E$.
	The \emph{cone of (nonnegative) flows} $\mathcal F_{\geq0}(G)$ is the space of nonnegative flows on $G$. The \emph{(unit) flow polyhedron} $\mathcal F_1=\mathcal F_1(G)$ is the set of all \emph{unit} flows of $G$, i.e., flows satisfying
	\[\sum_{\substack{v\text{ is a source}\\ e\in\out(v)}}F(e)=1.\]
	If $G$ is acyclic, then $\F_1(G)$ is a polytope and we call it the \emph{flow polytope} $\F_1(G)$.
\end{defn}

The left of Figure~\ref{fig:flow} shows a flow on a directed graph labelled in blue.

\begin{prop}\label{prop:dimflow}
The dimension of $\F_1(G)$ is
	\[\textup{dim}(\F_1(G))=|E|-\#\{v\in V\ :\ v\textup{ is an internal vertex}\}-1.\]
\end{prop}

Given a route $p$ of $G$, the \emph{indicator vector} $\I(p)$ with 1's in the coordinates of edges used by $p$ and 0's elsewhere, is a vertex of $\F_1(G)$. All vertices of $\F_1(G)$ are obtained in this way. 

\subsection{Framed DAGs}
\label{ssec:dkkback}

We give some definitions, following~\cite{WIWT}.

\begin{defn}
	Let $G=(V,E)$ be a DAG.  For each internal vertex $v$ of $G$, assign a linear order to the edges in $\inn(v)$ and assign a linear order to the edges in $\out(v)$. This assignment is called a \emph{framing} of $G$, which we denote by $\mathfrak{F}$. We call a DAG $G$ with a framing $\mathfrak{F}$ a \emph{framed DAG}, which we often denote by $\Gamma=(G,\mathfrak{F})$. If $e$ is less than $f$ in the linear order for $\mathfrak{F}$ on $\inn(v)$, we write $e<_{\mathfrak{F},\inn(v)}f$ (and similarly for $\out(v)$). When $\mathfrak{F}$ and/or $\inn(v)$ or $\out(v)$ is clear, we may drop one or both subscripts.
\end{defn}

In this paper, to denote a framing we label the edges of a DAG with integers.
For example, in the left of Figure~\ref{fig:difdagc}, the edge $\alpha$ is the lowest element of $\out(t(\alpha))$. 

If $\Gamma=(G,\mathfrak{F})$ is a framed DAG, then a route (vertex, edge, \dots) of $\Gamma$ is a route (vertex, edge, \dots) of $G$.

\begin{defn}\label{defn:compatible}
	Let $R$ be a path in $\Gamma$ from $v$ to $w$, where both vertices are internal (we may have $v=w$, in which case $R$ is the empty path from $v$ to $v$).
	Let $\alpha_1$ and $\alpha_2$ be edges of $G$ ending at $v$ such that $\alpha_1<_{\mathfrak{F},\inn(v)}\alpha_2$. Let $\beta_1$ and $\beta_2$ be edges of $G$ starting at $w$ such that $\beta_1<_{\mathfrak{F},\out(v)}\beta_2$. Then the paths $\alpha_1R\beta_2$ and $\alpha_2R\beta_1$ are \emph{incompatible}.
	Two routes $p$ and $q$ are \emph{incompatible} if there exist subpaths $p'$ of $p$ and $q'$ of $q$ which are incompatible. Otherwise, they are \emph{compatible}.
	If a route $p$ in $\Gamma$ is compatible with every other route in $\Gamma$, we say that $p$ is \emph{exceptional}. Otherwise, it is \emph{nonexceptional}.
\end{defn}

Similar to the case of gentle algebras, we will see that sets of pairwise compatible routes will form a pure simplicial complex.

\begin{defn}
	A \emph{clique} is a set of pairwise-compatible routes in $\Gamma$. A clique is \emph{reduced} if it contains no exceptional routes.
\end{defn}

Note that an exceptional route is a route which is in every maximal clique.
The \emph{clique complex} $\R({\Gamma})$ of $\Gamma$ is the simplicial complex of cliques of $\Gamma$. The \emph{reduced clique complex} $\R_{\text{red}}(\Gamma)$ is the simplicial complex of cliques of $\Gamma$ which contain no exceptional routes. The maximal cliques of $\R_{\text{red}}(\Gamma)$ are precisely the maximal cliques of $\R(\Gamma)$ with the exceptional routes removed.

Recall that taking indicator vectors gives a bijection from routes of $G$ onto vertices of $\F_1(G)$. Through this correspondence, we may view a maximal clique of $\Gamma$ as a collection of vertices of $\F_1(G)$ which form a simplex. The set of such simplices forms a regular unimodular triangulation of $\F_1(G)$.

\begin{thm}[\cite{DKK}]
	\label{thm:dkk}
	Let $\Gamma$ be a framed DAG. The set of maximal cliques of $\Gamma$ forms a regular unimodular triangulation of the flow polytope $\mathcal F_1(G)$.
\end{thm}

The triangulation from Theorem~\ref{thm:dkk} is called the \emph{DKK triangulation} of $\Gamma$. 
See Example~\ref{ex:cube}.
By Theorem~\ref{thm:dkk}, the dual graph on maximal cliques of $\Gamma$ is isomorphic to the dual graph of the DKK triangulation of $\Gamma$.

\begin{example}\label{ex:cube}
	Figure~\ref{fig:cube} shows a framed DAG whose flow polytope is a square. In the middle are its two maximal cliques and on the right is the corresponding triangulated flow polytope. The flow polytope is a square embedded in the four-dimensional coordinate system with a coordinate for each of the four edges of the DAG. The vertex labelled $e_2f_1$, for example, is the indicator vector $\I(e_2f_1)$.
	\begin{figure}
		\centering
		\def\svgscale{.21}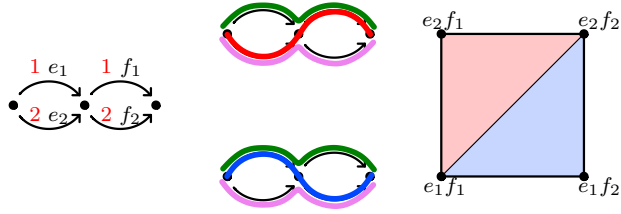
		\caption{A framed DAG with its maximal cliques and triangulated flow polytope.}
		\label{fig:cube}
	\end{figure}
\end{example}

\subsection{Ample framings}
\label{ssec:ampleframe}

A framed DAG is \emph{amply framed}~\cite{DKK} if every edge is contained in some exceptional route.
In~\cite{WIWT} it was shown that, up to contraction, amply framed DAGs may be assumed to be of the following form, which we henceforth take as the definition of ample framedness.

\begin{defn}\label{defn:amp}
	A DAG is \emph{full} if every internal vertex is incident to precisely two outgoing edges and two incoming edges.
	A framed full DAG $\Gamma=(G,\mathfrak{F})$ is \emph{amply framed} if
		there is a map $\phi_\mathfrak{F}:E\to\{1,2\}$ such that, if $\alpha$ and $\beta$ are edges with the same internal source or target, we have $\alpha<_{\mathfrak{F}}\beta$ (in the partial order of $\inn(h(\alpha))$ or $\out(t(\alpha))$) if and only if $\phi_\mathfrak{F}(\alpha)<\phi_\mathfrak{F}(\beta)$. We say that the framing $\mathfrak{F}$ is \emph{realized by} $\phi_{\mathfrak{F}}$.
\end{defn}

We will always assume that an amply framed DAG is of the form given in Definition~\ref{defn:amp}.
We say that an edge $\alpha$ of an amply framed DAG $\Gamma$ is a \emph{1-edge} (respectively \emph{2-edge}) if $\phi_\mathfrak{F}(\alpha)=1$ (respectively $\phi_\mathfrak{F}(\alpha)=2$).
The exceptional routes of an amply framed DAG are the sets of routes which consist entirely of 1-edges or entirely of 2-edges.
All examples of framed DAGs in this paper are amply framed.

\begin{defn}\label{defn:gent-from-dag}
	Let $\Gamma=(G,\mathfrak{F})$ be an amply framed DAG. As in~\cite{WIWT}, we associate to $\Gamma$ a gentle algebra $\Lambda(\Gamma)=\k Q/I$ as follows. The vertices of $Q$ correspond to internal vertices of $\Gamma$. For each internal 1-edge $\alpha$ of $\Gamma$ there is an arrow $\alpha_{\Lambda(\Gamma)}$ of $Q$ from $t(\alpha)$ to $h(\alpha)$. For each internal 2-edge $\beta$ of $\Gamma$ there is an arrow $\beta_{\Lambda(\Gamma)}$ from $h(\beta)$ to $t(\beta)$. In other words, we obtain $Q$ by taking the internal subgraph of $\Gamma$ and flipping the direction of all 2-edges. The set of relations $I$ is generated by all pairs of the form $\alpha\beta$, where either $\alpha$ came from a 1-edge and $\beta$ came from a 2-edge, or $\alpha$ came from a 2-edge and $\beta$ came from a 1-edge.
	Let $\tL(\Gamma)$ be the fringed quiver of $\Lambda(\Gamma)$.
\end{defn}
For an example of this map, see Figure~\ref{fig:difdagc}.

There is a natural bijection between internal arrows of $\Gamma$ and internal edges of $\tL(\Gamma)$, and a bijection between source/sink edges of $\Gamma$ which do not go directly from a source vertex to a sink vertex and fringe edges of $\tL(\Gamma)$.
This induces a bijection between nonexceptional routes of $\Gamma$ and bending routes of $\Lambda(\Gamma)$. It was observed in~\cite{WIWT} that this map preserves compatibility. In other words, two routes of $\Gamma$ are compatible if and only if the corresponding routes of $\Lambda(\Gamma)$ do not kiss.

\begin{thm}[{\cite[Theorem 5.7, Theorem 5.12, Theorem 5.13]{WIWT}}]
	\label{thm:gentdag}
	Let $\Gamma$ be an amply framed DAG. There is a natural bijection between non-exceptional routes of $\Gamma$ and bending routes of $\L(\Gamma)$ which preserves compatibility and induces an isomorphism between the reduced clique complex $\R_{\text{red}}(\Gamma)$ and the support $\tau$-tilting complex $\C_{\Lambda(\Gamma)}$.
\end{thm}

Figure~\ref{fig:difdagc} shows an example of a framed DAG $\Gamma$ and its corresponding algebra $\L(\Gamma)$ as in Theorem~\ref{thm:gentdag}. Note that routes on $\Gamma$ are in pairwise-compatibility-preserving bijection with routes on $\Lambda(\Gamma)$.
\begin{figure}
	\centering
	\def\svgscale{.21}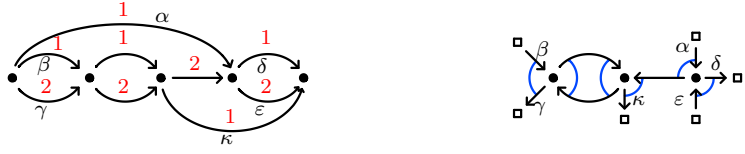
	\caption{Shown is an amply framed DAG on the left, and the fringed quiver of its corresponding gentle algebra on the right. The edge labels on the left and right show the natural bijection between source/sink edges of the DAG and fringe arrows of the fringed quiver.}
	\label{fig:difdagc}
\end{figure}
Figure~\ref{fig:cubegent} shows the fringed quiver associated to the framed DAG of Figure~\ref{fig:cube}. Shown also are the maximal cliques of this fringed quiver. Note the similarity to the maximal cliques of $\Gamma$ shown in Figure~\ref{fig:cube}.
\begin{figure}
	\centering
	\def\svgscale{.21}
\begingroup%
  \makeatletter%
  \providecommand\color[2][]{%
    \errmessage{(Inkscape) Color is used for the text in Inkscape, but the package 'color.sty' is not loaded}%
    \renewcommand\color[2][]{}%
  }%
  \providecommand\transparent[1]{%
    \errmessage{(Inkscape) Transparency is used (non-zero) for the text in Inkscape, but the package 'transparent.sty' is not loaded}%
    \renewcommand\transparent[1]{}%
  }%
  \providecommand\rotatebox[2]{#2}%
  \newcommand*\fsize{\dimexpr\f@size pt\relax}%
  \newcommand*\lineheight[1]{\fontsize{\fsize}{#1\fsize}\selectfont}%
  \ifx\svgwidth\undefined%
    \setlength{\unitlength}{737.77001953bp}%
    \ifx\svgscale\undefined%
      \relax%
    \else%
      \setlength{\unitlength}{\unitlength * \real{\svgscale}}%
    \fi%
  \else%
    \setlength{\unitlength}{\svgwidth}%
  \fi%
  \global\let\svgwidth\undefined%
  \global\let\svgscale\undefined%
  \makeatother%
  \begin{picture}(1,0.66510292)%
    \lineheight{1}%
    \setlength\tabcolsep{0pt}%
    \put(0,0){\includegraphics[width=\unitlength,page=1]{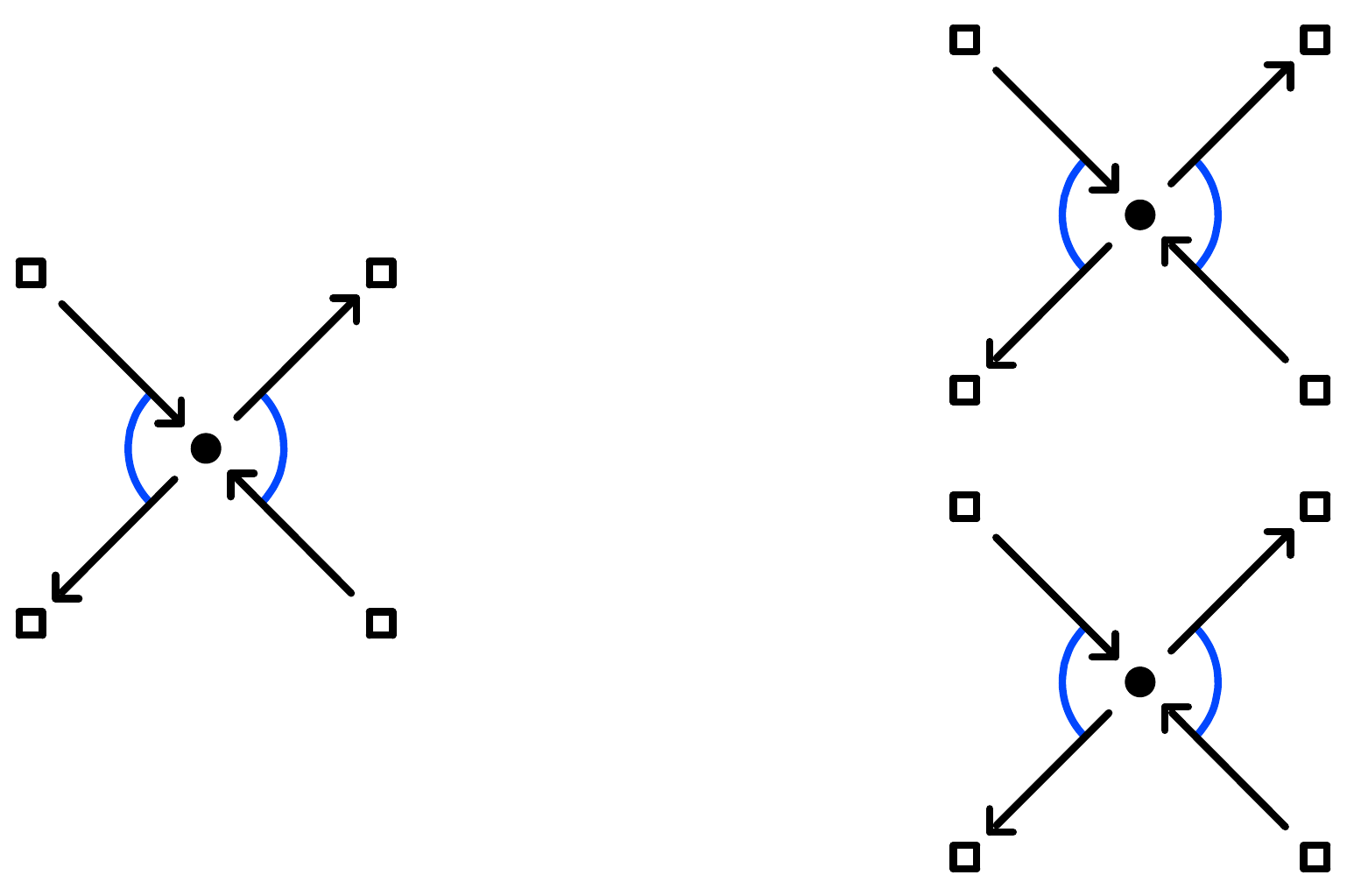}}%
    \put(0.08668135,0.40964526){\color[rgb]{0,0,0}\makebox(0,0)[lt]{\lineheight{1.25}\smash{\begin{tabular}[t]{l}$\scriptstyle{e_1}$\end{tabular}}}}%
    \put(0.17449815,0.40964526){\color[rgb]{0,0,0}\makebox(0,0)[lt]{\lineheight{1.25}\smash{\begin{tabular}[t]{l}$\scriptstyle{f_1}$\end{tabular}}}}%
    \put(0.17449815,0.23614946){\color[rgb]{0,0,0}\makebox(0,0)[lt]{\lineheight{1.25}\smash{\begin{tabular}[t]{l}$\scriptstyle{f_2}$\end{tabular}}}}%
    \put(0.08668135,0.23614946){\color[rgb]{0,0,0}\makebox(0,0)[lt]{\lineheight{1.25}\smash{\begin{tabular}[t]{l}$\scriptstyle{e_2}$\end{tabular}}}}%
    \put(0,0){\includegraphics[width=\unitlength,page=2]{squerg.pdf}}%
  \end{picture}%
\endgroup%

	\caption{The fringed quiver given by the framed DAG of Figure~\ref{fig:cube}.}
	\label{fig:cubegent}
\end{figure}

For the convenience of the reader, we include a table comparing analogous terms in the setting of amply framed DAGs and fringed quivers.
\begin{center}
\begin{tabular}{|c|c|}
	\hline
	\textbf{Amply Framed DAG} & \textbf{Fringed Quiver}\\
	\hline
	path & string\\
	\hline
	route & route\\
	\hline
	clique & clique\\
	\hline
	clique complex & clique complex \\
	\hline
	exceptional route & straight route\\
	\hline
	reduced clique & reduced clique\\
	\hline
	nonexceptional route & bending route\\
	\hline
	------ & band\\
	\hline
\end{tabular}
\end{center}

\section{Turbulence Polyhedra}
\label{sec:tp}

In Section~\ref{sec:DAG}, we saw that the $\tau$-tilting theory of certain gentle algebras matches the theory of amply framed DAGs and their triangulated flow polytopes.
In this section, we extend this situation by defining a notion of a (unit) flow on the fringed quiver $\tL$ of an arbitrary gentle algebra $\Lambda$. We refer to the space of unit flows as the \emph{turbulence polyhedron} $\F_1(\tL)$, and we refer to the space of nonnegative flows as the \emph{cone of flows} $\F_{\geq0}(\tL)$. In this section, we give some examples and we prove some basic results about turbulence polyhedra. We will connect turbulence polyhedra back to flow polyhedra of amply framed directed (possibly acyclic) graphs.

In the following, let $\tL$ be a fringed quiver with vertex set $V$ and edge set $E$.
We use the symbols $V_{\text{int}}$, $V_{\text{fringe}}$, $E_{\text{int}}$, and $E_{\text{fringe}}$ to refer to the internal and fringe vertices and edges of $\tL$.

\begin{defn}\label{defn:flowgent}
	A function $F:E\to\mathbb R$ is a \emph{(nonnegative) flow} on $\tL$ if it satisfies
	\begin{enumerate}
		\item \textbf{Nonnegativity:} for any arrow $\alpha\in E$, we have $F(\alpha)\geq0$,
		\item \textbf{Conservation of Flow:} for any internal vertex $v$ of $\tL$ with relations $\alpha_1\alpha_2$ and $\beta_1\beta_2$, we have 
			\[F(\alpha_1)+F(\alpha_2)=F(\beta_1)+F(\beta_2).\]
	\end{enumerate}
	The \emph{strength} of a flow $F$ is the sum $\frac{1}{2}\sum_{\alpha\in E\text{ fringe}}F(\alpha)$. A flow is \emph{unit} if it has strength 1, and it is a \emph{vortex} if it has strength 0.
	The \emph{cone of  flows} $\F_{\geq0}(\tL)$ is the space of nonnegative flows on $\tL$. The \emph{turbulence polyhedron} $\F_1(\tL)$ is the space of unit nonnegative flows on $\tL$.
\end{defn}

Both $\F_1(\tL)$ and $\F_{\geq0}(\tL)$ are rational polyhedra, since they can be defined as finite intersections of rational hyperplanes and half-spaces in $\mathbb R^E$. We make this explicit now.

\begin{defn}
We define the following subsets of the set $\mathbb R^{E}$ of functions from the arrows of $\tL$ to $\mathbb R$.
	\begin{enumerate}
		\item If $\alpha$ is an arrow of $\tL$, then let $H_\alpha^{\geq0}$ be the half-space of points $F\in\mathbb R^E$ satisfying $F(\alpha)\geq0$ and let $H_\alpha^0$ be the hyperplane satisfying $F(\alpha)=0$.
		\item If $v$ is an internal vertex of $\tL$ with relations $\alpha_1\alpha_2$ and $\beta_1\beta_2$, then let $H_v$ be the \emph{conservation of flow at $v$} hyperplane of points $F\in\mathbb R^{E}$ satisfying $F(\alpha_1)+F(\alpha_2)=F(\beta_1)+F(\beta_2)$.
		\item Let $H_U$ be the \emph{unit flow} hyperplane of points $F\in\mathbb R^{{E}}$ satisfying $\sum_{\alpha\in{E}\text{ fringe}}F(\alpha)=2$.
	\end{enumerate}
\end{defn}

Then $\F_{\geq0}(\tL)$ is the intersection $(\cap_\alpha H_\alpha^{\geq0})\cap(\cap_vH_v)$, where $\alpha$ and $v$ range over all arrows and internal vertices of $\tL$. Moreover, $\F_1(\tL)=H_U\cap\F_{\geq0}(\tL)$.

The \emph{recession cone} of a polyhedron $P$ in $\mathbb R^n$ is the set
\[\text{Rec}(P)=\{\xx\in\mathbb R^n\ :\ \yy+\xx\in P\text{ for all }\yy\in P\}.\]
Intuitively, the recession cone is the cone of unbounded directions of $P$. It is immediate that the recession cone of $\F_1(\tL)$ is precisely the set of vortices (vortexes) of $\tL$.

Recall the construction of Definition~\ref{defn:gent-from-dag}, which starts with an amply framed DAG $\Gamma$ and obtains a gentle algebra $\Lambda(\Gamma)$ and hence a fringed quiver $\tL(\Gamma)$ whose reduced clique complex agrees with the reduced clique complex of $\Gamma$ by Theorem~\ref{thm:gentdag}.
We remark that if $\Gamma$ has no arrows directly from source to sink, then arrows (resp. routes) of $\Gamma$ are in natural bijection with arrows (resp. routes) of $\tL(\Gamma)$. Moreover, the conservation of flow at an internal vertex $v$ of $\Gamma$ is the same as the conservation of flow at the analogous internal vertex of $\tL(\Gamma)$, so the turbulence polyhedron $\F_1(\tL(\Gamma))$ is the same as the flow polyhedron $\F_1(\Gamma)$. This will be made more explicit in Sections~\ref{sec:PRF} and~\ref{ssec:PGAaAFDG}, though we point it out now to motivate Definition~\ref{defn:flowgent}.
The left of Figure~\ref{fig:flow} has an example of a flow on the amply framed DAG of Figure~\ref{fig:difdagc} and the analogous flow on the corresponding fringed quiver.

\begin{figure}
	\centering
	\def\svgscale{.21}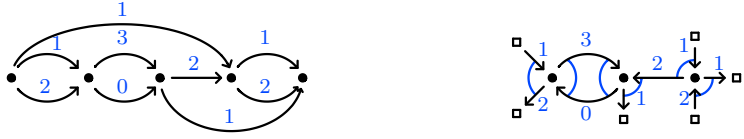
	\caption{Shown is an amply framed DAG and associated fringed quiver, with analogous flows labelled in blue.}
	\label{fig:flow}
\end{figure}

\begin{defn}
	If $p$ is a string of $\tL$, then the \emph{indicator vector} $\I(p)$ is the vector in $\mathbb R^E$ such that the coordinate of an arrow $\alpha$ is the number of times the arrow $\alpha^{\pm1}$ appears in $p$.
	The \emph{indicator vector} of a route or band of $\tL$ is the indicator vector of its underlying string.
\end{defn}

\begin{example}\label{ex:kron}
	Shown in Figure~\ref{fig:kron1} is the fringed quiver of the Kronecker quiver and its associated turbulence polyhedron. The closed dots are vertices and the open dots are lattice points which are not vertices.
	\begin{figure}
		\centering
		\def\svgscale{.21}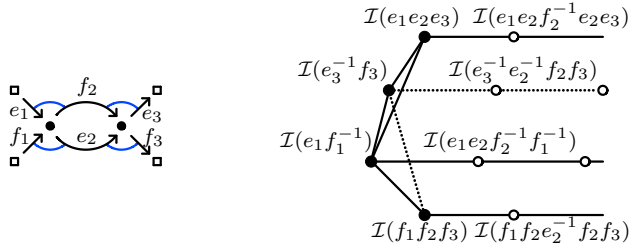
		\caption{The fringed quiver of the Kronecker quiver and its turbulence polyhedron.}
		\label{fig:kron1}
	\end{figure}

	The four vertices of the turbulence polyhedron are the indicator vectors $\I(e_1f_1^{-1})$, $\I(f_3^{-1}e_3)$, $\I(e_1e_2e_3)$, and $\I(f_1f_2f_3)$. The recession cone of $\F_1(\tL)$ is generated by the indicator vector of the band $e_2f_2^{-1}$.
	In Section~\ref{sec:vert-unb}, we will describe the vertices and recession cone of a general turbulence polyhedron in terms of certain routes and bands, called \emph{elementary}.
\end{example}

\begin{lemma}\label{lem:flowdim}
	The turbulence polyhedron $\F_1(\tL)$ has dimension $|E|-|V_{\text{int}}|-1=|V_{\text{int}}|+s-1$, where $s$ is the number of straight routes of $\tL$.
\end{lemma}
\begin{proof}
	The turbulence polyhedron $\F_1(\tL)$ is the intersection of $|V_{\text{int}}|+1$ hyperplanes (a conservation of flow hyperplane $H_v$ for each internal vertex $v$, and a unit flow hyperplane $H_U$) and $|E|$ half-spaces in $\mathbb R^E$ (a half-space $H_\alpha^{\geq0}$ for each arrow $\alpha\in{E}$).
	We first show that no one of the $|V_{\text{int}}|+1$ hyperplanes $\{H_U\}\cup\{H_w\ :\ w\in{V}\}$ is contained in the intersection of the others, and hence that the intersection of these hyperplanes is of dimension $|E|-|V_{\text{int}}|-1$.
	Indeed, choose $v$ to be an internal vertex and let $p=p^+ e_v p^-$ be any route of $\tL$ passing through $v$. Then $2\I(p^+)$ satisfies unit flow and conserves flow at every vertex except for $v$, hence $H_v\not\subseteq H_U\cap\left(\cap_{w\neq v}H_w\right)$.
	The zero flow (sending every edge to zero) conserves flow at every vertex, but does not lie in the unit flow hyperplane $H_v$, hence $H_U\not\subseteq\cap_{w\in{V}}H_w$. This shows that $H:=H_U\cap\left(\cap_{w\in{V}}H_w\right)$ is of dimension $|E|-|V_{\text{int}}|-1$.

	The turbulence polyhedron $\F_1(\tL)$ is the intersection of $H$ with the nonnegative orthant of $\mathbb R^E$. We claim that this intersection has the same dimension as $H$. It suffices to find a point $v\in H$ which is in the positive orthant (i.e., the interior of the nonnegative orthant).
	Let $S$ be a collection of routes of $\tL$ such that every arrow of $\tL$ is contained in a route of $S$. Then $v:=\frac{1}{|S|}\sum_{p\in S}\I(p)$ is in the intersection of $H$ with the positive orthant. This completes the proof that $\dim\F_1(\tL)=|E|-|V_{\text{int}}|-1$.

	It remains to show that $|E|-|V_{\text{int}}|-1=|V_{\text{int}}|+s-1$.
	Since every straight route of $\tL$ contains precisely one incoming (respectively outgoing) fringe arrow, we have that $s$ is the number of incoming (respectively outgoing) fringe arrows of $\tL$. Then~\cite[Lemma 2.3]{PPP} says that $|E|=4|V_{\text{int}}|-|E_{\text{int}}|$ and
	$s=2|V_{\text{int}}|-|E_{\text{int}}|$, giving the second and fourth equality of the following calculation:
	\begin{align*}
		|E|-|V_{\text{int}}|-1&=(4|V_{\text{int}}|-|E_{\text{int}}|)-|V_{\text{int}}|-1\\
		&=3|V_{\text{int}}|-|E_{\text{int}}|-1\\
		&=|V_{\text{int}}|-1+(2|V_{\text{int}}|-|E_{\text{int}}|)\\
		&=|V_{\text{int}}|-1+s.
	\end{align*}
\end{proof}
In accordance with Lemma~\ref{lem:flowdim}, the dimension of the turbulence polyhedron of the Kronecker fringed quiver of Figure~\ref{fig:kron1} is 
\[\dim\F_1(\tL)=|E|-|V_{\text{int}}|-1=6-2-1=3.\]

\begin{lemma}\label{lem:facet}
	All facets of $\F_1(\tL)$ are of the form $H_\alpha^{\geq0}$ for some $\alpha\in E$ (i.e., they are obtained by necessitating zero flow through some arrow $\alpha\in E$).
\end{lemma}
\begin{proof}
	The turbulence polyhedron $\F_1(\tL)$ is defined within the vector space $H:=H_U\cap(\cap_{w\in V}H_w)$ by the half-spaces $H_\alpha^{\geq0}$ for every $\alpha\in{E}$. It follows that every facet of $\F_1(\tL)$ is obtained by intersecting with some hyperplane $H_\alpha$, which amounts to requiring zero flow through an arrow $\alpha$.
\end{proof}

Example~\ref{ex:facetweird} shows that not all half-spaces $H_\alpha^{\geq0}$ define facets of $\F_1(\tL)$.

\begin{example}\label{ex:facetweird}
	Consider the fringed quiver $\tL$ and turbulence polyhedron $\F_1(\tL)$ of Figure~\ref{fig:facetweird}. If the coordinate space $\mathbb R^E$ is labelled $(e_1,e_2,e_3,e_4,e_5)$, then the point of $\F_1(\tL)$ labelled $e_1^2e_2e_3$, for example, is $(2,1,1,0,0)$.
	Note that every point of the turbulence polyhedron has a 1 in the coordinate of $e_3$, hence
	$H_{e_3}^{\geq0}$ is not a facet-defining hyperplane of $\F_1(\tL)$.
	\begin{figure}
		\centering
		\def\svgscale{.21}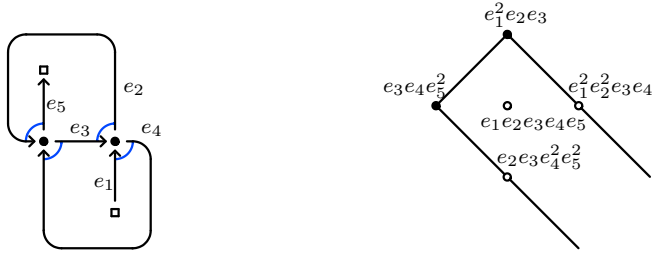
		\caption{A fringed quiver and its two-dimensional turbulence polyhedron.}
		\label{fig:facetweird}
	\end{figure}
\end{example}

\subsection{Extended strings}

We now define irrational strings, which are paths which behave like routes or bands which continue on indefinitely in one or both directions without repeating themselves.

\begin{defn}\label{defn:ext-complex}
	An \emph{irrational string} is an infinite walk $(a=\dots\alpha_{-1}^{\e_{-1}}\alpha_0^{\e_0}\alpha_1^{\e_1}\dots)$, $(a=\alpha_0^{\e_0}\alpha_1^{\e_1}\dots)$, or $a=(\dots\alpha_{-1}^{\e_{-1}}\alpha_0^{\e_0})$ such that
	\begin{enumerate}
		\item $\alpha_1^{\e_1}$ is defined if and only if $h(\alpha_0^{\e_0})$ is not fringe and $\alpha_{-1}^{\e_{-1}}$ is defined if and only if $t(\alpha_0^{\e_0})$ is not fringe,
		\item for any $j\in\mathbb Z$ such that $\alpha_j$ and $\alpha_{j+1}$ are defined, the walk $\alpha_j^{\e_j}\alpha_{j+1}^{\e_{j+1}}$ is a string, and
		\item there is no band $B$ such that $a=\dots B B B\dots$ is an infinite repetition of $B$.
	\end{enumerate}
	We consider the irrational string $a^{-1}$ given by reversing the direction and order of all arrows to be \emph{equivalent} to $a$.
	A \emph{substring} of $a$ is a (finite) string of the form $\alpha_j^{\e_j}\dots\alpha_k^{\e_k}$, for any $j\leq k$; it is a \emph{top substring} if $\e_{k+1}=1$ and $\e_{j-1}=-1$, and a \emph{bottom substring} if $\e_{k+1}=-1$ and $\e_{j-1}=-1$. An \emph{extended string} is a string, band, or irrational string.
	Generalizing Definition~\ref{defn:tilt-complex}, two extended strings $a$ and $b$ \emph{kiss} if there is a string $\sigma$ which is a top substring of $a$ and a bottom substring of $b$, without loss of generality. We call $\sigma$ an \emph{incompatibility} between $a$ and $b$. If $a$ and $b$ do not kiss, then they are \emph{compatible}.
\end{defn}

An extended string is \emph{self-compatible} if it is compatible with itself.
We will work with irrational strings as a technical tool. However, we are chiefly interested in the routes and bands.

\subsection{The bundle complex}

We now define the bundle complex of $\tL$, which is similar to the clique complex on $\tL$ but allows bands.

\begin{defn}
	A \emph{trail} of $\tL$ is a band or route of $\tL$. 
\end{defn}
Since $\tL$ is a finite-dimensional algebra, no band is an oriented path. We then say that a \emph{straight trail} is a straight route and a \emph{bending trail} is a band or a bending route.
Two trails are \emph{compatible} if they are compatible as extended strings as in Definition~\ref{defn:ext-complex}.
A trail is \emph{self-compatible} if it is compatible with itself.

\begin{defn}
	A set $\bK$ of pairwise compatible (self-compatible) trails is a \emph{bundle}. We consider bundles up to equivalence of routes and bands. The \emph{bundle complex} of $\tL$ is the simplicial complex of bundles.
	A bundle is \emph{reduced} if it contains no straight trails.
	We often write a bundle as $\bK=\K\cup\B$, where ${\K}$ is a set of routes and $\B$ is a set of bands.
\end{defn}

We remark that we recover the clique complex of $\tL$ from the bundle complex of $\tL$ by restricting to the cliques (i.e., excluding the bands).

\begin{example}\label{ex:kroncompat}
	In the Kronecker fringed quiver $\tL$ of Figure~\ref{fig:kron1}, there is only one band $B:=e_2f_2^{-1}$. This band is compatible with the straight routes $p:=e_1e_2e_3$ and $q:=f_1f_2f_3$, but is incompatible with every other route of $\tL$. For example, the string $e_2f_2^{-1}$ is an incompatibility between the band $B$ and the route $e_1e_2f_2^{-1}f_1^{-1}$. Hence, the bundle complex of $\tL$ consists of the clique complex as well as the bundles $\left\{\{B\},\ \{B, p\},\ \{B, q\},\ \{B, p, q\}\right\}$. Note that the bundle complex is not pure, as the maximal cliques have cardinality 4 while the maximal bundle containing a band has cardinality 3.
\end{example}

\subsection{Distinguished walks and the countercurrent order}

We loosely follow~\cite[\S2.2.2]{PPP}, though we make our definitions for extended strings rather than routes.

\begin{defn}
	A \emph{marked trail} of $\tL$ is a trail $p=\alpha_1^{\e_1}\dots\alpha_m^{\e_m}$ with a unique marked signed arrow $\alpha_i^{\e_i}$. If $p$ contains multiple instances of $\alpha_i^{\e_i}$, only one is marked. We say that $p$ is \emph{marked at $\alpha_i^{\e_i}$}.
	Similarly, a \emph{marked irrational string} of $\tL$ is an irrational string with a unique marked signed arrow $\alpha_i^{\e_i}$. A \emph{marked extended string} is a marked trail or marked irrational string.
\end{defn}

We remark that if a trail $p$ is marked at $\alpha_i^{-\e_i}$, then we may consider it to be marked at $\alpha_i^{\e_i}$ by reversing the string walk of $p$ to get $p^{-1}$ (which does not change the route up to equivalence). In fact, we consider the marked trail $p$ at some signed arrow $\alpha^\e$ to be equivalent to the marked trail $p^{-1}$ at the corresponding signed arrow $\alpha^{-\e}$.

\begin{defn}\label{defn:mark}
If $\bK$ is a bundle of $\tL$, let $\bK_{\alpha^\e}$ be the set of trails of $\bK$ marked at ${\alpha^\e}$.
\end{defn}

We remark that if $\bK$ contains a trail $p$ which uses some arrow $\alpha$ more than once, then there will be multiple ways to mark $p$ at $\alpha^\e$. All of these markings appear as separate elements of $\bK_{\alpha^\e}$.

\begin{defn}\label{defn:wkordG}
	Let $\alpha^\e$ be a signed arrow of $\tilde\Lambda$. We define the (weak) \emph{post-$\alpha^\e$-order} $\prec_{\alpha^\e}^+$ on the set of extended strings marked at $\alpha^{\pm1}$.
	Let $p$ and $q$ be extended strings marked at $\alpha^{\pm1}$. By possibly replacing $p$ with $p^{-1}$ and/or $q$ with $q^{-1}$, we may suppose that $p$ and $q$ are both marked at ${\alpha^\e}$.
	If $p$ and $q$ are equal after their marked occurrences of $\alpha^\e$, then we say that $p$ and $q$ are equal in the post-$\alpha^\e$ order and write $p=_{\alpha^\e}^+q$. Otherwise, let $\sigma$ be the maximal common substring of $p$ and $q$ beginning with the marked copies of $\alpha^\e$. Since $p$ and $q$ diverge after this common substring $\sigma$, it follows that one of these extended strings must proceed after $\sigma$ backwards along some arrow $\gamma$ and the other must proceed after $\sigma$ forwards along a different arrow $\sigma'$. Suppose without loss of generality that $q$ contains the substring $\sigma\gamma^{-1}$ and $p$ contains the substring $\sigma\gamma'$. We then say that $p\prec_{\alpha^\e}^+q$.

	Dually, we define the (weak) \emph{pre-${\alpha^\e}$-order} $\prec_{\alpha^\e}^-$ on the set of extended strings marked at $\alpha^{\pm1}$.
	Let $p$ and $q$ be extended strings marked at $\alpha^\e$. If these extended strings agree before their marked copies of $\alpha^\e$, then we say that $p=_{\alpha^\e}^-q$. Otherwise, let $\sigma$ be the maximal common substring of $p$ and $q$ ending with the marked copies of ${\alpha^\e}$. Say without loss of generality that $q$ contains the substring $\gamma^{-1}\sigma$ for some arrow $\gamma$, and $p$ contains the substring $\gamma'\sigma$ for some arrow $\gamma'$. We then say that $p\prec_{\alpha^\e}^-q$.
\end{defn}

Note that the straight route through $\alpha$ (marked at $\alpha$) is minimal with respect to both $\prec_\alpha^+$ and $\prec_\alpha^-$.
If $p$ and $q$ are compatible extended strings marked at ${\alpha^\e}$, then $p\preceq_{\alpha^\e}^+q$ if and only if $p\preceq_{\alpha^\e}^-q$. In this case, we say that $p\preceq_{\alpha^\e} q$.
Note that if in addition $p$ and $q$ are distinct, then $p\prec_{\alpha^\e} q$ or $q\prec_{\alpha^\e} p$.
Hence, if $\K\cup \B$ is a bundle, then $\prec_{\alpha^\e}$ gives a total order on the trails of $\K\cup \B$ marked at ${\alpha^\e}$ which we call the \emph{countercurrent order at $\alpha^\e$}.
When restricted to bundles of routes, this is the same as the countercurrent order of~\cite[Definition 2.18]{PPP}.
The lower a marked extended string at ${\alpha^\e}$ is in this partial order, the more positively oriented it stays around its marked copy of ${\alpha^\e}$ (i.e., the more it takes arrows $\alpha^1$ with positive sign).

\subsection{Cardinality of bundles}

It is known that the clique complex is pure.
On the other hand, recall from Example~\ref{ex:kroncompat} that the bundle complex may not be pure. We will now show that the cardinality of a bundle is bounded by
the cardinality of a maximal clique.

\begin{thm}[{\cite[Corollary 2.29]{PPP}, \cite[Theorem 5.1]{BDMTY}}]\label{thm:ccc}
	The cardinality of a maximal clique is $3|V_{\text{int}}|-|E_{\text{int}}|$. The cardinality of a maximal reduced clique is $|V_{\text{int}}|$.
\end{thm}

We will follow the proof of~\cite[Corollary 2.29]{PPP} while allowing bands to get $3|V_{\text{int}}|-|E_{\text{int}}|$ as an upper bound for the cardinality of a bundle.

\begin{defn}
	Let $\bK$ be a bundle of $\tL$. Let $p$ be a trail of $\bK$ and let $\alpha$ be an arrow of $p$. If $p$ or $p^{-1}$ marked at some copy of $\alpha$ is the $\prec_\alpha$-maximum element of $\bK_{\alpha^\e}$ (Definition~\ref{defn:mark}), then we say that $\alpha$ is a \emph{distinguished arrow of $p$ in $\bK$}.
\end{defn}

\begin{lemma}\label{lem:dist}
	Let ${\bK}$ be a bundle (which may have bands). Every straight route of ${\bK}$ has at least one distinguished arrow, and every bending route and band of ${\bK}$ has at least two distinguished arrows.
\end{lemma}
\begin{proof}
	We copy the proof of~\cite[Lemma 2.23]{PPP} with our notation; this lemma was proven for cliques, but works verbatim for bundles.
	It suffices to show that, if $p$ is a trail of $\bK$ and $\alpha^\e$ is a signed arrow of $p$, then there exists a signed arrow $\beta^\e$ of $p$ (with the same sign $\e$ as $\alpha^\e$) such that $\beta$ is a distinguished arrow of $p$.
	If necessary, replace $p$ with $p^{-1}$ so that $\e=1$.

	If $p$ is not maximal in $\bK$ for $\prec_\alpha$, consider $q:=\min_{\prec_\alpha}\{r\in \bK_\alpha\ :\ p\prec_\alpha r\}$. Let $\sigma$ denote the maximal common substring of $p$ and $q$ containing the marked copies of $\alpha$. Since $p$ and $q$ are distinct, they split at some endpoint of $\sigma$. Since $p\prec_\alpha q$, at this endpoint, the arrow $\beta$ of $p$ immediately after $\sigma$ points in the same direction as $\alpha$. Let $s$ be the $\prec_{\beta}$-maximal element of $\bK_{\beta}$ (so that $\alpha$ is a distinguished arrow of $s$ in $\bK$). If $s=p$, then we are done; if $s\neq p$, then $s$ either kisses $q$ or contradicts the minimality of $q$ among the walks $\{r\in \bK_\alpha\ :\ p\prec_\alpha r\}$.
\end{proof}

\begin{cor}\label{cor:card-upper-bound}
	The cardinality of a maximal bundle is less than or equal to $3|V_{\text{int}}|-|E_{\text{int}}|$. The cardinality of a maximal reduced bundle is less than or equal to $|V_{\text{int}}|$.
\end{cor}
\begin{proof}

	Let ${\bK}$ be a maximal (non-reduced) bundle of $\tL$. Let $b$ be the number of bending trails (including bands) of ${\bK}$ and let $s$ be the number of straight routes of ${\bK}$. By Lemma~\ref{lem:dist}, we must have $2b+s\leq|E|$. 

	Since ${\bK}$ is maximal, it contains all straight routes of $\tL$. Since each straight route contains precisely two fringe arrows, we have $s=\frac{1}{2}|E_{\text{fringe}}|=2|V_{\text{int}}|-|E_{\text{int}}|$.
	The inequality of the previous paragraph then gives $2b\leq |E|-s=|E|-2|V_{\text{int}}|+|E_{\text{int}}|=2|E_{\text{int}}|+|E_{\text{fringe}}|-2|V_{\text{int}}|$.
	Then 
	\[b\leq|E_{\text{int}}|+\frac{1}{2}|E_{\text{fringe}}|-|V_{\text{int}}|=|E_{\text{int}}|+\big(2|V_{\text{int}}|-|E_{\text{int}}|\big)-|V_{\text{int}}|=|V_{\text{int}}|.\]
	Since every maximal reduced bundle arises as a maximal bundle without its straight routes, we have shown the second statement. To show the first, we calculate
	$b+s=b+2|V_{\text{int}}|-|E_{\text{int}}|\leq 3|V_{\text{int}}|-|E_{\text{int}}|$.
\end{proof}

When ${\K}$ is a maximal clique (i.e., has no bands), every bending trail of ${\K}$ has \emph{exactly} two distinguished arrows and every straight route of $\K$ has \emph{exactly} one distinguished arrow~\cite[Proposition 2.28]{PPP}. On the other hand, this is not true for maximal bundles containing bands -- for example, both straight routes of the maximal bundle $\{e_1e_2e_3,f_1f_2f_3,e_2f_2^{-1}\}$ of the Kronecker fringed quiver of Figure~\ref{fig:kron1} have two distinguished arrows. This is why Corollary~\ref{cor:card-upper-bound} gives an upper bound whereas Theorem~\ref{thm:ccc} gives an equality.

\subsection{Paired and representation-finite gentle algebras}
\label{sec:PRF}

Gentle algebras and their turbulence polyhedra act as a generalization of amply framed DAGs and their flow polytopes. In this section, we will isolate two properties of gentle algebras that do not appear in amply framed DAGs in the form of non-pairedness and representation-infiniteness.
We first focus on the latter.
Note that
\[\Lambda\text{ representation-infinite}\iff\tL\text{ representation-infinite}\iff \tL\text{ has a band}.\]
In other words, representation-infiniteness is equivalent to the existence of a string which contains some arrow $\alpha$ multiple times with the same sign.
Of course, such behavior is impossible in an amply framed DAG due to the acyclicity condition.
We will see in Corollary~\ref{cor:finite-iff-bound} that representation-finiteness of $\tL$ is equivalent to boundedness of $\F_1(\tL)$. While it is possible to prove this directly at this time, we content ourselves to cite it here to save space.
We now define paired gentle algebras.

\begin{defn}
	A gentle algebra $\Lambda=\k Q/I$ is \textit{paired} if there is a map $\psi$ from the arrows of $Q$ to $\{1,2\}$ such that 
		 if $\alpha,\beta\in Q_1$ with $h(\alpha)=t(\beta)$ then $\alpha\beta\in I\iff \psi(\alpha)\neq\psi(\beta)$.
	We say that $\psi$ is a \emph{pairing function} of $\Lambda$.
\end{defn}

It is immediate that $\Lambda$ is paired if and only if $\tL$ is paired.
See the left of Figure~\ref{fig:kron-dg} for one of two possible pairings on the Kronecker fringed quiver.
Since arrows of an amply framed DAG may be labelled by elements of $\{1,2\}$, any gentle algebra coming from an amply framed DAG as in Definition~\ref{defn:gent-from-dag} is paired.
We also remark that a paired gentle algebra may not have loops (i.e., arrows starting and ending at the same vertex). Indeed, if $\alpha$ is a loop, then since $\psi(\alpha)$ must be well-defined we must have $\alpha\alpha\not\in I$. Then any power $\alpha^m$ of $\alpha$ is an element of $\Lambda$, hence $\Lambda$ is not finite dimensional as a vector space, a contradiction.

In practice, it is not difficult to verify whether or not a gentle algebra is paired. Begin by choosing any arrow $\alpha$ and giving it some label in $\{1,2\}$. The labeling of $\alpha$ determines labelings of all other arrows incident to $h(\alpha)$ or $t(\alpha)$. By iterating, this determines a label for every arrow of the connected component of $\Lambda$ containing $\alpha$. Do this for every connected component. If the result is a valid pairing, then $\Lambda$ is paired; otherwise, $\Lambda$ is not paired. One also observes that there are exactly two pairings of a connected paired gentle algebra, as switching the 1's and 2's of a pairing yields another pairing.

We now give examples showing that the properties of pairedness and representation-finiteness of a gentle algebra or fringed quiver may be independently varied.
Figure~\ref{fig:kron1} gives a fringed quiver which is paired but not representation-finite. Figure~\ref{fig:facetweird} gives a fringed quiver which is neither representation-finite nor paired. Figure~\ref{fig:nonv-turb} gives a fringed quiver which is representation-finite but not paired. The right of Figure~\ref{fig:difdagc} shows a fringed quiver which is representation-finite and paired.
\begin{figure}
	\centering
	\def\svgscale{.21}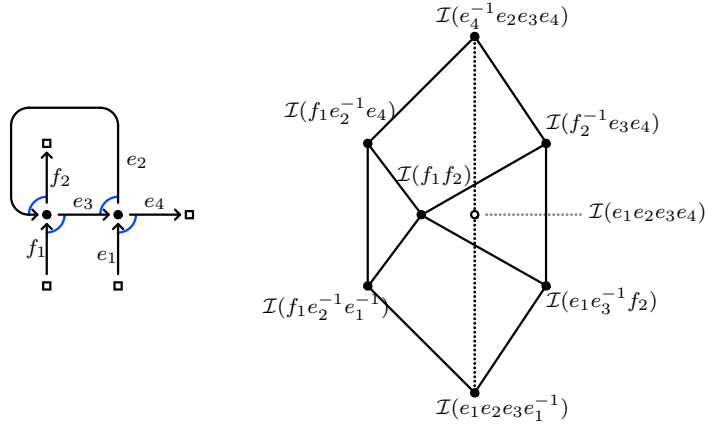
	\caption{The ``shard'' fringed quiver which is representation-finite and nonpaired.}
	\label{fig:nonv-turb}
\end{figure}

\subsection{Paired gentle algebras and amply framed directed graphs}
\label{ssec:PGAaAFDG}

We connect paired gentle algebras with a new definition of \emph{amply framed directed graphs}, which is a generalization of amply framed DAGs allowing cycles.

\begin{defn}\label{defn:framed-dg}
	An \emph{internal} vertex of a directed graph is one which is neither a source nor a sink. An \emph{internal} edge is an edge between two internal vertices. A \emph{framed directed graph} $(G,\mathfrak{F})$ is a directed graph $G$ and a \emph{framing} $\mathfrak{F}$. The framing $\mathfrak{F}$ consists of the data of
	 a linear order to the edges in $\inn(v)$ and a linear order to the edges in $\out(v)$ for every internal vertex $v$. If $e$ is less than $f$ in the linear order for $\mathfrak{F}$ on $\inn(v)$, we write $e<_{\mathfrak{F},\inn(v)}f$ (and similarly for $\out(v)$). When $\mathfrak{F}$ and/or $\inn(v)$ or $\out(v)$ is clear, we may drop one or both subscripts.
\end{defn}

\begin{defn}\label{defn:amply-framed-dg}
	A directed graph is \emph{full} if every internal vertex is incident to precisely two outgoing edges and two incoming edges.
	An \emph{amply framed directed graph} $(G,\mathfrak{F})$ is a framed directed graph $G$ such that
	\begin{enumerate}
		\item the directed graph $G$ is full,
		\item there is a map $\psi_\mathfrak{F}:E\to\{1,2\}$ such that, if $\alpha$ and $\beta$ are edges with the same internal source or target, we have $\alpha<_{\mathfrak{F}}\beta$ (in the partial order of $\inn(h(\alpha))$ or $\out(t(\alpha))$) if and only if $\psi_\mathfrak{F}(\alpha)<\psi_\mathfrak{F}(\beta)$, and
		\item\label{dgh3} if $e_1e_2\dots e_m$ is an oriented cycle of $G$, then the image of $\{e_1,\dots,e_m\}$ under $\psi_\mathfrak{F}$ is $\{1,2\}$.
	\end{enumerate}
\end{defn}

We quickly make a number of definitions extending notions defined on framed DAGs to framed directed directed graphs.
We may consider the \emph{flow polyhedron} $\F_1(\Gamma)$ as defined in Definition~\ref{defn:flow-polytope}.
A \emph{route} of $\Gamma$ is a path from a source to a sink, and a \emph{band} is an oriented cycle (up to cyclic equivalence) which is not a power of a smaller cycle. A \emph{trail} is a route or band.
We may define a notion of compatibility of trails of $\Gamma$ as in Definition~\ref{defn:compatible}. A \emph{clique} is a collection of pairwise compatible routes, and a \emph{bundle} is a collection of pairwise compatible trails.
Cliques and bundles form the \emph{clique complex} and \emph{bundle complex} of $\Gamma$, respectively.
A route is \emph{exceptional} if it consists entirely of 1-edges or entirely of 2-edges, and a clique is \emph{reduced} if it contains no exceptional routes.

We immediately remark that a (resp. amply) framed directed graph $\Gamma=(G,\mathfrak{F})$ is a (resp. amply) framed DAG if and only if $G$ is acyclic, in which case the flow polyhedron of $\Gamma$ is equal to the flow polytope of $\Gamma$.

\begin{remk}\label{remk:UQAM}
	{
	Intuitively, amply framed directed graphs generalize amply framed DAGs (Definition~\ref{defn:amp}) by allowing cycles under the condition that every cycle must contain edges of labels 1 and 2.
	A concurrent paper in progress~\cite{UQAM} studies an orthogonal generalization of amply framed directed graphs which allows cycles under the condition that no cycle must contain edges of labels 1 and 2 (i.e., every cycle contains edges of only one label).
	Among other things, they obtain analogs of our clique triangulation results in this setting.
	}
\end{remk}

We now define two extra conditions which we will assume in addition to ample framedness.

\begin{defn}
	We find a (resp. framed) directed graph $\Gamma$ \emph{convenient} if every source vertex and every sink vertex of $\Gamma$ is incident to precisely one edge.
\end{defn}

Given an arbitrary framed directed graph $\Gamma$, we may separate the source and sink vertices until each is incident to only one edge; this does not affect flow polyhedron, the routes, or the notion of compatibility on routes, so we may assume convenience of our framed directed graphs without meaningfully affecting the class of framed directed graphs.
We now define a condition which does have a mild effect on the class of flow polyhedra we consider.

\begin{defn}\label{defn:gentframe}
	A framed directed graph $\Gamma$ is \emph{gently framed} if $\Gamma$ is amply framed and has no edge directly from a source vertex to a sink vertex.
\end{defn}

Adding an edge to a framed DAG directly from a source to a sink will simply add a straight route to each maximal clique (resp. bundle) of the clique (resp. bundle) complex, and the resulting flow polyhedron will be a cone over the original flow polyhedron. So, Definition~\ref{defn:gentframe} does not meaningfully limit what clique complexes may look like and has only a mild effect on what flow polyhedra may look like.

Note that the map of Definition~\ref{defn:gent-from-dg} works particularly well on gently framed DAGs, as the edges of a gently framed DAG $\Gamma$ are in bijection with the arrows of $\tL(\Gamma)$. On the other hand, if an amply framed DAG $\Gamma$ has an edge $\alpha$ directly from the source to the sink then $\tL(\Gamma)$ has no arrow analogous to $\alpha$.
Building on this idea, we define a map from convenient gently framed directed graphs to fringed quivers.

\begin{defn}\label{defn:gent-from-dg}
	Let $\Gamma=(G,\mathfrak{F})$ be a convenient gently framed directed graph. Extending Definition~\ref{defn:gent-from-dag}, we associate to $\Gamma$ a fringed quiver $\tL(\Gamma)$ as follows.
	The vertices of $\tL(\Gamma)$ are in bijection with vertices of $\Gamma$. For each 1-edge $\alpha$ of $\Gamma$, there is an arrow $\tilde\alpha:t(\alpha)\to h(\alpha)$ of $\tL(\Gamma)$. For each 2-edge $\beta$ of $\Gamma$, there is an arrow $\tilde\beta:h(\beta)\to t(\beta)$ of $\tL(\Gamma)$. The relations of $I$ are all pairs of the form $\tilde\alpha\tilde\beta$, where either $\alpha$ is a 1-edge and $\beta$ is a 2-edge, or $\alpha$ is a 2-edge and $\beta$ is a 1-edge. Define a function $\psi_{\tL(\Gamma)}$ from arrows of $\tL$ to $\{1,2\}$ by $\psi_{\tL(\Gamma)}(\tilde\alpha)=\psi_{\mathfrak{F}}(\alpha)$.
\end{defn}

\begin{lemma}\label{lem:lll}
	If $\Gamma=(G,\mathfrak{F})$ is a convenient gently framed directed graph, then $\tL(\Gamma)$ is a fringed quiver with pairing function $\psi_{\tL(\Gamma)}$.
\end{lemma}
We leave the proof of Lemma~\ref{lem:lll} to the reader. We note specifically that Condition~\eqref{dgh3} of Definition~\ref{defn:amply-framed-dg} requiring that no cycle of $\Gamma$ is monolabelled implies that no band of $\tL(\Gamma)$ is oriented. This means that $\tL(\Gamma)$ is a finite-dimensional algebra, which is required for gentle algebras.

\begin{lemma}
	If $\Gamma$ is a convenient gently framed directed graph, then $\F_1(\Gamma)=\F_1(\tL(\Gamma))$.
\end{lemma}
\begin{proof}
	Immediate from the definitions.
\end{proof}

We now define a reverse of Definition~\ref{defn:gent-from-dg}:

\begin{defn}\label{defn:dg-from-gent}
	Suppose $\tL$ is paired with pairing function $\psi:E\to\{1,2\}$. We define a gently framed directed graph $\Gamma_\psi(\tL):=(G,\mathfrak{F})$ of $\tL$, which we call the \emph{directed flow-graph} of $\tL$.
	The vertex set of $\Gamma_\psi(\tL)$ is the vertex set of $\tL$. For any arrow $\alpha$ of $\tL$ with $\psi(\alpha)=1$, there is a 1-edge $\alpha':t(\alpha)\to h(\alpha)$ of $\Gamma_\psi(\tL)$, and for any arrow $\beta$ of $\tL$ with $\psi(\beta)=2$, there is a 2-edge $\beta':h(\beta)\to t(\beta)$ of $\Gamma_\psi(\tL)$. See Figure~\ref{fig:kron-dg}.
\end{defn}
See Figure~\ref{fig:kron-dg} for an example of this construction. The framing function $\psi$ is shown in red.
We remark that changing the choice of framing function $\psi$ by switching the 1's and 2's has the effect of switching the orientation and label of all arrows of $\Gamma_\psi(\tL)$. This does not affect the clique complex of $\Gamma_\psi(\tL)$ or the space of flows on $\Gamma_\psi(\tL)$.
\begin{figure}
	\centering
	\def\svgscale{.21}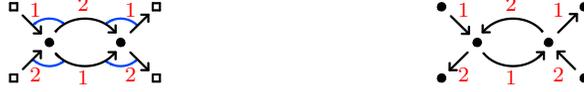
	\caption{The Kronecker fringed quiver and a directed flow-graph.}
	\label{fig:kron-dg}
\end{figure}

\begin{prop}\label{lem:god}
	The maps $\Gamma\mapsto(\tL(\Gamma),\psi_{\tL(\Gamma)})$ and $(\tL,\psi)\mapsto\Gamma_\psi(\tL)$ are mutual inverses realizing a bijection between paired gentle algebras with a pairing function and convenient gently framed directed graphs which preserves the space of unit flows.
	Moreover, $\Gamma_\psi(\tL)$ is a gently framed DAG if and only if $\tL$ is representation finite.
\end{prop}
\begin{proof}
	We leave to the reader the proof that these maps are mutual inverses.
	Since routes of $\Gamma_\psi(\tL)$ are in natural bijection with the routes of $\tL$, we have that
	\begin{align*}
		\Gamma\textup{ is acyclic}&\iff\textup{ there are a finite number of routes of }\Gamma\\
		&\iff\textup{ there are a finite number of routes of }\tL\\
		&\iff\ \tL\textup{ is representation finite}.
	\end{align*}
\end{proof}

\begin{cor}
	The fringed quiver $\tL$ appears as $\tL(\Gamma)$ for some gently framed DAG if and only if $\tL$ is paired and representation-finite.
\end{cor}

Proposition~\ref{lem:god} gives a two-to-one correspondence between convenient gently framed directed graphs and paired gentle algebras.
In future sections, we will see that the turbulence polyhedra of a paired gentle algebra $\tL$ is a well-behaved object which is deeply connected with the bundle complex of $\tL$.
In particular, the clique complex (respectively bundle complex) of $\tL$ gives a triangulation (respectively subdivision) of a dense subset of $\F_1(\tL)$.
This, through Proposition~\ref{lem:god}, will give analogous triangulations and subdivisions for convenient gently framed directed graphs.
Since convenient gently framed directed graphs capture most behavior of amply framed directed graphs, this can in turn be used to obtain corresponding results about general amply framed directed graphs.

In the above, we have focused on paired gentle algebras.
On the other hand, if $\tL$ is not paired, we cannot associate to $\tL$ a directed flow-graph
by Proposition~\ref{lem:god}.
In other words, we cannot give a notion of direction to the flow through each arrow of $\tL$. 
For example, a string walk may contain an arrow $\alpha$ as well as its inverse $\alpha^{-1}$, such as the string walk $e_1e_2e_3e_1^{-1}$ of the paired and representation-finite fringed quiver of Figure~\ref{fig:nonv-turb}. We intuitively interpret this as flow being sent ``in both directions'' through $e_1$; one cannot resolve which direction flow should be sent through $e_1$. This is a new behavior allowed in non-paired gentle algebras which we interpret as a relaxation of the condition of directedness.

We summarize the information of this subsection with the following table.

\begin{table}[htbp]
    \centering
    \begin{tabular}{|c||c|c|}
        \hline
	    & \begin{tabular}{c}
		     \textbf{Representation-Finite} \\
		     ($\F_1(\tL)$ bounded)
	    \end{tabular}
	     & \begin{tabular}{c}
		    \textbf{Representation-Infinite} \\
		    ($\F_1(\tL)$ unbounded)
	     \end{tabular} \\
        \hline
	\hline
	    \begin{tabular}{c}
		    \textbf{Paired} \\
		    (directed flow-graph)
		    \end{tabular} & \begin{tabular}{c}bounded,\\directed flow-graph is a\\gently framed DAG\\\textbf{(flow polytope)}\end{tabular} & \begin{tabular}{c}unbounded,\\directed flow-graph is an\\gently framed directed graph\\\textbf{(flow polyhedron)}\end{tabular} \\
	\hline
	    \begin{tabular}{c}
		    \textbf{Nonpaired} \\
		    (no directed flow-graph)
		    \end{tabular} & \begin{tabular}{c}bounded,\\no directed flow-graph\end{tabular} & \begin{tabular}{c}unbounded,\\no directed flow-graph\end{tabular}\\
        \hline
    \end{tabular}
\end{table}

\section{The Flow Algorithm and Bundle Subdivisions}
\label{sec:triang}

In Theorem~\ref{thm:dkk}, we saw that the clique complex of a framed DAG $\Gamma$ induces a unimodular triangulation on its flow polytope $\F_1(\Gamma)$ through taking indicator vectors. 
	    Given a convenient gently framed DAG $\Gamma$, we may consider the gentle algebra $\tL(\Gamma)$
	    (Definition~\ref{defn:gent-from-dg}).
	    Since the clique complex and turbulence polyhedron of $\tL(\Gamma)$ agree with the clique complex and flow polytope of $\Gamma$, it then follows that the clique complex of $\tL$ induces a unimodular triangulation on the flow polytope $\F_1(\tL)$.
	    We wish to extend this result to arbitrary fringed quivers by showing that the clique complex induces a unimodular triangulation on the turbulence polyhedron. When $\tL$ is representation-finite, this triangulation will cover every point of $\F_1(\tL)$, though this will not be the case when $\tL$ is representation-infinite. By considering the larger bundle complex, we will obtain a larger subdivision which covers all rational points of the turbulence polyhedron.

Our strategy will be to develop a \emph{flow algorithm} which takes a flow $F$ and obtains a representation of $F$ as a combination of indicator vectors of pairwise compatible trails, if such a representation exists, and indicates when it doesn't exist.

\begin{defn}
	Let $\bK=\K\cup \B$ be a bundle. A \emph{$\bK$-bundle combination} is a linear combination
	\[F=\sum_{p\in \bK}a_p\I(p)\]
	of indicator vectors of trails, such that each $a_p$ is nonnegative. The \emph{strength} of the bundle combination is $\sum_{p\in \K}a_p$ (note we iterate only over routes $\K$, not all trails $\bK$). Note that this is the strength of the resulting flow $F$. The bundle combination is \emph{unit} if its strength is 1. In other words, a unit bundle combination of $\bK$ is a convex combination of indicator vectors of routes in ${\K}$ plus a nonnegative combination of indicator vectors of bands in $\B$.
	A bundle combination is a \emph{clique combination} if $a_p$ is zero for every $B\in \B$ (in particular, this is the case when $\B=\emptyset$). It is \emph{positive} if $a_p$ is nonzero for every $p\in \bK$.
\end{defn}

\begin{defn}
	Let $\bK$ be a bundle with at least one route.
		    The \emph{bundle space} $\Delta_{\geq0}({\bK})$ of $\bK$ is the space of bundle combinations of ${\bK}$. The \emph{unit bundle space} $\Delta_1(\bK)$ is the space of unit bundle combinations of $\bK$.

		    We introduce some extra terminology to talk about unit bundle spaces.
		If ${\bK}$ has no bands, then $\Delta_1({\bK})$ is a simplex in $\mathbb R^{E}$ and we call $\Delta_1({\bK})$ a \emph{clique simplex}. In this case, we call $\Delta_{\geq0}(\bK)$ a \emph{clique cone}. Otherwise $\Delta_1({\bK})$ is unbounded and we call $\Delta_1({\bK})$ (and $\Delta_{\geq0}(\bK)$) a \emph{bundle wall}. The bundle space $\Delta_1(\bK)$ is \emph{maximal} if $\bK$ is maximal.
		    Note that if $\bK$ has only bands, and no routes, then $\Delta_1(\bK)=\Delta_{\geq0}(\bK)$.
\end{defn}
Figures~\ref{KRONINTRO} and~\ref{GEMINTRO} show examples of clique simplices.
	    Figure~\ref{KRONINTRO} has one maximal (two-dimensional) bundle wall whose vertices are the top and bottom vertices of the polyhedron and whose recession cone is the ray moving to the right.

\subsection{Arrow-flows and the flow algorithm}

Our goal now is to show that any rational flow may be uniquely realized as a positive bundle combination. Before we explain our strategy, consider the following definition.

\begin{defn}
	We say that an \emph{arrow-flow} of a flow $F$ of $\tL$ is a tuple $(\alpha^\e,C)$, where $\alpha^\e$ is an arrow of $\tL$ and $C\in[0,F(\alpha)]$.
\end{defn}

An arrow-flow represents an infinitesimally small amount of flow which travels through a signed arrow $\alpha^\e$ from $t(\alpha^\e)$ to $h(\alpha^\e)$.  We will now give a way to decide where any given arrow-flow must be sent to in order to maintain compatibility, to the end of realizing a flow $F$ as a bundle combination. We call this the \emph{flow algorithm} and it is the object of this section to prove that it returns a bundle.

\begin{defn}\label{defn:forbac}
	Let $(\alpha^\e,C)$ be an arrow-flow of a flow $F$ of $\tL$.
	\begin{enumerate}
		\item\label{fb1} Suppose $h(\alpha^\e)$ is not fringe. Let $\alpha',\beta,\beta'$ be distinct arrows of $\tL$ such that $\alpha^\e(\beta')^\e$ and $\beta\alpha'$ are relations of $h(\alpha^\e)$. In other words, such that $\alpha^\e\alpha',\beta(\beta')^{\e},\alpha^\e\beta^{-1}$ are strings of $\tL$. See Figure~\ref{fig:bet} to see how this looks for $\e\in\{1,-1\}$. Then we say that \[\for_F(\alpha^\e,C)=
			\begin{cases}
				(\alpha',C) & C\leq F(\alpha')\text{ and }\e=1\\
				(\beta^{-1},C-F(\alpha')) & C>F(\alpha')\text{ and }\e=1\\
				(\alpha',C+F(\beta')) & C+F(\beta')< F(\alpha')\text{ and }\e=-1\\
				(\beta^{-1},C+F(\beta')-F(\alpha')) & C+F(\beta')\geq F(\alpha')\text{ and }\e=-1.
			\end{cases}\]
		\item Suppose $t(\alpha^\e)$ is not fringe. Let $\alpha',\beta,\beta'$ be distinct arrows of $\tL$ such that
			$(\beta')^\e\alpha^\e$ and $\alpha'\beta$ are relations of $t(\alpha^\e)$.
Then we say that \[\bac_F(\alpha^\e,C)=
			\begin{cases}
				(\alpha',C) & C\leq F(\alpha')\text{ and }\e=1\\
				(\beta^{-1},C-F(\alpha')) & C>F(\alpha')\text{ and }\e=1\\
				(\alpha',C+F(\beta')) & C+F(\beta')< F(\alpha')\text{ and }\e=-1\\
				(\beta^{-1},C+F(\beta')-F(\alpha')) & C+F(\beta')\geq F(\alpha')\text{ and }\e=-1.
			\end{cases}\]
	\end{enumerate}
	\begin{figure}
		\centering
		\def\svgscale{.21}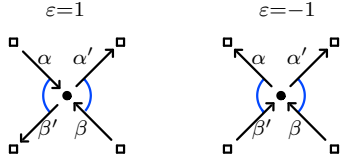
		\caption{Arrows labelled as in Definition~\ref{defn:forbac}~\eqref{fb1}.}
		\label{fig:bet}
	\end{figure}
	When the flow $F$ is clear by context, we omit the subscript $F$.
\end{defn}

Consider the fringed quiver $\tL$ with flow $F$ on the left of Figure~\ref{fig:forbac}. For $C\in[0,\frac{1}{2}]$, we have $\for(\alpha,C)=(\alpha',C)$, whereas for $C\in(\frac{1}{2},1]$, we have $\for(\alpha,C)=(\beta^{-1},C-\frac{1}{2})$.

\begin{lemma}
	Let $(\alpha^\e,C)$ be an arrow-flow of $F$. If $h(\alpha^\e)$ is not fringe, then $\for_F(\alpha^\e,C)$ is an arrow-flow of $F$. If $t(\alpha^\e)$ is not fringe, then $\bac_F(\alpha^\e,C)$ is an arrow-flow of $F$.
\end{lemma}
\begin{proof}
	We prove the $\for$ statement; the $\bac$ statement is symmetric.
	First, suppose $\e=1$.
	If $C\leq F(\alpha')$ then it is immediate that $\for_F(\alpha^\e,C)=(\alpha',C)$ is an arrow-flow of $F$. If $C>F(\alpha)$, then $C\leq F(\alpha)\leq F(\alpha)+F(\beta')=F(\alpha')+F(\beta)$ implies that $C-F(\alpha')\leq F(\alpha')+F(\beta)-F(\alpha')=F(\beta)$ so that $\for_F(\alpha^\e,C)=(\beta^{-1},C-F(\alpha'))$ is an arrow-flow of $F$.

	Now suppose $\e=-1$. If $C+F(\beta')<F(\alpha')$ then it is immediate that $\for_F(\alpha^\e,C)=(\alpha',C+F(\beta'))$ is an arrow-flow of $F$. If $C+F(\beta')\geq F(\alpha')$, then $C+F(\beta')\leq F(\alpha)+F(\beta')=F(\alpha')+F(\beta)$ implies that $C+F(\beta')-F(\alpha')\leq F(\alpha')+F(\beta)-F(\alpha')=F(\beta)$ so that $\for_F(\alpha^\e,C)=(\beta^{-1},C+F(\beta')-F(\alpha'))$ is an arrow-flow of $F$.
\end{proof}

\begin{lemma}\label{lem:norep}
	Let $(\alpha^\e,C)$ be an arrow-flow of $F$ and let $C\in[0,F(\alpha)]$.
	If $h(\alpha^\e)$ is not fringe, then $\bac(\for(\alpha^\e,C))=(\alpha^\e,C)$. If $t(\alpha^\e)$ is not fringe, then $\for(\bac(\alpha^\e,C))$.
\end{lemma}
\begin{proof}
	We prove that if $h(\alpha^\e)$ is not fringe, then $\bac(\for(\alpha^\e,C))=(\alpha^\e,C)$. We separate into cases based on which case is triggered in the definition of $\for(\alpha^\e,C)$.
	\begin{enumerate}
		\item If $C\leq F(\alpha')$ and $\e=1$, then $\for(\alpha^\e,C)=(\alpha',C)$. Then $\bac(\alpha',C)=\for(\alpha,C)$ through the first branch, since $C\leq F(\alpha)$.
		\item If $C>F(\alpha')$ and $\e=1$, then $\for(\alpha^\e,C)=(\beta^{-1},C-F(\alpha')$ through the second branch. Then $\bac(\beta^{-1},C-F(\alpha'))=(\alpha,C-F(\alpha')+F(\alpha'))=(\alpha,C)$ through the third branch.
		\item If $C+F(\beta')\leq F(\alpha')$ and $\e=-1$, then $\for(\alpha^\e,C)=(\alpha',C+F(\beta'))$ through the third branch. Then through the second branch, $\bac(\alpha',C+F(\beta'))=(\alpha^\e,C+F(\beta')-F(\beta'))=(\alpha^\e,C)$.
		\item If $C+F(\beta')\geq F(\alpha')$ and $\e=-1$, then $\for(\alpha^\e,C)=(\beta^{-1},C+F(\beta')-F(\alpha'))$ by the fourth branch. Then by the fourth branch again, $\bac(\beta^{-1},C+F(\beta')-F(\alpha'))=(\alpha^\e,C)$.
	\end{enumerate}
	One may show symmetrically that if $t(\alpha^\e)$ is not fringe, then $\for(\bac(\alpha^\e,C))=(\alpha^\e,C)$.
\end{proof}

In light of Lemma~\ref{lem:norep}, one may equivalently take the definition of $\bac_F(\alpha^\e,C)$ to be the arrow-flow $(\beta^\f,D)$ such that $\for(\beta^\f,D)=(\alpha^\e,C)$.

Definition~\ref{defn:forbac} will be used to realize a rational flow $F$ as a bundle combination. The idea is that applying $\for$ and/or $\bac$ repeatedly to an arrow-flow $(\alpha^\e,C)$ of $F$ will walk along an extended string $p_{(\alpha^\e,C)}^F$, which will be a trail under reasonable conditions. Definition~\ref{defn:forbac} is constructed so that lower values of $C$ will give extended strings which are lower in the order $\prec_\alpha$ of Definition~\ref{defn:wkordG} (which totally orders the members of any bundle which pass through $\alpha$). There will be an interval $I$ (containing $C$) such that $p^F_{(\alpha^\e,C)}=p^F_{(\alpha^\e,D)}$ for $D\in I$; the length of this interval $I$ gives the coefficient of $\I(p^F_{(\alpha^\e,C)})$ in the bundle combination. See the following example for a very small case.
\begin{example}\label{ex:sm}
	Consider the fringed quiver $\tL$ with flow $F$ on the left of Figure~\ref{fig:forbac}. We wish to realize $F$ as a bundle combination. For $C\in[0,\frac{1}{2}]$, we have $\for(\alpha,C)=(\alpha',C)$, whereas for $C\in(\frac{1}{2},1]$, we have $\for(\alpha,C)=(\beta^{-1},C-\frac{1}{2})$. This indicates that $\frac{1}{2}\I(\alpha\alpha')$ and $\frac{1}{2}\I(\alpha\beta^{-1})$ should both be in our bundle combination realizing $F$.
	On the other hand, $\for((\beta')^{-1},C)=(\beta^{-1},C+\frac{1}{2})$ for any $C\in[0,1]$, so our bundle combination should also include $1\I((\beta')^{-1}\beta^{-1})$. Indeed, $F$ is realized as the bundle combination
	\[F=\frac{1}{2}\I(\alpha\alpha')+\frac{1}{2}\I(\alpha\beta^{-1})+\I((\beta')^{-1}\beta^{-1}).\]
	This is represented on the right of Figure~\ref{fig:forbac}.
	\begin{figure}
		\centering
		\def\svgscale{.21}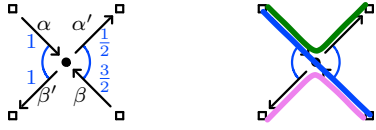
		\caption{On the left is a fringed quiver with a flow labelled in blue. On the right appear the paths indicated by the definitions of $\for$ and $\bac$.}
		\label{fig:forbac}.
	\end{figure}
\end{example}

We begin to formalize the above example with a definition.

\begin{defn}\label{defn:gr}
	Let $(\alpha^\e,C)$ be an arrow-flow of a flow $F$ of $\tL$.
	Take $(\alpha_0^{\e_0},C_0):=(\alpha^\e,C)$. Algorithmically proceed as follows. For any $(\alpha_j^{\e_j},C_j)$ such that $j\geq0$ and $h(\alpha_j^{\e_j})$ is not fringe, define $(\alpha_{j+1}^{\e_{j+1}},C_{j+1}):=\for(\alpha_j^{\e_j},C_j)$. For any $(\alpha_j,{\e_j},C_j)$ such that $j\leq0$ and $t(\alpha_j^{\e_j})$ is not fringe, define $(\alpha_{j-1}^{\e_{j-1}},C_{j-1}):=\bac(\alpha_j^{\e_j},C_j)$.
	The result is a sequence of arrow-flows $\{(\alpha_j^{\e_j},C_j)\ :\ j\in J\}$, where $J$ is some (possibly infinite) interval of integers containing $0$.

	\begin{enumerate}
		\item If $J$ is some finite interval $[m,m']$, then we define the route $p_{(\alpha^\e,C)}^F:=\alpha_{m}^{\e_m}\dots\alpha_{m'}^{\e_{m'}}$ marked at $\alpha^\e=\alpha_0^{\e_0}$.
		\item If $J$ consists of all integers, but there is some minimal positive integer $M$ such that $\alpha_j^{\e_j}=\alpha_i^{\e_i}$ if and only if $j\equiv i$ modulo $M$, then we define the band $p_{(\alpha^\e,C)}^F:=\alpha_0^{\e_0}\dots\alpha_{M-1}^{\e_{M-1}}$ marked at $\alpha_0^{\e_0}$.
		\item If neither of the above conditions are true, then $p_{(\alpha^\e,C)}^F:=\dots\alpha_j^{\e_j}\alpha_{j+1}^{\e_{j+1}}\dots$ is an irrational string (Definition~\ref{defn:ext-complex}) marked at $\alpha^\e$ (note that it may have a left or right endpoint, though not both). We typically wish to avoid this case.
	\end{enumerate}
	Let $I_{(\alpha^\e,C)}^F$ be the largest interval of $[0,F(\alpha)]$ containing $C$ such that $p_{(\alpha^\e,D)}^F=p_{(\alpha^\e,C)}^F$ (as marked paths) for any $D\in I$. Let $a_{(\alpha^\e,C)}^F$ be the length of the interval $I_{(\alpha^\e,C)}^F$. It is possible to have $a_{(\alpha^\e,C)}^F=0$, as shown in Example~\ref{ex:singleton}. We will often omit the superscripts $F$ when the underlying flow is clear by context.
\end{defn}

It is immediate from Definition~\ref{defn:forbac} that, in the notation of the definition, $\for(\alpha,0)=(\alpha',0)$. This gives the following remark.
\begin{remk}\label{remk:flow0}
	For any arrow $\alpha$ of $\tL$, the path $p_{(\alpha,0)}^F$ is the straight route containing $\alpha$ (marked at $\alpha$). Similarly, the path $p_{(\alpha^{-1},F(\alpha))}^F$ is the (equivalent) reverse of this marked route.
\end{remk}

\begin{example}\label{ex:singleton}
	Consider the fringed quiver of Figure~\ref{fig:singleton} with the flow $F$ indicated in blue (all arrows without a blue label are given zero flow). Shown in red is the path $p^F_{(\alpha,1)}=\alpha\gamma\delta^{-1}\kappa^{-1}$. On the other hand, we have $p_{(\alpha,C)}^F=\alpha\beta^{-1}$ for $C>1$ and $p^F_{(\alpha,C)}=\alpha\gamma\delta^{-1}\epsilon$ for $C<1$. Hence, the interval $I^F_{(\alpha,1)}=\{1\}$ and $a_{(\alpha,1)}^F=0$.
\begin{figure}
	\centering
	\def\svgscale{.21}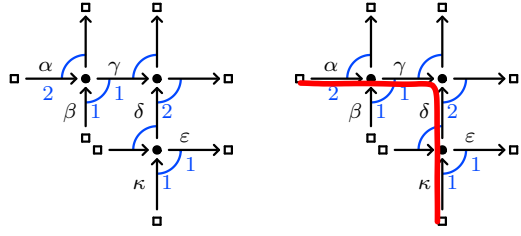
	\caption{A fringed quiver with a flow $F$ in blue and $p_{(\alpha,1)}^F$ in red.}
	\label{fig:singleton}
\end{figure}
\end{example}

\begin{example}\label{ex:gottabemarked}
	Consider the fringed quiver $\tL$ with flow $F$ given in Figure~\ref{fig:gottabemarked}. As noted in Remark~\ref{remk:flow0}, we have $p_{e_1}^0=e_1e_2e_3e_4$. Say $C\in(0,1]$. Then $\for_F(e_1,C)=(e_2,C)$ and $\for_F(e_2,C)=(e_3,C)$ and $\for_F(e_3,C)=(e_1^{-1},C)$. Hence, we get the marked route $p_{(e_1,C)}=e_1e_2e_3e_1^{-1}$ marked at the first $e_1$, and $I^F_{(e^\e,C)}=(0,1]$. Similarly, if $D\in(1,2]$, then $\for_F(e_1,D)=(e_3^{-1},D-1)$ and $\for_F(e_3^{-1},D-1)=(e_2^{-1},D-1)$ and $\for_F(e_2^{-1},D-1)=(e_1^{-1},(D-1)+1)=(e_1^{-1},D)$. Then $p_{(e_1,D)}=e_1e_3^{-1}e_2^{-1}e_1^{-1}$ is marked at the first copy of $e_1$ and $I^F_{(e_1,D)}=(1,2]$.

	Note that the unmarked routes $p_{(e_1,C)}$ and $p_{(e_1,D)}$ are the same, but their marked routes differ, as
	$p_{(e_1,D)}=e_1e_3^{-1}e_2^{-1}e_1^{-1}$ marked at the first copy of $e_1$ is equivalent to $p_{(e_1,C)}=e_1e_2e_3e_1^{-1}$ marked at the final arrow $e_1^{-1}$.
	Note also that $I^F_{(e_1,0)}=\{0\}$, as no nonzero value $E$ gives $p_{(e_1,E)}=e_1e_2e_3e_4$.
	\begin{figure}
		\centering
		\def\svgscale{.21}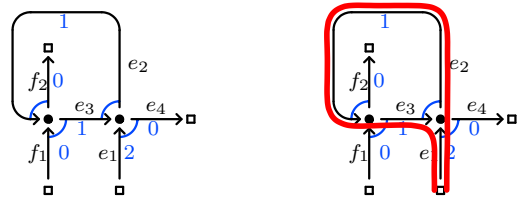
		\caption{The unmarked paths given by $(e_1,\frac{1}{2})$ and $(e_1,\frac{3}{2})$ are the same, but the marked paths differ.}
		\label{fig:gottabemarked}
	\end{figure}
\end{example}

It follows from Definition~\ref{defn:forbac} that, in the notation of the definition, for any arrow-flow $(\alpha^\e,C)$ such that $h(\alpha^\e)=t(\alpha^{-\e})$ is not fringe we have $\for(\alpha^\e,C)=\bac(\alpha^{-\e},F(\alpha)-C)$. Applying this result repeatedly yields the following remark.
\begin{remk}\label{remk:convenient}
	If $(\alpha^\e,C)$ is an arrow-flow of $\tL$, then $p^F_{(\alpha^\e,C)}=p^F_{(\alpha^{-\e},{F(\alpha)-C})}$ as marked extended strings (note that the former is marked at $\alpha^\e$, and the latter at $\alpha^{-\e}$).
\end{remk}

We now show that the extended strings obtained from arrow-flows are compatible with each other.

\begin{prop}\label{prop:arrow-order}
	Let $\alpha^\e$ be a signed arrow of $\tL$ and let $0\leq C\leq D\leq F(\alpha)$. Then the extended strings $p^F_{(\alpha^\e,C)}$ and $p^F_{(\alpha^\e,D)}$ satisfy $p^F_{(\alpha^\e,C)}\preceq_{\alpha^\e}p^F_{(\alpha^\e,D)}$.
\end{prop}
\begin{proof}
	Write $p_{(\alpha^\e,C)}^F=p_{(\alpha^\e,C)}^-\alpha^\e p_{(\alpha^\e,C)}^+$, where $\alpha^\e$ is the marked arrow of $p_{(\alpha^\e,C)}^F$ and $p_{(\alpha^\e,C)}^+$ is the (possibly infinite) part of this extended string after this marked arrow.
	We claim that $p_{(\alpha^\e,C)}^+$ than the analogously defined $p_{(\alpha^\e,D)}^+$ in the post-$\alpha^\e$ order $\prec_{\alpha^\e}^+$.

	For $j\in\mathbb N_{\geq0}$ such that $\for^j(\alpha^\e,C)$ is defined, let $\for^j(\alpha^\e,C)=(\alpha_j^{\e_j},C_j)$.
	For $j\in\mathbb N_{\geq0}$ such that $\for^j(\alpha^\e,D)$ is defined, let $\for^j(\alpha^\e,D)=((\alpha'_j)^{\e'_j},D_j)$.
	If $\alpha_j^{\e_j}=(\alpha'_j)^{\e_j}$ for all $j\geq0$, then $p_{(\alpha^\e,C)}^+=p_{(\alpha^\e,D)}^+$ and hence they are equal in $\prec_{\alpha^\e}^+$ and we have shown the claim.
	Otherwise, let $m\in\mathbb N_{\geq0}$ be maximal such that $\alpha_j^{\e_j}=(\alpha'_j)^{\e_j}$ for all $0\leq j\leq m$. If $\alpha_m^{\e_m}$ is fringe, then $p^F_{(\alpha^\e,C)}$ and $p^F_{(\alpha^\e,D)}$ agree after the marked $\alpha^\e$ and the claim is again shown. It remains to treat the case when $\alpha_m^{\e_m}$ is not fringe.

	We now prove that $C_k\leq D_k$ for all $0\leq k\leq m$.
	We prove this by induction on $k$. The base case $k=0$ is trivial, as $C_0=C\leq D=D_0$.
	Given $k<m$ such that $C_k\leq D_k$, we prove that $C_{k+1}\leq D_{k+1}$. Indeed, this is immediate by Definition~\ref{defn:forbac}; since $k<m$, we have $\alpha_k^{\e_k}=(\alpha'_k)^{\e'_k}$ and $\alpha_{k+1}^{\e_{k+1}}=(\alpha'_{k+1})^{\e'_{k+1}}$, hence applying $\for$ to $(\alpha_k^{\e_k},C_k)$ and $(\alpha_k^{\e_k},D_k)$ triggers the same branch of Definition~\ref{defn:forbac} and $D_{k+1}-C_{k+1}=D_k-C_k>0$.

	In particular, we have shown that $C_m\leq D_m$. Since we know that the signed arrows of $\for(\alpha_m^{\e_m},C_m)$ and $\for(\alpha_m^{\e_m},D_m)$ disagree and $C_m\leq D_m$, Definition~\ref{defn:forbac} forces that $\e_{m+1}=1$ and $\e'_{m+1}=-1$, finishing the proof of the claim.

	By applying the equivalence $p^F_{(\alpha^\e,C)}\equiv p^F_{(\alpha^{-\e},{F(\alpha)-C})}$ and $p^F_{(\alpha^\e,D)}\equiv p^F_{(\alpha^{-\e},{F(\alpha)-D})}$ with the claim, we prove the version of the claim with ``$+$'' replaced with ``$-$'' and ``$\for$'' replaced with ``$\bac$.'' This shows that $p^F_{(\alpha^\e,C)}\preceq_{\alpha^\e}p^F_{(\alpha^\e,D)}$.
\end{proof}

\begin{cor}\label{cor:compat}
	Let $(\alpha^\e,C)$ and $(\beta^\f,D)$ be arrow-flows of $F$. Then $p^F_{(\alpha^\e,C)}$ and $p^F_{(\beta^\f,D)}$ are compatible.
\end{cor}
\begin{proof}
	Let $\sigma:=\alpha_1^{\e_1}\dots\alpha_m^{\e_m}$ be a maximal shared segment of $p^F_{(\alpha^\e,C)}$ and $p^F_{(\beta^\f,D)}$. We will show that $\sigma$ does not realize an incompatibility between $p^F_{(\alpha^\e,C)}$ and $p^F_{(\beta^\f,D)}$.
	First, suppose $m\geq1$. Then we may choose $C'$ such that $p^F_{(\alpha^\e,C)}=p^F_{(\alpha_1^{\e_1},{C'})}$ as unmarked extended strings, and that the marked copy of $\alpha_1^{\e_1}$ of $p^F_{(\alpha_1^{\e_1},{C'})}$ is the beginning of $\sigma$. Similarly choose $D'$ such that $p^F_{(\beta^\f,D)}=p^F_{(\alpha_1^{\e_1},{D'})}$ as unmarked extended strings, and the marked $\alpha_1^{\e_1}$ of $p^F_{(\alpha_1^{\e_1},{D'})}$ is the start of $\sigma$.
	Then Proposition~\ref{prop:arrow-order} shows that $\sigma$ does not realize an incompatibility between $p^F_{(\alpha_1^{\e_1},D')}$ and $p^F_{(\alpha_1^{\e_1},{C'})}$, hence it does not realize an incompatibility between $p^F_{(\alpha^\e,C)}$ and $p^F_{(\beta^\f,D)}$. Since we chose $\sigma$ arbitrarily, we have shown that no shared segment may realize an incompatibility between these two paths, hence they are compatible.

	Now, suppose $m=0$, so that $\sigma$ is just a lazy string $e_v$. Let $\alpha_1\alpha_2$ and $\beta_1\beta_2$ be the two relations of $\tL$ at $v$.
	We must show that if, without loss of generality, $p^F_{(\alpha^\e,C)}$ contains $\alpha_2^{-1}\beta_2$ then $p^F_{(\beta^\f,D)}$ may not contain $\alpha_1\beta_1^{-1}$.
	Suppose to the contrary.
	Choose
	$C'$ such that $p^F_{(\alpha^\e,C)}=p^F_{(\alpha_2^{-1},{C'})}$ as unmarked extended strings and the marked copy of $\alpha_2^{-1}$ of $p^F_{(\alpha_2^{-1},C')}$ is followed by $\beta_2$. Then the third branch of Definition~\ref{defn:forbac} gives $F(\alpha_1)+C'<F(\beta_2)$. Similarly choose $D'$ such that $p^F_{(\beta^\f,D)}=p^F_{(\alpha_1^{1},{D'})}$ as unmarked extended strings, and the marked copy of $\alpha_1^{1}$ of $p^F_{(\alpha_1^{1},{D'})}$ is followed by $\beta_1^{-1}$.
	Then the second branch of Definition~\ref{defn:forbac} gives $D'>F(\beta_2)$. On the other hand,
	$D'\leq F(\alpha_1)\leq F(\alpha_1)+C'<F(\beta_2)$, a contradiction.
	This completes the proof.
\end{proof}

\begin{prop}\label{prop:interval-same-size}
	Let $(\alpha^\e,C)$ and $(\beta^\f,D)$ be arrow-flows of $F$ such that $p^F_{(\alpha^\e,C)}=p^F_{(\beta^\f,D)}$ are equal as unmarked extended strings. Then $|I^F_{(\alpha^\e,C)}|=|I^F_{(\beta^\f,D)}|$.
\end{prop}
\begin{proof}
	We first claim that, with the notation of Definition~\ref{defn:gr}, $|I^F_{(\alpha_j^{\e_j},{C_j})}|=|I^F_{(\alpha^\e,C)}|$ for all $j\in J$. It suffices to show that the length of the interval $I^F_{(\alpha,\e)}$ is preserved by the map $\for$. In other words, if $\for(\alpha^\e,C)=(\beta^\f,D)$, then $|I^F_{(\alpha^\e,C)}|=|I^F_{(\beta^\f,D)}|$.
	Define \[X:=\begin{cases}
		0&\e=1=\f\\
		-F(\alpha') & \e=1,\ \f=-1,\ \alpha\alpha'\text{ is a string}\\
		F(\beta') & \e=-1,\ \f=1,\ \alpha^{-1}(\beta')^{-1}\text{ is a string}\\
		F(\beta')-F(\alpha') & \e=-1=\f,\ \alpha^{-1}\alpha'\text{ and }\alpha^{-1}(\beta')^{-1}\text{ are strings}.
	\end{cases}\]
	Then by Definition~\ref{defn:forbac}, any arrow-flow
	$(\alpha^\e,C')\in I^F_{(\alpha^\e,C)}$ is sent to the arrow-flow $(\beta^\f,C'+X)$. It follows that $|I^F_{(\beta^\f,D)}|\geq|I^F_{(\alpha^\e,C)}|$. A symmetric argument, applying $\bac$ to arrow-flows $(\beta^\f,D')$ for $D'\in I^F_{(\beta^\f,D)}$, shows that $|I^F_{(\alpha^\e,C)}|\geq|I^F_{(\beta^\f,D)}|$ and completes the proof of the claim.

	Now suppose, for some arrow-flows $(\alpha^\e,C)$ and $(\beta^\f,D)$, that $p^F_{(\alpha^\e,C)}=p^F_{(\beta^\f,D)}$ as unmarked extended strings. Then there must be some arrow $\alpha^{\pm1}$ of $p^F_{(\beta^\f,D)}$ which can be marked to get the marked path $p^F_{(\alpha^\e,C)}$. Say this arrow is $\beta_j^{\f_j}$. First, suppose $\f_j=\e_j$. Then $D_j\in I^F_{\alpha^\e}$, hence $I^F_{(\alpha^\e,C)}=I^F_{(\beta_j^{\f_j},D_j)}$. Then the above claim shows that $|I^F_{(\alpha^\e,C)}|=|I^F_{(\beta_j^{\f_j},D_j)}|=|I^F_{(\beta^\f,D)}|$. On the other hand, if $\f_j=-\e_j$, then Remark~\ref{remk:convenient} shows that $I^F_{(\alpha^\e,C)}=\{F(\alpha)-D'\ :\ D'\in I^F_{(\beta_j^{\f_j},D_j)}\}$. In particular, $|I^F_{(\alpha^\e,C)}|=|I^F_{(\beta_j^{\f_j},D_j)}|$. Then the above claim shows that $|I^F_{(\alpha^\e,C)}|=|I^F_{(\beta_j^{\f_j},D_j)}|=|I^F_{(\beta^\f,D)}|$.
	This completes the proof.
\end{proof}

\begin{defn}\label{defn:THE-ALG}
	Let $F$ be a flow of $\tL$ with arrow-flow $(\alpha^\e,C)$.
	Define $a^F_{(\alpha^\e,C)}:=|I^F_{(\alpha^\e,C)}|$.
	Let $\bK_F$
	be the set of unmarked trails which arise as $p^F_{(\alpha^\e,C)}$ for some arrow-flow $(\alpha^\e,C)$ of $F$ (note that we are taking only the trails and leaving out any irrational strings appearing as some $p^F_{(\alpha^\e,C)}$). Let $\bK_F^+$ be the set of unmarked trails arising as $p^F_{(\alpha^\e,C)}$ for an arrow-flow $(\alpha^\e,C)$ such that $a^F_{(\alpha^\e,C)}\neq0$. For each trail $p\in \bK_F$, define $a_p^F$ to be $a^F_{(\alpha^\e,C)}$, where $p=p^F_{(\alpha^\e,C)}$ as unmarked extended strings (by Proposition~\ref{prop:interval-same-size}, this does not depend on the choice of arrow-flow $(\alpha^\e,C)$ realizing $p$).
\end{defn}

We call the process of calculating the trails of $\bK_F$ and their coefficients the \emph{flow algorithm}.

\begin{cor}\label{lem:mf-trail-clique}
	If $F$ is a flow of $\tL$, then $\bK_F$ (and hence $\bK_F^+$) is a bundle.
\end{cor}
\begin{proof}
	Corollary~\ref{cor:compat} shows that any two trails of $\bK_F$ are pairwise compatible.
\end{proof}

\begin{example}\label{ex:52}
	Consider the fringed quiver $\tL$ with flow $F$ given in Figure~\ref{fig:52}.
	We have
	\[p_{(f_1,C)}^F=\begin{cases}
		f_1f_2f_3&C=0\\
		f_1f_2e_2^{-1}f_2e_2^{-1}f_2e_2^{-1}e_1^{-1}&C\in(0,\frac{1}{2})\\
		f_1f_2e_2^{-1}f_2e_2^{-1}e_1^{-1}&C\in[\frac{1}{2},1].
	\end{cases}\]
	Hence, $\bK_F$ contains $p:=f_1f_2e_2^{-1}f_2e_2^{-1}e_1^{-1}$, $q:=f_1f_2e_2^{-1}f_2e_2^{-1}f_2e_2^{-1}e_1^{-1}$, and $r:=f_1f_2f_3$, with $a_p^F=a_q^F=\frac{1}{2}$ and $a_r^F=0$. The only other trail of $\bK_F$ is $s:=e_1e_2e_3$, with $a_s^F=0$.
	We then obtain $F$ as the bundle combination
	$F=\frac{1}{2}\left(\I(p)+\I(q)\right)$.
	\begin{figure}
		\centering
		\def\svgscale{.21}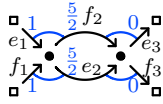
		\caption{The Kronecker double quiver with a flow labelled.}
		\label{fig:52}
	\end{figure}
\end{example}

We finally remark that the tools and results of this section (in particular, Definition~\ref{defn:forbac} and the flow algorithm) may also be phrased in the language of framed directed graphs. In this setting, we use the flow algorithm realize a rational flow as a convex combination indicator vectors of pairwise compatible routes and bands. See Appendix~\ref{sec:appendix} for more information.

\subsection{Existence and uniqueness of bundle combinations}

We now use our flow algorithm to show that any flow is obtained by at most one positive bundle combination, and that every rational flow is obtained by a bundle combination.
We will first use the following definition to relate a bundle combination to its arrow-flows.

\begin{defn}\label{defn:in-arrow}
	Let $F=\sum_{p\in \K\cup \B}b_p\I(p)$ be a bundle combination of a bundle $\K\cup \B$. The cardinality of $\K\cup \B$ is finite by Corollary~\ref{cor:card-upper-bound}. Let $\alpha^\e$ be a signed arrow of the fringed quiver and let $q_1,\dots,q_m$ be the marked trails of $\bK_{\alpha^\e}$ ordered so that $q_1\prec_{\alpha^\e}\dots\prec_{\alpha^\e} q_m$. For any $k\in[m]$, we say that an arrow-flow $(\alpha^\e,C)$ is \emph{in the marked arrow $\alpha^\e$ of $q_k$} if $C\in(\sum_{i=1}^{k-1}b_{q_i},\sum_{i=1}^kb_{q_i})$.
\end{defn}

This definition is justified by the following result.
\begin{lemma}\label{lem:in}
	Let $F=\sum_{q\in\K\cup\B}b_q\I(q)$ be a positive bundle combination.
	Let $(\alpha^\e,C)$ be an arrow-flow in a marked arrow $\alpha^\e$ of a trail $q\in\K\cup\B$ in the sense of Definition~\ref{defn:in-arrow}. Then $p_{(\alpha^\e,C)}^F=q$.
\end{lemma}
\begin{proof}
	It suffices to show that an arrow-flow $(\alpha^\e,C)$ in a marked arrow $\alpha^\e$ of a trail $q\in \K\cup \B$ is sent through $\for$ (respectively, $\bac$) to an arrow-flow in the arrow following (respectively preceding) the marked $\alpha^\e$ in $q$. We show this statement for $\for$; the $\bac$ statement is symmetric.
	To this end, let $\alpha^\e$ be an arrow of $\tL$. We assume $\e=1$; then $\e=-1$ case is similar.

	Let $q_{1},\dots,q_{M}$ be the complete list of marked trails of $\K\cup \B$ at $\alpha$, ordered so that $q_{1}\prec_{\alpha}\dots\prec_{\alpha} q_{M}$.
	If $h(\alpha)$ is fringe, then there is nothing to show, so suppose $h(\alpha)$ is internal and let $\gamma_1$ and $\gamma_2$ be the arrows of $\tL$ such that $\alpha\gamma_1$ and $\alpha\gamma_2^{-1}$ are strings. Let $\alpha_2$ be the arrow such that $\alpha\alpha_2$ is a relation of $\tL$.
	Define the intervals $I_1,\dots,I_{M}$ of $[0,F(\alpha)]$ as 
	\[I_j:=(\sum_{i=1}^{j-1}b_j,\sum_{i=1}^jb_j).\] Then an arrow-flow $(\alpha,C)$ is in the marked arrow $\alpha$ of $q_j$ if and only if $C\in I_j$.

	{Define $0\leq M'\leq M$ such that $q_j$ contains $\gamma_1$ after its marked copy of $\alpha$ if $1\leq j\leq M'$, and otherwise contains $\gamma_2^{-1}$ after its marked copy of $\alpha$.} Note that we may have $M'=0$. Then $\sum_{j=1}^{M'}b_j\leq F(\gamma_1)$ and $\sum_{j=M'+1}^{M}b_j\leq F(\gamma_2)$.
	For $1\leq j\leq M'$, let $q'_j$ be the trail $q_j$ marked at the arrow $\gamma_1$ following the marked copy of $\alpha$. It then follows that $q'_1,\dots,q'_{M'}$ are the $\prec_{\gamma_1}$-minimal marked trails of $\K\cup \B$ marked at $\gamma_1$, in order (any other trail of $\K\cup\B$ marked at $\gamma_1$ cannot have its marked arrow preceded by $\alpha$, hence must be above these trails in $\prec_{\gamma_1}$).
	Define the intervals $I'_1,\dots,I'_{M'}$ of $[0,F(\gamma_1)]$ as $I'_j=I_j$. The above shows that, for $j\in[M']$, an arrow-flow $(\gamma_1,C)$ is in the marked arrow $\gamma_1$ of $q'_j$ if and only if $C\in I'_j$.
	Since $\sum_{j=1}^{M'}b_j\leq F(\gamma_1)$, it follows from the frist branch of Definition~\ref{defn:forbac} that an arrow-flow $(\alpha,C)$ for $C\in[0,\sum_{i=1}^{M'}b_{i}]$ satisfies $\for(\alpha,C)=(\gamma_1,C)$. This shows that for $1\leq j\leq M'$, the arrow-flow $(\alpha,C)$ is in the marked arrow $\alpha$ of $q_j$ if and only if $\for(\alpha,C)$ is in the marked arrow $\gamma_1$ of $q'_j$, proving the claim for $j\in[M']$.

	It remains to show the claim for $M'<j\leq M$.
	There is nothing to show if $M'=M$, so suppose $M'<M$. Then there is some trail of $\K\cup\B$ containing $\alpha\gamma_2^{-1}$, hence there is no trail of $\K\cup \B$ containing $\alpha_2^{-1}\gamma_1$ since $\K\cup \B$ is a bundle. This shows that every marked trail of $\K\cup\B$ at $\gamma_1$ arises as some $q'_j$ for $j\in[M']$, hence that $\sum_{i=1}^{M'}b_i=F(\gamma_1)$.

	For $M'<j\leq M$, let $q'_j$ be the trail $q_j$ marked at the arrow $\gamma_2^{-1}$ following the marked copy of $\alpha$. Then $q'_1,\dots,q'_{M'}$ are the $\prec_{\gamma_2^{-1}}$-minimal marked trails of $\K\cup\B$ at $\gamma_2^{-1}$, in order. Define the intervals $I'_{M'+1},\dots,I'_{M}$ of $[0,F(\gamma_2)]$ as $I'_j=(\sum_{i=M'+1}^{j-1}b_i,\sum_{i=M'+1}^jb_i)=\{a-\sum_{i=1}^{M'}b_j\ :\ a\in I_j\}=\{a-F(\gamma_1)\ :\ a\in I_j\}$. Then, for $j\in\{M'+1,\dots,M\}$, an arrow-flow $(\gamma_2^{-1},C)$ is in the marked arrow $\gamma_2^{-1}$ of $q'_j$ if and only if $C\in I'_j$.

	An arrow-flow $(\alpha,C)$ for $C\in(\sum_{i=1}^{M'}b_{i},F(\alpha))=(F(\gamma_1),F(\alpha))$ is sent to $(\gamma_2^{-1},C-\sum_{i=1}^{M'}b_i)=(\gamma_2^{-1},C-F(\gamma_1))$ by the second branch of Definition~\ref{defn:forbac}. Hence, we see that for $j\in\{M'+1,\dots,M\}$, an arrow-flow $(\alpha,C)$ is in $I_j$ if and only if $\for(\alpha,C)$ is in $I'_j$. This completes the proof of the claim that, for all $j\in[M]$,
		an arrow-flow $(\alpha,C)$ is in the marked arrow $\alpha$ of $q_j$ if and only if $\for(\alpha,C)$ is in the marked arrow of $q'_j$.
\end{proof}

Lemma~\ref{lem:in} shows that an arrow-flow being in a marked arrow of a trail of a bundle combination of a flow $F$ in the sense of Definition~\ref{defn:in-arrow} has an interpretation only in terms of the flow algorithm on the flow $F$, without using any of the data of the bundle combination. We now use this to show that a bundle combination is determined by its underlying flow.

\begin{prop}\label{lem:only-way}
	If $F=\sum_{q\in \K\cup \B}b_q\I(q)$ is a positive bundle combination,
	then $\K\cup \B=\bK_F^+$.
	Moreover, for any $q\in\K\cup\B$ we have $b_q=a_q^F$.
\end{prop}
In other words, any flow $F$ is obtained by \emph{at most} one positive bundle combination; if this bundle combination exists, then it is $\sum_{p\in \bK_F^+}a_p\I(p)$.
\begin{proof}
	It follows from Lemma~\ref{lem:in} that if $q\in\K\cup\B$, then $q$ is also in $\bK_F^+$ with $a_q^F\geq b_q$.
	Since $\sum_{p\in\bK_F^+}a_p^F\I(p)\leq F=\sum_{q\in\K\cup\B}b_q\I(q)$, we must have $\bK_F^+=\bK\cup\B$ and equalities $a_q^F=b_q$ for all $q\in\\bK_F^+$.
\end{proof}

One might hope that for an arbitrary flow $F$, we will get a bundle combination $F=\sum_{p\in\bK_F^+}a_p^F\I(p)$.
The following example shows that when $F$ is not rational (i.e., $F$ has at least one irrational coordinate), this may fail.

\begin{example}\label{ex:irrational}
	See the double Kronecker fringed quiver and flow of Figure~\ref{fig:irrational}.
	For an arrow-flow $(f_2,C)$, we wish to describe the extended string $p_{(f_2,C)}^F$.
	To this end, we describe the effect of $\for$ on an arrow-flow $(f_2,C)$.
	\begin{enumerate}
		\item If $C>4-\pi$, then
			\begin{align*}
				\for(f_2,C)&=(e_2^{-1},C-(4-\pi))=(e_2^{-1},C+\pi-4)\\
				\for^2(f_2,C)&=(f_2,C+\pi-4).
			\end{align*}
		\item If $C\leq4-\pi$, then
			\begin{align*}
				\for(f_2,C)&=(f_3,C)\\
				\for^2(f_2,C)&=(e_3^{-1},C)\\
				\for^3(f_2,C)&=(e_2^{-1},C+\pi-(4-\pi))=(e_2^{-1},C+2\pi-4)\\
				\for^4(f_2,C)&=(f_2,C+2\pi-4).
			\end{align*}
	\end{enumerate}
	It follows that, for any $m>0$, if $\for^m(f_2,C)=(f_2,D)$ then $D=C+a\pi-b$ for some pair of nonnegative integers $(a,b)\neq(0,0)$. Since $\pi$ is not rational, this means that $\for^m(f_2,C)$ can never equal $(f_2,C)$, and hence that $p^F_{(f_2,C)}$ is an irrational string for any arrow-flow $(f_2,C)$. Then $\bK_F^+$ is empty and $F$ is not equal to $0=\sum_{p\in\bK_F^+}a_p^F\I(p)$.
	By adding the indicator vector of a straight route to $F$, we may obtain a unit flow displaying this same behavior.
\begin{figure}
	\centering
	\def\svgscale{.21}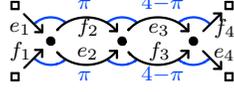
	\caption{The double Kronecker fringed quiver with an irrational flow. Fringe arrows have zero flow.}
	\label{fig:irrational}
\end{figure}
\end{example}

Section~\ref{sec:example} will discuss the bundle spaces of the fringed quiver of Example~\ref{ex:irrational} in depth.

We now show that when $F$ is a rational flow, the sum $\sum_{p\in\bK_F}a_{p}^F\I(p)$ is always a bundle combination realizing $F$. In other words, the issue of Example~\ref{ex:irrational} is specific to irrational flows.

\begin{defn}
	Let $F$ be a nonzero rational flow. Let $\GCD(F)$ be the largest positive (necessarily rational) number $c$ such that, for any edge $\alpha\in E$, we have $F(\alpha)=cm$ for some $m\in\mathbb Z_{\geq0}$.
\end{defn}

\begin{lemma}\label{lem:alg}
	Let $F$ be a rational flow. Then the extended string $p^F_{(\alpha^\e,C)}$ is a trail for any arrow-flow $(\alpha^\e,C)$.
\end{lemma}
\begin{proof}
	We adopt the notation of Definition~\ref{defn:gr}.
	We first argue that a finite number of arrow-flows $(\beta^\f,D)$ appear as $(\alpha_j^{\e_j},C_j)$ for some $j$. Indeed, each value $C_j$ is contained in the closed interval $[0,\max_{\alpha\in{E}}F(\alpha)]$. It follows from Definition~\ref{defn:gr} that $C_j-C_{j+1}$ is a multiple of $\GCD(F)$ for any valid $j$. This means, further, that for any $j\in J$ we have that $C_j-C$ is a multiple of $\GCD(F)$.
	It follows that there are a finite number of values which $C_j$ may take, and hence a finite number of arrow-flows appearing as $(\alpha_j^{\e_j},C_j)$ for some $j$.

	{If $J$ is a finite interval, then it is immediate that $p_{(\alpha^\e,C)}^F$ is a route.}
	If $J$ is not finite, then we must have $(\alpha_i^{\e_i},C_i)=(\alpha_j^{\e_j},C_j)$ for some $i\neq j$ -- suppose we have chosen $i$ and $j$ so that $i>j$ and $i-j$ is minimal with respect to this property. Then it follows that $(\alpha_a^{\e_a},C_a)=(\alpha_b^{\e_b},C_b)$ for any $a$ and $b$ which are equivalent modulo $i-j$. Then $p_{(\alpha^\e,C)}^F$ is a band. This ends the proof.
\end{proof}

We remark that Lemma~\ref{lem:alg} shows that when $F$ is rational, every $p^F_{(\alpha^\e,C)}$ is a trail, hence is part of the bundle $\bK_F$. Since bundles have finite cardinality by Corollary~\ref{cor:card-upper-bound}, this means that there are a finite number of extended strings (all of which are trails) appearing as $p^F_{(\alpha^\e,C)}$ for some arrow-flow $(\alpha^\e,C)$ of $F$.

\begin{prop}\label{prop:alg-realize-flow}
	Let $F$ be a rational flow. Then $F$ is realized as the positive bundle combination $F=\sum_{p\in \bK_F^+}a_p^F\I(p)$.
\end{prop}
\begin{proof}
	It suffices to show that $F=\sum_{p\in\bK_F}a_p^F\I(p)$, since $a_p^F=0$ for $p\in\bK_F\backslash\bK_F^+$.
	Corollary~\ref{lem:mf-trail-clique} shows that $\bK_F$ is a bundle, hence $\sum_{p\in\bK_F^+}a_p^F\I(p)$ is a bundle combination. It remains to show that this bundle combination gives $F$.

	Choose an arrow $\alpha$ of $\tL$.
	Define $q_1,\dots,q_m$ to be the marked trails of $(\bK_F)_{\alpha^\e}$ at $\alpha$ ordered by $\prec_\alpha$ (there are a finite number of these because Corollary~\ref{cor:card-upper-bound} implies that $\bK_F$ has finite cardinality). For $j\in[m]$, define $C_j\in[0,F(\alpha)]$ such that we have an equality of marked trails $q_j=p^F_{(\alpha,{C_j})}$.
	If $(\alpha,C)$ is any arrow-flow, then Lemma~\ref{lem:alg} shows that $p^F_{(\alpha,C)}$ is a trail, hence $C$ is contained in the interval $I_{(\alpha,{C_j})}$ for exactly one $j\in[m]$.
	Then 

	\[
		\sum_{p\in\bK_F}a_p^F\big(\I(p)(\alpha)\big)=
		\sum_{j=1}^ma^F_{(\alpha,{C_j})}=\sum_{j=1}^m|I_{(\alpha,{C_j})}|=F(\alpha).\]
	This is the restriction of the proposed equality $F=\sum_{p\in\bK_F}a_p^F\I(p)$ to the coordinate given by the arrow $\alpha$. Since this holds for any arrow $\alpha$, we must have $F=\sum_{p\in\bK_F}a_p^F\I(p)$, completing the proof.
\end{proof}

We may now phrase the first two statements of Theorem~\ref{ITHM:BRIQUE}.
\begin{thm}\label{thm:comb}
	Let $\tL$ be a fringed quiver.
	\begin{enumerate}
		\item\label{cco1} Any flow $F\in\F_{\geq0}(\tL)$ has at most one representation as a positive bundle combination $F=\sum_{p\in\bK}a_p\I(p)$.
		\item\label{cco3} Every rational flow may be represented as a positive bundle combination.
	\end{enumerate}
\end{thm}
\begin{proof}
	Proposition~\ref{lem:only-way} shows~\eqref{cco1}.
	Proposition~\ref{prop:alg-realize-flow} shows~\eqref{cco3}. 
\end{proof}

We now show that bundle combinations of integer flows have integer coefficients. This will be used to show that any maximal clique simplex is unimodular.

\begin{lemma}\label{prop:b}
	Let $F$ be a flow such that $F(\alpha)\in\mathbb Z$ for all arrows $\alpha$. Then all coefficients $a_p^F$ for $p\in\bK_F$ are integers.
\end{lemma}
\begin{proof}
	Choose a signed arrow $\alpha^\e$ of $\tL$. Choose two values $C$ and $C'$ of $[0,F(\alpha)]$ such that $C<C'$ and no integer lies in the interval $(C,C')$. We show that $p_{(\alpha^\e,C)}^F=p_{(\alpha^\e,C')}^F$.

	Let $a$ be maximal such that $\for^a(\alpha^\e,C)$ exists and notate $\for^j(\alpha^\e,C):=(\alpha_j^{\e_j},C_j)$ for $0\leq j\leq a$.
	{Since $\for(\alpha_j^{\e_j},C_j)=(\alpha_{j+1}^{\e_{j+1}},C_{j+1})$ and $F$ is integer-valued, it is immediate from Definition~\ref{defn:forbac} that $C_j$ and $C_{j+1}$ differ by an integer. This means that each $C_j$ differs from $C=C_0$ by an integer.}

	{We show by induction on $j$ that $\for^j(\alpha^\e,C')=(\alpha_j^{\e_j},C_j+C'-C)$ for any $0\leq j\leq a$. The base case $j=0$ is trivial, as $\for^0(\alpha^\e,C')=(\alpha^\e,C')=(\alpha^\e,C+C'-C)$. Now suppose $0\leq j<a$ and $\for^j(\alpha^\e,C')=(\alpha_j^{\e_j},C_j+C'-C)$.}

	By Definition~\ref{defn:forbac}, there is some value $D_j$ such that the arrow of $\for(\alpha_j^{\e_j},D)$ is oriented positively for $D<D_j$ and negatively for $D>D_j$ (we may have $D_j=0$ or $D_j=F(\alpha)$, in which case one of these conditions is impossible). Since $F$ is integer-valued, $D_j$ is an integer. 
	Since $C_j$ differs from $C$ by an integer, and there is no integer between $C$ and $C'$, it follows that there is no integer between $C_j$ and $C_j+C'-C$.
	Then the signed arrow of
	$\for(\alpha_j^{\e_j},C_j)$ and $\for(\alpha_j^{\e_j},C'_j)$ must agree. Hence, we have 
	\begin{align*}\for^{j+1}(\alpha^\e,C')=\for(\alpha_j^{\e_j},C_j+C'-C)&=(\alpha_{j+1}^{\e_{j+1}},C_{j+1}+(C_j+C'-C)-C_j)\\
	&=(\alpha_{j+1}^{\e_{j+1}},C_{j+1}+C'-C),\end{align*} completing the proof of the induction step. 

	A symmetric argument shows that for any $i$ such that $\bac^i(\alpha^\e,C)$ is defined, we have $\bac^i(\alpha^\e,C')=(\alpha_{-i}^{\e_{-i}},C_{-i}+C'-C)$.
	We have then shown that the marked paths $p^F_{(\alpha^\e,C)}$ and $p^F_{(\alpha^\e,{C'})}$ agree. 
	Since this holds for any $C$ and $C'$ such that no integer lies between them, we have shown that the endpoints of the interval $I_{(\alpha^\e,C)}^F$ are integers for any arrow-flow $(\alpha^\e,C)$. This shows that $a_{(\alpha^\e,C)}^F$ is an integer for every arrow-flow $(\alpha^\e,C)$, completing the proof.
\end{proof}

\begin{cor}\label{lem:simp-unimodular}
	Any maximal clique simplex is unimodular.
\end{cor}
\begin{proof}
	Let $\K$ be a maximal clique and consider its clique simplex $\Delta_1(\K)$. To show that $\Delta_1(\K)$ is unimodular, it suffices to show that any integer point in the cone over $\Delta_1(\K)$ appears as a nonnegative integer combination of the vertices of $\Delta_1(\K)$. 
	This is Lemma~\ref{prop:b}.
\end{proof}

\subsection{Representation-finiteness and boundedness}

Recall that in Section~\ref{sec:PRF} we claimed that representation-finiteness of $\tL$ is equivalent to boundedness of $\F_1(\tL)$, but deferred the proof. Having proven Theorem~\ref{thm:comb}, we now have the technology to prove this result efficiently.
\begin{lemma}\label{rep-finite-union}
	Suppose $\tL$ is representation-finite. Then $\F_1(\tL)$ is the union of the clique simplices of $\tL$.
\end{lemma}
\begin{proof}
	If $\tL$ is representation-finite, then $\tL$ has no bands. This means that for any flow $F\in\F_1(\tL)$ and any arrow-flow $(\alpha^\e,C)$ of $F$, the path $p^F_{(\alpha^\e,C)}$ is defined, since it cannot loop back on itself.
	Hence, any flow $F$ is realized in a clique simplex as $F=\sum_{p\in \bK_F}a_p\I(p)$. We have shown that $\F_1(\tL)$ is the union of the clique simplices.
\end{proof}

\begin{cor}\label{cor:finite-iff-bound}
	The following are equivalent:
	\begin{enumerate}
		\item\label{los1} $\L$ is representation-finite
		\item\label{los2} $\tL$ is representation-finite
		\item\label{los3} $\F_1(\tL)$ is bounded
	\end{enumerate}
\end{cor}
\begin{proof}
	\eqref{los1}$\iff$\eqref{los2} is immediate. We show~\eqref{los2}$\iff$\eqref{los3}.
	First, suppose $\tL$ is representation-infinite. Then there is some band $B$ of $\tL$. The indicator vector $\I(B)$ satisfies conservation of flow but is zero on all fringe arrows. Then for any $\xx\in\F_1(\tL)$, we have $\xx+\lambda\I(V)\in\F_1(\tL)$ for all $\lambda>0$. This shows that $\F_1(\tL)$ is not bounded.
	On the other hand, suppose $\tL$ is representation-finite. Lemma~\ref{rep-finite-union} shows that $\F_1(\tL)$ is the union of a finite number of clique simplices, hence is bounded.
\end{proof}

\subsection{Summarizing subdivision results and giving examples}

We are now able to define the bundle subdivision of the turbulence polyhedron and prove that it is, indeed, a subdivision of $\F_1(\tL)$. This completes the proof of Theorem~\ref{ITHM:BRIQUE} of the introduction.
\begin{defn}
	Define the \emph{(unit) bundle subdivision} of $\F_1(\tL)$ as
	\[\Sbs^{\spircle}(\F_1(\tL)):=\{\D_{\bK}\ :\ \bK\text{ is a maximal bundle of }\tL\}.\]
	Define the \emph{nonnegative bundle subdivision} of $\F_{\geq0}(\tL)$ as
	\[\Sbs^{\spircle}(\F_{\geq0}(\tL)):=\{\D_{\bK}^{\geq0}\ :\ \bK\text{ is a maximal bundle of }\tL\}.\]
\end{defn}
\begin{cor}\label{cor:bundle-subd}
	The bundle subdivision of $\F_1(\tL)$ is a subdivision which covers all rational points, and the nonnegative bundle subdivision of $\F_{\geq0}(\tL)$ is a subdivision which covers all rational points.
	When $\tL$ is representation-finite, there are no bundles and the clique triangulation is complete.
\end{cor}
\begin{proof}
	Theorem~\ref{thm:comb} shows both density and the strong intersection property, hence the bundle subdivisions are subdivisions.
	Lemma~\ref{rep-finite-union} shows that the subdivisions are complete in the representation-finite case.
\end{proof}

\begin{remk}
	It is possible at this time to prove by working directly with the turbulence polyhedron that the clique simplices are always dense in the turbulence polyhedron, which would complete the proof of Theorem~\ref{ITHM:CLIQUE} by showing that the clique triangulation of a turbulence polyhedron is a unimodular triangulation of $\F_1(\tL)$ in all cases. For brevity, we will instead inherit this density property from the g-vector fan in Corollary~\ref{cor:clique-subd} of Section~\ref{sec:g}.
\end{remk}

\begin{example}\label{ex:gemd}
	Consider the shard fringed quiver of Figure~\ref{fig:nonv-turb}. There are six maximal cliques, all of which are drawn in Figure~\ref{GEMINTRO} (drawn without straight routes for readability). Figure~\ref{GEMINTRO} matches one of these maximal cliques with its corresponding clique simplex. The six maximal clique simplices give a unimodular triangulation of $\F_1(\tL)$.
	\begin{figure}
		\centering
		\def\svgscale{.21}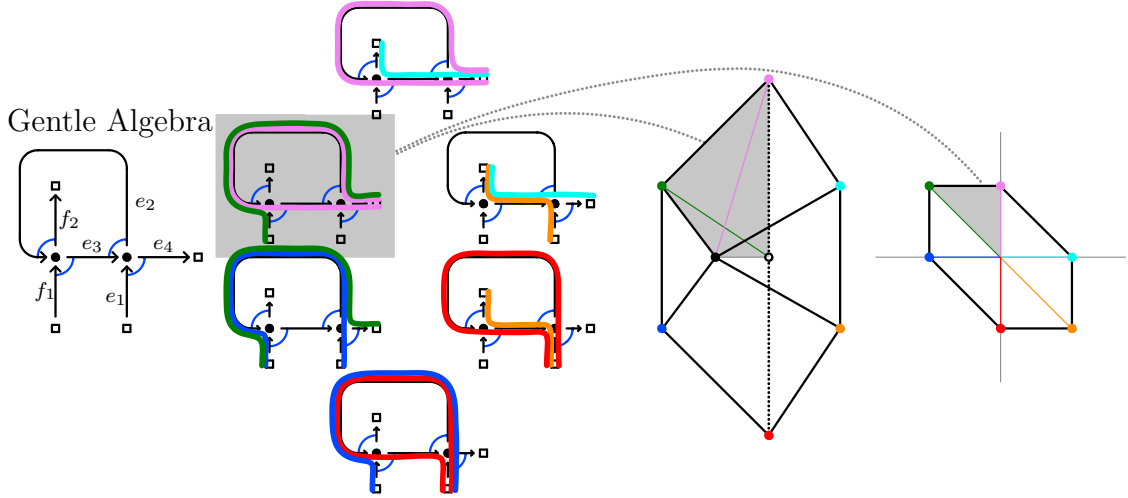
		\caption{The ``shard'' fringed quiver with its six maximal cliques (drawn without straight routes). On the right is its three-dimensional turbulence polyhedron and two-dimensional g-polyhedron.}
		\label{GEMINTRO}
	\end{figure}
\end{example}

\begin{example}\label{ex:kron-eddy-decomp}
	Let $\tL$ be the Kronecker fringed quiver of Figure~\ref{fig:kron1}.
	The vertices of $\F_1(\tL)$ are \[\{\I({f_1}{f_2}{f_3}),\ \I({e_3}^{-1}{f_3}),\ \I({e_1}{f_1}^{-1}),\ \I({e_1}{e_2}{e_3})\},\] and its recession cone is generated by $\I({e_2}{f_2}^{-1})$.
	Figure~\ref{KRONINTRO} shows some maximal bundles of $\tL$.
	The complete list of maximal bundles is as follows:
	\begin{align*}
		\{{e_1}{e_2}{e_3},\ {f_1}{f_2}{f_3},\ &{e_3}^{-1}{f_3},\ {e_1}{f_1}^{-1}\}, \\
		\{{e_1}{e_2}{e_3},\ {f_1}{f_2}{f_3},\ &{e_3}^{-1}({e_2}{f_2}^{-1})^j{f_3},\ {e_3}^{-1}({e_2}{f_2}^{-1})^{j+1}{f_3}\}\ (\textup{for any }j\geq0), \\
		\{{e_1}{e_2}{e_3},\ {f_1}{f_2}{f_3},\ &{e_1}({e_2}{f_2}^{-1})^j{f_1}^{-1},\ {e_1}({e_2}{f_2}^{-1})^{j+1}{f_1}^{-1}\}\ (\textup{for any }j\geq0), \\
		\{{e_1}{e_2}{e_3},\ {f_1}{f_2}{f_3},\ &{e_2}{f_2}^{-1}\}.
	\end{align*}
	Note that all maximal bundles contain the straight routes ${e_1}{e_2}{e_3}$ and ${f_1}{f_2}{f_3}$. The top three rows are all cliques. These simplices give a unimodular triangulation of a dense subset of $\F_1(\tL)$: 
	\begin{itemize} 
		\item the top row gives the clique simplex in $\F_1(\tL)$ with vertices \[\{\I({f_1}{f_2}{f_3}),\I({e_1}{e_2}{e_3}),\I({e_1}{f_1}^{-1}),\I({e_3}^{-1}{f_3})\}.\]
		\item the next row gives the clique simplices \[\{\I({f_1}{f_2}{f_3}),\I({e_1}{e_2}{e_3}),\I({e_3}^{-1}{f_3})+j\I({e_2}{f_2}^{-1}),\I({e_3}^{-1}{f_3})+(j+1)\I({e_2}{f_2}^{-1})\}\] for every $j\in\mathbb Z_{\geq0}$.
		\item the next row gives the clique simplices \[\{\I({f_1}{f_2}{f_3}),\I({e_1}{e_2}{e_3}),\I({e_1}{f_1}^{-1})+j\I({e_2}{f_2}^{-1}),\I({e_1}{f_1}^{-1})+(j+1)\I({e_2}{f_2}^{-1})\}\] for every $j\in\mathbb Z_{\geq0}$.
	\end{itemize}
	The above three clique simplices are, in fact, full-dimensional simplices in the affine span of $\F_1(\tL)$. The final maximal bundle $\{{e_1}{e_2}{e_3},{f_1}{f_2}{f_3},{e_2}{f_2}^{-1}\}$ is the only maximal bundle with a band. The space of bundle combinations of this bundle is the two-dimensional bundle wall
	\[\{a\I({f_1}{f_2}{f_3})+b\I({e_1}{e_2}{e_3})+c\I({e_2}{f_2}^{-1})\ :\ a,b,c\in\mathbb R_{\geq0},\ a+b=1\}.\]
	This bundle wall is a ``rectangle with infinite width.'' Note that the elements of this bundle wall with $c>0$ are precisely those points of $\F_1(\tL)$ which are in none of the clique simplices. Then every point of $\F_1(\tL)$ is in one of the above bundle spaces. The bundle subdivision is the collection of these bundle spaces, which in this case is complete.
	See Figure~\ref{KRONINTRO} for a visual picture of the indicator vectors of the routes and one of the maximal clique simplices.
\end{example}

The turbulence polyhedron of the double Kronecker quiver of Figure~\ref{fig:irrational} is not fully covered by its bundle subdivision. We will not discuss its bundle subdivision in depth until Section~\ref{sec:example} after we develop the theory of vortex dissections. Figure~\ref{fig:doubkron-g} shows a three-dimensional reduced picture of its four-dimensional turbulence polyhedron.

\section{Vertices and Recession Cones of Turbulence Polyhedra}
\label{sec:vert-unb}

We defined the turbulence polyhedron of a fringed quiver by giving a set of hyperplane inequalities; we now describe the vertices and recession cone of the turbulence polyhedron.
We will characterize vertices of $\F_1(\tL)$ as indicator vectors of \emph{elementary routes} and unbounded directions of $\F_1(\tL)$ as coming from \emph{elementary bands}. To define elementariness, we will need some preliminary definitions.

\begin{defn}\label{defn:boosted-crisscrossed}
	We now define what it means for a string $\sigma$ to be \emph{boosted} or \emph{criss-crossed} by a route or band $p$. We must consider separately the case when $\sigma$ is a lazy string and when $\sigma$ contains at least one arrow.
	\begin{enumerate}
		\item Suppose first that $\sigma$ contains at least one arrow. If $\sigma$ is contained as a subpath of $p$ multiple distinct times, then $\sigma$ is a \emph{boosted substring} of $p$. If both $\sigma$ and $\sigma^{-1}$ are contained as a subpath of $p$, then $\sigma$ is a \emph{criss-crossed substring} of $p$. In this case, $\sigma^{-1}$ is also a criss-crossed substring, and we consider $\sigma$ and $\sigma^{-1}$ to be equivalent and make up only one criss-crossed substring.
		\item Suppose now that $\sigma=e_v$ is the lazy string at an internal vertex $v$. Say $\alpha_1\alpha_2$ and $\beta_1\beta_2$ are the relations of $v$. Separate the length-two strings passing through $v$ into the sets $S=\{\alpha_1\beta_2,\alpha_1\beta_1^{-1},\alpha_2^{-1}\beta_2,\alpha_2^{-1}\beta_1\}$ and
			$T=\{\beta_2^{-1}\alpha_1^{-1},\beta_1\alpha_1^{-1},\beta_2^{-1}\alpha_2,\beta_1^{-1}\alpha_2\}$. If $p$ contains more than one element of $S$ or more than one element of $T$ (counting multiplicities), then $e_v$ is a \emph{boosted substring} of $p$. If $p$ contains an element of $S$ and an element of $T$, then $e_v$ is a \emph{criss-crossed substring} of $p$.
	\end{enumerate}
\end{defn}

Intuitively, a string $\sigma$ is a boosted substring of $p$ if $p$ passes through $\sigma$ multiple times in the same direction, and $\sigma$ is a criss-crossed substring of $p$ if $p$ passes through $\sigma$ both forwards and backwards.
We remark that $\tL$ is representation-infinite if and only if some string has a boosted substring.

\begin{example}
	The route $e_1e_2e_3e_1^{-1}$ drawn in red in Figure~\ref{fig:gottabemarked} contains no boosted substrings, and contains one maximal criss-crossed substring $e_1$. In the Kronecker fringed quiver of Figure~\ref{fig:kron1}, the path $e_1e_2f_2^{-1}e_2e_3$ contains no criss-crossed substrings, and contains one maximal boosted substring $e_2$.
\end{example}

\begin{remk}\label{remk:lazyboost}
	If a trail contains the lazy string $e_v$ at a vertex $v$ three or more distinct times, then $e_v$ is a boosted substring of $p$.
\end{remk}

\begin{defn}\label{defn:elband}
	A band $p$ of $\tL$ is \emph{elementary} if it is self-compatible, has no boosted substrings, and has at most one maximal criss-crossed substring.
\end{defn}

We now give a different interpretation of Definition~\ref{defn:elband}. We say that a band $B$ is \emph{simple} if it does not use the same vertex twice (not counting $t(B)=h(B)$). We say that a band is a \emph{barbell}
 if up to cyclic equivalence it is of the form $s\sigma_1 s^{-1}\sigma_2$, where every vertex used appears exactly once with the exception of the vertices of $s$, which are used twice.
Then a self-compatible band is elementary if and only if it is simple or it is a barbell.
In other words, an elementary band must look like one of the two left examples in Figure~\ref{fig:elband}.

\begin{figure}
\centering
	\begin{minipage}[b]{.22\textwidth}
	\centering
	\def\svgscale{.21}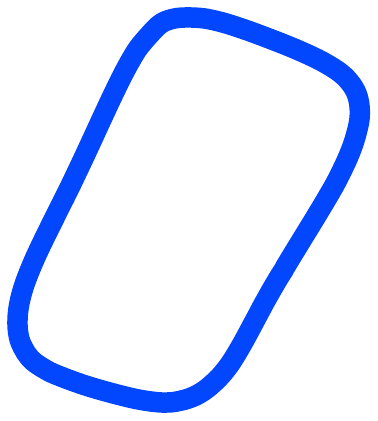

	simple

		(elementary)
\end{minipage}
	\begin{minipage}[b]{.22\textwidth}
	\centering
	\def\svgscale{.21}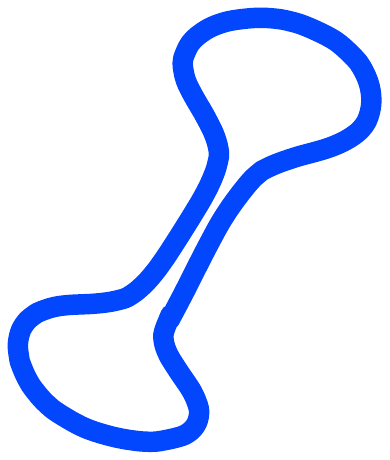

	barbell

		(elementary)
\end{minipage}
	\begin{minipage}[b]{.22\textwidth}
	\centering
	\def\svgscale{.21}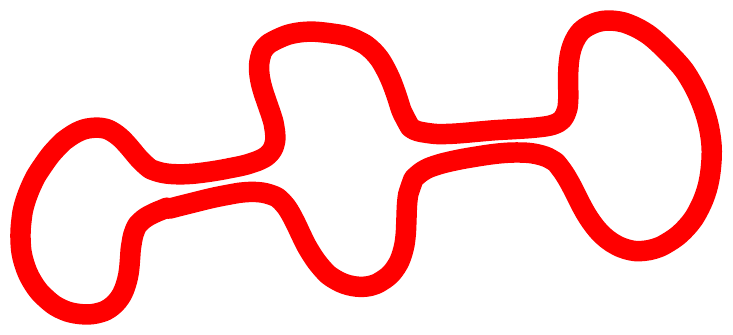

	nonelementary
\end{minipage}
	\begin{minipage}[b]{.22\textwidth}
	\centering
	\def\svgscale{.21}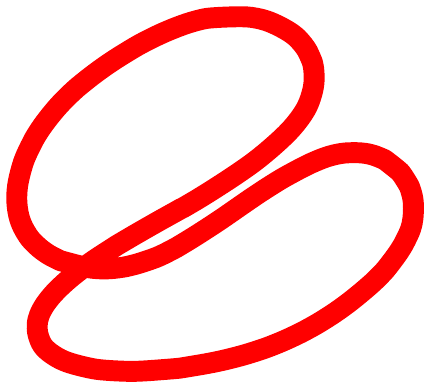

	nonelementary
\end{minipage}
	\caption{Examples of elementary (left) and nonelementary (right) bands.}
	\label{fig:elband}
\end{figure}

\begin{prop}\label{prop:elementary-bands}
	Let $\Lambda$ be a gentle algebra. 
	\begin{enumerate}
		\item\label{r1} The recession cone of $\F_1(\tL)$ is generated by the indicator vectors of bands.
		\item\label{r2} The set of indicator vectors of elementary bands is a minimal generating set for the recession cone of $\F_1(\tL)$.
	\end{enumerate}
\end{prop}

\begin{proof}
	We first show that the indicator vectors of elementary bands generate the recession cone. Then~\eqref{r1} follows.
	It is immediate that the recession cone of $\F_1(\tL)$ is precisely the cone of vortices of $\tL$ (recall that a vortex of $\tL$ is a flow giving zero weight to all fringe arrows). It then suffices to show that any vortex $W$ may be obtained as a nonnegative sum of indicator vectors of elementary bands. Since $\F_1(\tL)$ is defined by rational hyperplanes, its recession cone has a rational generating set, so it suffices to show this only for rational vortices $W$. We show this by induction on the number of arrows used by $W$ -- the base case, when $W$ uses zero arrows, is trivial. Now suppose we have a rational vortex $W$ and we have shown that any rational vortex using fewer arrows than $W$ is a nonnegative combination of indicator vectors of elementary bands.

	We begin by constructing a band $C$ such that every arrow of $C$ is given positive flow by $W$, and that $\I(C)$ is a positive combination of indicator vectors of elementary bands.
	First, choose any band $B\in\bK_W^+$. By Theorem~\ref{thm:comb}, such a band $B$ exists and every arrow of $B$ is given positive flow by $W$.
	If $B$ has no boosted substring, set $B_1:=B$. Otherwise, we may represent $B$ (up to cyclic equivalence) as $B=p\sigma q\sigma$, where $p,q,\sigma$ are strings, $\sigma$ is a boosted substring of $B$, and the string $p\sigma$ has no boosted substrings. Then $B_1:=p\sigma$ is a band with no boosted substrings using only arrows given positive flow by $W$.
	In either case, we have obtained a band $B_1$ with no boosted substrings using only arrows given positive flow by $W$.

	If $B_1$ has no criss-crossed substrings, then it is elementary and we set $C:=B_1$. If $B_1$ has exactly one maximal criss-crossed substring $s$, then up to cyclic equivalence we may write $B_1=s\sigma_1 s^{-1}\sigma_2$ for some strings $\sigma_1,\sigma_2$. If $B_1$ is self-compatible, then it must be elementary and set $C:=B_1$. Otherwise, the band $C:=s\sigma_1s^{-1}\sigma_2^{-1}$ must be self-compatible and hence elementary. In either of these cases, $C$ is itself elementary so $\I(C)$ is trivially a positive combination of indicator vectors of elementary bands.

	Now suppose $B_1$ has two or more maximal criss-crossed substrings.
	Let $s$ be a maximal criss-crossed substring of $B_1$. Then $B_1$ contains both $\alpha_1^{\e_1}s\alpha_2^{\e_2}$ and $\beta_1^{-\e_1}s^{-1}\beta_2^{-\e_2}$ for some signed arrows $\alpha_i^{\e_i}$ and $\beta_j^{-\e_j}$.
	We say that $\{\alpha_1,\alpha_2,\beta_1,\beta_2\}$ are the arrows of $B_1$ \emph{surrounding} $s$.
	In particular, note that the lazy string at $h(s)=t(\alpha_2^{\e_2})$ already appears in $s\alpha_2^{\e_2}$ and $\beta_1^{-\e_1}s^{-1}$, so by Remark~\ref{remk:lazyboost} this lazy string cannot appear any more times in $B_1$. In particular, this means that the arrow $\alpha_2$ (with any orientation) does not appear in $B_1$ outside of the substring $\alpha_1^{\e_1}s\alpha_2^{\e_2}$.
	Similarly, the arrows $\alpha_1,\beta_1,\beta_2$ are used by $B_1$ only once.
	Let $t$ be a different maximal criss-crossed substring of $B_1$. We now split into two cases.
	\begin{enumerate}
		\item Suppose that $B_1$ is equal (up to cyclic equivalence) to $s P_1 t P_2 t^{-1} P_3 s^{-1} P_4$, for some strings $P_i$. Since the arrows surrounding $s$ and $t$ are used by $B_1$ only once, the strings $P_1,P_2,P_3,P_4$ all contain at least one arrow. If necessary, redefine $\alpha_j^{\e_j}$ and $\beta_j^{-\e_j}$ so that the first arrow of $P_1$ is $\alpha_2^{\e_2}$; then this is the only occurrence of $\alpha_2$ in $B_1$.
			Then let $C:= s P_3^{-1} t P_2 t^{-1} P_3 s^{-1} P_4$. This band uses only arrows of $B_1$, but does not use the arrow $\alpha_2$ because we got rid of the substring $P_1$. Hence, the vortex $\I(C)$ uses strictly fewer arrows than $W$ and by the induction hypothesis $\I(C)$ is a nonnegative combination of indicator vectors of elementary bands.
		\item Suppose that $B_1$ is equal (up to cyclic equivalence) to $s P_1 t P_2 s^{-1} P_3 t^{-1} P_4$. Then let $C:= s P_2^{-1} t^{-1} P_4$. As before, this is a band which uses only arrows of $B_1$ but which does not use $\alpha_2$, hence our induction hypothesis gives that $\I(C)$ is a nonnegative combination of indicator vectors of elementary bands.
	\end{enumerate}

	In all cases, we have obtained a band $C$ such that every arrow of $C$ is given positive flow by $W$, and that $\I(C)$ is a positive combination of indicator vectors of elementary bands.
	Let $a_C\in\mathbb R_{\geq0}$ to be maximal such that for all arrows $\alpha\in\tL$, we have $a_C\I(C)\leq W(\alpha)$. Then the vortex $W-a_C\I(C)$ uses a strict subset of the arrows of $W$, hence by the induction hypothesis this vortex is a nonnegative combination of indicator vectors of elementary bands. Then the sum $W=(W-a_C\I(C))+(a_C\I(C))$ also has this property. We have shown that the set of indicator vectors of elementary bands generates the recession cone of $\F_1(\tL)$.

	It remains to show that this generating set is minimal.
	It is not hard to verify that if $B$ is an elementary band, then any band or route $p$ of $\tL$ which is not equal to $B$ must contain some arrow which is not in $B$.
	This shows that the indicator vector $\I(B)$ may not be obtained as a positive sum of indicator vectors of trails which are not $B$. In particular, this shows that no element may be removed from the generating set $\{\I(B)\ :\ B\text{ is an elementary band}\}$ without changing the cone, hence that this generating set is minimal.
\end{proof}

\begin{defn}\label{defn:elroute}
	A self-compatible route $p$ of $\tL$ is \emph{elementary} if either it is simple (i.e., simple as a path, so it does not use the same vertex multiple times) or it is of the form $R\sigma R^{-1}$, where
	\begin{enumerate}
		\item $R$ does not use the same vertex twice,
		\item $\sigma$ uses only the vertex $h(\sigma)=t(\sigma)$ twice, and
		\item $R$ and $\sigma$ share no vertices except $h(R)=t(\sigma)=h(\sigma)$.
	\end{enumerate}
	Equivalently, $p$ is elementary if and only if it either
	\begin{enumerate}
		\item $p$ has no boosted substrings or criss-crossed substrings, or
		\item $p$ has no boosted substring and exactly one maximal criss-crossed substring, and this substring contains a fringe vertex.
	\end{enumerate}
\end{defn}

See Figure~\ref{fig:elpath} for a visualization -- the left two paths are elementary, and the right two are not.
For a concrete example, the path $e_1e_2e_3e_1^{-1}$ shown in red in Figure~\ref{fig:gottabemarked} is elementary, but the path $e_1e_2e_3e_4$ is not elementary.

\begin{figure}
\centering
	\begin{minipage}[b]{.22\textwidth}
	\centering
	\def\svgscale{.21}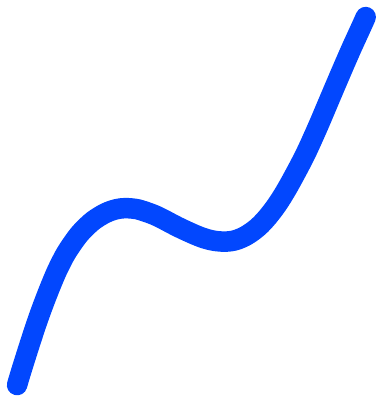

	simple

		(elementary)
\end{minipage}
	\begin{minipage}[b]{.22\textwidth}
	\centering
	\def\svgscale{.21}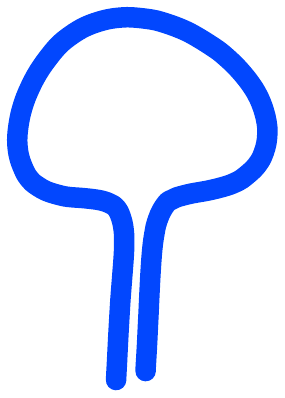

	lollipop

		(elementary)
\end{minipage}
	\begin{minipage}[b]{.22\textwidth}
	\centering
	\def\svgscale{.21}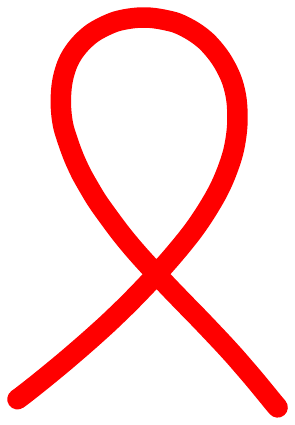

	nonelementary
\end{minipage}
	\begin{minipage}[b]{.22\textwidth}
	\centering
	\def\svgscale{.21}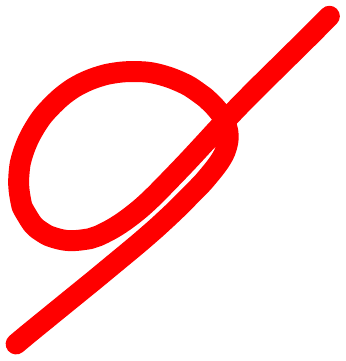

	nonelementary
\end{minipage}
	\caption{Examples of elementary (left) and nonelementary (right) paths.}
	\label{fig:elpath}
\end{figure}

\begin{prop}\label{prop:elementary-routes}
	The map $p\mapsto\I(p)$ gives a bijection between elementary routes of $\tL$ and vertices of $\F_1(\tL)$.
\end{prop}
\begin{proof}
	We first show that every vertex of $\F_1(\tL)$ must be the indicator vector of an elementary route.
	Take a flow $F$ which is a vertex of $\F_1(\tL)$. Since $\F_1(\tL)$ is a rational polyhedron, the flow $F$ must be rational, hence we may write $F=\sum_{p\in\bK_F^+}a_p^F\I(p)$ as a bundle combination by Theorem~\ref{thm:comb}. Since $F$ is a vertex, the bundle $\bK_F^+$ must consist of only one self-compatible route $\bK_F^+=\{p\}$. If $p$ is elementary, then we are done. If not, then we split into two cases to obtain a contradiction.

		Suppose first that $p$ has a boosted substring $\sigma$. Then write $p=P_1\sigma P_2\sigma P_3$. Define a route $q:=P_1\sigma P_3$ and a band $B:=\sigma P_2$. Then $\I(p)=\I(q)+\I(B)$, so $F=\I(p)$ is not a vertex, a contradiction.
		
		On the other hand, suppose $p$ is nonelementary and has no boosted substring. Since $p\in\bK_F^+$, the trail $p$ must be self-compatible, so $p$ has a maximal criss-crossed substring $\sigma$ which does not contain a fringe arrow.
		Then $p$ contains both $\alpha_1^{\e_1}\sigma\alpha_2^{\e_2}$ and $\beta_1^{-\e_1}\sigma^{-1}\beta_2^{-\e_2}$ for some signed arrows $\alpha_i^{\e_i}$ and $\beta_j^{-\e_j}$. 
			We say that $\{\alpha_1,\alpha_2,\beta_1,\beta_2\}$ are the arrows of $p$ \emph{surrounding} $\sigma$. In particular, note that the lazy string at $h(s)=t(\alpha_2^{\e_2})$ already appears as a substring of both these substrings. Since $p$ has no boosted substring, Remark~\ref{remk:lazyboost} shows that the arrow $\alpha_2$ (with any orientation) does not appear in $p$ outside of the substring $\alpha_1^{\e_1}\sigma\alpha_2^{\e_2}$. Similarly, the arrows $\alpha_1,\beta_1,\beta_2$ are used by $p$ only once.
		Write $p=P_1\sigma P_2\sigma^{-1} P_3$. If necessary, redefine $\alpha_j^{\e_j}$ and $\beta_j^{-\e_j}$ so that the last arrow of $P_1$ is $\alpha_1^{\e_1}$ and the first arrow of $P_2$ is $\alpha_2^{\e_2}$. Define the routes $q_1=P_1\sigma P_2\sigma^{-1} P_1^{-1}$ and $q_2=P_3^{-1}\sigma P_2\sigma^{-1} P_3$. Then $q_1$ does not contain the arrow $\alpha_2$ and $q_2$ does not contain the arrow $\alpha_1$, hence neither $\I(q_1)$ nor $\I(q_2)$ is equal to $\I(p)$. Moreover, $\I(p)=\frac{\I(q_1)+\I(q_2)}{2}$. This shows that $F=\I(p)$ is not a vertex of $\F_1(\tL)$, a contradiction.

	We have shown that the route $p$ must be elementary. This completes the proof that the vertex $F=\I(p)$ is the indicator vector of an elementary route.

	We now show that if $p$ is an elementary route, then $\I(p)$ is a vertex of $\F_1(\tL)$.
	It is not hard to verify that every route $q$ distinct from $p$ contains an arrow which is not in $p$. Since all vertices of $\F_1(\tL)$ are of the form $\I(q)$, and since the recession cone of $\F_1(\tL)$ is in the positive orthant, this means that $\I(p)$ may not be nontrivially realized as $I(p)=G+W$ where $G$ is a convex combination of vertices of $\F_1(\tL)$ and $W$ is in the recession cone of $\F_1(\tL)$. This shows that $\I(p)$ is a vertex.

	It remains only to show that the map $p\mapsto\I(p)$ is injective on elementary routes. Indeed, if two elementary routes $p$ and $q$ had the same indicator vectors $\I(p)=\I(q)$, then this flow $\I(p)=\I(q)$ would have two different expressions as a bundle combination, contradicting uniqueness of Theorem~\ref{thm:comb}. This ends the proof.
\end{proof}

Proposition~\ref{prop:elementary-routes} and Proposition~\ref{prop:elementary-bands} complete the proof of Theorem~\ref{thmZ} of the introduction:

\begin{thm}\label{thm:turb-vert-dir}
	The map $p\mapsto\I(p)$ is a bijection from elementary routes to vertices of $\F_1(\tL)$. The map $B\mapsto\I(B)$ on elementary bands is a bijection onto a minimal generating set for the recession cone of $\F_1(\tL)$.
\end{thm}
As a result, the turbulence polyhedron $\F(\tL)$ is equal to the Minkowski sum of (the convex hull of) all indicator vectors of elementary routes with the unbounded cone generated by indicator vectors of elementary bands.

\begin{example}\label{kron-el-routes-bands}
	Recall Example~\ref{ex:kron}, where we claimed that the four vertices of the turbulence polyhedron of the Kronecker fringed quiver are the indicator vectors $\I(\alpha_1\alpha_2^{-1})$, $\I(\gamma_1^{-1}\gamma_2)$, $\I(\alpha_1\beta_2\gamma_1)$, and $\I(\alpha_2\beta_1\gamma_2)$. One may check that these routes are precisely the elementary routes of the Kronecker fringed quiver. Moreover, the recession cone of $\F_1(\tL)$ is generated by the indicator vector $\I(\beta_1\beta_2^{-1})$ of the unique elementary band of $\tL$.
\end{example}

\begin{example}\label{ex:gemstone-vert-dir}
	Let $\tL$ be the shard fringed quiver of Figure~\ref{GEMINTRO} and Example~\ref{ex:gemd}. There are seven self-compatible routes:
	\[\{e_4^{-1}e_2e_3e_4,\ e_1e_2e_3e_1^{-1},\ e_1e_3^{-1}f_2,\ e_4^{-1}e_3^{-1}f_2,\ e_1e_2f_1^{-1},\ e_4^{-1}e_2f_1^{-1},\ f_1f_2,\ e_1e_2e_3e_4\}.\]
	The first six are elementary, and hence give vertices of the polyhedron $\F_1(\tL)$; the route $e_1e_2e_3e_4$ is not elementary and hence gives the unique lattice point of $\F_1(\tL)$ which is not a vertex. The only non-self-compatible route is $e_1e_3^{-1}e_2^{-1}e_4$, which has the same indicator vector as $e_1e_2e_3e_4$. There are no bands of $\tL$, hence there are no unbounded directions of $\F_1(\tL)$.
\end{example}

\begin{example}\label{ex:irrational-vert}
	Let $\tL$ be the double Kronecker quiver of Figure~\ref{fig:irrational}. In Example~\ref{ex:irrational} we saw that there exist flows of $\F_{\geq0}(\tL)$ which are not obtainable as a bundle combination.
	There are four elementary routes: $e_1e_2e_3e_4$, $f_1f_2f_3f_4$, $e_1f_1^{-1}$, and $e_4^{-1}f_4$. There are two elementary bands: $e_2f_2^{-1}$ and $e_3f_3^{-1}$.
	The vertices of $\F_1(\tL)$ are the indicator vectors of elementary routes, and the recession cone of $\F_1(\tL)$ is minimally generated by the two indicator vectors of elementary bands. The turbulence polyhedron $\F_1(\tL)$ is four-dimensional, so we do not include a picture here, but in Section~\ref{sec:example} we will see a three-dimensional reduced version in the form of the g-polyhedron (Figure~\ref{fig:doubkron-g}).
\end{example}

\section{$\g$-vector Fans and g-Polyhedra}
\label{sec:g}

The g-vector fan of a finite-dimensional algebra $\L$ (see Section~\ref{sec:bac}) is closely related to its $\tau$-tilting theory and carries rich representation-theoretic information about $\L$.
In this section, we define a map $\phi$ from the nonnegative flows $\F_{\geq0}(\tL)$ of a gentle algebra $\L$ into the ambient space of the g-vector fan of $\L$ which maps the indicator vector of a route of $\tL$ to the g-vector of its string $\L$-module.
In this way, we establish a tight connection between flows of $\tL$ and g-vectors of $\L$.
We define the \emph{g-polyhedron} $\g_1(\L)$ as the image of the unit turbulence polyhedron $\F_1(\tL)$ under $\phi$.
In the representation-finite case, this lines up with the g-polytope of~\cite{AHIKM}, proving that representation-finite gentle algebras are g-convex. Even in the representation-infinite case, our g-polyhedron is the completion of their g-polytope, showing that representation-infinite gentle algebras are g-convex in a reasonable sense.

\subsection{g-vector polyhedra}

Recall the definition of the g-vector of a support $\tau$-tilting module given in Definition~\ref{defn:g-vector}.
The authors of~\cite[Definition 5.2]{AHIKM} defined a bounded version of the g-vector fan of a finite dimensional algebra $\Lambda$ called the \emph{g-polytope} of $\Lambda$ (which may fail to be convex in general). 
We give a slightly altered definition now, which we call the \emph{g-polyhedron}.

\begin{defn}[{\cite[Definition 5.2]{AHIKM}}]
	If $M=\oplus_{i=1}^nM_i$ is a support $\tau$-tilting $\L$-module, define the \emph{g-polyhedron} $\g_{\geq0}(\Lambda)$
	to be the completion of $\bigcup_{M\in\sttilt\L}\g_{1}(M)$
	(i.e., the subset of $\mathbb R^{V_{\text{int}}}$ obtained by appending to $\bigcup_{M\in\sttilt\L}\g_{1}(M)$ all of its limit points in $\mathbb R^{V_{\text{int}}}$).
\end{defn}

The difference between our above definition and the g-polytope of~\cite[Definition 5.2]{AHIKM} is that we take a completion while they do not.
When $\Lambda$ is $\tau$-tilting finite, there are a finite number of support $\tau$-tilting modules and hence the completion does not do anything and our g-polyhedron agrees with their g-polytope. 
The authors of~\cite{AHIKM} focus mainly on the $\tau$-tilting finite case, and observe that the (noncompleted) g-polytope is never convex when $\Lambda$ is $\tau$-tilting infinite~\cite[Theorem 5.10]{AHIKM}.

The g-polytope of an arbitrary finite dimensional algebra $\Lambda$ may fail to be convex, even in the $\tau$-tilting finite case, and~\cite{AHIKM} is particularly interested in determining when a g-polytope is convex. We will see that when $\L$ is a gentle algebra, the g-polyhedron is always a convex polyhedron!
In this sense, we consider all gentle algebras to be g-convex.

\subsection{$\g$-vectors of routes and bands}

Recall Proposition~\ref{prop:g-vect}, which gave a combinatorial way to calculate the g-vector of the string $\L$-module $\tilde M(p)$ from its bending route $p$ of the fringed quiver. We begin by reinterpreting this as the g-vector of the route $p$, and extending it to define g-vectors for all routes and bands of $\tL$.

\begin{defn}
	\label{defn:g-vector-trail}
	Let $p$ be a route or band of $\tL$. Define $\g(s):=\textbf{a}-\textbf{b}$, where $\textbf{a}=(a_v)_{v\in V_{\text{int}}}$ such that $a_v$ is the number of times the lazy string $e_v$ is a top substring of $p$, and $\textbf{b}=(b_v)_{v\in V_{\text{int}}}$ such that $b_v$ is the number of times the lazy string $e_v$ is a bottom substring of $p$.
\end{defn}

\begin{remk}\label{remk:83}
	The g-vectors of bands and routes of $\tL$ are related to g-vectors of corresponding modules over the gentle algebra $\L$ as follows:
	\begin{enumerate}
		\item When $p$ is a bending route, the g-vector $\g(p)$ is equal to $\g(\tilde M(p))$ as given in Proposition~\ref{prop:g-vect}.
		\item Bands on $\tL$ (and their band modules) may be considered as bands (and band modules) on $\L$. Any band module $\tilde M(B,n,\lambda)$ over $\L$ is not $\tau$-rigid, but has g-vector $n\g(B)$.
		\item When $p$ is a straight route, no internal vertex is at the top or bottom of $v$ and hence $\g(p)$ is the zero vector. This makes sense because straight routes of $\tL$ are not associated to any strings or modules of $\L$,
	\end{enumerate}
\end{remk}

Recall that if $p=\alpha_1^{\e_1}\dots\alpha_m^{\e_m}$ is a band and $\e_m=1$ and $\e_1=-1$ (respectively $\e_m=-1$ and $\e_1=1$), then we consider $e_{t(\alpha_1^{\e_1})}=e_{h(\alpha_m^{\e_m})}$ to be a bottom (respectively top) substring of $p$.
For example, in the Kronecker fringed quiver of Figure~\ref{fig:kron1}, the route $p_1:=\alpha_1\beta_2\beta_1^{-1}\beta_2\beta_1^{-1}\alpha_1^{-1}$ has g-vector $\g(p_1)=(1,-2)$, and the band $p_1:=\beta_1\beta_2^{-1}$ has g-vector $\g(p_2)=(1,-1)$.

\begin{defn}
	Let $\bK=\K\cup\B$ be a maximal bundle or a maximal reduced bundle.
	\begin{itemize}
		\item A \emph{g-bundle combination} of $\bK$ is a linear combination \[\sum_{p\in\K\text{ bending}}a_p\g(p)+\sum_{p\in\B}a_p\g(p),\] where each $a_p$ is nonnegative.
		\item The g-bundle combination is \emph{unit} if $\sum_{p\in\K}a_p\leq1$.
		\item The \emph{g-bundle space} $\g_1(\bK)$ is the space of unit g-bundle combinations of $\bK$. 
			\begin{itemize}
				\item If $\bK$ is a clique (i.e., if $\B=\emptyset$), then we also call this a \emph{g-simplex}. 
				\item If $\bK$ contains a band (i.e., if $\B\neq\emptyset$), then we also call this a \emph{g-wall}.
			\end{itemize}
		\item The \emph{g-bundle cone} $\g_{\geq0}(\bK)$ is the space of (not necessarily unit) g-bundle combinations of $\bK$.
			\begin{itemize}
				\item If $\bK$ contains a band (i.e., if $\B\neq\emptyset$), then we also call this a \emph{g-wall}.
			\end{itemize}

	\end{itemize}
	Note that the straight routes of $\bK$ do not factor into any of these definitions, so that the g-bundle space (resp. g-bundle cone) of a maximal bundle $\bK$ is the same as the g-bundle space (resp. g-bundle cone) of $\bK_{\text{red}}$.
\end{defn}

Since the g-vector of a route $p$ is the same as the g-vector of the $\L$-module $\tilde M(p)$, we have the following fact about g-simplices of maximal cliques (recall the definition of g-simplices and g-cones of support $\tau$-tilting modules from Definition~\ref{defn:g-stuff}).

\begin{remk}\label{remk:g-sttilt-clique}
	Let $\K$ be a maximal clique and let $\tilde M(\K)=\oplus_{p\in\K\text{ bending}}\tilde M(p)$ be the corresponding support $\tau$-tilting module over $\Lambda$. Then we have the equality of g-simplices $\g_{1}(\K)=\g_{1}(\tilde M(\K))$ and g-cones $\g_{\geq0}(\K)=\g_{\geq0}(\tilde M(\K))$.
\end{remk}
It follows, then, that the g-polyhedron of $\L$ is the completion of the g-simplices of maximal cliques of $\tL$, and that the g-vector fan of $\L$ is the collection of g-cones of maximal cliques of $\tL$.

\subsection{The transformation $\phi$}

We relate the turbulence polyhedron $\F_1(\L)$ with the g-polyhedron $\g_1(\L)$ by defining a linear transformation $\phi$ which sends the indicator vector of a route (or band) to its g-vector.
Give $\mathbb R^E$ the basis $\{e_\alpha\ :\ \alpha\in E\}$ and give $\mathbb R^{V_{\text{int}}}$ the basis $\{e_v\ :\ v\in V_{\text{int}}\}$. 

\begin{defn}\label{defn:phi}
		Define a linear transformation 
	\begin{align*}
		\phi:\mathbb R^E&\longrightarrow\mathbb R^{V_{\text{int}}} \\
		e_\alpha&\longmapsto\frac{e_{t(\alpha)}-e_{h(\alpha)}}{2},
	\end{align*} where $e_{t(\alpha)}$ (respectively $e_{h(\alpha)}$) is treated as 0 if $t(\alpha)$ (respectively $h(\alpha)$) is fringe.
\end{defn}
The definition of $\phi$ is motivated by the following remark.
\begin{remk}\label{remk:collapse-exceptional}
	The map $\phi$ sends an indicator vector of a band or bending route of $\tL$ to its g-vector, and it sends any straight route of $\tL$ to zero.
\end{remk}

A consequence of Remark~\ref{remk:collapse-exceptional} and Remark~\ref{remk:g-sttilt-clique} is the following:
\begin{prop}\label{prop:phi-on-simplices}
	If $\K$ is a maximal clique and $M(\K)$ is the corresponding support $\tau$-tilting $\L$-module, then $\phi(\Delta_{\geq0}(\K))=\g_{\geq0}(\K)=\g_{\geq0}(M(\K))$ and $\phi(\Delta_1(\K))=\g_{1}(\K)=\g_{1}(M(\K))$.
\end{prop}
\begin{cor}
	The map $\phi$ restricted to the affine span of $\F_1(\tL)$ preserves normalized volume.
\end{cor}
Denote by $\text{aff}(\F_1(\tL))$ the affine span of $\F_1(\tL)$ in $\mathbb R^{E}$.
\begin{proof}
	The simplex $\Delta_1(\K)$ is unimodular by Lemma~\ref{lem:simp-unimodular} and $\g_{1}(\K)=\g_{1}(M(\K))$ is unimodular by Theorem~\ref{thm:uni-cite}. Since $\phi(\Delta_1(\K))=\g_1(\K)$ by Proposition~\ref{prop:phi-on-simplices}, the map $\phi|_{\text{Aff}(\F_1(\tL))}$ preserves normalized volume.
\end{proof}

The map $\phi$ is immediately seen to have a few other nice properties.

\begin{lemma}\label{lem:0-int1}
	The polyhedron $\phi(\F_1(\tL))$ is $|V_{\text{int}}|=n$-dimensional in $\mathbb R^{V_{\text{int}}}$. It contains the g-polyhedron $\g_1(\L)$ and contains the origin as an interior point.
\end{lemma}
\begin{proof}
	The g-simplex of any support $\tau$-tilting module $M$ is the image under $\phi$ of the clique simplex $\Delta_1(M)\subseteq\F_1(\tL)$. Hence, the union of all g-simplices is in $\phi(\F_1(\tL))$. Since $\phi(\F_1(\tL))$ is a polyhedron, it must contain the completion of all g-simplices, which is $\g_{\geq0}(\L)$.

	It remains to show that the g-polyhedron $\g_{\geq0}(\L)$ is $n$-dimensional in $\mathbb R^{V_{\text{int}}}$ and contains the origin as an interior point.
	For any $v\in V_{\text{int}}$, let $p_v$ be the route corresponding to the simple projective module at $v$ and $q_v$ be the route corresponding to the shifted projective module at $v$. Then $\phi(\I(p_v))=\g(P(v))$ has a 1 at $e_v$ and 0's at all other coordinates, and $\phi(\I(q_v))=\g(P_v[1])$ has a -1 at $e_v$ and 0's at all other coordinates. The convex hull of these vectors is $n$-dimensional and contains the origin in its interior, hence $\g_{\geq0}(\L)$ is $n$-dimensional and contains the origin in its interior.
\end{proof}

The following result amounts to saying that up to unimodular equivalence, the map $\phi$ restricted to the set of flows of $\tL$ precisely quotients out by the affine span of the indicator vectors of straight routes.

\begin{lemma}\label{lem:phi-preimage}
	Let $F_1$ and $F_2$ be nonnegative flows of $\tL$. Then $\phi(F_1)=\phi(F_2)$ if and only if $F_2=F_1+\sum_{p\text{ straight route}}a_p\I(p)$ for some (not necessarily positive) coefficients $a_p\in\mathbb R$. If $F_1,F_2\in\F_1(\tL)$, then the coefficients $a_p$ sum to zero.
\end{lemma}
\begin{proof}
	Since $\phi(\I(p))=0$ for any straight route $p$, it is immediate that if 
	$F_2=F_1+\sum_{p\text{ straight route}}a_p\I(p)$ then $\phi(F_1)=\phi(F_2)$. We must show the reverse direction.
	By Lemma~\ref{lem:flowdim}, the polyhedron $\F_1(\tL)$ is of dimension $|V_{\text{int}}|-1+s$, where $s$ is the number of straight routes.
	By Remark~\ref{remk:collapse-exceptional}, $\phi$ sends the indicator vector of any straight route to the origin, hence $\phi$ at least quotients out by the affine span of indicator vectors of straight routes. Quotienting out by the indicator vectors of the $s$ straight routes brings the dimension $\dim\F_1(\tL)=|V_{\text{int}}|-1+s$ down by $s-1$ to $|V_{\text{int}}|$. By Lemma~\ref{lem:0-int1}, the polyhedron $\phi(\F_1(\tL))$ is of dimension $n=|V_{\text{int}}|$ hence up to unimodular equivalence $\phi$ precisely quotients out the affine span of the indicator vectors of straight routes.
	This ends the proof.
\end{proof}

\begin{remk}\label{remk:collapse}
	We connect the map $\phi$ to the work of Danilov, Karzanov, and Koshevoy in flow polytopes~\cite{DKK}, particularly those coming from amply framed DAGs. Recall that amply framed DAGs and their flow polytopes may be converted to the language of gentle algebras and their turbulence polyhedra as in Section~\ref{ssec:PGAaAFDG} -- we will phrase their results in our language of gentle algebras and turbulence polyhedra.
	They work with the cone of flows $\F_{\geq0}(\tL)$, rather than working as we do primarily with the unit flows $\F_1(\tL)$, and define a map which collapses the span of the indicators of exceptional routes in $\mathbb R^E$. The equivalent map in our setting would be the map $\psi$ defined as the map quotienting out the affine span $\text{aff}\{\I(p)\ :\ p\in\tL\text{ is a straight route}\}$. They are particularly interested in this map when $\L$ comes from an amply framed DAG. On the other hand, they give no specific coordinates for this quotient map. Lemma~\ref{lem:phi-preimage} shows that the map $\phi$ precisely quotients out the affine span of the indicator vectors of straight routes -- in other words, up to unimodular equivalence, our map $\phi$ gives explicit coordinates to their map $\psi$.
	{A description of the map $\phi$ for a certain generalization of amply framed DAGs (see Remark~\ref{remk:UQAM}) appears in the concurrent work~\cite{UQAM}.}
\end{remk}

\subsection{The image $\phi(\F_1(\tL))$}

We show that the image $\phi(\F_1(\tL))$ is precisely the g-polyhedron of $\L$.

\begin{prop}\label{prop:g-vect-iff-nonbandy}
	Any nonnegative flow $F\in\F_{\geq0}(\tL)$ maps into the g-vector fan through $\phi$ if and only if $F$ is given by a clique combination.
\end{prop}
\begin{proof}
	If $F$ is given by a clique combination, then it is in the clique cone $\Delta_{\geq0}(\K)$ for some maximal clique $\K$. Then $\phi(F)$ is in $\phi(\Delta_{\geq0}(\K))$, which is the g-cone of the support $\tau$-tilting $\L$-module corresponding to $\K$. This shows that $\phi(F)$ is in the g-vector fan of $\L$.

	Conversely, suppose $F$ is a nonnegative flow and $\phi(F)$ is in the g-vector fan $g_{\geq0}(\L)$. Then $\phi(F)$ is in some g-cone $\g_{\geq0}(M)$, where $M=\oplus_{j\in[n]}M_j$ is a support $\tau$-tilting $\L$-module.
	Write $\phi(F)=\sum_{j\in[n]}a_j\g(M_j)$, where each $a_j\geq0$.
	For $j\in[n]$, let $p_j$ be the route of $\tL$ such that $\tilde M(p_j)=M_j$; then the flow
	\[F':=\sum_{j\in[n]}a_j\I(p_j)\]
	maps to $\phi(F)$ through $\phi$.
	Then Lemma~\ref{lem:phi-preimage} shows that $F=F'+\sum_{p\text{ straight route}}a_p\I(p)$ for some coefficients $a_p\in\mathbb R$.
	Let $p$ be a straight route of $\tL$. By Lemma~\ref{lem:dist}, the straight route $p$ has at least one distinguished arrow $\alpha_p$ in the clique $\K$ consisting of the straight routes as well as the routes $p_j$ for $j\in[n]$.
	Then no other route of $\K$ may use $\alpha_p$, hence $F(\alpha_p)=a_p$. Since $F$ is nonnegative, the coefficient $a_p$ is nonnegative. This holds for all straight routes $p$, hence
	\[F=\sum_{j\in[n]}a_j\I(p_j)+\sum_{p\text{ straight route}}a_p\I(p)\]
	realizes $F$ as a clique combination. This completes the proof.
\end{proof}

We now use density of g-cones and g-simplices to prove density of clique cones and clique simplices.

\begin{prop}\label{prop:density}
	Let $\L$ be a gentle algebra.
	\begin{enumerate}
		\item\label{ghsimp} The clique simplices are dense in $\F_1(\tL)$.
		\item\label{ghcone} The clique cones are dense in $\F_{\geq0}(\tL)$.
	\end{enumerate}
\end{prop}
\begin{proof}
	We prove~\eqref{ghsimp} first.
	More specifically, we show that for any $F\in\F_{1}(\tL)$ and $\delta>0$, we may find $F'\in\F_{1}(\tL)$ such that $|F'-F|<\delta$.
	It suffices to prove this result when $F$ is rational and gives every arrow of $\tL$ positive flow, since the class of such flows are dense in $\F_1(\tL)$.
	Hence, take a rational unit flow $F$ of $\tL$ which gives every arrow positive flow.

	Let $\K$ be a maximal clique of $\tL$ with bending routes $p_1,\dots,p_n$ and straight routes $q_1,\dots,q_s$. Let $A$ be the ($n$-dimensional) affine span of $p_1,\dots,p_n,q_1$.
	Let $A_F$ be the translation of $A$ in $\mathbb R^E$ which contains $F$.

	It is immediate that $\phi$ maps the convex hull of $p_1,\dots,p_n,q_1$ bijectively onto the g-simplex of $\K$, hence the affine transformation $\phi|_A$ given by restricting $\phi:\mathbb R^E\to\mathbb R^{V_{\text{int}}}$ to $A$ is a bijection onto $\mathbb R^{V_{\text{int}}}$.
	Since $A_F$ is a translation of $A$, it follows that the affine transformation $\phi|_{A_F}$ is a bijection from $A_F$ onto $\mathbb R^{V_{\text{int}}}$.

	Define $\delta'>0$ strictly less than $\min_{\alpha\in E}F(\alpha)$, so that any $F'\in\mathbb R^E$ satisfying $|F'-F|<\delta'$ has all positive coordinates. Define $\kappa:=\min(\delta,\delta')$.
	Let $A_F^\kappa$ be the intersection of $A_F$ with the open ball $\{F'\in\mathbb R^E\ :\ |F'-F|<\kappa\}$. Since $\kappa<\delta'$, we have $A_F^\kappa\subseteq\F_1(\tL)$.
	Since $\phi|_{A_F}$ is a bijection, the image $\phi(A_F^\kappa)$ is a full-dimensional subset of $\mathbb R^{V_{\text{int}}}$. The g-vector fan is dense by Theorem~\ref{thm:g-vector-dense}, so we may choose $\xx\in\phi(A_F^\kappa)$ which is in the g-vector fan. Then $F':=(\phi|_{A_F})^{-1}(\xx)\in A_F^\kappa$ is given by a clique combination by Proposition~\ref{prop:g-vect-iff-nonbandy}, hence is in a clique simplex. Since $F'\in A_F^\kappa$, we have $|F-F'|<\kappa<\delta$, completing the proof of~\eqref{ghsimp}.

	We now treat~\eqref{ghcone}.
	We wish to show that for any $F\in\F_{\geq0}(\tL)$ and $\delta>0$, we may find $F'\in\F_{\geq0}(\tL)$ such that $|F'-F|<\delta$. As before, it suffices to consider the case when $F$ is rational and gives every arrow of $\tL$ positive flow. Let $S=\sum_{\alpha\in E_{\text{fringe}}}F(\alpha)>0$ be the strength of $F$. Scale $F_1=\frac{1}{S}F\in\F_1(\tL)$. By the proof of~\eqref{ghcone}, we may find $F'_1\in\F_1(\tL)$ such that $|F_1-F'_1|<\frac{1}{S}\delta$. Setting $F':=SF'_1$, we then have $|F-F'|=S|F_1-F'_1|<\delta$, completing the proof of~\eqref{ghcone}.
\end{proof}

\begin{cor}\label{cor:phi-im-g}
	The image $\phi(\F_1(\tL))$ is precisely the g-polyhedron.
\end{cor}
\begin{proof}
	Since $\phi$ sends a clique simplex to the corresponding g-simplex, the image $\phi(\F_1(\tL))$ contains all of the g-simplices. 
Since $\phi(\F_1(\tL))$ is an image of a polyhedron under a linear transformation and hence is itself a polyhedron, it is closed. Then $\phi(F_1(\tL))$ must contain the {completion} of the g-simplices, which is the g-polyhedron. This shows that $\phi(\F_1(\tL))\supseteq\g_{\geq0}(\L)$.
	On the other hand, Proposition~\ref{prop:density} says that $\F_1(\tL)$ is equal to the completion of the clique simplices, hence $\phi(\F_1(\tL))$ is contained in the completion of the g-simplices. This completes the proof.
\end{proof}

Corollary~\ref{cor:phi-im-g} shows that the g-polyhedron of an arbitrary gentle algebra is a polyhedron. The author considers this important enough to cite separately as an analog of g-convexity of~\cite{AHIKM}.

\begin{thm}\label{thm:g-poly-is-poly}
	The g-polyhedron $\g_{\geq0}(\L)$ of a gentle algebra $\L$ is a convex polyhedron.
\end{thm}

\begin{remk}
	Lemma~\ref{lem:phi-preimage} shows that $\phi|_{\text{Aff}(\F_1(\tL))}$ is a quotient map along the affine span of indicator vectors of straight routes. Hence, it is a unimodular equivalence precisely when there is only one straight route (for example, as in Figure~\ref{fig:facetweird}). Then Corollary~\ref{cor:phi-im-g} shows that $\F_1(\tL)$ is unimodularly equivalent to $\g_1(\L)$ precisely when there is only one straight route.
\end{remk}

\subsection{Clique triangulation of the turbulence polyhedron}

Proposition~\ref{prop:density} allows us to finally complete our proof of Theorem~\ref{ITHM:CLIQUE}.

\begin{defn}
	Define the \emph{(unit) clique triangulation} of $\F_1(\tL)$ as
	\[\Tt(\F_1(\tL)):=\{\D_{\bK}\ :\ \bK\text{ is a maximal clique of }\tL\}.\]
	Define the \emph{nonnegative clique triangulation} of $\F_{\geq0}(\tL)$ as
	\[\Tt(\F_{\geq0}(\tL)):=\{\D_{\bK}^{\geq0}\ :\ \bK\text{ is a maximal clique of }\tL\}.\]
\end{defn}
Note that the clique triangulation is merely the bundle subdivision without the bundle walls.
\begin{cor}\label{cor:clique-subd}
	The clique triangulation is a unimodular triangulation of $\F_1(\tL)$, and the nonnegative clique triangulation is a unimodular triangulation of $\F_{\geq0}(\tL)$.
	When $\tL$ is representation-finite, both triangulations are complete.
\end{cor}
\begin{proof}
	Theorem~\ref{thm:comb} shows the strong intersection property and Proposition~\ref{prop:density} shows density, hence the clique triangulation is a subdivision.
	Corollary~\ref{lem:simp-unimodular} shows that the cells of the clique triangulation are full-dimensional unimodular simplices, hence the clique triangulation is a unimodular triangulation.
\end{proof}

\subsection{Cardinality of bundles}

The fact that the clique simplices are dense provides an interesting proof that the cardinality of a maximal bundle containing at least one band is strictly less than the cardinality of a maximal clique. 
\begin{lemma}\label{lem:trail-space-dim}
	Let $\bK$ be a bundle with at least one route. The dimension of $\Delta_1(\bK)$ is $|\bK|-1$.
\end{lemma}
\begin{proof}
	It is immediate that $\dim\Delta_1(\bK)\leq|\bK|-1$ by definition of the bundle space.
	If the dimension is less than $|\bK|-1$, then there must be a single point in $\Delta_1(\bK)$ with two different expressions as a bundle combination of $\bK$, contradicting Theorem~\ref{thm:comb}.
\end{proof}

\begin{cor}\label{cor:bandy-lowcard}
	The cardinality of a maximal bundle containing at least one band is strictly less than the cardinality of a maximal clique.
\end{cor}
Recall that $|E|-|V_{\text{int}}|$ is an upper bound for the cardinality of a maximal bundle by Corollary~\ref{cor:card-upper-bound}, and equal to the cardinality of a maximal clique by Theorem~\ref{thm:ccc}.
\begin{proof}
	Say $\bK$ is a bundle containing at least one band with cardinality $|E|-|V_{\text{int}}|$. Then the dimension of its bundle wall $\Delta_1(\bK)$ is $|E|-|V_{\text{int}}|-1=\dim\F_1(\tL)$ by Lemma~\ref{lem:trail-space-dim}.
	In particular, anything in the interior of $\Delta_1(\bK)$ is not in the completion of the clique simplices, contradicting Proposition~\ref{prop:density}.
\end{proof}

\section{Clique and Bundle Subdivisions, Facets, and Presentations of g-polyhedra}
\label{sec:CBS}

In this section, we take results which we have proven about triangulations, subdivisions, vertices, and unbounded directions of $\F_1(\tL)$ and $\F_{\geq0}(\tL)$ and we convert them into statements about the g-polyhedron and g-vector fan of $\L$.

\subsection{Clique and bundle subdivisions of the g-polyhedron}
\label{ssec:CBSG}

We have proven Theorems~\ref{ITHM:CLIQUE} and~\ref{ITHM:BRIQUE}, showing critical facts about existence and uniqueness of bundle combinations and phrasing the clique and bundle subdivisions of the turbulence polyhedron. We now apply the map $\phi$ to these results to obtain analogous results about the g-polyhedron.
We first use Theorem~\ref{ITHM:BRIQUE} to show the analogous result about g-bundle combinations.
\begin{thm}\label{thm:g-decomp}
	Let $\L$ be a gentle algebra. Any point $\xx\in\mathbb R^{V_{\text{int}}}$ has at most one representation as a positive g-bundle combination. Every rational point in $\mathbb R^{V_{\text{int}}}$ may be represented as a g-bundle combination.
\end{thm}
\begin{proof}
	We first show that $\xx\in\mathbb R^{V_{\text{int}}}$ may have at most one expression as a positive g-bundle combination. Suppose that we have an equality of positive g-bundle combinations
	\begin{equation}\label{tmb}\sum_{p\in\K_1\text{ bending}}a_p\g(p)+\sum_{p\in\B_1}a_p\g(p)=\sum_{p\in\K_2\text{ bending}}b_p\g(p)+\sum_{p\in\B_2}b_p\g(p).\end{equation}
		By Lemma~\ref{lem:phi-preimage}, the values 
			$\sum_{p\in\K_1\text{ bending}}a_p\I(p)+\sum_{p\in\B_1}a_p\I(p)$ and $\sum_{p\in\K_2\text{ bending}}b_p\I(p)+\sum_{p\in\B_2}b_p\I(p)$ differ only by indicator vectors of straight routes. Then for some partition $A_1\sqcup A_2$ of the set of straight routes of $\tL$ and nonnegative coefficients $a_p$ for $p\in A_1\sqcup A_2$, we have
		the equality of positive bundle combinations
	\begin{equation}\label{tmthree}\sum_{p\in\K_1\text{ bending}}a_p\I(p)+\sum_{p\in\B_1}a_p\I(p)+\sum_{p\in A_1}a_p\I(p)=\sum_{p\in\K_2\text{ bending}}b_p\I(p)+\sum_{p\in\B_2}b_p\I(p)+\sum_{p\in A_2}a_p\I(p).\end{equation}
		Bundle combinations for a fixed flow are unique by Proposition~\ref{lem:only-way}, so the two positive bundle combinations of equation~\eqref{tmthree} are the same, hence the positive g-bundle combinations of equation~\eqref{tmb} are also the same. This shows the first statement.
	It remains to show that if $\xx$ is rational, then it may be expressed as a g-bundle combination. This follows by applying $\phi$ to Theorem~\ref{thm:comb}.
\end{proof}

We now are able to phrase the g-clique triangulations and g-bundle subdivisions.

\begin{defn}\label{defn:ct}
	Define the \emph{g-clique triangulations}
	\begin{align*}\Tt(\g_1(\L))&:=\{\g_1(\K)\ :\ \K\text{ is a maximal clique of }\tL\}
	\text{ and }\\
		\Tt(\g_{\geq0}(\L))&:=\{\g_{\geq0}(\K)\ :\ \K\text{ is a maximal clique of }\tL\}.
	\end{align*}
	Define the \emph{g-bundle subdivisions}
	\begin{align*}\Sbs^{\spircle}(\g_1(\L))&:=\{\g_1(\bK)\ :\ \bK\text{ is a maximal bundle of }\tL\}
	\text{ and }\\
		\Sbs^{\spircle}(\g_{\geq0}(\L))&:=\{\g_{\geq0}(\bK)\ :\ \bK\text{ is a maximal bundle of }\tL\}.
	\end{align*}
\end{defn}

It is a standard construction in representation theory that $\Tt(\g_{\geq0}(\L))$ is a complete triangulation of the g-vector fan $\g_{\geq0}(\L)$, and hence a triangulation of $\mathbb R^{V_{\text{int}}}$ since the g-vector fan is dense in $\mathbb R^{V_{\text{int}}}$ by~\cite{AY}. A corollary of our turbulence polyhedron results is that the other three subdivisions defined above behave well.

\begin{cor}\label{thmDNI}
	The g-clique triangulation $\Tt(\g_1(\L))$ is a unimodular triangulation of the g-polyhedron $\g_1(\L)$.
	The g-bundle subdivision $\Sbs^{\spircle}(\g_1(\L))$ is a subdivision of the g-polyhedron $\g_1(\L)$ which covers all rational points.
\end{cor}
\begin{proof}
	Theorem~\ref{thm:g-decomp} gives the intersection and density properties in all cases.
\end{proof}

We also prove the analogous subdivision results about the g-vector fan.
\begin{cor}
	\label{thmENI}
	The g-clique subdivision $\Tt(\g_{\geq0}(\L))$ recovers the usual g-vector fan of $\L$.
	The g-bundle subdivision $\Sbs^{\spircle}(\g_{\geq0}(\L))$ is a subdivision of $\mathbb R^{V_{\text{int}}}$ which covers all rational points.
\end{cor}
\begin{proof}
	The first statement is immediate.
	The second is given by Theorem~\ref{thm:g-decomp}.
\end{proof}

We remark specifically that $\Sbs^{\spircle}(\g_{\geq0}(\tL))$ is obtained from the g-vector fan by adding in the lower-dimensional g-walls of maximal bundles containing bands.
These new g-walls are hence giving us some information about the complement of the g-vector fan.

\begin{example}\label{ex:kron-g-decomp}
	Recall Example~\ref{ex:kron-eddy-decomp}, which gave the bundle subdivision of the turbulence polyhedron of the Kronecker quiver $\L$ in Figure~\ref{fig:kron1}.
	Figure~\ref{KRONINTRO} shows the g-polyhedron of $\L$, where the left internal vertex is the x-axis and the right is the y-axis.
	Note that the g-polyhedron $\g_1(\L)$ is obtained by quotienting the turbulence polyhedron $\F_1(\tL)$ by the affine span of the indicator vectors of straight routes, drawn as black vertices at the top and bottom. This quotient respects all subdivisions which we consider.

	The vertices of $\g_1(\L)$ are $\g({e_3}^{-1}{f_3})=(0,1)$ and $\g({e_1}{f_1}^{-1})=(-1,0)$. Note that these are the images of $\I({e_3}^{-1}{f_3})$ and $\I({e_1}{f_1}^{-1})$ under $\phi$. The other two vertices of $\F_1(\tL)$ come from straight routes and are hence sent to the origin under $\phi$. The recession cone of $\g_1(\L)$ is the ray generated by $\g({e_2}{f_2}^{-1})=(1,-1)$.
	Example~\ref{ex:kron-eddy-decomp} listed the maximal bundles as
	\begin{align*}
		\{{e_1}{e_2}{e_3},\ {f_1}{f_2}{f_3},\ &{e_3}^{-1}{f_3},\ {e_1}{f_1}^{-1}\}, \\
		\{{e_1}{e_2}{e_3},\ {f_1}{f_2}{f_3},\ &{e_3}^{-1}({e_2}{f_2}^{-1})^j{f_3},\ {e_3}^{-1}({e_2}{f_2}^{-1})^{j+1}{f_3}\} \ (\textup{for any }j\geq0), \\
		\{{e_1}{e_2}{e_3},\ {f_1}{f_2}{f_3},\ &{e_1}({e_2}{f_2}^{-1})^j{f_1}^{-1},\ {e_1}({e_2}{f_2}^{-1})^{j+1}{f_1}^{-1}\}\ (\textup{for any }j\geq0), \text{ and } \\
		\{{e_1}{e_2}{e_3},\ {f_1}{f_2}{f_3},\ &{e_2}{f_2}^{-1}\}.
	\end{align*}
	The top three rows consist only of routes, hence they give g-simplices. Their g-simplices give a unimodular triangulation of a dense subset of $\g_1(\L)$: 
	\begin{itemize} 
		\item The top row gives the g-simplex with vertices 
			\[\{(0,0),\g({e_1}{f_1}^{-1}),\g({e_3}^{-1}{f_3})\}=\{(0,0),(-1,0),(0,1)\}.\]
		\item The next row gives the g-simplex with vertices
			$\{(0,0),(j,1-j),(j+1,-j)\}$
			for any $j\in\mathbb Z_{\geq0}$.
		\item the next row gives the g-simplex with vertices
			$\{(0,0),(j-1,-j),(j,-j-1)\}$
			for any $j\in\mathbb Z_{\geq0}$.
	\end{itemize}
	The final maximal bundle $\{{e_1}{e_2}{e_3},{f_1}{f_2}{f_3},{e_2}{f_2}^{-1}\}$ is the only bundle with a band. This gives the g-wall $\{(a,-a)\ :\ a\in\mathbb R_{\geq0}\}$ consisting of all scalar multiples of the ray $\g({e_2}{f_2}^{-1})=(1,-1)$.
	This g-wall covers every point of $\g_1(\L)$ which is not part of a g-simplex.

	Taking the cone over this subdivided g-polyhedron gives a subdivision for all of $\mathbb R^{V_{\text{int}}}$. The g-vector fan of $\L$ is the union of the g-cones of maximal cliques, whereas the bundle wall of the sole maximal bundle containing a band gives the one-dimensional complement of the g-vector fan.
\end{example}

\begin{example}\label{ex:gem-g}
	Let $\L$ be the gentle algebra of Example~\ref{ex:gemd} (Figure~\ref{fig:nonv-turb}).
	Figure~\ref{GEMINTRO} gives the triangulated g-polyhedron, where the left internal vertex is the x-axis and the right is the y-axis.
\end{example}

We will see an example when the g-bundle subdivision is not complete in Section~\ref{sec:example}.

\begin{remk}\label{remk:bundle-subd-reduced}
	Recall that the map $\bK\mapsto\bK_{\text{red}}$ removing any straight routes from a bundle gives a bijection from maximal bundles to maximal reduced bundles, and that $\g_1(\bK)=\g_1(\bK_{\text{red}})$ and $\g_{\geq0}(\bK)=\g_{\geq0}(\bK_{\text{red}})$. Hence, the g-clique triangulations and g-bundle subdivisions may be thought of as given by bundle spaces of maximal \emph{reduced} cliques and maximal \emph{reduced} bundles, rather than of maximal cliques and maximal bundles as in Definition~\ref{defn:ct}.
\end{remk}

\subsection{Faces of $\F_1(\tL)$ and $\g_{1}(\L)$}

We now wish to describe the faces of $\g(\tL)$, especially the facets and vertices. In order to do so, we first must obtain a better understanding of the proper faces of $\F_1(\tL)$.

\begin{defn}
	Let $W$ be a set of arrows of $\tL$. Define the face $Q_W$ of $\F_1(\tL)$ as
	\[Q_W:=\{F\in\F_1(\tL)\ :\ F(\alpha)=0\text{ for all }\alpha\in W.\}\]
\end{defn}

By Lemma~\ref{lem:facet}, all faces of $\F_1(\tL)$ are of the form $Q_W$ for some $W\subseteq E$ (recall that $E$ is the set of arrows of $\tL$).
In fact, we can do a bit better. We need more terminology to describe sets of arrows.

\begin{defn}
	A set $W\subseteq E$ is
		\emph{closed} if there exists a unit flow $F_W$ on $\tL$ such that $F_W(\alpha)=0$ if and only if $\alpha$ is an arrow of $W$.
\end{defn}

Let $W\subseteq E$. If $W_1$ and $W_2$ are two closed arrow sets containing $W$, then $F_{W_1\cap W_2}:=\frac{F_{W_1}+F_{W_2}}{2}$ is a unit flow such that $F_{W_1\cap W_2}(\alpha)=0$ if and only if $\alpha$ is an arrow of $W_1\cap W_2$, hence $W_1\cap W_2$ is closed. We may then define the \emph{closure} ${\overline{W}}$ of $W$ to be the smallest closed set of arrows containing $W$.

\begin{lemma}\label{lem:closure-respects-face}
	If $W\subseteq E$, then $Q_W=Q_{{\overline{W}}}$.
\end{lemma}
\begin{proof}
	Since ${\overline{W}}\supseteq W$, it is immediate any $F\in Q_{{\overline{W}}}$ is also in $Q_W$, hence $Q_{{\overline{W}}}\subseteq Q_W$.
	We wish to show that any $F\in Q_W$ is in $Q_{{\overline{W}}}$. So, take $F\in Q_W$.
	Let $W_F$ be the set of arrows $\alpha\in E$ such that $F(\alpha)=0$. Since $F\in Q_W$, we have $W_F\supseteq W$. Moreover, since $F$ is a flow, $W_F$ is closed, hence $W_F\supseteq{\overline{W}}$. It then follows that $F(\alpha)=0$ for all $\alpha\in{\overline{W}}$, hence $F\in Q_{{\overline{W}}}$. This completes the proof.
\end{proof}

\begin{prop}\label{prop:closed-arrow-sets-face}
	The map $W\mapsto Q_W$ bijects between closed arrow sets and faces of $\F_1(\tL)$.
\end{prop}
\begin{proof}
	We first show surjectivity. By Lemma~\ref{lem:facet}, any facet of $\F_1(\tL)$ is obtained by necessitating zero flow through some arrow of $\tL$, hence any face $Q$ of $\F_1(\tL)$ arises as $Q_W$ for some set $W$ of arrows of $\tL$. It follows from Lemma~\ref{lem:closure-respects-face} that $Q_W=Q_{{\overline{W}}}$, so $Q$ is in the image of this map.

	We now show injectivity. Let $W_1$ and $W_2$ be different closed arrow sets. Then, without loss of generality, there must be some arrow $\alpha$ which is in $W_1$ but not $W_2$. By the definition of closedness, there exists a flow $F_{W_2}$ which zeros out on precisely the arrows of $W_2$ -- in particular, $F_{W_2}$ is in the face $Q_{W_2}$, but is not in the face $Q_{W_1}$, as $F_{W_2}(\alpha)>0$. This shows that $Q_{W_2}\neq Q_{W_1}$ and that the map is injective.
\end{proof}

We now obtain two alternative characterizations of closed arrow sets. Given $W\subseteq E$, we say that a route or band $p$ of $\tL$ is \emph{$W$-avoiding} if $p$ uses no arrows of $W$.
\begin{lemma}\label{lem:closed-then-trail}
	The following are equivalent for $W\subsetneq E$:
	\begin{enumerate}
		\item\label{th1} $W$ is closed.
		\item\label{th2} There exists a $W$-avoiding route of $\tL$ and for every $\alpha\not\in W$ there is a $W$-avoiding trail containing $\alpha$.
		\item\label{th3} There is a bundle $\bK$ which contains at least one route and which avoids precisely the arrows of $W$.
	\end{enumerate}
\end{lemma}
In particular, condition~\eqref{th2} is easy to verify in practice for a set of arrows $W\subseteq E$.
\begin{proof}
	We first show that~\eqref{th1}$\implies$\eqref{th3}.
	Suppose $W\subsetneq E$ is closed. We first find a rational flow $F'_W$ which gives zero flow to precisely the arrows of $W$. Since $W$ is closed, there exists a flow $F_W$ of $\tL$ such that $F_W(\alpha)=0$ if and only if $\alpha\in W$. Let $0<\epsilon<\min_{\alpha\not\in W}F_W(\alpha)$ and choose a rational flow $F'_W$ in the face $Q_W$ of $\F_1(\tL)$ such that $|F'_W-F_W|<\epsilon$. Since $F'_W$ is rational, it may be realized as a positive bundle combination
	$F'_W=\sum_{p\in\bK}a_p\I(p)$.
	Because $|F'_W-F_W|<\epsilon$, we know that $F'_W(\alpha)>0$ for all $\alpha\not\in W$ and hence each $\alpha\not\in W$ is in some trail of $\bK$. Since $F'_W\in Q_W$, we know that $F'_W(\alpha)=0$ for all $\alpha\in W$, hence no $\alpha\in W$ appears in any trail of $\bK$. Since $F'_W$ is a unit flow, it gives positive flow to some fringe arrow, hence $\bK$ must contain at least one route. This completes the proof that~\eqref{th1}$\implies$\eqref{th3}.

	It is immediate that~\eqref{th3}$\implies$\eqref{th2}.
	We finish the proof by showing that~\eqref{th2}$\implies$\eqref{th1}.
	Assuming~\eqref{th2}, we must have a collection $A\cup B$, where $A$ is a set of $W$-avoiding routes, $B$ is a set of $W$-avoiding bands, and every arrow not in $W$ is in some route of $A$ or band of $B$. Choose positive constants $\{a_p\ :\ p\in A\}$ and $\{b_p\ :\ p\in B\}$ such that $\sum_{p\in A}a_p=1$. Then $\sum_{p\in A}a_p\I(p)+\sum_{p\in B}b_p\I(B)$ is a flow zeroing out on precisely the arrows of $W$, hence $W$ is closed.
\end{proof}

\begin{defn}
	Given an arrow $\alpha$ of $\tL$, let $s_\alpha$ denote the straight route of $\alpha$. Let $s_{\alpha}^+$ denote the subpath of $s_\alpha$ beginning with $\alpha$ and ending with a fringe arrow, and let $s_\alpha^-$ denote the subpath of $s_\alpha$ beginning with a fringe arrow and ending with $\alpha$. Given any trail $p$ of $\tL$ and $W\subseteq E$, we use $W\cap p$ to refer to the set of arrows of $W$ which appear in $p$.
\end{defn}

\begin{prop}\label{prop:QW_hyperplane}
	If $W$ is a closed arrow set, then the half-space of $\mathbb R^E$ defined by the inequality 
	\[\sum_{\alpha\in W}\frac{1}{|W\cap s_\alpha|}F(\alpha)\geq0\]
	is a defining half-space of $Q_W$.
\end{prop}
Note that for any $\alpha\in W$, the set $W\cap s_{\alpha}$ contains at least the arrow $\alpha$, so we are not dividing by zero in the coefficients of Proposition~\ref{prop:QW_hyperplane}.
\begin{proof}
	Since flows of $\F_1(\tL)$ are nonnegative, it is immediate that this inequality is satisfied everywhere in $\F_1(\tL)$. Moreover, it is an equality precisely when all arrows of $W$ are given zero flow, which is the condition for a flow to be in $Q_W$.
\end{proof}

In fact, Proposition~\ref{prop:QW_hyperplane} is true with the constants $\frac{1}{|W\cap s_\alpha|}$ replaced by 1. However, we need this version of the defining hyperplanes for the proof of Theorem~\ref{thm:halfspaces}.
To this end, we now focus on faces of $\g_{1}(\L)$. We will now show that, similar to the face(t)s of $\F_1(\tL)$, face(t)s of $\g_{1}(\L)$ are indexed by certain closed arrow sets.

\begin{defn}
	Let $W\subseteq E$.
	\begin{enumerate}
		\item $W$ is \emph{crooked} if $W$ is closed and contains at least one arrow from every straight route of $\tL$.
		\item $W$ is \emph{barely crooked} if $W$ is closed and contains exactly one arrow from every straight route of $\tL$.
	\end{enumerate}
\end{defn}

\begin{example}
	Consider the fringed quiver on the left of Figure~\ref{GEMINTRO}. The set $\{f_2,e_1\}$ is closed by Lemma~\ref{lem:closed-then-trail} (3), since the highlighted bundle avoids precisely the arrows $e_1$ and $f_2$. Since it contains exactly one arrow from both straight routes of $\tL$, it is barely crooked.
	The set $\{f_2,e_1,e_3\}$ is closed by Lemma~\ref{lem:closed-then-trail} (3) because the bundle $\{f_1e_2^{-1}e_4\}$ avoids precisely the arrows $f_2$, $e_1$, and $e_3$. Because this set contains one arrow from the straight route $f_2f_2$ and two from the straight route $e_1e_2e_3e_4$, this set is crooked but not barely crooked.
	Finally, the set $\{f_1,f_2\}$ is closed because the bundle $\{e_1e_2e_3e_4\}$ avoids precisely $f_1$ and $f_2$, but the set $\{f_1,f_2\}$ is not crooked because it does not contain any arrow from the straight route $e_1e_2e_3e_4$.
\end{example}

\begin{lemma}\label{lem:barelycrooked}
	Let $\K$ be a maximal reduced clique of $\tL$. Then the set of arrows avoided by $\K$ is barely crooked.
\end{lemma}
\begin{proof}
	Let $\tilde\K$ be the maximal clique obtained by adding the straight routes to $\K$. Lemma~\ref{lem:dist} states that if $p$ is a straight route of $\tilde\K$, then $p$ has at least one distinguished arrow $\alpha_p$. Since $p$ is straight, this means that $\alpha_p$ is the an arrow of $p$ which is not used by any other route of $\tilde\K$. The result follows since this holds for every straight route.
\end{proof}

\begin{defn}\label{defn:fgfgfg}
	Let $v\in V_{\text{int}}$ be an internal vertex of $\tL$. Say $\alpha_v^-\beta_v^+$ and $\beta_v^-\alpha_v^+$ are the relations of $v$, so that $\alpha_v^-\alpha_v^+$ and $\beta_v^-\beta_v^+$ are the length-two oriented strings through $v$.
		Let
			$A_v:=\frac{|W\cap s_{\alpha_v^+}^+|}{|W\cap s_{\alpha}|}-\frac{1}{2}$
			and
			$B_v:=\frac{|W\cap s_{\beta_v^+}^+|}{|W\cap s_{\beta}|}-\frac{1}{2}$.
			Let $S_v=A_v+B_v$.
			We make two remarks about these definitions:
			\begin{enumerate}
				\item In other words, $A_v+\frac{1}{2}$ is the fraction of arrows of $W$ on the straight route containing $\alpha_v^-\alpha_v^+$ which appear after the vertex $v$.
				\item The values $A_v$ and $B_v$ depend on which incoming arrow to $v$ we denote $\alpha_v^-$ and which we denote $\beta^-$, which is a choice we fix arbitrarily for each $v\in V_{\text{int}}$ through the proof of the following theorem. Switching this choice switches the values $A_v$ and $B_v$. On the other hand, the sum $S_v=A_v+B_v$ is a well-defined value depending only on $v$ with no extra choice.
			\end{enumerate}
\end{defn}
\begin{thm}\label{thm:halfspaces}
	Let $W$ be a crooked arrow set. The preimage of the half-space of $\mathbb R^{V_{\text{int}}}$ defined by $\sum_{v\in V_{\text{int}}}S_v\xx(v)\geq-1$ through $\phi$ is precisely the half-space $\sum_{\alpha\in W}\frac{1}{|W\cap s_\alpha|}F(\alpha)\geq0$ of $\mathbb R^E$.
\end{thm}
Recall that the latter half-space is a defining half-space of $Q_W$ by Proposition~\ref{prop:QW_hyperplane}.
\begin{proof}
	Recall from Definition~\ref{defn:phi} that if $F$ is a flow and $v\in V_{\text{int}}$, then $\phi(F)(v)=\frac{1}{2}\big(F(\alpha_v^+)+F(\beta_v^+)-F(\alpha_v^-)-F(\beta_v^-)\big)$. Then pulling back the inequality $\sum_{v\in V_{\text{int}}}(A_v+B_v)\xx(v)\geq-1$ through $\phi$ gives the inequality 
	\begin{equation}\label{eqeq}\sum_{v\in V_{\text{int}}}\frac{A_v+B_v}{2}\big(F(\alpha_v^+)+F(\beta_v^+)-F(\alpha_v^-)-F(\beta_v^-)\big)\geq-1\end{equation}
of $\mathbb R^{E}$.
	Now, for a fixed $v$, we rewrite $\frac{A_v+B_v}{2}\big(F(\alpha_v^+)+F(\beta_v^+)-F(\alpha_v^-)-F(\beta_v^-)\big)$:
	\begin{align*}
		\frac{A_v+B_v}{2}\big(&F(\alpha_v^+)+F(\beta_v^+)-F(\alpha_v^-)-F(\beta_v^-)\big)\\
		=&\frac{A_v}{2}\big(F(\alpha_v^+)+F(\beta_v^+)-F(\alpha_v^-)-F(\beta_v^-)\big) \\
		&\ \ \ \ \ +\frac{B_v}{2}\big(F(\alpha_v^+)+F(\beta_v^+)-F(\alpha_v^-)-F(\beta_v^-)\big) & (\text{Distribute})\\
		=&\frac{A_v}{2}\big(2F(\alpha_v^+)-2F(\alpha_v^-)\big)
		+\frac{B_v}{2}\big(2F(\beta_v^+)-2F(\beta_v^-)\big) & (\text{Conservation of flow at $v$})\\
		=&A_vF(\alpha_v^+)+B_vF(\beta_v^+)-A_vF(\alpha_v^-)-B_vF(\beta_v^-)
		&(\text{Combine like terms}).
	\end{align*}
	This shows that inequality~\eqref{eqeq} is equivalent to 
	\begin{equation}\label{eqeqeq}\sum_{v\in V_{\text{int}}}\big(A_vF(\alpha_v^+)+B_vF(\beta_v^+)-A_vF(\alpha_v^-)-B_vF(\beta_v^-)\big)\geq-1.\end{equation}
	 We now wish to write the sum of inequality~\eqref{eqeqeq} as a sum indexed by arrows of $\tL$ rather than internal vertices of $\tL$. We split into four classes of arrows.
	\begin{itemize}
		\item Suppose $\alpha$ is an internal arrow and $\alpha\in W$. Then $F(\alpha)$ appears exactly twice in~\eqref{eqeqeq}. It appears once with the coefficient $B_{h(\alpha)}$ or $A_{h(\alpha)}$ -- in either case, the coefficient is $\frac{|W\cap s_\alpha^+|}{|W\cap s_\alpha|}-\frac{1}{2}$. The other appearance has the coefficient $-B_{t(\alpha)}$ or $-A_{t(\alpha)}$, which in either case amounts to $\frac{|W\cap s_\alpha^-|}{|W\cap s_\alpha|}-\frac{1}{2}$. So, the total coefficient of $F(\alpha)$ in the sum of~\eqref{eqeqeq} is $\frac{|W\cap s_\alpha^+|+|W\cap s_\alpha^-|}{|W\cap s_\alpha|}-1$. The expression $|W\cap s_\alpha^+|+|W\cap s_\alpha^-|$ counts $\alpha\in W$ twice and every other arrow of $W\cap s_\alpha$ once, so the coefficient of $F(\alpha)$ is equal to $\frac{|W\cap s_\alpha|+1}{|W\cap s_\alpha|}-1=\frac{1}{|W\cap s_\alpha|}$.
		\item Suppose $\alpha$ is an internal arrow and $\alpha\not\in W$. We argue as in the previous case that the total coefficient of $F(\alpha)$ in the sum of~\eqref{eqeqeq} is $\frac{|W\cap s_\alpha^+|+|W\cap s_\alpha^-|}{|W\cap s_\alpha|}-1$. Since $\alpha\not\in W$, the expression $|W\cap s_\alpha^+|+|W\cap s_\alpha^-|$ counts every arrow of $W\cap s_\alpha$ once, so the coefficient of $F(\alpha)$ is equal to $\frac{|W\cap s_\alpha|}{|W\cap s_\alpha|}-1=0$.
		\item Suppose $\alpha$ is a fringe arrow and $\alpha\in W$. Then $F(\alpha)$ appears exactly once in~\eqref{eqeqeq} with a coefficient of $\frac{1}{|W\cap s_\alpha|}-\frac{1}{2}$.
		\item If $\alpha$ is a fringe arrow and $\alpha\not\in W$, then $F(\alpha)$ appears exactly once in~\eqref{eqeqeq} with a coefficient of $-\frac{1}{2}$.
	\end{itemize}
	So, we may rearrange the sum on the left of inequality~\eqref{eqeqeq} to get 
	$\sum_{\alpha\in W}\frac{1}{|W\cap s_\alpha|}F(\alpha)-\sum_{\beta\text{ fringe}}\frac{1}{2}F(\beta)\geq-1$. 
	Since $\sum_{\beta\text{ fringe}}\frac{1}{2}F(\beta)=1$ by unit flow, this is equivalent to \[\sum_{\alpha\in W}\frac{1}{|W\cap s_\alpha|}F(\alpha)\geq0.\]
\end{proof}

\begin{thm}\label{thm:faces-crooked}
	The map $W\mapsto \phi(Q_W)=\{\xx\ :\ \sum_{v\in V_{\text{int}}}S_v\xx(v)=-1\}$ is a bijection from crooked arrow sets to proper faces of $\g_{1}(\L)$ which restricts to a bijection from barely crooked arrow sets to facets of $\g_{1}(\L)$.
\end{thm}
\begin{proof}
	Recall that the face $\phi(Q_W)$ is defined by $\sum_{v\in V_{\text{int}}}S_v\xx(v)=-1$ by Theorem~\ref{thm:halfspaces}.
	Let $W$ be a crooked arrow set.
	Theorem~\ref{thm:halfspaces} shows that the half-space $H_V$ of $\mathbb R^{V_{\text{int}}}$ defined by $\sum_{v\in V_{\text{int}}}S_v\xx(v)\geq-1$ pulls back to the half-space $H_E$ of $\mathbb R^E$ defined by $\sum_{\alpha\in W}\frac{1}{|W\cap s_\alpha|}F(\alpha)\geq0$.
	Proposition~\ref{prop:QW_hyperplane} shows that $H_E$ is the defining hyperplane of the face $Q_W$ of $\F_1(\tL)$. It follows that the image of $Q_W$ under $\phi$ is the face of $\g_{1}(\L)$ obtained by intersecting $\g_{1}(\L)$ with the hyperplane $\sum_{v\in V_{\text{int}}}S_v\xx(v)=-1$.
	This shows that $W\mapsto\phi(Q_W)$ is a well-defined map from crooked arrow sets to proper faces of $\g_{1}(\L)$.

	We now show injectivity. Say $W_1$ and $W_2$ are different crooked arrow sets.
	Then Proposition~\ref{prop:closed-arrow-sets-face} shows that, without loss of generality, $Q_{W_1}$ contains some flow $F$ which is not in $Q_{W_2}$.
	By Lemma~\ref{lem:phi-preimage}, any flow $F'\in\F_1(\tL)$ with $\phi(F)=\phi(F')$ must be of the form $F+\sum_{p\text{ straight route}}a_p\I(p)$, where the coefficients $a_p$ sum to zero. Since $F\in Q_{W_1}$, we have $F(\alpha)=0$ for all $\alpha\in W_1$. Since $W_1$ is crooked, every straight route $s$ contains an arrow $\beta_s$ with $F(\beta_s)=0$.
	This means that if any coefficient $a_p$ is negative, then $F'(\beta_{p})<0$, contradicting that $F'\in\F_1(\tL)$, hence no coefficient $a_p$ is negative and $F'=F$.
	We have shown that there is no flow $F'\in\F_1(\tL)$ distinct from $F$ with $\phi(F')=\phi(F)$. In particular, no flow in $Q_{W_2}$ maps to $\phi(F)$ through $\phi$, hence the face $\phi(Q_{W_2})$ does not contain $\phi(F)\in\phi(Q_{W_1})$. This shows that the faces $\phi(Q_{W_1})$ and $\phi(Q_{W_2})$ are distinct.
	
	We now show surjectivity. Any proper face $Q$ of $\g_{1}(\L)$ pulls back through $\phi$ to a face of $\F_1(\tL)$; say, to the face $Q_W$ for some closed arrow set $W$. Since $Q$ is proper, it does not contain the origin, hence $Q_W$ does not contain the indicator vector of any straight route. This means that $W$ must be crooked. Then $Q=\phi(Q_W)$ and the proof of the first statement is complete.

	It remains to show that $\phi(Q_W)$ is a facet of $\g_{1}(\L)$ if and only if $W$ is barely crooked. First, if $W$ is barely crooked, then there is no crooked arrow set properly contained in $W$, hence $Q_W$ is maximal among faces of $\F_1(\tL)$ defined by crooked arrow sets, hence $\phi(Q_W)$ is a facet. 
	We now show that if $W$ is crooked but not barely crooked, then there is some barely crooked $W'$ contained in $W$. {Let $W$ be crooked but not barely crooked. Since $W$ is closed, we may choose a flow $F$ such that $F(\alpha)=0\iff\alpha\in W$. Define $M:=\min\{F(\alpha)\ :\ \alpha\in\tL\text{ and }F(\alpha)>0\}$. By Proposition~\ref{prop:density}, we may choose a flow $F'$ in a unit clique simplex $\Delta\subset\F_1(\tL)$ such that $|F'-F|<\epsilon$. Since $|F'-F|<\epsilon$, if $\alpha$ is an arrow of $\tL$ such that $F(\alpha)>0$, then $F'(\alpha)=0$. This means that if $F'(\alpha)=0$ then $\alpha\in W$. Then the clique $\K_{F'}^+$ avoids only arrows of $W$. Complete $\K_{F'}^+$ to a maximal reduced clique $\K$ -- this maximal clique $\K$, then, also avoids only arrows of $W$. Set $W'$ to be the set of arrows avoided by $\K_{F'}^+$. By Lemma~\ref{lem:barelycrooked}, $W'$ is barely crooked. Since $W'\subsetneq W$, we have $Q_{W'}\supsetneq Q_W$ and the proof is complete.}
\end{proof}

Before giving examples, we give a more explicit description of the facet-defining hyperplanes of $\g(\tL)$.
\begin{thm}\label{thm:facets-crooked}
	Let $W$ be a barely crooked arrow set. For any vertex $v\in V_{\text{int}}$, let $\alpha$ and $\beta$ be the arrows of $\tL$ beginning at $v$. Define 
	\[T_v=
	\begin{cases}
		1&s_\alpha^+\text{ and }s_\beta^+\text{ both contain arrows of }W \\
		-1&\text{neither }s_\alpha^+\text{ nor }s_\beta^+\text{ contain an arrow of }W\\
		0&\text{else}.
	\end{cases}\]
	Then $\sum_{v\in V_{\text{int}}}T_v\xx(v)\leq1$ is a facet-defining half-space of $\g_{1}(\L)$, and all facet-defining half-spaces are uniquely obtained from a barely crooked arrow set in this way.
\end{thm}
\begin{proof}
	It is not hard to check that $T_v=-S_v$ in the notation of Definition~\ref{defn:fgfgfg}. The result then follows from Theorem~\ref{thm:faces-crooked}.
\end{proof}

\begin{example}
	Figure~\ref{GEMINTRO} shows a fringed quiver $\tL$ and its g-polyhedron, where the left internal vertex $v_1$ is the x-axis and the right $v_2$ is the y-axis. First, let $W$ be the barely crooked arrow set $\{f_2,e_1\}$. We may calculate $A_{v_1}=-\frac{1}{2}$ and $B_{v_1}=\frac{1}{2}$, hence $S_{v_1}=0$. Similarly, $A_{v_2}=-\frac{1}{2}$, $B_{v_2}=-\frac{1}{2}$, and $S_{v_2}=-1$. By Theorem~\ref{thm:faces-crooked}, the facet-defining half-space for $\phi(Q_W)=\text{conv}\{(-1,1),(0,1)\}$ is $-\xx(v_2)\geq-1$. Equivalently, $\xx(v_2)\leq1$, as can be retrieved using Theorem~\ref{thm:facets-crooked}.
	This is the far edge of the shaded clique simplex.

	Now, let $W$ be the crooked arrow set $\{f_2,e_1,e_3\}$. To find $S_{v_1}$, we choose $\alpha_{v_1}^+=e_3$ and $\beta_{v_1}^+=f_2$. Then $A_{v_1}=0$, $B_{v_1}=\frac{1}{2}$, and $S_{v_1}=\frac{1}{2}$. To find $S_{v_2}$, we choose $\alpha_{v_2}^+:=e_2$ and $\beta_{v_2}^+:=e_4$. Then $A_{v_2}=0$, $B_{v_2}=-\frac{1}{2}$, and $S_{v_2}=-\frac{1}{2}$. Then the half-space giving $\phi(Q_W)=\{(-1,1)\}$ is given by the equation $-\frac{1}{2}\xx(v_1)+\frac{1}{2}\xx(v_2)\leq1$ as desired.
\end{example}

\subsection{Vertices and unbounded directions of the g-polyhedron}

We use the map $\phi$ and our new understanding of faces of $\g_{1}(\L)$ to port our characterization of vertices and unbounded directions of $\F_1(\tL)$ from Theorem~\ref{thm:turb-vert-dir} to $\g_{1}(\L)$.

\begin{cor}\label{cor:g-vert-dir}
	The vertices of the g-polyhedron $\g_1(\L)$ are precisely the g-vectors of elementary bending routes, and the unbounded directions of the g-polyhedron $\g_1(\L)$ are minimally generated by the g-vectors of elementary bands.
\end{cor}
As a result, the g-polyhedron $\g_1(\L)$ is equal to the Minkowski sum of (the convex hull of) all g-vectors of elementary bending routes with the unbounded cone generated by g-vectors of elementary bands.
\begin{proof}
	Any vertex of $\g_1(\L)=\phi(\F_1(\tL))$ must have a vertex $v$ of $\F_1(\tL)$ in its preimage through $\phi$. Vertices of $\F_1(\tL)$ are indicators of elementary routes, and $\phi$ sends the indicator vector of a route to its g-vector, so all vertices of $\g_1(\L)$ are g-vectors of elementary routes. Since the g-vector of any straight route is the origin, an interior point of $\g_1(\L)$, all vertices of $\g_1(\L)$ are g-vectors of elementary bending routes. It remains to show that the g-vector of any elementary bending route is a vertex of $\g_1(\L)$. 
	Let $p$ be an elementary bending route and let $W$ be the arrows of $\tL$ not appearing in $p$. It follows from the definition of elementary routes that $p$ cannot contain every arrow from any straight route, so $W$ is crooked.
	The vertex $\I(p)$ of $\F_1(\tL)$ is equal to the face $Q_W$ of $\F_1(\tL)$. Theorem~\ref{thm:faces-crooked} gives that $\phi(\I(p))=\g(p)$ is a face of $\g_1(\L)$, finishing the proof.

	We now consider the unbounded directions of $\g_1(\L)$.
	Since the indicator vectors \[\{\I(p)\ :\ p\text{ is an elementary band of }\tL\}\] generate the unbounded directions of $\F_1(\tL)$, it follows that 
	\[\{\g(p)=\phi(\I(p))\ :\ p\text{ is an elementary band of }\tL\}\] generates the unbounded directions of $\phi(\F_1(\tL))=\g_1(\L)$.
	Suppose now that this generating set is not minimal. Then, for some elementary bands $p_1,\dots,p_m$ and coefficients $a_j>0$, we have $\g(p_m)=\sum_{j\in[m-1]}a_j\g(p_j)$. Pulling this back through $\phi$ and applying Lemma~\ref{lem:phi-preimage}, we see that \[\I(p_m)=\sum_{j\in[m-1]}a_j\I(p_j)+\sum_{r\text{ straight route}}b_r\I(r),\] where the coefficients $b_r$ sum to zero. If some $b_r$ is nonzero, then $b_{r'}$ is negative for a straight route $r'$.
	If $\alpha$ is a fringe arrow of $r'$, then no band $p_j$ and no other straight route $b_r$ (for $r\neq r'$) passes through $\alpha$, hence the flow
	$\I(p_m)=\sum_{j\in[m-1]}a_j\I(p_j)+\sum_{r\text{ straight route}}b_r\I(r)$ gives (negative) flow $b_{r'}$ to $\alpha$, a contradiction. Then all coefficients $b_r$ are equal to zero, hence $\I(p_m)=\sum_{j\in[m-1]}a_j\I(p_j)$, contradicting that the indicators vectors
	\[\{\I(p)\ :\ p\text{ is an elementary band of }\tL\}\] are a minimal generating set for the unbounded directions of $\F_1(\tL)$.
\end{proof}

\begin{remk}
	The author finds doing polyhedral geometry with indicator vectors of routes to be in many ways more intuitive than doing polyhedral geometry with their g-vectors.
	This speaks to the methods used in proving Theorem~\ref{thm:g-decomp}, Theorem~\ref{thm:faces-crooked}, and Corollary~\ref{cor:g-vert-dir}, where we first obtain the analogous results on the turbulence polyhedron $\F_1(\tL)$ and use the map $\phi$ to translate these to the g-vector picture.
	This strategy will continue to be used in the following sections, even in the practical computations of the example studied in Section~\ref{sec:example}.
	We hope this strategy may continue to be useful in creating and proving statements about g-polyhedra of gentle algebras.
\end{remk}

\begin{example}
	Let $\tL$ be the Kronecker fringed quiver. Figure~\ref{KRONINTRO} shows its g-polyhedron and Example~\ref{kron-el-routes-bands} gives the vertices and unbounded directions of its turbulence polyhedron.
	In accordance with Corollary~\ref{cor:g-vert-dir}, the vertices of $\g_1(\L)$ are precisely the g-vectors $\g(e_1f_1^{-1})=(-1,0)$ and $\g(e_3^{-1}f_3)=(0,1)$ of the elementary bending routes. The recession cone of $\g_1(\L)$ is the ray generated by the g-vector $\g(e_2f_2^{-1})=(1,-1)$ of the only (elementary) band.
\end{example}

\begin{example}\label{ex:doub-kron-routes-bands}
	Let $\tL$ be the ``double Kronecker fringed quiver'' of Figure~\ref{fig:irrational}. Its g-polyhedron appears in Figure~\ref{fig:doubkron-g}.
		There are four elementary routes: $e_1e_2e_3e_4$, $f_1f_2f_3f_4$, $e_1f_1^{-1}$, and $e_4^{-1}f_4$.
		There are two elementary bands: $e_2f_2^{-1}$ and $e_3f_3^{-1}$.
	We see in Figure~\ref{fig:doubkron-g} that the vertices of $\g_1(\L)$ are the g-vectors of elementary routes, and the recession cone of $\g_1(\L)$ is minimally generated by the g-vectors of the two elementary bands.  
\end{example}

\section{Vortex Dissections}
\label{sec:vortex}

We have developed the theory of bundle subdivisions of turbulence polyhedra.
When $\tL$ is representation-infinite, there may be flows which cannot be expressed as a bundle combination (Example~\ref{ex:irrational}), and hence the bundle subdivision may fail to be complete. In this section, we define the \emph{vortex dissection} of a turbulence polyhedron or g-polyhedron. This dissection is indexed by \emph{band-stable cliques} rather than maximal bundles and has the advantage of being complete.

\subsection{Technical results}

We obtain a few technical results which will be necessary to prove results about vortex dissections.

Recall Example~\ref{ex:irrational}, which showed that a flow $F$ may not be equal to $F':=\sum_{p\in\bK_F^+}a_p^F\I(p)$. We now show that $F$ and $F'$ must agree on all fringe arrows (Corollary~\ref{cor:blos-fl}). First, we need a lemma.

\begin{lemma}\label{lem:blos-fl}
	Let $\alpha^\e$ be a signed fringe arrow of $\tL$. Let $F$ be a flow of $\tL$. For any interval $I\subseteq[0,F(\alpha)]$, there is some value $C\in I$ such that $p^F_{(\alpha^\e,C)}$ is defined.
\end{lemma}
\begin{proof}
	Let $\alpha^\e$ be a signed fringe arrow of $\tL$. It suffices to consider the case when $h(\alpha^\e)$ is internal. Indeed, if $h(\alpha^\e)$ is not internal, then $h(\alpha^{-\e})$ is internal so finding $C\in\{F(\alpha)-x\ :\ x\in I\}$ such that $p^F_{(\alpha^{-\e},C)}$ is well-defined guarantees that the equivalent $p^F_{(\alpha^\e,{F(\alpha)-C})}$ is well-defined, with $F(\alpha)-C\in I$. So, we now suppose that $h(\alpha^\e)$ is internal.

	Define $m\in\mathbb N$ such that $m>\frac{1}{|I|}\sum_{\beta\in E}F(\beta)$.
	We claim that there is some $C\in I$ such that $p^F_{(\alpha^\e,C)}$ is a route of some length $k<m$. In other words, such that $\for^{k-1}(\alpha^\e,C)=(\beta^\f,D)$, where $h(\beta^\f)$ is a fringe vertex.

	Suppose to the contrary. Then to any arrow-flow $(\alpha^\e,C)$ for $C\in I$ we may associate the string  $p_{(\alpha^\e,C)}^{\leq m}:=\alpha_0^{\e_0}\alpha_1^{\e_1}\dots\alpha_m^{\e_m}$, where the signed arrows $\alpha_j^{\e_j}$ are defined such that $\for^j(\alpha^\e,C)=(\alpha_j^{\e_j},C_j)$ for some $C_j$. 

	Let $S$ be the (finite) set of length-$m$ strings appearing as $p_{(\alpha^\e,C)}^{\leq m}$ for some $C\in I$. For any $p\in S$, let $I_p:=\{C\in I\ :\ p_{(\alpha^\e,C)}^{\leq m}=p\}$.
	We claim that, fixing $p\in S$, the set $I_p$ is an interval. Indeed, take $C_1,C_2\in I_p$ and $D\in I$ with $C_1<D<C_2$; we will now show that $D\in I_p$.
	By Proposition~\ref{prop:arrow-order}, we have $p_{(\alpha^\e,C_1)}\preceq_{\alpha^\e} p_{(\alpha^\e,{D})}\preceq_{\alpha^\e} p_{(\alpha^\e,{C_2})}$. The former and latter paths begin with $p_{(\alpha^\e,{C_1})}^{\leq m}=p_{(\alpha^\e,C_2)}^{\leq m}=p$, hence the middle path $p_{(\alpha^\e,D)}$ must also begin with $p$, showing that $D\in I_p$.

	If $p\in S$, then we consider an arrow-flow $(\beta^\f,D)$ to be ``in $p$'' if $(\beta^\f,D)=\for^j(\alpha^\e,C)$ for some $j\leq m$ and $C\in I_p$. It is immediate that an arrow-flow may only be in one path of $S$.
	Fix $p\in S$. For any $j\in[m]$, the flow values of $\{\for^j(\alpha^\e,C)\ :\ C\in I_p\}$ is an interval of length $I_p$ of arrow-flows in $p$. So, in total, the intervals of arrow-flows in $p$ add up to $m|I_p|$ in length.

	The interval $I$ is partitioned by the finite set of intervals $\{I_p\ :\ p\in S\}$. Then $\sum_{p\in S}|I_p|=|I|$. By the previous paragraph, the intervals of arrow-flows in paths of $S$ add up to $\sum_{p\in S}m|I_p|=m\sum_{p\in S}|I_p|=m|I|>\sum_{\beta\in E}F(\beta)$, where the last inequality follows by choice of $m$. This is a contradiction, as the sum of intervals of arrow-flows of paths of $S$ clearly may not exceed $\sum_{\beta\in E}F(\beta)$.
\end{proof}

\begin{cor}\label{cor:blos-fl}
	Let $F$ be an arbitrary (not necessarily rational) flow of $\tL$.
	Let $\alpha$ be a fringe arrow of $\tL$. Then $\sum_{p\in\bK_F^+}a_p^F\I(p)(\alpha)=F(\alpha)$.
\end{cor}
\begin{proof}
	Let $\alpha$ be a fringe arrow. Let $\{p_1,\dots,p_m\}$ be the marked routes of $\bK_F$ at $\alpha$ (this is a finite set because $\bK_F$ is a bundle).
	For any $j\in[m]$, let $I_{p_j}:=\{C\in[0,F(\alpha)]\ :\ p^F_{(\alpha,C)}=p_j\text{ as marked paths}\}.$
	In fact, $I_{p_j}$ is an interval. Given $C_1<C_2$ in $I_{p_j}$ and $C_1<D<C_2$, it follows from Proposition~\ref{prop:arrow-order} that $p^F_{(\alpha,{C_1})}\preceq_{\alpha}p^F_{(\alpha,{D})}\preceq_{\alpha}p^F_{(\alpha,{C_2})}$. Since the former and latter marked paths are equal, we have equality throughout, hence $D\in I_{p_j}$. This shows that $I_{p_j}$ is an interval for all $j\in[m]$.

	By Lemma~\ref{lem:blos-fl}, the complement of the (finite) union of intervals $\cup_{j\in[m]}I_{p_j}$ is a (finite) union of singletons. It follows that $\sum_{j\in[m]}a_{p_j}=\sum_{j\in[m]}|I_{p_j}|=F(\alpha)$.
\end{proof}

Note that Corollary~\ref{cor:blos-fl} does not always work if we take $\alpha$ to be an internal arrow (even after replacing ``routes'' with ``trails'') as evidenced by Example~\ref{ex:irrational}.

\begin{defn}
	Let $F$ be a flow of $\tL$. Define $\K_F$ (respectively $\K_F^+$) to be the restriction of $\bK_F$ (respectively $\bK_F^+$) to the set of routes of $\tL$.
\end{defn}

Corollary~\ref{cor:blos-fl} shows that the sum $\sum_{p\in\K_F^+}a_p\I(p)$ is equal to $F$ on all fringe arrows. In other words, 
\begin{equation}\label{EQX}F=G_F+\sum_{p\in\K_F^+}a_p\I(p)\end{equation}
	where $G_F:=F-\sum_{p\in\K_F^+}a_p^F\I(p)\in\mathbb R_{\geq0}^E$ is vortex (i.e., a flow of strength zero).
We say that $\sum_{p\in\K_F^+}a_p\I(p)$ is the \emph{canonical clique combination} of $F$, and $G_F$ is the \emph{canonical vortex} of $F$.
Our goal in this section  is to use canonical clique combinations and canonical vortices to obtain a complete dissection of the turbulence polyhedron by arguing that~\eqref{EQX} is the unique description of $F$ as a \emph{vortex combination} (Definition~\ref{defn:vortexcomb}).

We now give a technical result which allows us to add and subtract from a flow without changing its canonical clique combination.
We first require a number of technical definitions.

\begin{defn}
	Let $\K$ be a clique of $\tL$. A self-compatible band $B$ is \emph{$\K$-compatible} if $B$ is compatible with every route of $\K$.
\end{defn}

\begin{defn}\label{defn:blank}
	Let $F$ be a flow of $\tL$ and let $\alpha^\e$ be a signed arrow of $\tL$. Let $p_1,\dots,p_m$ be the marked routes of $\K_F^+$ at $\alpha^\e$ in order of $\prec_{\alpha^\e}$.
	For any $j\in[m]$, let $I_j$ be the interval of values $C\in[0,F(\alpha)]$ such that $p_{(\alpha^\e,C)}^F=p_j$. Since each $p_j$ is in $\K_F^+$, these intervals $I_j$ all have positive length.
	Define $I_0:=\{0\}$ and $I_{m+1}:=\{F(\alpha)\}$.
	Pick an integer $0\leq c\leq m$. 
	Define
	\[J_c:=\{x\in\mathbb R\ :\ a<x\text{ for any }a\in I_c\text{ and }x<b\text{ for any }b\in I_{c+1}\}.\]
	We say that $J_c$ is a \emph{blank space} of $F$ at $\alpha$. The interval $J_c$ may be empty or a singleton. In particular, if $0\in I_1$ then the lowest blank space $J_0$ of $F$ at $\alpha^\e$ is empty, and if $F(\alpha)\in I_m$ then the highest blank space $J_{m+1}$ is empty. Any blank space of $F$ at $\alpha^\e$ is an interval, and it is \emph{proper} if it has positive length.
	We then say that the blank space $J_c$ is \emph{bounded below by $I_c$} and \emph{bounded above by $I_{c+1}$}.
	By formally defining $p_0:=0$ and $p_{m+1}:=\infty$, we may also say that $J_c$ is \emph{bounded below by $(p_c,I_c)$} and \emph{bounded above by $(p_{c+1},I_{c+1})$}.
\end{defn}

\begin{remk}\label{remk:proactive}
	By Corollary~\ref{cor:blos-fl}, if $\alpha^\e$ is a fringe arrow then there are no proper blank spaces of $F$ at $\alpha^\e$.
\end{remk}

\begin{remk}\label{remk:blank-space-card}
	If $F$ is a flow, then the number of blank spaces of $F$ is
	\begin{align*}
		\sum_{\alpha\in E}(1+\text{number of marked routes of $\K_F^+$ at $\alpha$})
		=| E|+\sum_{p\in\K_F^+}|p|
	\end{align*}
	where if $p$ is a route, then $|p|$ is the number of arrows in $p$. Note that this depends only on the clique $\K_F^+$.
\end{remk}

In this section, we will be adding and subtracting indicator vectors of bands to and from the flow $F$. We wish to do this in a way which preserves the routes of $\K_F^+$ and their coefficients (and hence the canonical clique combination of $F$), but may alter the blank spaces of $F$ or the bands of $\bK_F^+$ and their coefficients.
The following definition is important in relating a band $B$ to the blank spaces of $F$.

\begin{defn}
	Let $F$ be a flow of $\tL$. Let $B$ be a $\K_F^+$-compatible band. Let $B_{\alpha^\e}$ be a marking of $B^{\pm1}$ at a signed $\alpha^\e$.
	Let $p_1,\dots,p_m$ be the marked routes of $\K_F^+$ at $\alpha^\e$ in order of $\prec_{\alpha^\e}$. Pick the integer $0\leq j\leq m$ such that $p_j\prec_{\alpha^\e}B_{\alpha^\e}$ (if $j\neq0$) and $B_{\alpha^\e}\prec_{\alpha^\e}p_{j+1}$ (if $j\neq m$).
	Then let $J_{F,B_{\alpha^\e}}$ be the blank space of $F$ at $\alpha^\e$ bounded below by $p_j$ and bounded above by $p_{j+1}$.
	We say that $B_{\alpha^\e}$ \emph{splits the blank space $J_{F,B_{\alpha^\e}}$}.
	If $J$ is a blank space, then
	define $N_{B,J}$ to be the number of markings $B'$ of $B^{\pm1}$ which split the blank space $J$.
	Then define $a_{B,J}:=\frac{|J|}{N_{B,J}}$.
	Defining $M_{F,B}:=\min\{a_{B,J}\ :\ J\text{ is a blank space of }F\}$, we say that the (unmarked) band $B$ \emph{splits $F$ with strength $M_{F,B}$}.
\end{defn}

Intuitively, $M_{F,B}$ is the largest scalar multiple of $\I(B)$ which we might reasonably want to subtract from $F$ without changing the routes of $\K_F^+$ and their coefficients. This is made precise by the following proposition.

\begin{prop}\label{prop:nnew}
	Let $F$ be a flow of $\tL$. Let $B$ be a band which is compatible with $\K_F$ and which splits $F$ with strength $M_{F,B}$. Let $b_B\in\mathbb R$ such that $b_B\geq(-M_{F,B})$.
	Let $F':=F+b_B\I(B)$. Then $\K_{F'}=\K_F$ and we have $a_{p}^{F'}=a_{p}^{F}$ for every $p\in\K_{F'}$.
	Moreover, if
	$M_{F,B}>0$ and $b_B=-M_{F,B}$ then $F'$ has strictly fewer proper blank spaces than $F$.
\end{prop}
\begin{proof}
	Define $F$, $B$, $b_B$, and $F'$ as in the theorem statement.
	
	\noindent\textbf{\underline{Claim A:}} Let $q=\alpha_0^{\e_0}\dots\alpha_m^{\e_m}$ be a route of $\K_F^+$. For any integer $0\leq j\leq m$, let $D_j$ be the number of markings of $B^{\pm1}$ at $\alpha_j^{\e_j}$ which are below $q$ marked at $\alpha_j^{\e_j}$ in the order $\prec_{\alpha_j^{\e_j}}$. Let $I_j$ be the interval of $[0,F(\alpha_j)]$ such that $p_{(\alpha_j^{\e_j},C)}^F=q$ (as marked paths) for $C\in I_j$. We claim that for any value $C$ in the interior of $I_j$, we have $p_{(\alpha_j^{\e_j},C+D_jb_B)}^{F'}=p_{(\alpha_j^{\e_j},C)}^F$.
	
	To prove Claim A, pick $C_0$ in the interior of $I_0$ and define $C_j$ for $j\in[m]$ so that $\for_{F}^j(\alpha_0^{\e_0},C_0)=(\alpha_j^{\e_j},C_j)$. Then each $C_j$ is in the interior of $I_j$. We will show for every integer $0\leq j<m$ that $\for_{F'}(\alpha_j^{\e_j},C_j+D_jb_B)=(\alpha_{j+1}^{\e_{j+1}},C_{j+1}+D_{j+1}b_B)$. This shows that $p^{F'}_{(\alpha_0^{\e_0},C_0+D_0b_B)}=p_{(\alpha_0^{\e_0},C_0)}^F$. Since this holds for any $C_0$ in the interior of $I_0$, this will prove Claim A.
	To show this, we will use the following intermediate claim:

	\noindent\underline{\textbf{Claim B:}}
	We claim that, for any integer $0\leq j\leq m$, the tuple $(\alpha_j^{\e_j},C_j+D_jb_B)$ is a valid arrow-flow of $F'$. In other words, we claim that $0\leq C_j+D_jb_B\leq F'(\alpha_j)$. 
	We first show that $0\leq C_j+D_jb_B$. This is immediate if $b_B\geq0$, so suppose $b_B<0$.
	Let $S_j$ be the set of blank spaces $J$ of $F$ at $\alpha_j^{\e_j}$ such that every $x\in J$ satisfies $x<C_j$. For each $J\in S_j$, let $N_J$ be the number of markings of $B^{\pm1}$ at $\alpha_j^{\e_j}$ which split $J$. Then $\sum_{J\in S_j}N_J=D_j$. Since $b_B\geq-M_{F,B}$, for every $J\in S_j$ we have  $N_J(-b_B)\leq|J|$. Then
	\[D_j(-b_B)=\sum_{J\in S_j}N_J(-b_B)\leq\sum_{J\in S_j}|J|\leq C_j,\]
	proving that $C_j+D_jb_B\geq0$.

	To prove Claim B, it remains to show that $C_j+D_jb_B\leq F'(\alpha_j)$.
	Let $D'_j$ be the number of markings of $B^{\pm1}$ at $\alpha_j^{\e_j}$ which are above $q$ marked at $\alpha_j^{\e_j}$ in the order $\prec_{\alpha_j^{\e_j}}$.
	Then $D_j+D'_j$ is the number of times $\alpha_j^{\pm1}$ appears in $B$, so $F'(\alpha_j)=F(\alpha_j)+D_jb_B+D'_jb_B$. If $b_B\geq0$, then the inequality $C_j\leq F(\alpha_j)$ implies the desired $C_j+D_jb_B\leq F(\alpha_j)+D_jb_B\leq F(\alpha_j)+D_jb_B+D'_jb_B=F'(\alpha_j)$. We now suppose $b_B<0$.
	Let $S'_j$ be the set of blank spaces $J$ of $F$ at $\alpha_j^{\e_j}$ such that every $x\in J$ satisfies $x>C_j$. For each $J\in S_j$, let $N_J$ be the number of markings of $B^{\pm1}$ at $\alpha_j^{\e_j}$ which split $J$. Then $\sum_{J\in S'_j}N_J=D'_j$. Since $b_B\geq-M_{F,B}$, for every $J\in S'_j$ we have $N_J(-b_B)\leq|J|$.
	Then
	\[D'_j(-b_B)=\sum_{J\in S'_j}N_J(-b_B)\leq F(\alpha_j)-C_j.\]
	This implies the inequality $F'(\alpha_j)=F(\alpha_j)+D_jb_B+D'_jb_B\geq C_j+D_jb_B$, ending the proof of Claim B.

	To prove Claim A, it now remains only to show for every $0\leq j<m$ that $\for_{F'}(\alpha_j^{\e_j},C_j+D_jb_B)=(\alpha_{j+1}^{\e_{j+1}},C_{j+1}+D_{j+1}b_B)$.
	Letting $\alpha^\e:=\alpha_j^{\e_j}$, define the arrows $\alpha',\beta,\beta'$ as in Definition~\ref{defn:forbac}. We split into four cases matching the four cases of Definition~\ref{defn:forbac}.
	\begin{enumerate}
		\item Suppose $\e_j=\e_{j+1}=1$, so that $\alpha'=\alpha_{j+1}$.
			Since $B$ is compatible with $q$, any marked copy of $B$ at $\alpha_j$ (respectively $\alpha_{j+1}$) which is below $q$ marked at $\alpha_j$ in $\prec_{\alpha_j}$ (respectively $\alpha_{j+1}$ in $\prec_{\alpha_{j+1}}$) must continue on to $\alpha_{j+1}$ (respectively from $\alpha_j$). This shows that $D_{j+1}=D_j$.
			We have already shown by Claim B that $(\alpha_{j+1}^{\e_{j+1}},C_j+D_{j+1}b_B)$ is an arrow-flow of $F'$, hence $C_j+D_jb_B=C_j+D_{j+1}b_B\leq F(\alpha_{j+1}^{\e_{j+1}})=F(\alpha')$. Then the first branch of Definition~\ref{defn:forbac} gives
			\begin{align*}
				\for_{F'}(\alpha_j^{\e_j},C_j+D_jb_B)&=(\alpha',C_j+D_jb_B)\\
				&=(\alpha_{j+1}^{\e_{j+1}},C_j+D_{j+1}b_B)\\
				&=(\alpha_{j+1}^{\e_{j+1}},C_{j+1}+D_{j+1}b_B),
			\end{align*}
			where the last equality follows because $C_{j+1}=C_j$ by Definition~\ref{defn:forbac}.
			This completes the proof in this case.
		\item Suppose $\e_j=1$ and $\e_{j+1}=-1$, so that $\alpha_{j+1}^{\e_{j+1}}=\beta^{-1}$. Let $B'$ be a marking of $B^{\pm1}$ at a copy of $\alpha_j$ which is below $q$ marked at $\alpha_j$. Since $B$ is compatible with $q$, the marked $\alpha_j$ must be followed by either $\alpha'$ or $\alpha_{j+1}^{-1}$; in the latter case $B$ marked at $\alpha_{j+1}^{-1}$ is below $q$ marked at $\alpha_{j+1}^{-1}$. Similarly, if $B''$ is any marking of $B$ at $\alpha_{j+1}^{-1}$ which is below $q$ marked at $\alpha_{j+1}^{-1}$ then the marked arrow $\alpha_{j+1}^{-1}$ of $B''$ must be preceded by $\alpha_j$. We hence see that if $D_{\alpha'}$ is the number of times $B$ passes through $(\alpha')^{\pm1}$, then $D_j=D_{\alpha'}+D_{j+1}$. 
			We calculate
			\begin{align}
				C_j+D_jb_B-F'(\alpha')&=C_j+D_jb_B-(F(\alpha')+D_{\alpha'}b_B)\nonumber\\
				&=(C_j-F(\alpha'))+(D_jb_B-D_{\alpha'}b_B)\nonumber\\
				&=(C_j-F(\alpha'))+D_{j+1}b_B\nonumber\\ 
				&=C_{j+1}+D_{j+1}b_B\label{eqn:me},
			\end{align}
			where the last equality follows because the second branch of Definition~\ref{defn:forbac} gives $C_{j+1}=C_j-F(\alpha')$.
			By Claim B, $C_{j+1}+D_{j+1}b_B\geq0$. Then by~\eqref{eqn:me} we have $C_j+D_jb_B-F'(\alpha')\geq0$ and hence $C_j+D_jb_B\geq F'(\alpha')$. {Since $C_j$ is in the interior of $I$, this inequality is strict.} Then the second branch of Definition~\ref{defn:forbac} gives
				\[\for_{F'}(\alpha_j,C_j+D_jb_B)=(\beta^{-1},C_j+D_jb_B-F'(\alpha'))
				=(\alpha_{j+1}^{\e_{j+1}},C_{j+1}+D_{j+1}b_B),
				\]
			where the last equality follows by $\beta^{-1}=\alpha_{j+1}^{\e_{{j+1}}}$ and~\eqref{eqn:me}. This ends the proof in this case.
		\item Suppose $\e_j=-1$ and $\e_{j+1}=1$, so that $\alpha_{j+1}^{\e_{j+1}}=\alpha'$.
			Similarly to the previous case, we may argue using compatibility of $q$ and $B$ that if $D_{\beta'}$ is the number of times $B$ passes through $(\beta')^{\pm1}$, then $D_{\beta'}+D_j=D_{j+1}$.
			We calculate
			\begin{align}
				C_j+D_jb_B+F'(\beta')&=C_j+D_jb_B+(F(\beta')+D_{\beta'}b_B)\nonumber\\
				&=(C_j+F(\beta'))+(D_jb_B+D_{\beta'}b_B)\nonumber\\
				&=(C_j+F(\beta'))+D_{j+1}b_B\nonumber\\
				&=C_{j+1}+D_{j+1}b_B\label{eqn:me2},
			\end{align}
			where the last equality follows because $C_{j+1}=C_j+F(\beta')$ by the third branch of Definition~\ref{defn:forbac}.
			By Claim B, $C_{j+1}+D_{j+1}b_B\leq F'(\alpha')$. Then by~\eqref{eqn:me2} we have $C_j+D_jb_B+F'(\beta')\leq F'(\alpha')$. {Since $C_j$ is in the interior of $I_j$, this inequality is strict.} Then the third branch of Definition~\ref{defn:forbac} gives
			\[\for_{F'}(\alpha_j^{\e_j},C_j+D_jb_B)=(\alpha',C_j+D_jb_B+F'(\beta'))=(\alpha_{j+1}^{\e_{j+1}},C_{j+1}+D_{j+1}b_B)\]
			and the proof is complete in this case.
		\item Suppose $\e_j=-1$ and $\e_{j+1}=-1$, so that $\alpha_{j+1}^{\e_{j+1}}=\beta^{-1}$.
			Similarly to the previous cases, we may argue using compatibility of $q$ and $B$ that if $D_{\beta'}$ is the number of times $B$ passes through $(\beta')^{\pm1}$ and $D_{\alpha'}$ is the number of times $B$ passes through $(\alpha')^{\pm1}$, then $D_{\beta'}+D_j=D_{\alpha'}+D_{j+1}$.
			We calculate
			\begin{align}
				C_j+D_jb_B+F'(\beta')-F'(\alpha')&=C_j+D_jb_B+(F(\beta)+D_{\beta'}b_B)-(F(\alpha')+D_{\alpha'}b_B))\nonumber\\
				&=(C_j+F(\beta')-F(\alpha'))+D_jb_B+D_{\beta'}b_B-D_{\alpha'}b_B\nonumber\\
				&=(C_j+F(\beta')-F(\alpha'))+D_{j+1}b_B\nonumber\\
				&=C_{j+1}+D_{j+1}b_B\label{eqn:me3},
			\end{align}
			where the last equality follows because $C_{j+1}=C_j+F(\beta')-F(\alpha')$ by the fourth branch of Definition~\ref{defn:forbac}.
			By Claim B, $C_{j+1}+D_{j+1}b_B\geq0$, hence~\eqref{eqn:me3} gives that $C_j+D_jb_B+F'(\beta')\geq F'(\alpha')$.
			Then the fourth branch of Definition~\ref{defn:forbac} gives
			\[\for_{F'}(\alpha_j^{\e_j},C_j+D_jb_B)
			=(\beta^{-1},C_j+D_jb_B+F'(\beta')-F'(\alpha'))
			=(\alpha_{j+1}^{\e_{j+1}},C_{j+1}+D_{j+1}b_B)\]
			and the proof is complete in this case.
	\end{enumerate}
	We have now proven Claim A. It immediately follows that if $q\in\mathcal K_F^+$, then $q\in\mathcal K_{F'}^+$ with $a_q^{F'}\geq a_q^F$. On the other hand, take any fringe arrow $\alpha^\e$ of $\tL$. By Corollary~\ref{cor:blos-fl}, we have $F(\alpha)=\sum_{p\in\K_F^+}a_p^F\I(p)(\alpha)$ and $F'(\alpha)=\sum_{p\in\K_{F'}^+}a_p^{F'}\I(p)(\alpha)$. Then
	\[\sum_{p\in\K_F^+}a_p^F\I(p)(\alpha)=F(\alpha)=F'(\alpha)=\sum_{p\in\K_{F'}^+}a_p^{F'}\I(p)(\alpha),\]
	showing that $a_p^F=a_p^{F'}$ for any route passing through $\alpha^\e$. Since this holds for any signed fringe arrow $\alpha^\e$, we have proven that $\K_F^+=\K_{F'}^+$ and that $a_p^F=a_p^{F'}$ for all $p\in\K_F^+$.

	It remains to show the final statement of the proposition, which we prove using a final claim.
	We will define a map $\Psi$ from blank spaces of $F$ to blank spaces of $F'$.
	Let $J$ be a blank space of $F$ at $\alpha^\e$ bounded below by $(p,I_p)$ and bounded above by $(q,I_q)$.
	Define the intervals
	\[I'_p:=\{C\in[0,F'(\alpha)]\ :\ p_{(\alpha^\e,C)}=p\} ~~~ \text{and} ~~~
	I'_q:=\{C\in[0,F'(\alpha)]\ :\ p_{(\alpha^\e,C)}=q\}\]
	if $p$ and $q$, respectively, exist. If $p=0$, then define $I'_p:=\{0\}$, and if $q=\infty$, then define $I'_q:=\{F'(\alpha)\}$.
	Let $\Psi(J)$ be the blank space of $F'$ at $\alpha^\e$ bounded below by $(p,I'_p)$ and bounded above by $(q,I'_q)$.

	\noindent\textbf{\underline{Claim C:}} 	
	Since $\K_F^+=\K_{F'}^+$, it is immediate that the map $\Psi$ is a bijection from (not necessarily proper) blank spaces of $F$ to (not necessarily proper) blank spaces of $F'$.
	We claim in addition that, if $N_{B,J}$ is the number of markings of $B^{\pm1}$ at $\alpha^\e$ which split $J$, then $|\Psi(J)|=|J|+b_BN_{B,J}$.

	We now prove Claim C.
	Let $D_p$ (resp. $D_q$) be the number of markings of $B^{\pm1}$ at $\alpha^\e$ which are below $p$ (resp. $q$) marked at $\alpha_j^{\e_j}$ in the order $\prec_{\alpha_j^{\e_j}}$.
	Then $D_q-D_p=N_{B,J}$.
	By Claim A and the fact that $a_p^F=a_p^{F'}$ and $a_q^F=a_q^{F'}$, we have
	\[(I'_p)^\circ=\{x+b_BD_p\ :\ x\in I_p^\circ\} ~~~ \text{and} ~~~
	(I'_q)^\circ=\{x+b_BD_q\ :\ x\in I_q^\circ\}\]
	where the superscript $\circ$ denotes the interior of an interval. In other words, the interval $I'_p$ is of the same length as $I_p$ and is shifted $b_BD_p$ to the right, and the interval $I'_q$ is of the same length as $I_q$ and is shifted $b_BD_q$ to the right, where $D_q\geq D_p$. Then
	\[|J'|=\text{dist}(I'_p,I'_q)=\text{dist}(I_p,I_q)+b_B(D_q-D_p)=|J|+b_B(D_q-D_p)=|J|+b_BN_{B,J},\]
	where $\text{dist}$ denotes the distance between two intervals. This ends the proof of Claim C.

	We are now able to prove the final statement of the proposition. If $b_B\leq0$, then Claim C  shows that the bijection $\Psi^{-1}$ from blank spaces of $F'$ to blank spaces of $F$ can only increase the length of a blank space, hence restricts to an injection from proper blank spaces of $F'$ to proper blank spaces of $F$. This shows that $F'$ has weakly fewer proper blank spaces than $F$. Moreover, if $b_B=-M_{F,B}$ and $M_{F,B}>0$, then there exists a proper blank space $J$ of $F$ at some marked arrow $\alpha^\e$ with $b_B=-a_{B,J}=-\frac{|J|}{N_{B,J}}$.
	Then Claim C shows that $\Psi(J)$ has length 
	\[|J|+b_BN_{B,J}=|J|+\left(-\frac{|J|}{N_{B,J}}\right)N_{B,J}=0,\]
	meaning that $\Psi(J)$ is not a proper blank space and hence that $F'$ has strictly fewer proper blank spaces than $F$.
	This ends the proof.
\end{proof}

\subsection{Vortex combinations and vortex spaces}

We define vortex combinations and vortex spaces of a clique. The collection of maximal vortex spaces will eventually give us a complete dissection of the turbulence polyhedron.

\begin{defn}\label{defn:vortexcomb}
	Let $\K$ be a clique of $\tL$. A vortex $G$ is \emph{$\K$-compatible} if it is a nonnegative combination of indicator vectors of $\K$-compatible bands. A \emph{(resp. unit) $\K$-vortex combination} is a combination of the form
	\[G+\sum_{p\in\K}a_p\I(p),\]
	where $\sum_{p\in\K}a_p\I(p)$ is a (resp. unit) clique combination of $\K$ and $G$ is a $\K$-compatible vortex.
	A $\K$-vortex combination is \emph{positive} if $a_p>0$ for every $p\in\K$.
	The \emph{nonnegative vortex space} $\D_{\geq0}^{\spiral}(\K)$ of $\K$ is the set of flows which appear as $\K$-vortex combinations. The \emph{unit vortex space} $\D_1^{\spiral}(\K)$ of $\K$ is the set of flows which appear as unit $\K$-vortex combinations.
	If $\K$ is a maximal clique, then $\Delta_1^{\spiral}(\K)=\Delta_1(\K)$ is the clique simplex of $\K$. If $\K$ is not a maximal clique, then we call $\Delta_1^{\spiral}(\K)$ a \emph{vortex wall}.
\end{defn}

If $\K$ is a clique, then there may be an infinite number of $\K$-compatible bands.
Hence, it is not immediately clear that the vortex space of $\K$ is a polyhedron, since its recession cone may a priori fail to be finitely generated.
We will see in this section that the recession cone of $\D_1^{\spiral}(\K)$ is always finitely generated, so that $\D_1^{\spiral}(\K)$ and $\D_{\geq0}^{\spiral}(\K)$ are polyhedra.

\begin{prop}\label{prop:vort-1dir}
	Let $F$ be a flow of $\tL$. Then $F=G_F+\sum_{p\in\mathcal K_F^+}a_p^F\mathcal I(p)$ is a $\K_F^+$-vortex combination. Moreover, $G_F$ may be obtained as a nonnegative sum of indicator vectors of bands, with each band having less than or equal to $| E|+\sum_{p\in\K_F^+}|p|$ arrows.
\end{prop}
Note that $|E|+\sum_{p\in\K_F^+}|p|$ is the number of blank spaces of $F$ by Remark~\ref{remk:blank-space-card}
\begin{proof}
	Let $N:=|E|+\sum_{p\in\K_F^+}|p|$.
	It suffices to show that $G_F$ may be obtained as a nonnegative sum of indicator vectors of $\K_F^+$-compatible bands with $\leq N$ arrows.
	We prove this by induction on the number of proper blank spaces of $F$.
	When $F$ has no proper blank spaces, we must have $G_F=0$ and there is nothing to show. So, suppose $F$ has at least one proper blank space and that we have shown the result for flows with fewer proper blank spaces than $F$.

	Our goal is to find a $\K_F^+$-compatible band $B$ of $\tL$ with $\leq N$ arrows and which splits $F$ with positive strength $M_{F,B}$.
	Let $M$ be an integer such that $\frac{M}{2}$ is greater than the number of proper blank spaces of $F$.
	Let $J$ be a proper blank space of $F$ at a signed arrow $\alpha^\e$.
	For every $C\in J$, let $p_{(\alpha^\e,C)}^{\leq M}$ be $\alpha_0^{\e_0}\alpha_1^{\e_1}\dots\alpha_{M_C}^{\e_{M_C}}$, where $M_C$ is the largest value in $\{0,1,\dots,M\}$ such that $\for^{M_C}(\alpha^\e,C)$ is defined and the signed arrows $\alpha_j^{\e_j}$ are defined such that $\for^j(\alpha^\e,C)=(\alpha_j^{\e_j},C_j)$ for some $C_j$.
	Let $S$ be the (finite) set of strings appearing as $p_{(\alpha^\e,C)}^{\leq M}$ for some $C\in J$.
	By Proposition~\ref{prop:arrow-order}, the subset
	\[I_p:=\{C\in J\ :\ p_{(\alpha^\e,C)}^{\leq M}=p\}\]
	is an interval for any $p\in S$. Since $S$ is finite and $J$ is partitioned by $\{I_p\ :\ p\in S\}$, there must exist $q\in S$ such that $I_q$ has positive length.
	Write $q=\alpha_0^{\e_0}\alpha_1^{\e_1}\dots\alpha_{M_C}^{\e_{M_q}}$.
	Since $I_q$ has positive length, every arrow $\alpha_j$ for $j\in[M_q]$ is given positive weight $G_F(\alpha_j)$ by the canonical vortex $G_F$.

	Suppose for contradiction that $M_q<M$. Then $\alpha_{M_q}$ is a fringe arrow. Define the value $A\in\mathbb R$ so that, for any $C\in I_q$, we have $\for^{M_q}(\alpha^\e,C)=(\alpha_{M_q}^{\e_{M_q}},C+A)$. Since the map $\for$ sends blank spaces to blank spaces, the interval $\{x+A\ :\ x\in I_q\}$ is in a proper blank space of $F$ at $\alpha_{M_q}^{\e_{M_q}}$. This contradicts Remark~\ref{remk:proactive}. This shows that $M_q=M$.

	Since $\for$ sends blank spaces to blank spaces, for all $i\in\{0,\dots,M\}$ we have that $C_i$ is in a proper blank space $J_i$ of $F$ at $\alpha_i^{\e_i}$. Since $\frac{M}{2}$ is greater than the number of proper blank spaces of $F$, we must be able to find two values $a<b$ of $\{0,\dots,M\}$ such that $\alpha_a^{\e_a}=\alpha_b^{\e_b}$ and the blank spaces $J_a$ and $J_b$ are equal. Choose these so that $b-a$ is minimal. Define the band $B:=\alpha_a^{\e_a}\dots\alpha_{b-1}^{\e_{b-1}}$ (since $b-a$ is minimal, we know that $\alpha_a^{\e_a}\dots\alpha_{b-1}^{\e_{b-1}}$ is not a power of a smaller cycle and hence gives a band).

	We claim that $B$ is $\K_F^+$-compatible and that $B$ splits $F$ with positive strength. Indeed, pick any route $q$ of $\K_F^+$.
	Suppose $q$ and $B$ have an incompatibility $\sigma$. Then for some signed arrows $\alpha_i$ and $\beta_j$ and a sign $\f\in\{-1,1\}$, we have $\alpha_1^{\f}\sigma\beta_1^{-\f}\in q$ and $\alpha_2^{-\f}\sigma\beta_2^{\f}\in B$.

	First, suppose $\sigma$ is a lazy string. Note that every length-two string of $B$ is also a length-two string of $p^F_{(\alpha^\e,C)}$; in particular, $\alpha_2^{-\f}\sigma\beta_2^{\f}=\alpha_2^{-\f}\beta_2^{-\f}\in B$. Then $q$ and $p_{(\alpha^\e,C)}$ are incompatible, which contradicts Corollary~\ref{cor:compat}. This shows that $\sigma$ cannot be a lazy string.

	Then write $\sigma=\beta_1^{\f_1}\beta_2^{\f_2}\dots\beta_M^{\f_M}$. For $i\in[M]$, choose $D_i\in[0,F(\beta_i)]$ such that $p_{(\beta_i^{\f_i},D_i)}^F=q$. 
	Mark $q$ and $B$ at the first arrow $\beta_1^{\f_1}$ of $\sigma$.
	Say this marked copy of $B$ is $\alpha_c^{\e_c}$, for $c\in[a,b-1]$.
	Suppose that $D_1<C_c$; the case $D_1>C_c$ is symmetric.
	We show that $q\prec_{\beta_1^{\e_1}}B$, and hence that $\sigma$ is not an incompatibility between $q$ and $B$.

	{For $i\in[M]$, write $Y_i$ as the element of $\{a,a+1,\dots,b-1\}$ which is equivalent to $i+c-1$ modulo $b-a$ (so that $Y_1=c$, and $Y_2=c+1$, and so on, where we wrap back around from $b-1$ to $a$).}
	We claim that, for each $i\in[M]$, we have $D_i<C_{Y_i}$.
	This holds by hypothesis for $D_1$. Moreover, if it holds for $i<M$, then it follows from Proposition~\ref{prop:arrow-order} that $D_{i+1}<C_{Y_i+1}$. If $Y_i<b-1$, then $Y_{i+1}=Y_i+1$ so this is the same as $D_{i+1}<C_{Y_{i+1}}$.
	If $Y_i=b-1$, then $Y_{i+1}=a$ we may have $C_{Y_{i+1}}=C_{a}\neq C_b= C_{Y_i+1}$. On the other hand, since the markings $B_a$ of $B$ at $\alpha_a^{\e_a}$ and $B_b$ of $B$ at $\alpha_b^{\e_b}$ split the same blank space $J_a=J_b$ of $F$ at $\alpha_a^{\e_a}=\alpha_b^{\e_b}$, the values $C_{Y_{i+1}}$ and $C_{Y_{i}+1}$ are both in the blank space $J_a=J_b$. The value $D_{i+1}$ is not in $J_a=J_b$ because $p_{(\beta_{i+1}^{\e_{i+1}})}^F$ is a route with positive coefficient $a_{(\beta_{i+1}^{\e_{i+1}})}^F$, hence 
$D_{i+1}<C_{Y_i+1}$ implies that $D_{i+1}<C_{Y_{i+1}}$. This ends the proof that $D_i<C_{Y_i}$ for all $i$.
By assumption the arrows of $\for(\beta_M^{\e_M},D_M)$ and $\for(\alpha_{Y_M}^{\e_{Y_M}},C_{Y_M})$ differ; since $D_M<C_{Y_M}$ it then follows from Definition~\ref{defn:forbac} that the signed arrow $\for(\beta_M^{\e_M},D_M)$ has positive sign $\e=1$ while the arrow of $\for(\alpha_{Y_M}^{\e_{Y_M}},C_{Y_M})$ has negative sign. Similarly, the arrow of $\bac(\beta_1^{\e_1},D_1)$ has positive sign while $\bac(\alpha_c^{\e_c},C_c)$ has negative sign. This shows that $\sigma$ is not an incompatibility between $B$ and $q$, completing our contradiction.
We have now shown that $B$ is $\K$-compatible.

Moreover, we have shown that for any $c\in\{a,a+1,\dots,b-1\}$, the marked band $B_c$ at $\alpha_c^{\e_c}$ splits the proper blank space $J_c$.
This completes the proof that the band $B$ splits $F$ with positive coefficient.
	Moreover, by choice of $a$ and $b$, no two markings of $B$ split the same blank space of $F$ at any signed arrow $\alpha^\e$. Since $N$ is equal to the number of blank spaces of $F$ by Remark~\ref{remk:blank-space-card}, this shows that $B$ has $\leq N$ arrows.

	Now let $F':=F-M_{B,F}\I(B)$.
	Then Proposition~\ref{prop:nnew} shows that $\K_{F'}^+=\K_{F}^+$ and, for any $p\in\K_{F'}^+$, we have $a_{p}^{F'}=a_p^F$. Moreover, the same proposition shows that $F'$ has strictly fewer proper blank spaces than $F$. Then our induction hypothesis shows that $G_{F'}$ is a nonnegative sum of indicator vectors of $\K_{F'}^+$-compatible (hence $\K_F^+$-compatible) bands with $\leq N$ arrows.
	Then write
	\begin{align*}
		F&=M_{B,F}\I(B)+F'\\
		&=M_{B,F}\I(B)+G_{F'}+\sum_{p\in\K_{F'}^+}a_p^{F'}\I(p)\\
		&=M_{B,F}\I(B)+G_{F'}+\sum_{p\in\K_F^+}a_p^F\I(p)
	\end{align*}
	and observe that by definition, $G_F=M_{B,F}\I(B)+G_{F'}$. This shows that $G_F$ is a positive combination of $\K$-compatible bands with $\leq N$ arrows and the proof is complete.
\end{proof}

\begin{prop}\label{prop:vort-2dir}
	Let $F$ be a flow of $\tL$ and say $F=G+\sum_{p\in\K}a_p\I(p)$ is a positive vortex combination. Then $\K=\K_F^+$ and $G=G_F$.
\end{prop}
\begin{proof}
	Since $F=G+\sum_{p\in \K}a_p\I(p)$ is a vortex combination, we may choose a set of $\K$-compatible bands $S=\{B_1,\dots,B_m\}$ with nonnegative coefficients $\{b_i\ :\ i\in[m]\}$ so that $G=\sum_{i=1}^mb_i\I(B_i)$.
	For any $j\in\{0,1,\dots,m\}$, define $F_j:=\sum_{i=1}^jb_i\I(B_i)+\sum_{p\in\K}a_p\I(p)$. First, it follows from Theorem~\ref{thm:comb} that $\K_{F_0}^+=\K$ and that $a_p^{F_0}=a_p$ for any $p\in\K$.
	Say for some $j\in\{0,1,\dots,m-1\}$ that we have $\K_{F_j}^+=\K$ and that $a_p^{F_j}=a_p$ for any $p\in\K$.
	It then follows from Proposition~\ref{prop:nnew} that $F_{j+1}=b_{j+1}\I(B_{j+1})+F_j$ satisfies this property as well.
	Applying this for all $j\in\{0,1,\dots,m-1\}$ in turn gives that $F_m=F$ satisfies $\K_F^+=\K$ and that $a_p^F=a_p$ for any $p\in\K$.
\end{proof}

\begin{cor}\label{cor:vort-poly}
	If $\K$ is a clique, then $\D_1^{\spiral}(\K)$ is a polyhedron with vertex set $\{\I(p)\ :\ p\in\K\}$, and $\D_{\geq0}^{\spiral}(\K)$ is a polyhedral cone.
\end{cor}
\begin{proof}
	Let $S$ be the (finite) set of $\K$-compatible bands of $\tL$ with less than or equal to $N:=|E|+\sum_{p\in\K}|p|$ arrows. 
	By Proposition~\ref{prop:vort-2dir}, any vortex combination is of the form $G_F+\sum_{p\in\K_F^+}a_p^F\I(p)$ for some flow $F$. Then by the second part of Proposition~\ref{prop:vort-1dir}, $G_F$ is a positive combination of indicator vectors of bands of $S$. This shows that $\D_1^{\spiral}(\K)$ is the polyhedron given by the Minkowski sum of the clique simplex $\D_\K$ with the indicator vectors of bands of $S$. This similarly shows that $\D^{\spiral}_{\geq0}(\K)$ is the (polyhedral) cone generated by the indicator vectors of routes of $\K$ and of the bands of $S$. 

	It remains to find the vertices of $\D_1^{\spiral}(\K)$. It is immediate that all such vertices must be indicator vectors of routes of $\K$.
	Moreover, if $p\in\K$, then the only way to represent $\I(p)$ as a $\K$-vortex combination is as the trivial sum $\I(p)$ by Proposition~\ref{prop:vort-2dir}, hence $\I(p)$ is a vertex of $\D_1^{\spiral}(\K)$.
\end{proof}

We may now show the first statement of Theorem~\ref{ITHM:VORTEX} from the introduction.

\begin{thm}\label{thm:vort-decomp}
	Let $F$ be a flow of $\tL$. Then $F=G_F+\sum_{p\in\K_F^+}a_p^F\I(p)$ is the \emph{unique} way to express $F$ as a positive $\K$-vortex combination, for any clique $\K$.
\end{thm}
\begin{proof}
	$F=G_F+\sum_{p\in\K_F^+}a_p^F\I(p)$ is a $\K_F^+$-vortex combination by Proposition~\ref{prop:vort-1dir} and uniqueness follows from Proposition~\ref{prop:vort-2dir}.
\end{proof}

We would like to say that Theorem~\ref{thm:vort-decomp} gives us a dissection of $\F_1(\tL)$. To do this, we must figure out which cliques $\K$ give maximal vortex spaces $\D_1^{\spiral}(\K)$.

\subsection{Vortex dissections}

We wish to show that the set of maximal vortex spaces gives us a dissection for the turbulence polyhedron. When we worked with bundle dissections, it was immediate that maximal bundle spaces were precisely the bundle spaces of maximal bundles, but this is no longer the case for vortex spaces. Hence, we first need to figure out which cliques give maximal vortex spaces.

\begin{lemma}\label{lem:fgh}
	Let $\K$ be a clique and let $B$ be a self-compatible band. Then $B$ is $\K$-compatible if and only if $\I(B)$ is $\K$-compatible.
\end{lemma}
\begin{proof}
	It is immediate by Definition~\ref{defn:vortexcomb} that if $B$ is $\K$-compatible then the vortex $\I(B)$ is $\K$-compatible.
	We now show that if $B$ is not $\K$-compatible, then $\I(B)$ is not $\K$-compatible.
	Let $\sigma$ be an incompatibility between $B$ and a route $q$ of $\K$.
	Define
	\begin{align*}
		S_{\text{in}}:=\{\alpha\in\tL\ :\ h(\alpha)\in\sigma\}\text{~~~ and ~~~}S_{\text{out}}:=\{\alpha\in\tL\ :\ t(\alpha)\in\sigma\}.
	\end{align*}
	Let $\alpha_1,\beta_1\in S_{\text{in}}$ and $\alpha_2,\beta_2\in S_{\text{out}}$ and suppose that $q$ contains $\alpha_1\sigma\beta_2^{-1}$ and $B$ contains $\alpha_2^{-1}\sigma\beta_1$. 
		The case where $q$ contains $\alpha_2^{-1}\sigma\beta_1$ and $B$ contains $\alpha_1\sigma\beta_2^{-1}$ is symmetric.
		
		Since $B$ is self-compatible and contains $\alpha_2^{-1}\sigma\beta_1$, the band $B$ may not contain any substring of the form $\gamma_1\sigma'\gamma_2^{-1}$, where $\gamma_i$ are arrows and $\sigma'$ is a substring of $\sigma$. It follows that $\sum_{\alpha\in S_{\text{in}}}\mathcal I(B)(\alpha)<\sum_{\alpha\in S_{\text{out}}}\mathcal I(B)(\alpha)$.
		On the other hand, any $p\in\K$ is compatible with $q$, hence $\sum_{\alpha\in S_{\text{in}}}\mathcal I(p)(\alpha)\geq\sum_{\alpha\in S_{\text{out}}}\mathcal I(p)(\alpha)$. Then any positive combination $G$ of $\K$-compatible bands
		satisfies the same inequality
	$\sum_{\alpha\in S_{\text{in}}}G(\alpha)\geq\sum_{\alpha\in S_{\text{out}}}G(\alpha)$.
	Since $\sum_{\alpha\in S_{\text{in}}}\mathcal I(B)(\alpha)<\sum_{\alpha\in S_{\text{out}}}\mathcal I(B)(\alpha)$, we have shown that $\I(B)$ is not a positive combination of $\K$-compatible bands, and hence is not $\K$-compatible.
\end{proof}

We will now define the class of cliques $\K$ giving maximal vortex spaces $\D_1^{\spiral}(\K)$.

\begin{defn}\label{defn:bandstable}
	A clique $\K$ of $\tL$ is \emph{band-stable} if for any clique $\K'\supsetneq\K$, there exists a band $B_{\K'}$ which is $\K$-compatible but not $\K'$-compatible.
\end{defn}

\begin{thm}\label{thm:vort-sub}
	The following are equivalent for a clique $\K$ of $\tL$:
	\begin{enumerate}
		\item the clique $\K$ is band-stable,
		\item the vortex space $\D_1^{\spiral}(\K)$ is maximal, and
		\item the vortex space $\D_{\geq0}^{\spiral}(\K)$ is maximal.
	\end{enumerate}
\end{thm}
\begin{proof}
	Suppose first that $\K$ is band-stable. We show that $\D_1^{\spiral}(\K)$ and $\D_{\geq0}^{\spiral}(\K)$ are maximal by showing that for any clique $\K'\neq\K$, there exists a point of $\D_1^{\spiral}(\K)\subseteq\D_{\geq0}^{\spiral}(\K)$ which is not in $\D_1^{\spiral}(\K')$ or $\D_{\geq0}^{\spiral}(\K')$.
	First, suppose $\K'\not\supseteq\K$. Then there exists $p\in\K$ which is not in $\K'$. Then $\I(p)\in\D_1^{\spiral}(\K)$ and $\I(p)\not\in\D_{\geq0}^{\spiral}(\K')$. On the other hand, suppose $\K'\supsetneq\K$. By band-stability of $\K$, there exists some $\K$-compatible band $B$ which is not $\K'$-compatible. Let $p$ be any route of $\K$. Then $\I(B)+\I(p)$ is a $\K$-vortex combination, but $\I(B)$ is not $\K'$-compatible by Lemma~\ref{lem:fgh} so $\I(B)+\I(p)$ is not a $\K'$-vortex combination.
	This completes the proof of this direction.

	Now suppose $\K$ is not band-stable. Then there exists a clique $\K'\supsetneq\K$ such that every $\K$-compatible band is $\K'$-compatible. Then it is immediate that $\D_1^{\spiral}(\K')\supseteq\D_1^{\spiral}(\K)$. Moreover, taking some $p\in\K'\backslash\K$, the vector $\I(p)$ is in $\D_1^{\spiral}(\K')$ but not $\D_1^{\spiral}(\K)$, so $\D_1^{\spiral}(\K')\supsetneq\D_1^{\spiral}(\K)$. Similarly, $\D_{\geq0}^{\spiral}(\K')\supsetneq\D_{\geq0}^{\spiral}(\K)$. This completes the proof.
\end{proof}

We may now define the vortex dissections of $\F_1(\tL)$ and $\F_{\geq0}(\tL)$.
\begin{defn}
	Define the \emph{(unit) vortex dissection} of $\F_1(\tL)$ as
	\[\Sbs^{\spiral}(\F_1(\tL)):=\{\D_{1}^{\spiral}(\K)\ :\ \K\text{ is a band-stable clique of }\tL\}.\]
	Define the \emph{nonnegative vortex dissection} of $\F_{\geq0}(\tL)$ as
	\[\Sbs^{\spiral}(\F_{\geq0}(\tL)):=\{\D_{\geq0}^{\spiral}(\K)\ :\ \K\text{ is a band-stable clique of }\tL\}.\]
\end{defn}
We may now complete the proof of Theorem~\ref{ITHM:VORTEX} of the introduction.
\begin{thm}\label{thmCbody}
	The vortex dissections $\Sbs^{\spiral}(\F_1(\tL))$ and $\Sbs^{\spiral}(\F_{\geq0}(\tL))$ are complete dissections of $\F_1(\tL)$ and $\F_{\geq0}(\tL)$, respectively.
\end{thm}
\begin{proof}
	Corollary~\ref{cor:vort-poly} shows that the cells are polyhedra.
	Theorem~\ref{thm:vort-decomp} and Theorem~\ref{thm:vort-sub} give the weak intersection property.
	Because every flow may be obtained as a vortex combination by Theorem~\ref{thm:vort-decomp}, the dissection is complete.
\end{proof}

Like bundle dissections, a vortex dissection always contains the clique triangulation because every maximal clique is band-stable.  When $\tL$ is representation-finite, all three dissections are equal. Unlike bundle subdivisions, vortex dissections are always complete.
Because maximal clique simplices are dense in $\F_1(\tL)$, no vortex wall is full-dimensional in $\F_1(\tL)$ or $\F_{\geq0}(\tL)$. 

\begin{example}\label{ex:kron-vortex-decomp}
	Let $\tL$ be the Kronecker fringed quiver of Figure~\ref{KRONINTRO}.
	As is always the case, every maximal clique of $\tL$ is band-stable. Since the unique band $e_2f_2^{-1}$ of $\tL$ is not compatible with any bending route, the only band-stable clique which is not maximal is the set of straight routes $\{e_1e_2e_3,f_1f_2f_3\}$.
	The corresponding vortex space is 
	\[\Delta_1^{\spiral}(\{e_1e_2e_3,f_1f_2f_3\})=\{x\I(e_1e_2e_3)+y\I(f_1f_2f_3)+z\I(e_2f_2^{-1})\ :\ x,y,z\in\mathbb R_{\geq0}\text{ and }x+y=1\}.\]
	Note that this is the same as the bundle space $\Delta_1(\{e_1e_2e_3,f_1f_2f_3,e_2f_2^{-1}\})$. Hence, the vortex dissection of the Kronecker fringed quiver is the same as its bundle subdivision given in Example~\ref{ex:kron-eddy-decomp}, with the caveat that the unique cell which is not a simplex is indexed by the band-stable clique $\{e_1e_2e_3,f_1f_2f_3\}$ in the vortex dissection and is indexed by the maximal bundle $\{e_1e_2e_3,f_1f_2f_3,e_2f_2^{-1}\}$ in the bundle subdivision.
\end{example}

The fringed quiver of Example~\ref{ex:gemd} is representation-finite, so the vortex dissection and bundle subdivision of its turbulence polyhedron both agree with the clique triangulation as given in Example~\ref{ex:gemd}.

\subsection{Vortex dissections of the g-polyhedron and g-vector fan}

Recall Section~\ref{ssec:CBSG}, wherein we took the clique and bundle subdivisions of the turbulence polyhedron $\F_1(\tL)$ and cone of flows $\F_{\geq0}(\tL)$ and ported them through $\phi$ to get analogous subdivisions of the g-polyhedron $\g_1(\L)$ and the space $\mathbb R^{V_{\text{int}}}$, respectively.
We now apply the same strategy to the vortex dissections of $\F_1(\tL)$ and $\F_{\geq0}(\tL)$.

\begin{defn}
	Let $\K$ be a clique. Let $\mathcal B_\K$ be the set of $\K$-compatible bands.
	A \emph{g-vortex combination} of $\K$ is a sum of the form
	\[z+\sum_{p\in\K}a_p\g(p),\]
	where all coefficients $a_p$ are nonnegative and
	$z=\sum_{B\in\mathcal B_K}b_B\g(B)$ is a finite nonnegative combination of g-vectors of $\K$-compatible bands.
	We call $z$ the \emph{canonical g-vortex} of the combination.
	The combination is \emph{unit} if $\sum_{p\in\K}a_p\leq0$.
	Define the \emph{(resp. unit) g-vortex space} $\g_{\geq0}^{\spiral}(\K)$ (resp. $\g_1^{\spiral}(\K)$) as the space of (resp. unit) vortex combinations of $\K$.
	If $\K$ is not a maximal clique, then we refer to $\g_{1}^{\spiral}(\K)$ and $\g_{\geq0}^{\spiral}(\K)$ as \emph{g-vortex walls}.
\end{defn}

\begin{prop}\label{prop:g-vort-max}
	The following are equivalent for a clique $\K$ of $\tL$:
	\begin{enumerate}
		\item the clique $\K$ is band-stable,
		\item the g-vortex space $\g_1^{\spiral}(\K)$ is maximal, and
		\item the g-vortex space $\g_{\geq0}^{\spiral}(\K)$ is maximal.
	\end{enumerate}
\end{prop}
\begin{proof}
	Apply the map $\phi$ to Theorem~\ref{thm:vort-sub}.
\end{proof}

\begin{thm}\label{thm:g-vort}
	Every point in $\mathbb R^{V_{\text{int}}}$ has a unique description as a vortex combination.
\end{thm}
\begin{proof}
	This proceeds as the proof of Theorem~\ref{thm:g-decomp}, replacing (g-)bundle combinations with (g-)vortex combinations.
\end{proof}

\begin{defn}
	Define the \emph{g-vortex dissections}
	\begin{align*}
		&\Sbs^{\spiral}(\g_1(\L)):=\{\g_1^{\spiral}(\K)\ :\ \K\text{ is a band-stable clique of }\tL\} \text{ and } \\
		&\Sbs^{\spiral}(\g_{\geq0}(\L)):=\{\g_{\geq0}^{\spiral}(\K)\ :\ \K\text{ is a band-stable clique of }\tL\}.
	\end{align*}
\end{defn}

\begin{cor}
	$\Sbs^{\spiral}(\g_1(\L))$ is a complete dissection of the g-polyhedron $\g_1(\L)$ containing the g-clique triangulation. $\Sbs^{\spiral}(\g_{\geq0}(\L))$ is a complete dissection of the ambient space $\mathbb R^{V_{\text{int}}}$ of the g-vector fan containing the standard triangulation of the g-vector fan into g-cones.
\end{cor}
\begin{proof}
	Corollary~\ref{cor:vort-poly} shows that the spaces $\D_1^{\spiral}(\K)$ and $\D_{\geq0}^{\spiral}(\K)$ are polyhedra for any band-stable clique $\K$, hence their images $\phi(\D_1^{\spiral}(\K))=\g_1^{\spiral}(\K)$ and $\phi(\D_{\geq0}^{\spiral}(\K))=\g_{\geq0}^{\spiral}(\K)$ are polyhedra.
	The weak intersection property is then given by Proposition~\ref{prop:g-vort-max} and Theorem~\ref{thm:g-vort}.
\end{proof}

\begin{example}
	As we discussed in Example~\ref{ex:kron-vortex-decomp}, the g-vortex dissection of the g-polyhedron of the Kronecker quiver is the same as its g-bundle subdivision, with the caveat that the cell given by the ray $\{(a,-a)\ :\ a\in\mathbb R_{\geq0}\}$ is indexed by the band-stable clique $\{e_1e_2e_3,f_1f_2f_3\}$ rather than the maximal bundle $\{e_1e_2e_3,f_1f_2f_3,e_2f_2^{-1}\}$.
\end{example}

Recall that the usual g-vector fan of $\L$ is a part of the vortex dissection $\g_{\geq0}^{\spiral}(\L)$. The rest of the vortex dissection $\g_{\geq0}^{\spiral}(\L)$, then, partitions the complement of the g-vector fan into vortex walls. It would be interesting to interpret this partitioning representation-theoretically.
For example, perhaps it could be connected to non-functorially finite torsion classes~\cite{DIRRT,AI24}, the wall-and-chamber picture of $\L$~\cite{ITW,BSH,BHIT,BRIDGE,BST}, the purely non-rigid region of $\tL$~\cite{Asa22}, or scattering diagrams~\cite{GHKK,READING}.

Recall Remark~\ref{remk:bundle-subd-reduced}, which stated that a g-bundle subdivision may be viewed as the set of bundle spaces of maximal \emph{reduced} bundles. To make the analogous remark for g-vortex dissections, we must define a \emph{band-stable reduced clique} to be $\K_{\text{red}}$, where $\K$ is a band-stable clique. Note that $\g_1^{\spiral}(\K_{\text{red}})=\g_1^{\spiral}(\K)$ for any clique $\K$, leading to the following remark.
\begin{remk}\label{remk:vortex-subd-reduced}
The g-vortex dissection of $\tL$ is the union of vortex spaces of band-stable reduced cliques.
\end{remk}

\section{A Sufficiently Complicated Example}
\label{sec:example}

We have fully described the maximal bundles and band-stable cliques (and hence the bundle subdivisions and vortex dissections) of the fringed quivers of Figure~\ref{KRONINTRO} and Figure~\ref{GEMINTRO}.
In both of these examples, the bundle subdivisions are complete and are equal to the vortex dissections. We now describe in greater detail the bundle subdivision and vortex dissection of the double Kronecker fringed quiver of Figure~\ref{fig:irrational}, which we have already seen has some irrational flows which may not be obtained as bundle combinations.
For the rest of this section, $\tL$ will refer to this fringed quiver unless otherwise specified.

In particular, we will give a full description of the maximal bundles containing bands as well as the band-stable cliques which are not maximal cliques. These index two different interpretations for the complement of the g-vector fan $g_{\geq0}(\L)$.
Naively, it is not clear what these bundles and cliques should look like, and it is not even clear what the self-compatible routes and bands are.
While it may be possible to work this out using compatibility on trails directly, we find it more convenient to answer both of these questions through a clever use of the flow algorithm.
A key idea is that self-compatible trails are determined by their indicator vectors, and we may determine whether an integer flow $F$ is the indicator vector of a trail by applying the flow algorithm to $F$.

\subsection{Self-compatible routes and bands}

We would like to characterize the self-compatible routes and bands. It suffices to characterize the indicator vectors of self-compatible routes and bands by Theorem~\ref{thm:comb}.
To this end, we need notation to talk about certain flows on $\tL$.
For any set of nonnegative integers $(a,b,c,d)$, define the flow $F_{(a,b,c,d)}$ so that
\begin{align*}
	&F_{(a,b,c,d)}(e_1)=F_{(a,b,c,d)}(f_1)=a, \\
	&F_{(a,b,c,d)}(e_2)=F_{(a,b,c,d)}(f_2)=b, \\
	&F_{(a,b,c,d)}(e_3)=F_{(a,b,c,d)}(f_3)=c, \text{ and } \\
	&F_{(a,b,c,d)}(e_4)=F_{(a,b,c,d)}(f_4)=d.
\end{align*}

\begin{prop}\label{prop:dbands}
	For any pair of coprime nonnegative integers $(b,c)$, there is a self-compatible band $B_{(b,c)}$ with indicator vector $F_{(0,b,c,0)}$.
	This is a complete list of all self-compatible bands, and no two distinct self-compatible bands are compatible.
\end{prop}
For Proposition~\ref{prop:dbands} and the following, we adopt the convention that $0$ is coprime to 1 but not to any other nonnegative integer. Hence, $B_{(0,1)}$ and $B_{(1,0)}$ are the only self-compatible bands with 0's in the subscripts.
\begin{proof}
	Let $B$ be a self-compatible band of $\tL$. Of course, $\I(B)(e_1)=\I(B)(f_1)=0$. Then conservation of flow at the leftmost internal vertex gives $\I(B)(e_2)=\I(B)(f_2)$. Similarly, $\I(B)(e_3)=\I(B)(f_3)$. This shows that $\I(B)$ is of the form $F_{(0,b,c,0)}$, where $(b,c)=(\I(B)(e_2),\I(B)(e_3))$.

	We now show that if $(b,c)$ are coprime, then there exists a self-compatible band $B_{(b,c)}$ with indicator vector $F_{(0,b,c,0)}$. If $(b,c)=(1,0)$, then the band $e_2f_2^{-1}$ is the desired band. Similarly, if $(b,c)=(0,1)$, then the band $e_3f_3^{-1}$ is the desired band. We now assume that $b$ and $c$ are both nonzero. We use the flow algorithm to find $p_{(f_2,0)}^{F_{(0,b,c,0)}}$ and show that its indicator vector is $F_{(0,b,c,0)}$.

	Suppose $b\geq c$; the case where $b\leq c$ is symmetric. We calculate the following for the flow $F_{(0,b,c,0)}$ using Definition~\ref{defn:forbac}:
	\begin{enumerate}
		\item If $c<C$, then
			\begin{align*}
				\for(f_2,C)&=(e_2^{-1},C-c) & (C>c\text{ activates branch (2) of Definition~\ref{defn:forbac}})\\
				\for(e_2^{-1},C-c)&=(f_2,C-c) & (C-c<b \text{ activates branch (3)}).
			\end{align*}
		\item If $0<C\leq c$, then
			\begin{align*}
				\for(f_2,C)&=(f_3,C) & (C\leq c\text{ activates branch (1)})\\
				\for(f_3,C)&=(e_3^{-1},C) & (C>0\text{ activates branch (2)}) \\
				\for(e_3^{-1},C)&=(e_2^{-1},C+b-c) & (C+b>b>c\text{ activates branch (4)})\\
				\for(e_2^{-1},C+b-c)&=(f_2,C+b-c) & (C+b-c<c+b-c=b\text{ activates branch (3)}).
			\end{align*}
	\end{enumerate}
	We wish to calculate the trail $p_{(f_2,1)}^{F_{(0,b,c,0)}}$. We will do this by starting at the arrow-flow $(f_2,1)$ and applying $\for$ repeatedly. The above calculation shows that starting from some flow $(f_2,C)$ and applying $\for$ until we reach the arrow $f_2$ again gives us the arrow-flow
	\begin{equation}\label{vin}\begin{cases}
		(f_2,C-c) & c<C\\
		(f_2,C+b-c) & 0<C\leq c.
	\end{cases}
	\end{equation}
	This amounts to subtracting $c$ from $C$ modulo $b$, while potentially adding a multiple of $b$ to stay in the interval $[0,b]$ (whenever the bottom branch is activated).
	This process will eventually lead us back to the arrow-flow $(f_2,C)$ -- say after $x$ total applications of the top branch of~\eqref{vin} and $y$ applications of the bottom branch.
	Since $b$ and $c$ are coprime, every integer in $\{1,2,\dots,b\}$ will be reached exactly once and hence we must have $x+y=b$. Moreover, $C=C+x(-c)+y(b-c)$, hence $y=c$ and $x=b-c$. Each application of the top branch of~\eqref{vin} adds the path $f_2e_2^{-1}$ and each application of the bottom adds the path $f_2f_3e_3^{-1}e_2^{-1}$, so we have
	\[\I(p_{(f_2,1)}^{F_{(0,b,c,0)}})=F_{(0,x+y,y,0)}=F_{(0,b,c,0)}.\]
	Then the band $B_{(b,c)}:=\I(p_{(f_2,1)}^{F_{(0,b,c,0)}})$ has the desired indicator vector when $(b,c)$ are coprime.

	It remains to show that all self-compatible bands are of this form, and that no two self-compatible bands form a bundle.
	It is immediate that every band must have indicator vector $F_{(0,b,c,0)}$ for some pair of nonnegative integers $(b,c)$. If $(b,c)$ are not coprime, then we may express $(b,c)=(nb',nc')$ for some $n>1$ and some nonnegative coprime integers $(b',c')$. Then $F_{(0,b,c,0)}=n\I(B_{(b',c')})$ is a bundle combination realizing $F_{(0,b,c,0)}$, so by uniqueness of Theorem~\ref{thm:comb} it is the \emph{only} bundle combination realizing $F_{(0,b,c,0)}$, hence there is no self-compatible band with indicator vector $F_{(0,b,c,0)}$. This completes the proof that the self-compatible bands of $\tL$ are precisely of the form $B_{(b,c)}$, where $(b,c)$ are coprime nonnegative integers.

	Suppose that two self-compatible bands $B_{(b_1,c_1)}$ and $B_{(b_2,c_2)}$ are compatible. Choose $n\geq1$ and coprime $(b,c)$ so that $(b_1+c_1,b_2+c_2)=(nb,nc)$. As in the above paragraph, $F_{(0,b_1+c_1,b_2+c_2,0)}=n\I(0,b,c,0)$ is the unique bundle combination realizing $F_{(0,b_1+c_1,b_2+c_2,0)}$. This forces $(b_1,c_1)=(b_2,c_2)$, completing the proof that no two distinct self-compatible bands form a bundle.
\end{proof}

Let $\mathcal B$ be the set of self-compatible bands $B_{(b,c)}$ of $\tL$.

We have gained an understanding of the self-compatible bands of $\tL$ by applying the flow algorithm to obtain flows $F_{(0,b,c,0)}$ as bundle combinations, where $b$ and $c$ are nonnegative integers. We now wish to understand self-compatible routes of $\tL$ through a similar strategy. Our first step will be the following lemma obtaining flows $F_{(1,b,c,0)}$ as bundle combinations, where $b$ and $c$ are nonnegative integers.

\begin{lemma}\label{lem:tLleftF}
	Fix nonnegative integers $(b,c)$ such that at least one is positive.
	\begin{enumerate}
		\item If $b=0$ and $c=0$, then $F_{(1,b,c,0)}=F_{(1,0,0,0)}$ is the indicator vector of the self-compatible route $l_{(0,1)}:=e_1f_1^{-1}$.
		\item If $b=0$ and $c>0$, then $F_{(1,b,c,0)}=F_{(1,0,c,0)}$ is realized as the bundle combination
			\[\I(l_{(0,1)})+c\I(B_{(0,1)}).\]
		\item If $b>0$ and $c=0$, then $F_{(1,b,c,0)}=F_{(1,b,0,0)}$ is the indicator vector of a self-compatible route $l_{(b,-1)}$.
		\item If $b>0$ and $c>0$ and $(b,c-1)$ are coprime, then $F_{(1,b,c,0)}$ is the indicator vector of a self-compatible route $l_{(b,c-1)}$.
		\item If $b>0$ and $c>0$ and $(b,c-1)$ are not coprime, then write $(b,c-1)=m(b',c')$ where $m>1$ and $(b',c')$ are coprime nonnegative integers. Then $F_{(1,b,c,0)}$ is realized as the bundle combination
			\[\I(l_{(b',c')})+(m-1)\I(B_{(b',c')}).\]
	\end{enumerate}
\end{lemma}
We summarize Lemma~\ref{lem:tLleftF} with the following table showing the bundle combination for $F_{(1,b,c,0)}$ for all pairs of nonnegative integers $(b,c)$.
\begin{center}
	\begin{tabular}{|c||c|c|}
		\hline
		$\bf{F_{(1,b,c,0)}=}$ & $b=0$ & $b>0$ \\
		\hline
		\hline
		$c=0$ & $\I(l_{(0,1)})$ & $\I(l_{(b,-1)})$\\
		\hline
		$c>0$ & $\I(l_{(0,1)})+c\I(B_{(0,1)})$ & 
			\begin{tabular}{c}
				$(b,c-1)$ coprime:  \ \ $\I(l_{(b,c-1)})$ \ \ \ \ \ \ \ \ \ \ \ \ \ \ \ \ \ \ \ \ \ \  \\
				$(b,c-1)=m(b',c')$: \ \ $\I(l_{(b',c')})+(m-1)\I(B_{(b',c')})$
			\end{tabular}
		\\
		\hline
	\end{tabular}
\end{center}
\begin{proof}
	First, it is immediate that $l_{(0,1)}:=e_1f_1^{-1}$ is a self-compatible route with indicator vector $F_{(1,0,0,0)}$ which is compatible with the band $B_{(0,1)}=e_3f_3^{-1}$. This shows (1) and (2). If $b>0$, then define $l_{(b,-1)}:=e_1(e_2f_2^{-1})^bf_1^{-1}$.
	It is not hard to check that this is a self-compatible band with indicator vector $F_{(1,b,0,0)}$, proving (3).

	It remains to show (4) and (5). We do these simultaneously. Suppose $b>0$ and $c>0$. Suppose also that $b\geq c$; the case where $b\leq c$ is similar.
	As in the proof of Proposition~\ref{prop:dbands}, we study the effect of applying $\for$ to an arrow-flow $(f_2,C)$ until a new arrow-flow on $f_2$ or a fringe arrow is reached.
	\begin{enumerate}
		\item[(i)] First, suppose $0<C<c-1$. Then
		\begin{align*}
			\for(f_2,C)&=(f_3,C) & (C<c-1<c\text{ activates (1) of Definition~\ref{defn:forbac}}) \\
			\for(f_3,C)&=(e_3^{-1},C) & (0<C\text{ activates branch (2)}) \\
			\for(e_3^{-1},C)&=(e_2^{-1},C+b-c) & (b+C>b\geq c\text{ activates branch (4)}) \\
			\for(e_2^{-1},C+b-c)&=(f_2,C+b-c+1) & (C+b-c<b-1\text{ activates branch (3)}).
		\end{align*}
		\item[(ii)] Now suppose $c-1< C< c$. Then
		\begin{align*}
			\for(f_2,C)&=(f_3,C) & (C\leq c\text{ activates branch (1)}) \\
			\for(f_3,C)&=(e_3^{-1},C) & (0<C\text{ activates branch (2)}) \\
			\for(e_3^{-1},C)&=(e_2^{-1},C+b-c) & (b+C>b\geq c\text{ activates branch (4)}) \\
			\for(e_2^{-1},C+b-c)&=(e_1^{-1},C+b-c-b+1)
			& (C+b-c+1>b\text{ activates branch (4)}).
		\end{align*}
		\item[(iii)] Now suppose $c<C<b$. Then
		\begin{align*}
			\for(f_2,C)&=(e_2^{-1},C-c) & (C>c\text{ activates branch (2)}) \\
			\for(e_2^{-1},C-c)&=(f_2,C-c+1) & (C-c+1\leq C<b\text{ activates branch (3)}).
		\end{align*}
	\end{enumerate}
	The above calculation shows us how to find the path $p_{(f_1,0.5)}^{F_{(1,b,c,0)}}$. Of course, $\for(f_1,0.5)=(f_2,0.5)$. Continuing on from any arrow-flow $(f_2,C)$, we proceed as follows. If $C-(c-1)\in(0,b)$, then this triggers (iii) and we add the path $e_2^{-1}f_2$ to get $(f_2,C-(c-1))$. If $C-(c-1)$ is negative, then (i) triggers and we add the path $f_3e_3^{-1}e_2^{-1}f_2$ to get $(f_2,C+b-(c-1))$. This amounts to subtracting $(c-1)$ modulo $b$, while correcting by adding multiples of $b$ where necessary so that our flow value stays in $(0,b)$. The process ends when we reach $(f_2,c-0.5)$, in which case (ii) triggers and we add the path $f_3e_3^{-1}e_2^{-1}e_1^{-1}$ to finish off our route. If $x$ is the number of times (i) activated and $y$ is the number of times (iii) activated, then the indicator vector of $p:=p_{(f_1,0.5)}^{F_{(1,b,c,0)}}$ is $F_{(1,x+y+1,x+1,0)}$.

	Write $(b,c-1)=(mb',mc')$, where $m\geq1$ and $(b',c')$ are coprime nonnegative integers. If $(b,c-1)$ are coprime, then $(b',c')=(b,c-1)$ and $m=1$.
	We wish to show that $\I(p)=F_{(1,x+y+1,x+1,0)}$ is equal to $F_{(1,b',c'+1,0)}$, thus proving that that $p$ is our desired path $l_{(b',c')}$.

	The process of repeatedly subtracting $(c-1)$ modulo $b$ while adding increments of $b$ in order to stay in the interval $[0,b)$ will reach $(f_2,D+0.5)$ for precisely those $D\in\{0,1,\dots,b-1\}$ which are multiples of $m=\textup{GCD}(b,c-1)$, and it ends with $(f_2,(c-1)+0.5)$.
	There are $b'$ multiples of $m$ in $\{0,1,\dots,b-1\}$, meaning that $p$ passes through $f_2$ (and $e_2$) precisely $b'$ times and $x+y+1=b'$.

	We now argue that $x+1=c'+1$. Since our final arrow-flow at $f_2$ is $(f_2,c-0.5)$, we have the equation
	$x(b-(c-1))-y(c-1)=c-1$.
	Plugging in $y=b'-1-x$, we get the equation
	$x(b-(c-1))-(b'-1-x)(c-1)=c-1$.
	Solving this equation for $x$ gives
			\begin{align*}
				x&=\frac{(c-1)+(c-1)(b'-1)}{b-(c-1)+(c-1)}\\
				&=(c-1)\frac{b'}{b}&(\text{simplifying})\\
				&=(c-1)\frac{c'}{c-1}\\
				&=c',
			\end{align*}
	where the third equality follows because $\frac{b'}{b}=\frac{c'}{c-1}$.
	This proves that $x+1=c'+1$.
	We have shown that the indicator vector of $p$ is $F_{(1,b',c'+1,0)}$, so call this path $l_{(b',c')}$. If $b$ and $c-1$ are coprime, then $F_{(1,b',c'+1,0)}=F_{(1,b,c,0)}$ and we are done with the proof of (4); we now continue on with the proof of (5) under the assumption that $b$ and $c-1$ are not coprime. The flow $F_{(1,b',c'+1,0)}$ must appear in the bundle combination for $F_{(1,b,c,0)}$. The remaining flow is
	\[F_{(0,b-b',c-(c'+1),0)}=F_{(0,(m-1)b',(m-1)c',0)}=(m-1)F_{(0,b',c',0)}=(m-1)\I(B_{(b',c')}).\]
	Since bundle combinations are unique, $(m-1)\I(B_{(b',c')})$ is the only bundle combination for $F_{(0,b-b',c-(c'+1),0)}$, hence the bundle combination for $F_{(1,b,c,0)}$ must be
	\[F_{(1,b,c,0)}=\I(p_{(b,c-1)})+(m-1)\I(B_{(b',c')}).\]
	This completes the proof of (5).
\end{proof}

The following lemma is dual to Lemma~\ref{lem:tLleftF} and its proof is symmetric.

\begin{lemma}\label{lem:tLrightF}
	Fix nonnegative integers $(b,c)$ such that at least one is positive.
	\begin{enumerate}
		\item If $b=0$ and $c=0$, then $F_{(0,b,c,1)}=F_{(0,0,0,1)}$ is the indicator vector of the self-compatible route $r_{(1,0)}:=e_4f_4^{-1}$.
		\item If $b>0$ and $c=0$, then $F_{(0,b,c,1)}=F_{(0,b,0,1)}$ is realized as the bundle combination
			\[\I(r_{(1,0)})+c\I(B_{(1,0)}).\]
		\item If $b=0$ and $c>0$, then $F_{(0,b,c,1)}=F_{(0,0,c,1)}$ is the indicator vector of a self-compatible route $r_{(-1,c)}$.
		\item If $b>0$ and $c>0$ and $(b-1,c)$ are coprime, then $F_{(0,b,c,1)}$ is the indicator vector of a self-compatible route $r_{(b-1,c)}$.
		\item If $b>0$ and $c>0$ and $(b-1,c)$ are not coprime, then write $(b-1,c)=m(b',c')$ where $m>1$ and $(b',c')$ are coprime nonnegative integers. Then $F_{(0,b,c,1)}$ is realized as the bundle combination
			\[\I(r_{(b',c')})+(m-1)\I(B_{(b',c')}).\]
	\end{enumerate}
\end{lemma}

We now describe the self-compatible routes of $\tL$.
\begin{thm}\label{thm:mememe}
	The complete list of self-compatible routes of $\tL$ is as follows.
	\begin{enumerate}
		\item If $(b,c)$ is a pair of coprime nonnegative integers or $b\geq1$ and $c=-1$, then there is a self-compatible route $l_{(b,c)}$ of $\tL$ with indicator vector
			\[\I(l_{(b,c)})=
			\begin{cases}
				F_{(1,b,c+1,0)} & (b,c)\neq(0,1)\\
				F_{(1,0,0,0)} & (b,c)=(0,1).
			\end{cases}\]
			Let $\mathcal L$ be the set of all such routes $l_{(b,c)}$.
		\item If $(b,c)$ is a pair of coprime nonnegative integers or $c\geq1$ and $b=-1$, then there is a self-compatible route $r_{(b,c)}$ of $\tL$ with indicator vector
			\[\I(r_{(b,c)})=
			\begin{cases}
				F_{(0,b+1,c,1)} & (b,c)\neq(1,0)\\
				F_{(0,0,0,1)} & (b,c)=(1,0).
			\end{cases}\]
			Let $\mathcal R$ be the set of all such routes $r_{(b,c)}$.
	\end{enumerate}
	The set $\mathcal L\cup \mathcal R$ is the set of all self-compatible routes of $\tL$.
\end{thm}
\begin{proof}
	Lemma~\ref{lem:tLleftF} and Lemma~\ref{lem:tLrightF} show that the desired routes of $\mathcal L\cup\mathcal R$ exist and are self-compatible. It remains to show that all self-compatible routes are of this form.

	It is immediate that any $p$ route beginning with $f_1$ must end with either $e_1^{-1}$ or $f_4$. Suppose that a self-compatible route $p$ begins with $f_1$ and ends with $f_4$. Then it must be true that $p$ contains $f_2$ and $f_3$. Then $p$ is the distinguished route of the clique $\{p,f_1f_2f_3f_4\}$ at the arrows $f_1$, $f_2$, $f_3$, and $f_4$, contradicting the result~\cite[Proposition 2.28]{PPP} which states that every route in a clique must be the distinguished route at some arrow.
	Then any self-compatible route beginning with $f_1$ must end with $e_1^{-1}$. Similarly, any self-compatible route beginning with $f_4^{-1}$ must end with $e_4$. This shows that any self-compatible route must have indicator vector $F_{(1,b,c,0)}$ or $F_{(0,b,c,1)}$, for some nonnegative integers $(b,c)$. Lemmas~\ref{lem:tLleftF} and~\ref{lem:tLrightF} show that the routes of $\mathcal L\cup\mathcal R$ are the only self-compatible routes with such indicator vectors, completing the proof.
\end{proof}

We now describe the maximal bundles and band-stable cliques of $\tL$.

\begin{thm}\label{thm:tLmaxlbundles}
	The maximal reduced bundles of $\tL$ containing bands are precisely the bundles
	\[\big\{\{l_{(b,c)},B_{(b,c)}\}\ :\ (b,c)\in\mathbb Z_{\geq0}\text{ coprime}\big\}
	\cup
	\big\{\{r_{(b,c)},B_{(b,c)}\}\ :\ (b,c)\in\mathbb Z_{\geq0}\text{ coprime}\big\}.
	\]
\end{thm}
\begin{proof}
	Since all self-compatible bands are of the form $B_{(b,c)}$, it suffices to show that for fixed coprime nonnegative integers $(b,c)$ the only maximal reduced bundles of $\tL$ containing $B_{(b,c)}$ are $\{l_{(b,c)},B_{(b,c)}\}$ and $\{r_{(b,c)},B_{(b,c)}\}$.

	It is immediate from Theorem~\ref{lem:tLleftF} that $B_{(b,c)}$ is compatible with $l_{(b,c)}$ and $r_{(b,c)}$. By Proposition~\ref{prop:dbands}, it is compatible with no other self-compatible band of $\tL$. We now show that it is not compatible with any other self-compatible routes.

	Indeed, if $B_{(b,c)}$ is compatible with $l_{(b_2,c_2)}$ where $(b,c)\neq(b_2,c_2)$ then $\I(B_{(b,c)})+\I(l_{(b_2,c_2)})$ must be a bundle combination. This is impossible because Lemma~\ref{lem:tLleftF} describes all bundle combinations of flows $F_{(1,b',c',0)}$ (for $b',c'$ nonnegative integers), and all are of the form $\I(l_{(b'',c'')})+m\I(B_{(b'',c'')})$ for some coprime $(b'',c'')$ and $m\geq0$. This completes the proof that $B_{(b,c)}$ is not compatible with any self-compatible route $l_{(b_2,c_2)}$. Similarly, it is not compatible with any $r_{(b_2,c_2)}$. We have shown that the only bending trails compatible with $B_{(b,c)}$ are $l_{(b,c)}$ and $r_{(b,c)}$.

	It remains only to show that $l_{(b,c)}$ and $r_{(b,c)}$ are not compatible with each other. If $(b,c)=(0,1)$, then $l_{(b,c)}=e_1f_1^{-1}$ and $r_{(b,c)}=f_4^{-1}f_3^{-1}f_2^{-1}e_2e_3e_4$ are incompatible. Similarly, $l_{(1,0)}$ and $r_{(1,0)}$ are incompatible. We now show that $l_{(b,c)}$ and $r_{(b,c)}$ are incompatible when $(b,c)$ are coprime positive integers. Suppose to the contrary that they are compatible. Then
	\[F_{(1,2b+1,2c+1,0)}=\I(l_{(b,c)})+\I(r_{(b,c)})\]
	is a bundle combination. This is a contradiction, because the flow algorithm applied to the arrow-flow $(e_1,C)$ of $F_{(1,2b+1,2c+1,1)}$ for any $C\in[0,1]$ gives $p_{(e_1,C)}^{F_{(1,2b+2,2+1,1)}}=e_1e_2e_3e_4$. This shows that $l_{(b,c)}$ and $r_{(b,c)}$ are incompatible.

	We have shown in all cases that that $B_{(b,c)}$ is compatible only with the bending trails $l_{(b,c)}$ and $r_{(b,c)}$, and that all three of these trails do not make up a bundle. It follows that the maximal reduced bundles containing $B_{(b,c)}$ are $\{B_{(b,c)},l_{(b,c)}\}$ and $\{B_{(b,c)},r_{(b,c)}\}$.
\end{proof}

\begin{thm}\label{thm:tLbscliques}
	The band-stable cliques of $\tL$ which are not maximal cliques are
	\[\big\{\{l_{(b,c)}\}\ :\ (b,c)\in\mathbb Z_{\geq0}\text{ coprime}\big\}
	\cup
	\big\{\{r_{(b,c)}\}\ :\ (b,c)\in\mathbb Z_{\geq0}\text{ coprime}\big\}\cup\{\emptyset\}.
	\]
\end{thm}
\begin{proof}
	We first show that the given cliques are band-stable.
	If $(b,c)$ are coprime nonnegative integers, then $l_{(b,c)}$ is compatible with $B_{(b,c)}$. 
	By Theorem~\ref{thm:tLmaxlbundles}, no route is compatible with both $l_{(b,c)}$ and $B_{(b,c)}$, hence $\{l_{(b,c)}\}$ is a band-stable clique. Similarly, $\{r_{(b,c)}\}$ is a band-stable clique.
	The empty set is band-stable because any route is incompatible with some band.

	It remains to show that there are no additional band-stable cliques.
	Any clique containing two distinct routes is not compatible with any band, so all band-stable cliques are of cardinality less than or equal to 1. Any route $l_{(b,-1)}$ for $b\geq1$ is compatible with no bands, and hence is not band-stable. Similarly, the routes $r_{(-1,c)}$ for $c\geq1$ are compatible with no bands and hence are not band-stable. This completes the proof.
\end{proof}

\subsection{Bundle subdivisions and vortex dissections}

The previous two theorems described the maximal bundles containing bands, and the band-stable cliques which are not maximal, respectively.
The former is the combinatorial model for the (g-)bundle walls of $\tL$ and the latter is the combinatorial model for the (g-)vortex walls of $\tL$.
We now use this to show that the vortex dissection of $\F_1(\tL)$ is equal to the union of the bundle subdivision of $\F_1(\tL)$ with the vortex wall $\Delta_{1}^{\spiral}(\emptyset)$.

\begin{lemma}\label{lem:ggsf}
	For any coprime nonnegative integers $(b,c)$, we have $\g_{1}(\{l_{(b,c)},B_{(b,c)}\})=\g_{1}^{\spiral}(\{l_{(b,c)}\})$ and $\g_{1}(\{r_{(b,c)},B_{(b,c)}\})=\g_{1}^{\spiral}(\{r_{(b,c)}\})$.
	Hence, every maximal g-bundle space is a maximal g-vortex space.
\end{lemma}
Analogous statements hold for unit g-bundle and g-vortex spaces, as well as (unit or nonunit) bundle and vortex spaces.
\begin{proof}
	We prove that
$\g_{1}(\{l_{(b,c)},B_{(b,c)}\})$ and $\g_{1}^{\spiral}(\{l_{(b,c)}\})$ are equal as desired. The analogous $r$ statement is symmetric.
	Theorem~\ref{thm:tLmaxlbundles} shows that the only band compatible with $l_{(b,c)}$ is $B_{(b,c)}$, hence the g-vortex space $\g_{\geq}^{\spiral}(\{l_{(b,c)}\})$ is the set
	\[\{x\g(l_{(b,c)})+y\I(B_{(b,c)})\ :\ x,y\in\mathbb R_{{\geq0}},\ x\leq1\}.\]
	This is the same as the g-bundle space $\g_{1}(\{l_{(b,c)},B_{(b,c)}\})$ by definition.

	We now show the final statement. Every g-simplex is both a maximal g-bundle space and a maximal g-vortex space. The remaining maximal g-bundle spaces are all of the form $\g_{1}(\{l_{(b,c)},B_{(b,c)}\})$ or $\g_{1}(\{r_{(b,c)},B_{(b,c)}\})$ by Theorem~\ref{thm:tLmaxlbundles}, and the above shows that these are also maximal g-vortex spaces. It follows that every maximal g-bundle space is a maximal g-vortex space.
\end{proof}

We have shown that every maximal bundle wall of $\tL$ is a maximal vortex wall.
The g-vortex wall $\D_{{\geq0}}^{\spiral}(\emptyset)$ is the unique cell of the vortex dissection of $\tL$ which is not part of the bundle subdivision of $\tL$. This contains flows like the one in Figure~\ref{fig:irrational} which are not in any bundle space.

\begin{example}
	We finally speak of the g-bundle and g-vortex dissections of the g-polyhedron $\g_1(\L)$, since $\g_1(\L)$ is three-dimensional rather than four-dimensional. Figure~\ref{fig:doubkron-g} shows an image of the g-polyhedron.
	Its vertices are indicator vectors of the elementary bending routes $e_1f_1^{-1}$ and $f_4^{-1}e_4$.
	Its unbounded directions are generated by $\I(e_2f_2^{-1})$ and $\I(e_3f_3^{-1})$.

	We say nothing about the g-simplices of $\tL$, which are modelled by maximal cliques. These are in both the g-bundle and g-vortex dissection of $\g_1(\L)$.
	For any coprime nonnegative integers $(b,c)$, the two-dimensional spaces $\g_1(\{l_{(b,c)},B_{(b,c)}\})=\g_1^{\spiral}(\{l_{(b,c)}\})$
	and
$\g_1(\{r_{(b,c)},B_{(b,c)}\})=\g_1^{\spiral}(\{r_{(b,c)}\})$
appear in both the g-bundle and g-vortex dissections.
	Finally, the two-dimensional vortex wall
	\[\g_1^{\spiral}(\emptyset)=\{(x,y-x,y)\ :\ x,y\in\mathbb R_{\geq0}\}\]
	is in the vortex dissection but not the bundle subdivision; this space is highlighted in Figure~\ref{fig:doubkron-g}.
This is the unique cell of the vortex dissection which is not a cell of the bundle subdivision. Irrational rays in this space are not realized as any g-bundle combination, as in Example~\ref{ex:irrational}, while rational rays along this space are all given by g-vectors of some band $B_{(b,c)}$.

	Note that the intersection of $\g_1^{\spiral}(\emptyset)$ and, say, $\g_1^{\spiral}(\{l_{1,1}\})$ is the ray generated by $\I(l_{(1,1)})$, which is a face of $\g_1^{\spiral}(\{l_{(1,1)}\})$ but not a face of $\g_1^{\spiral}(\emptyset)$. This shows that the vortex dissection does not satisfy the strong intersection property, and is hence not a subdivision.
	\begin{figure}
		\centering
		\def\svgscale{.21}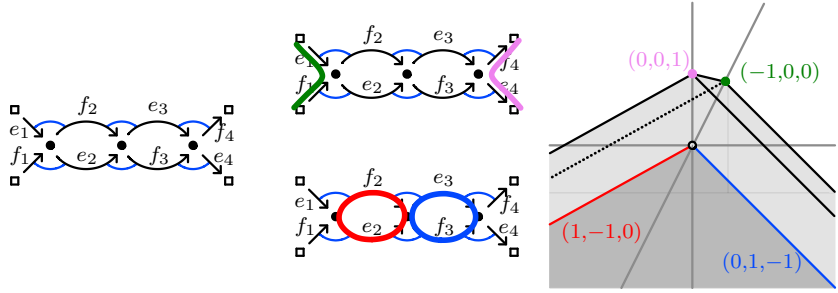
		\caption{A gentle algebra and its g-polyhedron with a g-vortex space highlighted.}
		\label{fig:doubkron-g}
	\end{figure}
\end{example}

\subsection{Complement of the g-vector fan}

We now describe the complement of the g-vector fan of $\tL$.

\begin{cor}\label{cor:cordra}
	The complement of the g-vector fan of $\tL$ is
	\begin{align*}&\{\mathbf{l}_{\geq0}{(b,c)}\ :\ (b,c)\in\mathbb Z_{\geq0}\text{ coprime}\} \\
		\cup
		&\{\mathbf{r}_{\geq0}{(b,c)}\ :\ (b,c)\in\mathbb Z_{\geq0}\text{ coprime}\} \\
	\cup 
	&\{(x,y-x,y)\ :\ (0,0)\neq(x,y)\in\mathbb R_{\geq0}\},\end{align*}
	where for $(b,c)$ coprime the cones $\mathbf{l}_{\geq0}(b,c)$ and $\mathbf{r}_{\geq0}(b,c)$ are defined as
	\begin{align*}
		\mathbf{l}_{\geq0}(b,c)&=
		\begin{cases}
			\{x(b-1,c+1-b,-c-1)+y(b,c-b,-c)\ :\ x\geq0,\ y>0\}&(b,c)\neq(0,1)\\
			\{x(-1,0,0)+y(0,1,-1)\ :\ x\geq0,\ y>0\}&(b,c)=(0,1)
		\end{cases}
		\\
		\mathbf{r}_{\geq0}(b,c)&=
		\begin{cases}
			\{x(b+1,c-b-1,1-c)+y(b,c-b,-c)\ :\ x\geq0,\ y>0\}&(b,c)\neq(1,0)\\
			\{x(0,0,1)+y(1,-1,0)\ :\ x\geq0,\ y>0\}&(b,c)=(1,0)
		\end{cases}
	\end{align*}
\end{cor}
\begin{proof}
	The complement of the g-vector fan of $\tL$ is the set of g-vortex combinations with nonzero canonical g-vortex.
	This includes all nonzero points of the g-vortex space $\g_{\geq0}(\emptyset)$, which are described as
	$\{(x,y-x,y)\ :\ (0,0)\neq(x,y)\in\mathbb R_{\geq0}\}$.
	The spaces $\mathbf{l}_{\geq0}(b,c)$ and $\mathbf{r}_{\geq0}(b,c)$ are the relevant parts of $\g_{\geq0}^{\spiral}(\{l_{(b,c)}\})$ and $\g_{\geq0}^{\spiral}(\{r_{(b,c)}\})$, respectively. This is a complete list of the g-vortex combinations with nonzero canonical g-vortex by Corollary~\ref{cor:cordra}.
\end{proof}

\subsection{Triple Kronecker quiver}

In the double Kronecker fringed quiver studied in this section, we saw that every maximal (g-)bundle space appeared as a maximal (g-)vortex space. This is not always the case. Let $\tL$ be the triple Kronecker quiver of Figure~\ref{fig:tripkron-g}. Take the maximal reduced bundle $\bK:=\{e_2f_2^{-1},e_4f_4^{-1}\}$. Its bundle wall is the two-dimensional cone
generated by $\I(e_2f_2^{-1})=(1,-1,0,0)$ and $\I(e_4f_4^{-1})=(0,0,1,-1)$.
It is immediate that this two-dimensional cone is contained in the maximal vortex space $\g_1^{\spiral}(\emptyset)$. On the other hand, $\g_1^{\spiral}(\emptyset)$ also contains the ray $\I(e_3f_3^{-1})=(0,1,-1,0)$, hence $\g_1(\bK)\subsetneq\g_1^{\spiral}(\emptyset)$.

\begin{figure}
	\centering
	\def\svgscale{.21}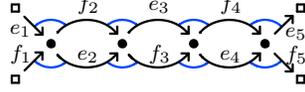
	\caption{The triple Kronecker fringed quiver.}
	\label{fig:tripkron-g}
\end{figure}

\appendix
\section{The Flow Algorithm on Framed Directed Graphs}
\label{sec:appendix}

Our main technical method of proving results about clique triangulations, bundle subdivisions, and vortex dissections of turbulence polyhedra of fringed quivers has been the constructive \emph{flow algorithm} to calculate the description of a flow as a bundle combination.
In fact, the specialization of this flow algorithm to amply framed DAGs is new; the original proof of Danilov, Karzanov, and Koshevoy~\cite{DKK} that a framing on a DAG gives a regular unimodular triangulation of its flow polytope is an inductive proof on the length of the DAG rather than a constructive proof. We believe that this algorithm is a valuable technical tool even in the setting of amply framed DAGs, so we phrase it for amply framed DAGs here.
We remark that the algorithm works verbatim for amply framed (not necessarily acyclic) directed graphs, and should work with slight modification for general (not necessarily amply) framed directed graphs.

\begin{defn}
	Let $F$ be a flow on an amply framed directed graph $(G,\mathfrak{F})$. We say that an \emph{arrow-flow} of $F$ with respect to $(G,\mathfrak{F})$ is a pair $(\alpha,C)$, where $\alpha$ is an edge of $G$ and $C\in[0,F(\alpha)]$.
\end{defn}

We remark that our arrow-flows on gentle algebras included choosing a direction on the arrows; the consistent directions of our framed directed graphs means that we need not do this.
We now translate Definition~\ref{defn:forbac} to the language of amply framed DAGs.
\begin{defn}
	Let $\alpha$ be an arrow of $G$ such that $h(\alpha)$ is an internal vertex.
	\begin{itemize}
		\item Suppose $\psi(\alpha)=1$. Let $\alpha_2$ be the edge of $G$ such that $h(\alpha_2)=h(\alpha)$ and $\psi(\alpha_2)=2$. Let $\beta_1$ and $\beta_2$ be the two edges starting at $h(\alpha)=h(\alpha_2)$ such that $\psi(\beta_1)=1$ and $\psi(\beta_2)=2$. Then for $C\in[0,F(\alpha)$, define
			\[\for(\alpha,C)=
			\begin{cases}
				(\beta_1,C) & C\leq F(\beta_1)\\
				(\beta_2,C-F(\beta_1)) & C>F(\beta_1).
			\end{cases}\]
		\item Suppose $\psi(\alpha)=2$. Let $\alpha_1$ be the edge of $G$
	 such that $h(\alpha_1)=h(\alpha)$ and $\psi(\alpha_1)=1$. Let $\beta_1$ and $\beta_2$ be the two edges starting at $h(\alpha)=h(\alpha_1)$ such that $\psi(\beta_1)=1$ and $\psi(\beta_2)=2$. Then for $C\in[0,F(\alpha)$, define
			\[\for(\alpha,C)=
			\begin{cases}
				(\beta_1,C) & C+F(\alpha_1)<F(\beta_1)\\
				(\beta_2,C) & C+F(\alpha_1)\geq F(\beta_1).
			\end{cases}\]
	\end{itemize}
	If $(\alpha,C)$ is an edge of $G$ such that $t(\alpha)$ is an internal vertex, then let $\bac(\alpha,C)$ be the unique arrow-flow $(\beta,D)$ such that $\for(\beta,D)=(\alpha,C)$.
\end{defn}

We now phrase the flow algorithm in the language of amply framed DAGs.
\begin{defn}\label{defn:forbacdag}
	Suppose $(G,\mathfrak{F})$ is an amply framed DAG and $F$ is a flow of $G$.
	Let $(\alpha,C)$ be an arrow-flow of $F$.
	Let $m_1$ be maximal among values $m\geq0$ such that $\for^m(\alpha,C)$ is defined. For $j\in[m_1]$, define $\alpha_j$ so that $\for^m(\alpha,C)=(\alpha_j,C_j)$ for some $C_j$. Similarly, let $m_2$ be maximal among values $m\geq0$ such that $\bac^m(\alpha,C)$ is defined and for $j\in[m_2]$ define $\alpha_{-j}$ so that $\for^m(\alpha,C)=(\alpha_{-j},C_{-j})$ for some $C_j$.
	Define $p_{(\alpha,C)}^F$ to be the path $\alpha_{-m_2}\dots\alpha_{-1}\alpha\alpha_1\dots\alpha_{m_1}$ marked at $\alpha$.
	Define $a_{(\alpha,C)}^F$ to be the interval of those values $D\in[0,F(\alpha)]$ such that $p_{(\alpha,D)}^F=p_{(\alpha,C)}^F$ as marked paths.
\end{defn}

\begin{defn}
	Let $(G,\mathfrak{F})$ be an amply framed DAG and let $F$ be a flow of $G$. Let $\K_F$ be the set of routes which appear as $p_{(\alpha,C)}^F$ for some arrow-flow $(\alpha,C)$. For any $p\in\K_F$ such that $p=p_{(\alpha,C)}^F$, write $a_p^F:=a_{(\alpha,C)}^F$.
	Let $\K_F^+\subseteq\K_F$ be the set of those routes $p\in\K_F$ such that $a_p^F>0$.
\end{defn}

Note that to calculate $\K_F$, it suffices to check only the arrow-flows $(\alpha,C)$ where $\alpha$ is a source edge, since every route starts with a source edge.

\begin{thm}[{Corollary~\ref{cor:bundle-subd}}]\label{thm:cliquy}
	Let $(G,\mathfrak{F})$ be an amply framed DAG and let $F$ be a flow of $G$. Then
	\[F=\sum_{p\in\K_F^+}a_p^F\I(p)\]
	is the unique expression of the flow $F$ as a positive clique combination.
\end{thm}

Note that the methods of this paper do not show regularity of the triangulations, while this is shown in~\cite{DKK}. It would be interesting to find a combinatorial interpretation for some height functions on vertices realizing these triangulations as regular triangulations.

\begin{example}
	Consider the amply framed DAG of Figure~\ref{fig:cube} with flow
	\[F(e_1)=1\ \ F(e_2)=3\ \ F(f_1)=3\ \ F(f_2)=2.\]
	We wish to express $F$ as a clique combination.
	We have $p_{(e_1,C)}^F=e_1f_1$ for all $C\in[0,1]$, so the route $e_1f_1$ occurs in this clique combination with coefficient 1. We have $p_{(e_2,C)}^F=
	\begin{cases}
		e_2f_1 & C\in[0,2)\\
		e_2f_2 & C\in[2,3],
	\end{cases}$ so $e_2f_1$ and $e_2f_2$ occur in this clique combination with coefficients 2 and 1, respectively. This tells us as desired that $F$ is the clique combination
	\[F=\I(e_1f_1)+2\I(e_2f_1)+\I(e_2f_2).\]
\end{example}

\begin{remk}
	Recall from Section~\ref{ssec:PGAaAFDG} that we may define amply framed directed graphs without requiring acyclicity. Definition~\ref{defn:forbacdag} is then the specialization of the flow algorithm from gentle algebras to amply framed directed graphs. Hence, it may be used to prove the specialization of our gentle algebra results to the setting of amply framed directed graphs.
\end{remk}

\bibliographystyle{plain}
\bibliography{biblio} 

@article{UQAM,
      title={Flows on graphs with cycles, locally gentle algebras, and the Mutoperhedron}, 
      author={Antoine Abram and Jose Bastidas and Benjamin Dequêne and Alejandro H. Morales and GaYee Park and Hugh Thomas},
      year={2026},
      journal={arXiv preprint},
      publisher={arXiv}
}

@article {BMR,
    AUTHOR = {Buan, Aslak Bakke and Marsh, Robert J. and Reiten, Idun},
     TITLE = {Cluster-tilted algebras},
   JOURNAL = {Trans. Amer. Math. Soc.},
  FJOURNAL = {Transactions of the American Mathematical Society},
    VOLUME = {359},
      YEAR = {2007},
    NUMBER = {1},
     PAGES = {323--332},
      ISSN = {0002-9947,1088-6850},
   MRCLASS = {16G20},
  MRNUMBER = {2247893},
MRREVIEWER = {Lidia\ Angeleri H\"ugel},
       DOI = {10.1090/S0002-9947-06-03879-7},
       URL = {https://doi.org/10.1090/S0002-9947-06-03879-7},
}

@article {HR,
    AUTHOR = {Happel, Dieter and Ringel, Claus Michael},
     TITLE = {Tilted algebras},
   JOURNAL = {Trans. Amer. Math. Soc.},
  FJOURNAL = {Transactions of the American Mathematical Society},
    VOLUME = {274},
      YEAR = {1982},
    NUMBER = {2},
     PAGES = {399--443},
      ISSN = {0002-9947,1088-6850},
   MRCLASS = {16A46 (16A64 18E40)},
  MRNUMBER = {675063},
MRREVIEWER = {Sheila\ Brenner},
       DOI = {10.2307/1999116},
       URL = {https://doi.org/10.2307/1999116},
}

@incollection {BB,
    AUTHOR = {Brenner, Sheila and Butler, M. C. R.},
     TITLE = {Generalizations of the {B}ernstein-{G}elfand-{P}onomarev reflection functors},
 BOOKTITLE = {Representation theory, {II} ({P}roc. {S}econd {I}nternat.
              {C}onf., {C}arleton {U}niv., {O}ttawa, {O}nt., 1979)},
    SERIES = {Lecture Notes in Math.},
    VOLUME = {832},
     PAGES = {103--169},
 PUBLISHER = {Springer, Berlin},
      YEAR = {1980},
      ISBN = {3-540-10264-7},
   MRCLASS = {16A64 (16A46)},
  MRNUMBER = {607151},
MRREVIEWER = {Idun\ Reiten},
}

@article {AY,
    AUTHOR = {Aoki, Toshitaka and Yurikusa, Toshiya},
     TITLE = {Complete gentle and special biserial algebras are {$g$}-tame},
   JOURNAL = {J. Algebraic Combin.},
  FJOURNAL = {Journal of Algebraic Combinatorics. An International Journal},
    VOLUME = {57},
      YEAR = {2023},
    NUMBER = {4},
     PAGES = {1103--1137},
      ISSN = {0925-9899,1572-9192},
   MRCLASS = {16G60 (05E16 16G20)},
  MRNUMBER = {4588126},
MRREVIEWER = {Simion\ Sorin P. Breaz},
       DOI = {10.1007/s10801-023-01216-8},
       URL = {https://doi.org/10.1007/s10801-023-01216-8},
}

@article {BST,
    AUTHOR = {Br\"ustle, Thomas and Smith, David and Treffinger, Hipolito},
     TITLE = {Wall and chamber structure for finite-dimensional algebras},
   JOURNAL = {Adv. Math.},
  FJOURNAL = {Advances in Mathematics},
    VOLUME = {354},
      YEAR = {2019},
     PAGES = {106746, 31},
      ISSN = {0001-8708,1090-2082},
   MRCLASS = {16G10 (13F60)},
  MRNUMBER = {3989130},
MRREVIEWER = {Xueqing\ Chen},
       DOI = {10.1016/j.aim.2019.106746},
       URL = {https://doi.org/10.1016/j.aim.2019.106746},
}

@article{BRIDGE,
      title={Spaces of stability conditions}, 
      author={Tom Bridgeland},
      year={2006},
      publisher={arXiv},
      journal={arXiv preprint}
}

@article {BHIT,
    AUTHOR = {Br\"ustle, Thomas and Hermes, Stephen and Igusa, Kiyoshi and
              Todorov, Gordana},
     TITLE = {Semi-invariant pictures and two conjectures on maximal green
              sequences},
   JOURNAL = {J. Algebra},
  FJOURNAL = {Journal of Algebra},
    VOLUME = {473},
      YEAR = {2017},
     PAGES = {80--109},
      ISSN = {0021-8693,1090-266X},
   MRCLASS = {13F60 (16G20)},
  MRNUMBER = {3591142},
MRREVIEWER = {Christof\ Gei\ss},
       DOI = {10.1016/j.jalgebra.2016.10.025},
       URL = {https://doi.org/10.1016/j.jalgebra.2016.10.025},
}

@article {BSH,
    AUTHOR = {Br\"ustle, Thomas and Smith, David and Treffinger, Hipolito},
     TITLE = {Stability conditions and maximal {G}reen sequences in
              {A}belian categories},
   JOURNAL = {Rev. Un. Mat. Argentina},
  FJOURNAL = {Revista de la Uni\'on Matem\'atica Argentina},
    VOLUME = {63},
      YEAR = {2022},
    NUMBER = {1},
     PAGES = {203--221},
      ISSN = {0041-6932,1669-9637},
   MRCLASS = {16P20 (13D30 18E10)},
  MRNUMBER = {4451134},
MRREVIEWER = {Septimiu\ Crivei},
       DOI = {10.33044/revuma.1110},
       URL = {https://doi.org/10.33044/revuma.1110},
}

@article {ITW,
    AUTHOR = {Igusa, Kiyoshi and Orr, Kent and Todorov, Gordana and Weyman,
              Jerzy},
     TITLE = {Cluster complexes via semi-invariants},
   JOURNAL = {Compos. Math.},
  FJOURNAL = {Compositio Mathematica},
    VOLUME = {145},
      YEAR = {2009},
    NUMBER = {4},
     PAGES = {1001--1034},
      ISSN = {0010-437X,1570-5846},
   MRCLASS = {16G20 (13A50 14L24 16G70)},
  MRNUMBER = {2521252},
MRREVIEWER = {Gregoire\ Dupont},
       DOI = {10.1112/S0010437X09004151},
       URL = {https://doi.org/10.1112/S0010437X09004151},
}

@article {BMRRT,
    AUTHOR = {Buan, Aslak Bakke and Marsh, Robert and Reineke, Markus and
              Reiten, Idun and Todorov, Gordana},
     TITLE = {Tilting theory and cluster combinatorics},
   JOURNAL = {Adv. Math.},
  FJOURNAL = {Advances in Mathematics},
    VOLUME = {204},
      YEAR = {2006},
    NUMBER = {2},
     PAGES = {572--618},
      ISSN = {0001-8708,1090-2082},
   MRCLASS = {16G20 (16G70 16S99)},
  MRNUMBER = {2249625},
MRREVIEWER = {Lidia\ Angeleri H\"ugel},
       DOI = {10.1016/j.aim.2005.06.003},
       URL = {https://doi.org/10.1016/j.aim.2005.06.003},
}

@article {READING,
    AUTHOR = {Reading, Nathan},
     TITLE = {Scattering fans},
   JOURNAL = {Int. Math. Res. Not. IMRN},
  FJOURNAL = {International Mathematics Research Notices. IMRN},
      YEAR = {2020},
    NUMBER = {23},
     PAGES = {9640--9673},
      ISSN = {1073-7928,1687-0247},
   MRCLASS = {13F60 (05E14 14J33)},
  MRNUMBER = {4182806},
MRREVIEWER = {Fan\ Qin},
       DOI = {10.1093/imrn/rny260},
       URL = {https://doi.org/10.1093/imrn/rny260},
}

@article{Asa22,
	author={Asai, Sota},
	title={Non-rigid regions of real Grothendieck groups of gentle and special biserial algebras},
	year={2022},
	publisher={arXiv},
	journal={arXiv preprint}
}

@article {AI24,
    AUTHOR = {Asai, Sota and Iyama, Osamu},
     TITLE = {Semistable torsion classes and canonical decompositions in
              {G}rothendieck groups},
   JOURNAL = {Proc. Lond. Math. Soc. (3)},
  FJOURNAL = {Proceedings of the London Mathematical Society. Third Series},
    VOLUME = {129},
      YEAR = {2024},
    NUMBER = {5},
     PAGES = {Paper No. e12639, 58},
      ISSN = {0024-6115,1460-244X},
   MRCLASS = {16G10 (16E20 16G20 16S90)},
  MRNUMBER = {4814844},
       DOI = {10.1112/plms.12639},
       URL = {https://doi.org/10.1112/plms.12639},
}

@article {WIWT,
    AUTHOR = {Bell, Matias von and Braun, Benjamin and Bruegge, Kaitlin and
              Hanely, Derek and Peterson, Zachery and Serhiyenko, Khrystyna
              and Yip, Martha},
     TITLE = {Triangulations of flow polytopes, ample framings, and gentle
              algebras},
   JOURNAL = {Selecta Math. (N.S.)},
  FJOURNAL = {Selecta Mathematica. New Series},
    VOLUME = {30},
      YEAR = {2024},
    NUMBER = {3},
     PAGES = {Paper No. 55, 34},
   MRCLASS = {52B20 (05C20 05C21 05E45 16G20 52B05)},
  MRNUMBER = {4754385},
       DOI = {10.1007/s00029-024-00942-6},
       URL = {https://doi.org/10.1007/s00029-024-00942-6},
}

@article {PPP,
    AUTHOR = {Palu, Yann and Pilaud, Vincent and Plamondon, Pierre-Guy},
     TITLE = {Non-kissing complexes and tau-tilting for gentle algebras},
   JOURNAL = {Mem. Amer. Math. Soc.},
  FJOURNAL = {Memoirs of the American Mathematical Society},
    VOLUME = {274},
      YEAR = {2021},
    NUMBER = {1343},
     PAGES = {vii+110},
   MRCLASS = {16G10 (05E10 05E45 06B10 16G20 52B05)},
  MRNUMBER = {4346482},
MRREVIEWER = {Jos\'{e}\ A.\ V\'{e}lez-Marulanda},
       DOI = {10.1090/memo/1343},
}

@incollection {DKK,
    AUTHOR = {Danilov, Vladimir I. and Karzanov, Alexander V. and Koshevoy,
              Gleb A.},
     TITLE = {Coherent fans in the space of flows in framed graphs},
 BOOKTITLE = {24th {I}nternational {C}onference on {F}ormal {P}ower {S}eries
              and {A}lgebraic {C}ombinatorics ({FPSAC} 2012)},
    SERIES = {Discrete Math. Theor. Comput. Sci. Proc.},
    VOLUME = {AR},
     PAGES = {481--490},
 PUBLISHER = {Assoc. Discrete Math. Theor. Comput. Sci., Nancy},
      YEAR = {2012},
   MRCLASS = {05C21 (05E18)},
  MRNUMBER = {2958022},
}

@article {BDMTY,
    AUTHOR = {Br\"{u}stle, Thomas and Douville, Guillaume and Mousavand,
              Kaveh and Thomas, Hugh and Y\i ld\i r\i m, Emine},
     TITLE = {On the combinatorics of gentle algebras},
   JOURNAL = {Canad. J. Math.},
  FJOURNAL = {Canadian Journal of Mathematics. Journal Canadien de
              Math\'{e}matiques},
    VOLUME = {72},
      YEAR = {2020},
    NUMBER = {6},
     PAGES = {1551--1580},
   MRCLASS = {16G20},
  MRNUMBER = {4176701},
MRREVIEWER = {Piotr\ Malicki},
       DOI = {10.4153/s0008414x19000397},
}

@article {Y2,
    AUTHOR = {Gonz\'alez D'Le\'on, Rafael S. and Morales, Alejandro H. and
              Philippe, Eva and Tamayo Jim\'enez, Daniel and Yip, Martha},
     TITLE = {Realizing the {$s$}-permutahedron via flow polytopes},
   JOURNAL = {S\'em. Lothar. Combin.},
  FJOURNAL = {S\'eminaire Lotharingien de Combinatoire},
    VOLUME = {91B},
      YEAR = {2024},
     PAGES = {Art. 60, 12},
      ISSN = {1286-4889},
   MRCLASS = {52B05 (05C21)},
  MRNUMBER = {4818692},
}

@article{AHIKM,
      title={Fans and polytopes in tilting theory {I}: {F}oundations}, 
      author={Toshitaka Aoki and Akihiro Higashitani and Osamu Iyama and Ryoichi Kase and Yuya Mizuno},
      year={2024},
      journal={arXiv preprint},
      publisher={arXiv}
}

@article {LMD,
    AUTHOR = {Liu, Ricky I. and M\'esz\'aros, Karola and St. Dizier, Avery},
     TITLE = {Gelfand-{T}setlin polytopes: a story of flow and order
              polytopes},
   JOURNAL = {SIAM J. Discrete Math.},
  FJOURNAL = {SIAM Journal on Discrete Mathematics},
    VOLUME = {33},
      YEAR = {2019},
    NUMBER = {4},
     PAGES = {2394--2415},
      ISSN = {0895-4801,1095-7146},
   MRCLASS = {05E10},
  MRNUMBER = {4039518},
MRREVIEWER = {Robert\ Davis},
       DOI = {10.1137/19M1251242},
}

@article {MD,
    AUTHOR = {M\'esz\'aros, Karola and St. Dizier, Avery},
     TITLE = {From generalized permutahedra to {G}rothendieck polynomials
              via flow polytopes},
   JOURNAL = {Algebr. Comb.},
  FJOURNAL = {Algebraic Combinatorics},
    VOLUME = {3},
      YEAR = {2020},
    NUMBER = {5},
     PAGES = {1197--1229},
      ISSN = {2589-5486},
   MRCLASS = {05E10 (05C21 52B20)},
  MRNUMBER = {4166815},
MRREVIEWER = {Andrew\ Berget},
       DOI = {10.5802/alco.136},
}

@article {WV,
    AUTHOR = {Baldoni, Welleda and Vergne, Mich\`ele},
     TITLE = {Kostant partitions functions and flow polytopes},
   JOURNAL = {Transform. Groups},
  FJOURNAL = {Transformation Groups},
    VOLUME = {13},
      YEAR = {2008},
    NUMBER = {3-4},
     PAGES = {447--469},
      ISSN = {1083-4362,1531-586X},
   MRCLASS = {52B12 (05A17 11D45 52C35)},
  MRNUMBER = {2452600},
MRREVIEWER = {Jes\'us\ A.\ De Loera},
       DOI = {10.1007/s00031-008-9019-8},
}

@article {EM,
    AUTHOR = {Escobar, Laura and M\'esz\'aros, Karola},
     TITLE = {Toric matrix {S}chubert varieties and their polytopes},
   JOURNAL = {Proc. Amer. Math. Soc.},
  FJOURNAL = {Proceedings of the American Mathematical Society},
    VOLUME = {144},
      YEAR = {2016},
    NUMBER = {12},
     PAGES = {5081--5096},
      ISSN = {0002-9939,1088-6826},
   MRCLASS = {14M25 (05E45 14M15 52B11)},
  MRNUMBER = {3556254},
MRREVIEWER = {Thomas\ Kahle},
       DOI = {10.1090/proc/13152},
}

@article {AH81,
    AUTHOR = {Assem, Ibrahim and Happel, Dieter},
     TITLE = {Generalized tilted algebras of type {$A\sb{n}$}},
   JOURNAL = {Comm. Algebra},
  FJOURNAL = {Communications in Algebra},
    VOLUME = {9},
      YEAR = {1981},
    NUMBER = {20},
     PAGES = {2101--2125},
      ISSN = {0092-7872,1532-4125},
   MRCLASS = {16A46 (16A64)},
  MRNUMBER = {640613},
MRREVIEWER = {Christine\ Riedtmann},
       DOI = {10.1080/00927878108822697},
       URL = {https://doi.org/10.1080/00927878108822697},
}

@article{AS87,
	author = {Assem, Ibrahim and Skowronski, Andrzej},
	journal = {Mathematische Zeitschrift},
	pages = {269-290},
	title = {Iterated Tilted Algebras of Type {$\tilde A_n$}.},
	url = {http://eudml.org/doc/183692},
	volume = {195},
	year = {1987},
}

@article {BR87,
    AUTHOR = {Butler, M. C. R. and Ringel, Claus Michael},
     TITLE = {Auslander-{R}eiten sequences with few middle terms and
              applications to string algebras},
   JOURNAL = {Comm. Algebra},
  FJOURNAL = {Communications in Algebra},
    VOLUME = {15},
      YEAR = {1987},
    NUMBER = {1-2},
     PAGES = {145--179},
      ISSN = {0092-7872,1532-4125},
   MRCLASS = {16A64 (16A35)},
  MRNUMBER = {876976},
MRREVIEWER = {Christine\ Riedtmann},
       DOI = {10.1080/00927878708823416},
       URL = {https://doi.org/10.1080/00927878708823416},
}

@article {GHKK,
    AUTHOR = {Gross, Mark and Hacking, Paul and Keel, Sean and Kontsevich,
              Maxim},
     TITLE = {Canonical bases for cluster algebras},
   JOURNAL = {J. Amer. Math. Soc.},
  FJOURNAL = {Journal of the American Mathematical Society},
    VOLUME = {31},
      YEAR = {2018},
    NUMBER = {2},
     PAGES = {497--608},
      ISSN = {0894-0347,1088-6834},
   MRCLASS = {13F60 (14J33)},
  MRNUMBER = {3758151},
MRREVIEWER = {Ralf\ Schiffler},
       DOI = {10.1090/jams/890},
       URL = {https://doi.org/10.1090/jams/890},
}

@article{Plamondon,
   title={tau-tilting finite gentle algebras are
representation-finite},
   volume={302},
   ISSN={0030-8730},
   url={http://dx.doi.org/10.2140/pjm.2019.302.709},
   DOI={10.2140/pjm.2019.302.709},
   number={2},
   journal={Pacific Journal of Mathematics},
   publisher={Mathematical Sciences Publishers},
   author={Plamondon, Pierre-Guy},
   year={2019},
   month=nov, pages={709–716}
}

@article {KPY,
    AUTHOR = {Plamondon, Pierre-Guy and Yurikusa, Toshiya and Keller,
              Bernhard},
     TITLE = {Tame algebras have dense g-vector fans},
   JOURNAL = {Int. Math. Res. Not. IMRN},
  FJOURNAL = {International Mathematics Research Notices. IMRN},
      YEAR = {2023},
    NUMBER = {4},
     PAGES = {2701--2747},
      ISSN = {1073-7928,1687-0247},
   MRCLASS = {16G60 (13F60 16G20)},
  MRNUMBER = {4565625},
MRREVIEWER = {Emine\ Y\i ld\i r\i m},
       DOI = {10.1093/imrn/rnab105},
       URL = {https://doi.org/10.1093/imrn/rnab105},
}

@article {DIRRT,
    AUTHOR = {Demonet, Laurent and Iyama, Osamu and Reading, Nathan and
              Reiten, Idun and Thomas, Hugh},
     TITLE = {Lattice theory of torsion classes: beyond {$\tau$}-tilting
              theory},
   JOURNAL = {Trans. Amer. Math. Soc. Ser. B},
  FJOURNAL = {Transactions of the American Mathematical Society. Series B},
    VOLUME = {10},
      YEAR = {2023},
     PAGES = {542--612},
      ISSN = {2330-0000},
   MRCLASS = {16G10 (06A07 20F55)},
  MRNUMBER = {4579952},
MRREVIEWER = {Liangang\ Peng},
       DOI = {10.1090/btran/100},
       URL = {https://doi.org/10.1090/btran/100},
}

@article{AIR,
  author = {Adachi, Takahide and Iyama, Osamu and Reiten, Idun},
  title = {$\tau$-tilting theory},
  year = {2012},
      journal={arXiv preprint},
      publisher={arXiv}
}

\end{document}